\documentclass[200pt]{amsart}
\usepackage[top=30mm, bottom=30mm, left=30mm, right=30mm]{geometry}

\usepackage{ifthen, url, amsmath, amsthm, amssymb,url, amssymb, relsize,placeins, tikz, pgffor, multicol,mathrsfs, abstract, graphicx, fp, hyperref, rotating}

\usepackage{sidenotes} 
\usepackage{lipsum} 
\usepackage{mwe} 
\usepackage[whole]{bxcjkjatype}
\usepackage{newunicodechar}
\usepackage{pdfpages}

\usetikzlibrary{positioning,shapes,shadows}
\usetikzlibrary{arrows,automata}
  \newtheoremstyle{break}
  {\topsep}
  {\topsep}
  {\upshape}
  {}
  {\bfseries}
  {.}
  {\newline}
  {}
\theoremstyle{break}
\newtheorem{theore}{Theorem}[section]

\newtheorem{defn}[theore]{Definition}
\newtheorem{remark}[theore]{Remark}
\newtheorem{example}[theore]{Example}

  \newtheoremstyle{breaktheorem}
  {\topsep}
  {\topsep}
  {\itshape}
  {}
  {\bfseries}
  {.}
  {\newline}
  {}
\theoremstyle{breaktheorem}
\newtheorem{theorem}[theore]{Theorem}
\newtheorem{lemma}[theore]{Lemma}
\newtheorem{proposition}[theore]{Proposition}
\newtheorem{corollary}[theore]{Corollary}

\let\oldmarginpar\marginpar
\renewcommand\marginpar[1]{\-\oldmarginpar[\raggedleft\footnotesize #1]%
{\raggedright\footnotesize #1}}

\newcommand{\speeddictzero}[2]{
\adddict{#1 ref}{#2}
\adddict{\dict{#1 ref} count}{0}}

\newcommand{\speeddictone}[3]{
\adddict{#1 ref}{#2}
\adddict{\dict{#1 ref} count}{1}
\adddict{\dict{#1 ref} ass 1}{#3}}

\newcommand{\speeddicttwo}[4]{
\adddict{#1 ref}{#2}
\adddict{\dict{#1 ref} count}{2}
\adddict{\dict{#1 ref} ass 1}{#3}
\adddict{\dict{#1 ref} ass 2}{#4}}

\newcommand{\speeddictthree}[5]{
\adddict{#1 ref}{#2}
\adddict{\dict{#1 ref} count}{3}
\adddict{\dict{#1 ref} ass 1}{#3}
\adddict{\dict{#1 ref} ass 2}{#4}
\adddict{\dict{#1 ref} ass 3}{#5}}

\newcommand{\speeddictfour}[6]{
\adddict{#1 ref}{#2}
\adddict{\dict{#1 ref} count}{4}
\adddict{\dict{#1 ref} ass 1}{#3}
\adddict{\dict{#1 ref} ass 2}{#4}
\adddict{\dict{#1 ref} ass 3}{#5}
\adddict{\dict{#1 ref} ass 4}{#6}}

\newcommand{\speeddictfive}[7]{
\adddict{#1 ref}{#2}
\adddict{\dict{#1 ref} count}{5}
\adddict{\dict{#1 ref} ass 1}{#3}
\adddict{\dict{#1 ref} ass 2}{#4}
\adddict{\dict{#1 ref} ass 3}{#5}
\adddict{\dict{#1 ref} ass 4}{#6}
\adddict{\dict{#1 ref} ass 5}{#7}}

\newcommand{\speeddictsix}[8]{
\adddict{#1 ref}{#2}
\adddict{\dict{#1 ref} count}{6}
\adddict{\dict{#1 ref} ass 1}{#3}
\adddict{\dict{#1 ref} ass 2}{#4}
\adddict{\dict{#1 ref} ass 3}{#5}
\adddict{\dict{#1 ref} ass 4}{#6}
\adddict{\dict{#1 ref} ass 5}{#7}
\adddict{\dict{#1 ref} ass 6}{#8}}

\newcommand{\speeddictseven}[9]{
\adddict{#1 ref}{#2}
\adddict{\dict{#1 ref} count}{7}
\adddict{\dict{#1 ref} ass 1}{#3}
\adddict{\dict{#1 ref} ass 2}{#4}
\adddict{\dict{#1 ref} ass 3}{#5}
\adddict{\dict{#1 ref} ass 4}{#6}
\adddict{\dict{#1 ref} ass 5}{#7}
\adddict{\dict{#1 ref} ass 6}{#8}
\adddict{\dict{#1 ref} ass 7}{#9}}

\newcommand{\Assumed}[1]{
\ifthenelse{\equal{\dict{\dict{#1 ref} count}}{1}}
{Assumed Knowledge: \ref{\dict{\dict{#1 ref} ass 1}}}
{Assumed Knowledge: \ref{\dict{\dict{#1 ref} ass 1}}\foreach \n in {2,...,\dict{\dict{#1 ref} count}}{, \ref{\dict{\dict{#1 ref} ass \n}}}}}

\newcommand{\SN}{\operatorname{Sym}(\mathbb{N})}
\newcommand{\normalsubgroup}{\unlhd}

\DeclareMathOperator{\End}{End}
\DeclareMathOperator{\A}{Aut}
\DeclareMathOperator{\Out}{Out}
\DeclareMathOperator{\supt}{supt}
\DeclareMathOperator{\fix}{fix}
\newcommand{\union}[1]{\bigcup_{\substack{#1}}}
\newcommand{\intersection}[1]{\bigcap_{\substack{#1}}}

\newcommand{\nbhd}[2]{\operatorname{Nbhds}_{#1} (#2)}
\newcommand{\QR}{\operatorname{QR}}
\newcommand{\PS}{\operatorname{PS}}
\newcommand{\Clo}{\operatorname{Clo}}
\newcommand{\CloP}{\operatorname{CloP}}

\newcommand{\Sym}{\operatorname{Sym}}
\newcommand{\Aut}{\operatorname{Aut}}
\newcommand{\Z}{\mathbb{Z}}
\newcommand{\N}{\mathbb{N}}
\newcommand{\F}{Fr\"aiss\'e }
\newcommand{\im}{\operatorname{img}}
\newcommand{\dom}{\operatorname{dom}}

\newcommand{\inj}{\operatorname{Inj}}
\newcommand{\Hom}{\operatorname{Hom}}

\newcommand{\conj}{\operatorname{conj}}
\newcommand{\ar}{\operatorname{ar}}

\newcommand{\de}{\Gamma}

\newcommand{\PFL}[1]{\operatorname{Flim}_{\leftarrow}(#1)}

\newcommand{\C}{\operatorname{Core}}

\newcommand{\makeset}[2]{\left\lbrace #1 \;\middle|\;
  \begin{tabular}{@{}l@{}}
    #2
   \end{tabular}
  \right\rbrace}

\title{On constructing topology from algebra}

\def\adddict#1#2{\expandafter\def\csname MY@#1\endcsname{#2}}

\def\dict#1{\csname MY@#1\endcsname}

\begin{document}
\includepdf[pages=-]{Thesis_title_page_final.pdf}
\maketitle
\begin{abstract}

In this thesis we explore natural procedures through which topological structure can be constructed from specific semigroups. 
We will do this in two ways: 1) we equip the semigroup object itself with a topological structure, and 2) we find a topological space for the semigroup to act on continuously.

We discuss various minimum/maximum topologies which one can define on an arbitrary semigroup (given some topological restrictions).
We give explicit descriptions of each these topologies for the monoids of binary relations, partial transformations, transformations, and partial bijections on the set \(\N\).
Using similar methods we determine whether or not each of these semigroups admits a unique Polish semigroup topology. 
We also do this for the following semigroups: the monoid of all injective functions on \(\N\), the monoid of continuous transformations of the Hilbert cube \([0, 1]^\N\), the monoid of continuous transformations of the Cantor space \(2^\N\), and the monoid of endomorphisms of the countably infinite atomless boolean algebra.
With the exception of the continuous transformation monoid of the Hilbert cube, we also show that all of the above semigroups admit a second countable semigroup topology such that every semigroup homomorphism from the semigroup to a second countable topological semigroup is continuous.

In a recent paper, Bleak, Cameron, Maissel, Navas, and Olukoya use a theorem of Rubin to describe the automorphism groups of the Higman-Thompson groups \(G_{n, r}\) via their canonical Rubin action on the Cantor space. 
In particular they embed these automorphism groups into the rational group \(\mathcal{R}\) of transducers introduced by Grigorchuk,  Nekrashevich, and Sushchanskii.
We generalise these transducers to be more suitable to higher dimensional Cantor spaces and give a similar description of the automorphism groups of the Brin-Thompson groups \(dV_n\) (although we do not give an embedding into \(\mathcal{R}\)).
Using our description, we show that the outer automorphism group \(\Out(dV_2)\) of \(dV_2\) is isomorphic to the wreath product of \(\Out(1V_2)\) with the symmetric group on \(d\) points.
\end{abstract}
\newpage
\renewcommand{\abstractname}{General acknowledgements}
\begin{abstract} I would like to thank my PhD supervisors, Collin Bleak and James Mitchell, both of which have been a great source of support (both mathematically and otherwise) for as long as I have known them. 何時も楽しませてくれる大空スバルさんには感謝しております。
\end{abstract}
\renewcommand{\abstractname}{Funding}
\begin{abstract}
This work was supported by the University of St Andrews (School of Mathematics).
\end{abstract}
\newpage
\tableofcontents
\part{Introduction}

This thesis is concerned with methods for choosing an appropriate or ``best" topological structure in a given context.
We are particularly interested in doing this when we are provided with a nice algebraic object to work with.
This usually involves finding a topological structure which is in some way ``compatible" with our algebraic object and investigating its uniqueness given some ``reasonable" topological assumptions.

We have three main parts. 
In Part~\ref{as section}, we establish the definitions and background knowledge that we require in the remaining parts. 
In particular we include proofs of some of the key theorems in the literature which Part~\ref{semigroups section} and Part~\ref{nv section} are dependant on.
The most important of these being:
\begin{itemize}
    \item Comparable Polish (\ref{Polish space defn}) topologies on a group are equal (Theorem~\ref{comparable Polish group topologies theorem}).
    \item All homomorphisms from the symmetric group on \(\N\) to second countable groups are continuous, and similarly all homomorphisms from the homeomorphism group of the Cantor space to second countable groups are continuous (Theorem~\ref{automatic continuity examples thm})
    \item Rubin's Theorem (Theorem~\ref{Rubin's Theorem}).
\end{itemize}
The term ``automatic continuity" is often used in the literature to describe properties like the one in Theorem~\ref{automatic continuity examples thm}.
However to avoid confusion we have avoided this terminology in all the formal statements made in this document.

In Part~\ref{semigroups section}, we are concerned with identifying natural ways of defining topologies on a given semigroup.
We are also interested in finding topologies compatible with important semigroups which are unique up to some ``reasonable" topological restrictions.

To this end, we will give five ways of defining a topology on an arbitrary semigroup, each of which is in some sense a bound on the set of ``nice" topologies compatible with the semigroup. 
These topologies are the minimum Fr\'echet topology (\ref{minimum topology defn}), the Fr\'echet/Hausdorff-Markov topologies (\ref{Markov topology defn}), the Zariski topology (\ref{Zariski topology defn}), and the second countable continuity topology (\ref{ac topology defn}).
Notably, some of these topologies are always comparable (see Proposition~\ref{topology comparison proposition}).

We explore and describe these topologies for various naturally occurring semigroups and as a result show that each of them is compatible with either 0, 1 or many Polish topologies (see Theorems~\ref{main binary relation theorem}, \ref{main N^N theorem}, \ref{main inverse monoid theorem}, \ref{much Polish Theorem}, \ref{main inj theorem}, \ref{hilber cube main theorem}, \ref{Cantor space main theorem} and \ref{main boolean algebra theorem}).

Clones (\ref{Abstract Clones defn}) are a natural generalisation of monoids, and many of the monoids we discuss have natural clone analogues.
We also extend some of the results in Part 3 to these analogues (see Corollaries~\ref{full func clone cor}, \ref{Hilbert clone cor}, \ref{Cantor clone cor} and \ref{boolean clone cor}).

In Part~\ref{nv section}, we use Rubin's Theorem (see Corollary~\ref{Rubin Automorphisms cor}) to study the automorphisms of the Brin-Thompson groups \(dV_n\) (see Definition~\ref{prefix codes and dVn defn} and \cite{Brin2004}).
In \cite{bleak2016} it was shown that, via Rubin's Theorem, it is possible to categorise which homeomorphisms of the Cantor set correspond to automorphisms of the Higman-Thompson groups \(G_{n, r}\).
This description is based on the combinatorial properties of the minimum transducers representing these homeomorphisms as described in \cite{GNS2000}.

We extend the notion of a transducer given in \cite{GNS2000} so as to allow transducers to describe continuous maps on ``higher dimensional" Cantor spaces (see Theorem~\ref{transducerable = continuous theorem}).
We then extend the methods of \cite{bleak2016} to give an analogous description of the groups \(\Aut(dV_n)\) (see Theorem~\ref{autnV}).
We then use this description to draw an algebraic connection between the outer automorphism groups of the groups \(dV_n\) for different values of \(d\) (see Theorems~\ref{main theorem 2} and \ref{main theorem 1}).

\part{Assumed Knowledge and Definitions}\label{as section}

The purpose of this part is to establish the previously known results from the literature which we need for later parts, as well as introducing some relevant notation and terminology.
It is also assumed that the reader is familiar with the following concepts:
\begin{itemize}
    \item Zorn's Lemma.
    \item Cardinals and taking the cardinalities of arbitrary sets (the cardinality of a set \(X\) is denoted by \(|X|\)).
    \item The claims about compact and Hausdorff spaces in Remark~\ref{compactness facts}.
    \item Brouwer's Theorem (Theorem~\ref{Brouwer' Theorem}).
    \item The Baire Category Theorem (Theorem~\ref{baire category theorem}).
\end{itemize}

In the Hilbert cube subsection (\ref{hilbert subsection}) of Part~\ref{semigroups section} we also make use of Theorem~\ref{extending Hilbert homeomorphisms theorem} (Theorem 5.2.4 of \cite{Mill2001aa}), and in the boolean algebra subsection (\ref{boolean subsection}) we make use of the Stone Duality Theorem (\ref{stone duality theorem}) but no other subsections of this document are dependent of these results.

\section{Sets}
Most of the concepts in this section are well-known.
However it will be important (particularly in Part~\ref{semigroups section}) that the reader is comfortable with how these objects are viewed in this thesis. 
At times we will write ``:=" instead of ``=" to indicate that the given equality if defining the object on the left.

\speeddictzero{1}{Numbers defns}

\begin{defn}[Numbers]\label{\dict{1 ref}}
We denote the set of natural numbers (including \(0\)) by \(\N\).
We also denote the set of integers by \(\Z\), the set of rational numbers by \(\mathbb{Q}\), and the set of real numbers by \(\mathbb{R}\).

If \(a, b\in \mathbb{R}\), then we write \([a, b]\), \((a, b]\), \([a, b)\) or \((a, b)\) to denote the usual closed, half-open and open intervals in \(\mathbb{R}\).
If \(a, b\in \Z\), then \(\{a, a+1, \ldots, b\}:= \makeset{i\in \Z}{\(a\leq i \leq b\)}\).
\end{defn}

\speeddictzero{2}{power set defn}
\begin{defn}[Power set]\label{power set defn}
If \(X\) is a set, then we denote the \textit{power set} of \(X\) (the set of subsets of \(X\)) by \(\mathcal{P}(X)\).
\end{defn}

\speeddictone{3}{function defn}{power set defn}
\begin{defn}[Functions, \Assumed{3}]\label{function defn}
If \(X, Y\) are sets, then we say that \(f\) is a \textit{function} from \(X\) to \(Y\) (denoted \(f:X\to Y\)) if \(f\) is a subset of the set \(X\times Y\) which satisfies the following condition:
\begin{enumerate}
    \item For all \(x\in X\), there is a unique \(y\in Y\) with \((x, y)\in f\).
\end{enumerate}
We will compose our functions from left to right. Note that we later (Definition~\ref{product sets and projection maps def}) define the set \(X \times Y\) using functions, so this is technically a circular definition. Formally this can be avoiding by defining \(X\times Y\) without the use of functions first and then redefining the notation later, but we avoid doing this for simplicity.
\end{defn}

\speeddictone{4}{function sets def}{function defn}
\begin{defn}[Function sets, \Assumed{4}]\label{function sets def}
If \(X\) and \(Y\) are sets, then we define \(X^Y\) to be the set of all functions from \(Y\) to \(X\).
If \(Y\) is a natural number \(n\), then we treat \(n\) as the set \(\{0, 1, \ldots n-1\}\) for these purposes.
\end{defn}

\speeddictone{5}{product sets and projection maps def}{function sets def}
\begin{defn}[Product sets and projection maps, \Assumed{5}]\label{\dict{5 ref}}
If \((X_i)_{i\in I}\) are sets, then we define 
\[\prod_{i\in I} X_i := \makeset{t\in \left(\union{i\in I} X_i\right)^I}{for all \(i\in I\) we have \((i)t\in X_i\)}.\]
Note that if \(X, Y\) are sets and for all \(y\in Y\) we define \(X_y:= X\), then \(X^Y = \prod_{y\in Y} X_y\).
If \(X_0, X_1, \ldots, X_{n-1}\) are sets, then we will also use the notation:
\[X_0 \times X_1 \times \cdots \times X_{n-1} := \prod_{i\in \{0, 1, \ldots, n-1\}} X_i.\]
We will often use the notation \(((0)t, (1)t, \ldots, (n-1)t)\), to define an element \(t\) of \(X_0 \times X_1 \times \ldots \times X_{n}\).
Moreover, we will generally denote the \textit{projection maps} \(\pi_i: \prod_{j\in I} X_j \to X_i\) (defined by \((t)\pi_i = (i)t\)) by \(\pi_i\). We rely on context for the domain of the map \(\pi_i\).
In particular if \(X, Y\) are sets and \((x, y)\in X\times Y\), then \((x, y)\pi_0 = x\) and \((x, y)\pi_1 = y\).
\end{defn}

\speeddicttwo{6}{binary relations defns}{power set defn}{product sets and projection maps def}
\begin{defn}[Binary relations, \Assumed{6}]\label{binary relations defns}
If \(X\) and \(Y\) are sets and \(b\subseteq X\times Y\), then we say that \(b\) is a \textit{binary relation from \(X\) to \(Y\)} (or a \textit{binary relation on \(X\)} if \(X= Y\)). 
Let \(b\) be a binary relation from a set \(X\) to a set \(Y\).
If \(S\subseteq X\), then we define the \textit{image of }\(S\)\textit{ under }\(b\) by
\[(S)b := \makeset{y\in Y}{there is \(x\in S\) with \((x, y)\in b\)}.\]
If \(S\subseteq \mathcal{P}(X)\), then we similarly define the image of \(S\) by \((S)b := \makeset{(A)b}{\(A\in S\)}\), we also extend this to \(\mathcal{P}(\mathcal{P}(X)), \mathcal{P}(\mathcal{P}(\mathcal{P}(X)))\) etc.
We define the \textit{inverse} of \(b\) by
\[b^{-1} := \makeset{(y, x)\in Y\times X}{\((x, y)\in b\)}.\]
If \(b'\) is a binary relation from \(Y\) to another set \(Z\), then we define the \textit{composition} of \(b\) and \(b'\) by
\[b b' := \makeset{(x, z)\in X\times Z}{there is \(y\in Y\) with \((x, y)\in b\) and \((y, z)\in b'\)}.\]
 We define the \textit{domain}, \textit{image} and \textit{kernel} of \(b\) by \(\dom(b) :=  (Y)b^{-1}\), \(\im(b) := (X)b\) and \(\ker(b):= \makeset{(x, y)\in \dom(b)^2}{\((\{x\})b= (\{y\})b\)}\) respectively.
 Note that functions are examples of binary relations and composition of functions is a special case of composition of binary relations, so these definitions can be applied to functions.
 
 If \(b\) is a binary relation from \(X\) to \(Y\) and \(S\) is a set, then we define the \textit{restriction} of \(b\) to \(S\) by
 \[b\restriction_{S} := b\cap (S\times Y).\]
 
 We denote the relation \(\makeset{(x, x)}{\(x\in X\)}\) by \(\text{id}_X\) (this relation is often called the identity function, the diagonal relation or equality relation on \(X\)).
 We say that a binary relation \(b\) from \(X\) to \(Y\) is \textit{surjective} if \(\im(b) = Y\) and \textit{injective} if \(\ker(b) = \text{id}_{X}\).
 If \(f\) is a function, then we say that \(f\) is \textit{bijective} if it is both injective and surjective.
\end{defn}

\speeddicttwo{7}{countable defn}{Numbers defns}{binary relations defns}
\begin{defn}[Countable, Assumed Knowledge: \ref{Numbers defns}, \ref{binary relations defns}]\label{countable defn}
We say that a set \(X\) is \textit{countable} if there is an injective function from \(X\) to \(\N\). 
In particular, finite sets are countable.
\end{defn}

\speeddictzero{8}{union and intersection convensions}
\begin{defn}[Union and intersection conventions]\label{union and intersection convensions}
If \(X\) is a set, then we define
\[\union{} X  := \union{S\in X} S, \quad\text{and} \quad \intersection{} X := \intersection{S\in X} S.\]
Where \(\union{S\in \varnothing} S = \varnothing\), and \(\intersection{S\in \varnothing}S\) will be the ``universe" or ``space" we are working in (when this is ambiguous the notation is avoided).
\end{defn}

\speeddictzero{9}{complement defn}
\begin{defn}[Complement]\label{complement defn}
If \(X\) is a set and \(S\subseteq X\), then we will denote the \textit{complement} of \(S\) in \(X\) by \(S^c\). That is
\[S^c := X\backslash S=\makeset{x\in X}{\(x\notin S\)}.\]
In the rare cases when the set \(X\) is ambiguous the notation \(S^c\) is avoided.
\end{defn}

\speeddictone{11}{types of binary relation defns}{binary relations defns}
\begin{defn}[Types of binary relation, \Assumed{11}]\label{types of binary relation defns}
Suppose that \(X\) is a set and \(b\) is a binary relation on \(X\). 
We will use the following terms to describe these potential properties of \(b\):
\begin{enumerate}
    \item Reflexive: if \(x\in X\) then we have \((x, x)\in b\).
    \item Transitive: if \((x, y)\in b\) and \((y, z)\in b\) we also have \((x, z)\in b\).
    \item Anti-Symmetric: if \((x, y)\in b\) and \((y, x)\in b\) then \(x=y\).
    \item Symmetric: if \((x, y)\in b\) then \((y, x)\in b\).
    \item Total: if \(x, y\in X\) then either \((x, y)\in b \) or \((y, x)\in b\).
\end{enumerate}
We say that \(b\) is a \textit{preorder} if it is reflexive and transitive. 
We say that a preorder \(b\) is a \textit{partial order} if it is anti-symmetric. 
We say that a preorder \(b\) is an \textit{equivalence relation} if it is symmetric. 
We say that a partial order \(b\) is a \textit{total order} if it is total.
If \(\leq\) is one of these types of binary relations we will often write \(x\leq y\) to mean \((x, y)\in\ \leq\), and \(x< y\) to mean that \(x, y\) are distinct and satisfy \(x\leq y\).

If \(X\) is a set and \(\leq\) is a partial order on \(X\), then we say that \(x\in X\) is
\begin{enumerate}
    \item \textit{Minimum} if \(x\leq y\) for all \(y\in X\).
    \item \textit{Minimal} if for all \(y\in X\) we have \(y\leq x \Rightarrow y = x\).
    \item \textit{Maximum} if \(y\leq x\) for all \(y\in X\).
    \item \textit{Maximal} if for all \(y\in X\) we have \(x\leq y \Rightarrow y = x\).
\end{enumerate}

If \(\sim\) is an equivalence relation on a set \(X\) and \(x\in X\), then we define the \textit{equivalence class} of \(x\) by \([x]_{\sim} := \makeset{y\in X}{\(x\sim y\)}\). Furthermore we define \(X/\sim\ := \makeset{[x]_{\sim}}{\(x\in X\)}\), and we define the \textit{index} of \(\sim\) to be \(|X/\sim |\).

\end{defn}

\speeddictone{90}{filters defn}{types of binary relation defns}
\begin{defn}[Filters, \Assumed{90}]\label{filters defn}
Suppose that \((X, \leq)\) is a partially ordered set. We say that a subset \(F\) of \(X\) is a \textit{filter} if:
\begin{enumerate}
    \item For all \(x, y\in F\), there is \(z\in F\) with \(z\leq x\) and \(z\leq y\).
    \item If \(z\in X\) and there is \(x\in F\) with \(x\leq z\), then \(z\in F\).
\end{enumerate}

We say that a filter \(F\) on \(X\) is \textit{proper} if \(F\neq X\), and we say that \(F\) is a \textit{ultrafilter} if it is maximal (with respect to containment) among the proper filters on \(X\).

Filters are often considered on the power set of some set ordered by inclusion, in this case it follows from Zorn's Lemma that every proper filter in contained in an ultrafilter.
\end{defn}

\section{Categories}
We will use many categories throughout this document. 
In particular we introduce a category of transducers which is used heavily throughout Part~\ref{nv section}, and we have a categorical perspective when discussing clones in Part~\ref{semigroups section}.
Most of the categories of this document have products (Definition~\ref{Products in Categories Defn}).
We make heavy use of this fact throughout the document and in particular the notation  
\[\langle (f_i)_{i\in I} \rangle_{P}\]
(where \(I\) is an index set, \((f_i)_{i\in I}\) are morphisms with a common source object, and \(P\) is a product object)
will be used heavily.

The following lengthy ``Definition" is for the purposes of those who (like the author) are scared when proper classes are used at length. If the reader is comfortable with proper classes then they are encouraged to skip it.
\speeddicttwo{105}{Proper Classes defn}{function defn}{product sets and projection maps def}
\begin{defn}[Proper classes, \Assumed{105}]\label{Proper Classes defn}
If \(P\) is a property which assigns the value true or false to every set (being countable for example), then it will often be useful to talk simultaneously about all the sets which satisfy \(P\).
We will informally use the term \textit{class} to simultaneously refer to all the sets with some property.
Recall that every element of a set is itself a set, we make no distinction between a set \(S\) and the class associated with the property of being an element of \(S\).
Many classes are not sets, a famous example being the class of sets which do not contain themselves.
We refer to a class that is not a set as a \textit{proper class}. If \(\mathcal{C}\) is a class defined by a property \(P\), then we will often write \(c\in \mathcal{C}\) or \(P(c)\) to mean that \(c\) satisfies \(P\).

As proper classes are not formal ZFC objects, much care is required when working with them.
Nevertheless as we will see later, many natural and useful examples of categories make use of proper classes.
Thus it will be useful to be able to make subclasses, tuples of classes, products of classes, and make functions between classes.

Note that we cannot define a function whose image consists of classes. In particular when we give subscripts to classes here we are not defining a function, we are only saying that we allow this notation at different times with a different number of classes being discussed.
\begin{enumerate}
\item If \(\mathcal{A}, \mathcal{B}\) are classes defined by properties \(P_A, P_B\) respectively, then we say that \(\mathcal{A}\) is a subclass of \(\mathcal{B}\) if \(P_A\) implies \(P_B\). 
    \item If \(\mathcal{C}_0, \mathcal{C}_1, \ldots, \mathcal{C}_{n-1}\) are classes (including at least one proper class) defined by the properties \(P_0, P_1, \ldots, P_{n-1}\) respectively, then we write
    \[(\mathcal{C}_0, \mathcal{C}_1, \ldots ,\mathcal{C}_{n-1})\]
    to refer to the class of pairs \((c_i, i)\) where \(c_i\) satisfies property \(P_i\). This is very different to how tuples where defined for sets (Definition~\ref{product sets and projection maps def}), it is in fact much more like how disjoint unions are typically defined, but all that will matter in this document is that our original classes are ``recoverable" from this class. Note that \(\mathcal{C}_i\) is the class of \(c\) such that \((c, i)\in (\mathcal{C}_0, \mathcal{C}_1, \ldots ,\mathcal{C}_{n-1})\). 
    
    \item If \(\mathcal{C}_0, \mathcal{C}_1, \ldots, \mathcal{C}_{n-1}\) are classes defined by the properties \(P_0, P_1, \ldots, P_{n-1}\) respectively, then we write
    \[\mathcal{C}_0\times \mathcal{C}_1 \times \ldots \times\mathcal{C}_{n-1}\]
    to refer to the class of \(n\)-tuples \((c_0, c_1, \ldots, c_{n-1})\) such that each \(c_i\) is a set which satisfies property \(P_i\). Note that in the case that our classes are sets, this definition agrees with Definition~\ref{product sets and projection maps def}.
    \item If \(\mathcal{A}, \mathcal{B}\) are classes defined by properties \(P_A, P_B\) respectively, then we say that \(f:\mathcal{A} \to \mathcal{B}\) is \textit{class function} to mean that \(f\) is a subclass of \(\mathcal{A}\times \mathcal{B}\) and for all \(a\in \mathcal{A}\), there is a unique \(b\in \mathcal{B}\) with \((a, b)\in f\). If \(a\in \mathcal{A}\), then we write \((a)f\) to denote the unique \(b\in \mathcal{B}\) with \((a, b)\in f\). Note that this agrees with Definition~\ref{function defn}.
\end{enumerate}
\end{defn}

\speeddictone{106}{Categories defn}{Proper Classes defn}
\begin{defn}[Categories and functors, \Assumed{106}]\label{Categories defn}
We say that \textit{\(\mathcal{C}\) is a category} to mean that \(\mathcal{C}\) is a tuple \((\mathcal{O}_{\mathcal{C}}, \mathcal{M}_{\mathcal{C}}, s_{\mathcal{C}}, t_{\mathcal{C}}, \circ_{\mathcal{C}})\), where
\begin{enumerate}
    \item \(\mathcal{O}_{\mathcal{C}}\) is a class (called the \textit{objects} of the category).
    \item \(\mathcal{M}_{\mathcal{C}}\) is a class (called the \textit{morphisms} of the category).

    \item \(s_{\mathcal{C}}:\mathcal{M}_{\mathcal{C}} \to \mathcal{O}_{\mathcal{C}}\) and \(t_{\mathcal{C}}:\mathcal{M}_{\mathcal{C}} \to \mathcal{O}_{\mathcal{C}}\) are class functions (which assign to morphisms a \textit{source} and \textit{target} respectively).

    \item \(\operatorname{Comp}_{\mathcal{C}}\) is the subclass of \(\mathcal{M}_{\mathcal{C}} \times \mathcal{M}_{\mathcal{C}}\) consisting of those \((m_0, m_1)\) with \((m_0)t_{\mathcal{C}} =(m_1)s_{\mathcal{C}}\) (called the \textit{composable} pairs).
    
    \item \(\circ_{\mathcal{C}}:\operatorname{Comp}_{\mathcal{C}}\to \mathcal{M}_{\mathcal{C}}\) is a class function (called \textit{composition}) such that if \((a, b), (b, c)\in \operatorname{Comp}_{\mathcal{C}}\) then 
    \[((a, b)\circ_{\mathcal{C}})s_\mathcal{C} =(a)s_\mathcal{C},\quad ((a, b)\circ_{\mathcal{C}})t_\mathcal{C} =(b)t_\mathcal{C},\quad ((a, b)\circ_{\mathcal{C}}, c)\circ_{\mathcal{C}} = (a, (b, c)\circ_{\mathcal{C}})\circ_{\mathcal{C}}.\]
    
    \item For all \(O\in \mathcal{O}_{\mathcal{C}}\), there is a unique \(1_O\in \mathcal{M}_{\mathcal{C}}\) such that \((1_O)s_{\mathcal{C}} = (1_O)t_{\mathcal{C}}= O\) and if \((1_O, m_0), (m_1, 1_O)\in \operatorname{Comp}_{\mathcal{C}}\) then 
    \[(1_O, m_0)\circ_{\mathcal{C}} = m_0 \quad \text{ and }\quad (m_1, 1_O)\circ_{\mathcal{C}} = m_1.\]
\end{enumerate}

If \(\mathcal{C}\) is a category and \(A, B\) are objects of \(\mathcal{C}\), then we say \(A, B\) are \textit{isomorphic} in \(\mathcal{C}\) if there are morphisms \(f_A, f_B\) of \(\mathcal{C}\) with
\[(f_A, f_B)\circ_{\mathcal{C}} = 1_{A} \quad \text{and} \quad (f_B, f_A)\circ_{\mathcal{C}} = 1_{B}.\]
If \(\mathcal{A}, \mathcal{B}\) are categories, then we write \(f:\mathcal{A}\to \mathcal{B}\)\textit{ is a functor} to mean that \(f\) is a tuple \((f_{\mathcal{O}}, f_{\mathcal{M}})\) where
\begin{enumerate}
    \item \(f_{\mathcal{O}}: \mathcal{O}_{\mathcal{A}} \to  \mathcal{O}_{\mathcal{B}}\) is a class function.
    \item \(f_{\mathcal{M}}: \mathcal{M}_{\mathcal{A}} \to  \mathcal{M}_{\mathcal{B}}\) is a class function.
    \item For all \(a, b\in \operatorname{Comp}_{\mathcal{C}}\) and \(O\in \mathcal{O}_C\) we have
    \[(1_O)f_{\mathcal{M}} = 1_{(O)f_{\mathcal{O}}}, \quad  ((a)t_{\mathcal{A}})f_{\mathcal{O}} = ((a)f_{\mathcal{M}})t_{\mathcal{B}}, \quad  ((b)s_{\mathcal{A}})f_{\mathcal{O}} = ((b)f_{\mathcal{M}})s_{\mathcal{B}}\]
    \[((a, b)\circ_{\mathcal{A}})f_{\mathcal{M}} = ((a)f_{\mathcal{M}}, (b)f_{\mathcal{M}})\circ_{\mathcal{B}}.\]
\end{enumerate}

We say that a category is \textit{small} if it is a tuple of sets and \textit{big} if it is not.
\end{defn}

The notation \(\pi_i\) in the following definition is used because in the category of sets and functions (see Example~\ref{Important Categories}) the natural projection maps are precisely theses maps from Definition~\ref{product sets and projection maps def}.
While these are not the only products or projection maps, all concrete examples discussed in this document will be (essentially) of this type.
For example in Definition~\ref{transducer products defn}, the projections will not even be functions, but they are triples of the usual projection maps.

\speeddicttwo{108}{Products in Categories Defn}{Numbers defns}{Categories defn}
\begin{defn}[Categorical products, \Assumed{108}]\label{Products in Categories Defn}
If \(\mathcal{C}\) is a category, \(I\) is a set and \((O_{i})_{i\in I}\) are objects of \(\mathcal{C}\), then a \textit{product} of \((O_{i})_{i\in I}\) in \(\mathcal{C}\) is a pair \((P, (\pi_i)_{i\in I})\) such that
\begin{enumerate}
    \item \(P\) is an object of \(\mathcal{C}\) (called the \textit{product object}).
    
    \item For each \(i\in I\), \(\pi_i\) is a morphism of \(\mathcal{C}\) with source \(P\) and target \(O_i\) (called the \textit{projection morphisms}).
    
    \item If \(B\) is an object of \(\mathcal{C}\) and \((f_i)_{i\in I}\) are morphisms of \(C\) with \(f_i:B\to O_i\), then there is a unique morphism \(\langle (f_i)_{i\in I} \rangle_{P}:B\to P\) of \(\mathcal{C}\) such that
    \[(\langle (f_i)_{i\in I} \rangle_{P} , \pi_i) \circ_{\mathcal{C}} = f_i\]
    for all \(i\in I\).
\end{enumerate}
The morphism described in the third condition above is dependant not only on the product object \(P\), but also on the choice of projection morphisms. However in this document we will never consider a single object as a product in multiple ways so the notation with always be unambiguous.

We are often interested in the cases when \(I\) is \(\{0, 1, \ldots, n-1\}\) for some \(n\in \N\). In this case we will often write \(\langle f_0, f_1, \ldots, f_{n-1} \rangle_{P}\text{ instead of }\langle (f_i)_{i\in \{0, 1, \ldots, n-1\}} \rangle_{P}.\)
\end{defn}

\speeddictthree{107}{Important Categories}{product sets and projection maps def}{Categories defn}{Products in Categories Defn}
\begin{example}[Important categories, \Assumed{107}]\label{Important Categories}
The most notable example of a category is the category of sets and functions. All sets are objects of this category and the morphisms are tuples \((A, f, B)\) where \(f:A\to B\) is a function. The source and target of \((A, f, B)\) are \(A, B\) respectively, and the composition of \((A, f, B), (B, g, C)\) is \((A, fg, C)\).
We cannot simply use functions as the morphisms of this category because the target of a function is not deducible from the function itself, however we will sometimes abuse language and treat functions as morphisms to avoid clunky notation.
As one would expect this category has products for arbitrary collections of sets and they are precisely as described in Definition~\ref{product sets and projection maps def}.

The ``category of categories and functors" is also a natural candidate, but this fails to be a category as big categories cannot be elements of classes. However small categories and functors do form a category in the natural fashion.

\end{example}

\section{Topology}
This section consists mostly of introducing terminology and facts from topology which will be using throughout the document.
Polish spaces (\ref{Polish space defn}) will be of particular interest, and Baire Category (\ref{baire category theorem}) plays a prominent role in the study of these spaces.

In our view of topology we are occasionally interested the topology of a space as an object in its own right.
For example in the Rubin construction (\ref{canonical Rubin action defn}) the partial order on a given topology is crucial, and in Part~\ref{semigroups section} we frequently compare multiple topologies on the same set.
\speeddictthree{12}{topological space defn}{power set defn}{union and intersection convensions}{complement defn}
\begin{defn}[Topological space, \Assumed{12}]\label{topological space defn}
A \textit{topological space} is a pair \((X, \mathcal{T})\) where \(X\) is a set and \(\mathcal{T}\subseteq \mathcal{P}(X)\) satisfies:
\begin{enumerate}
    \item If \(F\subseteq \mathcal{T}\) is finite, then \(\intersection{} F \in \mathcal{T}\).
    \item If \(S\subseteq \mathcal{T}\), then \(\union{}S\in \mathcal{T}\).
    \item Both \(\varnothing\) and \(X\) are elements of \(\mathcal{T}\).
\end{enumerate}
In this case we say that \(\mathcal{T}\) is a topology on \(X\).
Recall that we allow empty unions/intersections (Definition~\ref{union and intersection convensions}) so the third condition above is technically redundant.
When convenient we will often identity the pair \((X, \mathcal{T})\) with the set \(X\), particularly in the use of the symbols \(\in, \subseteq\) etc. We will refer to elements of \(\mathcal{T}\) as \textit{open} sets and sets whose complement is in \(\mathcal{T}\) as \textit{closed} sets. If a set is both open and closed then we will call it \textit{clopen}.
\end{defn}

\speeddictone{13}{subspaces}{topological space defn}
\begin{defn}[Subspaces, \Assumed{13}]\label{subspaces}
If \((X, \mathcal{T})\) is a topological space and \(A\subseteq X\), then we will view \(A\) as a topological space with the topology 
\[\mathcal{T}\restriction_{A}:=\makeset{U\cap A}{\(U\in \mathcal{T}\)}.\]
This is called the \textit{subspace topology} on \(A\).
\end{defn}

\speeddictfour{14}{continuous maps and homeomorphisms defn}{binary relations defns}{Categories defn}{Important Categories}{topological space defn}
\begin{defn}[Continuous maps and homeomorphisms, \Assumed{14}]\label{continuous maps and homeomorphisms defn}
If \(X\) and \(Y\) are topological spaces, then we say that a function \(f: X\to Y\) is \textit{continuous} if whenever \(U\subseteq Y\) is open, \((U)f^{-1}\) is also open. We say that \(f\) is a \textit{homeomorphism} if it is a bijection and \(f^{-1}\) is also continuous (in this case we say that \(X\) and \(Y\) are \textit{homeomorphic}). 
\end{defn}

It is routine to verify that topological spaces and continuous maps form a category (\ref{Categories defn}), and homeomorphisms are the isomorphisms of this category.

\speeddictone{15}{closure and interior defn}{topological space defn}
\begin{defn}[Closure and interior, \Assumed{15}]\label{\dict{15 ref}}
If \(X\) is a topological space and \(S\subseteq X\), then we define
\[S^- := \intersection{}\makeset{F\subseteq X}{\(F\) is closed and \(F\supseteq S\)}, \quad S^\circ := \union{}\makeset{U\subseteq X}{\(U\) is open and \(U\subseteq S\)}.\]
These are called the \textit{closure} and \textit{interior} of \(S\) respectively. Note that the closure of a set is always closed and the interior of a set is always open.
We often use these notations together with the notation for complement (\ref{complement defn}) horizontally together in one superscript to denote doing them in succession, for example \(S^{-c}(=S^{c\circ})\) denotes the complement of the closure of the set \(S\).
\end{defn}

\speeddictone{16}{nbhds defn}{closure and interior defn}
\begin{defn}[Neighbourhoods, \Assumed{16}]\label{nbhds defn}
If \(X\) is a topological space \(S\subseteq X\) and \(x\in X\), then we say that \(S\) is a \textit{neighbourhood} (often written \textit{nbhd}) of \(x\) if \(x\in S^\circ\). We denote the set of all neighbourhoods of \(x\) by \(\nbhd{X}{x}\). Note that a set \(S\) is open if and only if \(S\) is a neighbourhood of each of its elements.
\end{defn}

The set of neighbourhoods (\ref{nbhds defn}) of a point in a topological space form a filter (\ref{filters defn}).
We will later see that it is precisely these filters which allow the Rubin construction (\ref{canonical Rubin action defn}) to construct the elements of Rubin spaces.

\speeddictone{27}{discrete trivial defn}{topological space defn}
\begin{defn}[Discrete and trivial topologies, \Assumed{27}]\label{\dict{27 ref}}
If \(X\) is a set, then the \textit{discrete topology on \(X\)} is \(\mathcal{P}(X)\) and the \textit{trivial topology on \(X\)} is \(\{\varnothing, X\}\). A topological space equipped with the discrete topology is called a \textit{discrete space}.
\end{defn}

\speeddicttwo{18}{subbasis}{topological space defn}{discrete trivial defn}
\begin{defn}[Subbasis,  \Assumed{18}]\label{subbasis}
Suppose that \(X\) is a set, \(B\) is a collection of subsets of \(X\), and \(\mathcal{T}\) is the smallest topology on \(X\) containing \(B\) (this topology always exists because \(B\) is contained in the discrete topology on \(X\)).
In this case we say that \(B\) is a \textit{subbasis for \(\mathcal{T}\)} or that \(\mathcal{T}\) is the \textit{topology generated by \(B\)}.
\end{defn}

\speeddictone{19}{basis defn}{subbasis}
\begin{defn}[Basis, \Assumed{19}]\label{\dict{19 ref}}
Suppose that \(X\) is a set, \(B\) is a collection of subsets of \(X\), and \(\mathcal{T} = \makeset{\union{} S}{\(S\subseteq B\)}\) is a topology. In this case we say that \(B\) is a \textit{basis} for \(\mathcal{T}\). In particular any basis for a topology is a subbasis for that topology. Also if \(S\) is a subbasis for a topology \(\mathcal{T}\) on \(X\), then \(\makeset{\intersection{} F}{\(F\subseteq  S\) is finite}\) is a basis for \(\mathcal{T}\).
\end{defn}

\speeddicttwo{20}{second countable defn}{countable defn}{basis defn}
\begin{defn}[Second countable, \Assumed{20}]\label{second countable defn}
We say that a topology/topological space is \textit{second countable} if the topology admits a countable basis.
\end{defn}

\speeddictone{21}{isolated points defn}{topological space defn}
\begin{defn}[Isolated points, \Assumed{21}]\label{\dict{21 ref}}
If \(X\) is a topological space and \(x\in X\), then we say that \(x\) is an \textit{isolated point} of \(X\) if \(\{x\}\) is open.
\end{defn}

\speeddictfour{22}{product topology defn}{product sets and projection maps def}{Products in Categories Defn}{continuous maps and homeomorphisms defn}{subbasis}
\begin{defn}[Product topology, \Assumed{22}]\label{product topology defn}
If \((X_i)_{i\in I}\) are topological spaces, then we view \(\prod_{i\in I} X_i\) as a topological space with the topology generated by the sets of the form \((U)\pi_i^{-1}\) where \(U\) is open in \(X_i\). We call this topology the \textit{pointwise topology} or the \textit{product topology}.

Note that (with respect to the product topology) the projections maps are all continuous and map open sets to open sets. It is routine to verify that this construction gives a categorical product in the category of topological spaces and continuous functions (see Definition~\ref{Products in Categories Defn}).
\end{defn}

As we will see later, the products above are very useful for constructing continuous maps.
They are also very useful for proving that nice looking maps are indeed continuous, as we can often show that they are continuous by constructing them as compositions of maps of the form \(\langle f, g\rangle_P\) (see Definition~\ref{Products in Categories Defn}) for appropriate choices of \(f, g\) and \(P\).

\speeddictone{300}{union topology defn}{subspaces}
\begin{defn}[Union topology, \Assumed{300}]\label{disjoint union topology defn}
If \((X_i)_{i\in I}\) are disjoint topological spaces, then we view \(\union{i\in I} X_i\) as a topological space with the topology 
\[\makeset{U\subseteq \union{i\in I} X_i}{\(U\cap  X_i\) is open in \(X_i\) for all \(i\in I\)}.\]
We call this topology the \textit{disjoint union topology}.

It is routine to verify that this is a topology and for all \(i\in I\), the original topology on \(X_i\) agrees with the subspace topology on \(X_i\subseteq \union{i\in I} X_i\).
\end{defn}

\speeddictone{23}{zero-dimensional defn}{subbasis}
\begin{defn}[Zero-dimensional, \Assumed{23}]\label{zero-dimensional defn}
We say that a topological space \(X\) is \textit{zero-dimensional} if it has a subbasis consisting of clopen sets.
\end{defn}

\speeddictfour{24}{Hausdorff defn}{binary relations defns}{nbhds defn}{product topology defn}{compactness facts}
\begin{defn}[Hausdorff, \Assumed{24}]\label{Hausdorff defn}
We say that a topological space \(X\) is Hausdorff if one of the following equivalent (see Remark~\ref{compactness facts}) conditions holds :
\begin{enumerate}
    \item If \(x, y\in X\) and \(x\neq y\), then there exist disjoint neighbourhoods \(U_x,U_y\) of \(x\) and \(y\) respectively.
    \item The diagonal \(\operatorname{id}_X\) (recall Definition~\ref{binary relations defns}) of \(X\) is a closed subset of \(X\times X\).
\end{enumerate}
\end{defn}

\speeddictone{171}{frechet defn}{topological space defn}
\begin{defn}[Fréchet, \Assumed{171}]\label{frechet defn}
We say that a topological space \(X\) is \textit{Fréchet} if for all \(x\in X\), the set \(\{x\}\) is a closed subset of \(X\).
\end{defn}

The notion of a Fréchet space is a weakening of the notion of a Hausdorff space which is particularly useful as it gives us concrete examples of sets which must be open/closed with respect to our topology.
In particular in Part~\ref{semigroups section}, we will often make use of the observation that every finite subspace of a Fréchet space is discrete.

\speeddictfive{25}{compact defn}{power set defn}{filters defn}{topological space defn}{nbhds defn}{compactness facts}
\begin{defn}[Compact, \Assumed{25}]\label{compact defn}
We say that a topological space \(X\) is \textit{compact} if any of the following equivalent (see Remark~\ref{compactness facts}) conditions holds:
\begin{enumerate}
    \item If \(S\) is a collection of open subsets of \(X\) with \(\union{}S = X\), then there is a finite \(F\subseteq S\) with \(\union{} F = X\).
    
    \item If \(S\) is a collection of closed subsets of \(X\) and \(\intersection{} F \neq \varnothing\) for all finite \(F\subseteq S\), then \(\intersection{} S\neq \varnothing\).
    
    \item If \(F\) is a proper filter on \((\mathcal{P}(X), \subseteq)\), then there is a proper filter \(F'\) on \((\mathcal{P}(X), \subseteq)\) and \(x\in X\) such that \(\nbhd{X}{x}\cup F\subseteq F'\).
\end{enumerate}
\end{defn}

\speeddicttwo{94}{locally compact defn}{subspaces}{compact defn}
\begin{defn}[Locally compact, \Assumed{94}]\label{locally compact defn}
We say that a topological space \(X\) is \textit{locally compact} if every element of \(X\) has a compact neighbourhood.
\end{defn}

\speeddictfive{26}{compactness facts}{filters defn}{continuous maps and homeomorphisms defn}{product topology defn}{Hausdorff defn}{compact defn}
\begin{remark}[Compactness and Hausdorffness facts, \Assumed{26}]\label{compactness facts}
The following facts about compactness and Hausdorffness are well known and we use them frequently throughout the document.
\begin{enumerate}
    \item The three definitions of compactness given in Definition~\ref{compact defn} are equivalent.
    
    \item The two definitions of Hausdorffness given in Definition~\ref{Hausdorff defn} are equivalent.    
    \item If \(\phi:X\to Y\) is a continuous bijection between compact Hausdorff spaces then \(\phi^{-1}\) is also continuous.
    
    \item A closed subspace is a compact space is compact.
    
    \item A continuous image of a compact space is compact.
    
    \item A compact subspace of a Hausdorff space is closed.
    
    \item An arbitrary product of Hausdorff spaces is Hausdorff.
    
    \item An arbitrary product of compact spaces is compact (Tychonoff's Theorem).
    
    \item If \(X\) is a compact Hausdorff space and \(x\in X\) then every neighbourhood of \(x\) contains a compact neighbourhood of \(x\).
    
    \item If \(X\) is a compact topological space and \(\mathcal{F}\) is an ultrafilter (Definition~\ref{filters defn}) on \(X\), then there is \(x\in X\) such that \(\nbhd{X}{x}\subseteq \mathcal{F}\).
\end{enumerate}
\end{remark}
\begin{proof}
\((1):\) For \((1)\iff (2)\) see Theorem 26.9  of \cite{munkres2016topology}. 
The equivalence of \((2)\) and \((3)\) follows from the observation that a filter \(F\subseteq \mathcal{P}(X)\) is proper if and only if \(\varnothing\notin F\), and thus \(\nbhd{X}{x}\cup F\) is contained in a proper filter if and only if \(x\in \intersection{S\in F} S^-\).

\((2):\) Note that \(x, y \in X\) and \(x\neq y\), then \(U_x, U_y\) are neighbourhoods of \(x, y\) respectively if and only if \(U_x\times U_y\) is a neighbourhood of \((x, y)\).
Moreover \(U_x, U_y\) are disjoint if and only if \(U_x \times U_y\) is disjoint from \(\text{id}_X\).

\((3):\) See Theorem 26.6  of \cite{munkres2016topology}.

\((4):\) See Theorem 26.2  of \cite{munkres2016topology}.

\((5):\) See Theorem 26.5  of \cite{munkres2016topology}.

\((6):\) See Theorem 26.3  of \cite{munkres2016topology}.

\((7):\) Let \((X_i)_{i\in I}\) be Hausdorff topological spaces.
If \(x, y \in \prod_{i\in I} X_i\), and \(i\in I\) is such that \((x)\pi_i\neq (y)\pi_i\), then there are disjoint neighbourhoods \(U_x\), \(U_y\) of \((x)\pi_i, (y)\pi_i\) respectively (in \(X_i\)).
The sets \((U_x)\pi_i^{-1}\) and \((U_y)\pi_i^{-1}\) are then the required neighbourhoods of \(x\) and \(y\).

\((8):\) See Theorem 37.3  of \cite{munkres2016topology}.

\((9):\) See Theorem 26.4  of \cite{munkres2016topology}.

\((10):\) This is immediate from the third definition of compactness in Definition~\ref{compact defn}.
\end{proof}

\speeddicttwo{28}{Cantor space defn}{product topology defn}{discrete trivial defn}
\begin{defn}[Cantor space, \Assumed{28}]\label{\dict{28 ref}}
We define the \textit{Cantor space} \(2^\N\) to be the set \(\{0,1\}^\N\) with the product topology (where \{0, 1\} has the discrete topology). 
\end{defn}

\speeddictseven{29}{Brouwer' Theorem}{continuous maps and homeomorphisms defn}{second countable defn}{isolated points defn}{zero-dimensional defn}{Hausdorff defn}{compact defn}{Cantor space defn}
\begin{theorem}[Brouwer's Theorem, \Assumed{29}]\label{\dict{29 ref}}
Any second countable, zero-dimensional, Hausdorff, compact topological space with no isolated points is homeomorphic to either \(\varnothing\) or \(2^\N\).
\end{theorem}
There are many known proofs of Brouwer's Theorem, see for example Theorem 7.4 of \cite{kechris2012classical}.

\speeddicttwo{30}{metric space def}{Numbers defns}{product sets and projection maps def}
\begin{defn}[Metric spaces, \Assumed{30}]\label{\dict{30 ref}}
If \((X, d)\) is a pair where \(d:X\times X \to [0, \infty)\) is a function satisfying the following for all \(x, y, z\in X\):
\begin{enumerate}
    \item \((x, y)d = 0 \iff x=y\).
    
    \item \((x, y)d = (y, x)d\).
    
    \item \((x, z)d \leq (x, y)d + (y, z)d\).
\end{enumerate}
then we call \((X, d)\) a \textit{metric space} and we call \(d\) a \textit{metric on \(X\)}.
\end{defn}

\speeddicttwo{31}{topology on a metric space defn}{subbasis}{metric space def}
\begin{defn}[Topology of a metric space, \Assumed{31}]\label{\dict{31 ref}}
If \((X, d)\) is a metric space, then we define the topology \textit{induced by \(d\)} to be the topology \(\mathcal{T}\) generated by the sets of the form
\[B_{d}(x, r) := \makeset{y\in Y}{\((x, y)d < r\)}\]
for all \((x, r)\in X\times \mathbb{R}\).
In this case we say that \(\mathcal{T}\) is \textit{compatible with \(d\)}. We will always assume a metric space is equipped with this topology unless otherwise stated. We say that a topological space is \textit{metrizable} if it is compatible with a metric.
\end{defn}

\speeddicttwo{17}{convergence defn}{nbhds defn}{topology on a metric space defn}
\begin{defn}[Convergence of sequences, \Assumed{17}]\label{\dict{17 ref}}
If \(X\) is a topological space and \((x_i)_{i\in \N}\) is a sequence of points in \(X\), then we say that \((x_i)_{i\in \N}\) \textit{converges to a limit \(x\in X\)} (written \((x_i)_{i\in \N} \rightarrow x\)) if:
\begin{enumerate}
    \item For all \(N\in \nbhd{X}{x}\), the set \(\makeset{i\in \N}{\(x_i\notin N\)}\) is finite. 
\end{enumerate}
This concept is particularly useful in metrizable spaces as in these cases a function is continuous if and only if it preserves limits of sequences.
\end{defn}

\speeddicttwo{32}{complete metric defn}{convergence defn}{topology on a metric space defn}
\begin{defn}[Complete metric, \Assumed{32}]\label{\dict{32 ref}}
Let \((X, d)\) be a metric space. 
A sequence \((x_i)_{i\in \N}\) in \(X\) is called \textit{Cauchy} if it satisfies the following condition: 
\begin{enumerate}
    \item For all \(\varepsilon > 0\) there is \(N\in \N\) such that for all \(n, m\geq N\) we have \((x_n, x_m)d \leq \varepsilon\).
\end{enumerate}
We say that \(d\) is \textit{complete} if all Cauchy sequences in \((X, d)\) converge.
\end{defn}

\speeddicttwo{33}{Polish space defn}{second countable defn}{complete metric defn}
\begin{defn}[Polish spaces, \Assumed{33}]\label{Polish space defn}
We say that a topological space \(X\) is \textit{Polish}, if it is second countable and there is a complete metric on \(X\) which induces its topology.
\end{defn}

Note that it is not true in general that if \(X\) is a Polish space, then all metrics compatible with \(X\) are complete.
For example the open unit interval with the standard metric is not complete but it is homeomorphic to \(\mathbb{R}\) which is complete with respect to the standard metric.

\speeddictone{34}{denseness defn}{closure and interior defn}
\begin{defn}[Denseness, \Assumed{34}]\label{denseness defn}
Suppose that \(X\) is a topological space and \(A, B\subseteq X\). We say that \(A\) is \textit{dense in }\(B\) if \(B\subseteq A^-\).

We say that \(A\) is \textit{somewhere dense} (in \(X\)) if it is dense in a non-empty open subset of \(X\) (equivalently if \(A^{-\circ}\neq \varnothing\)). We say that \(A\) is \textit{nowhere dense} (in \(X\)) if \(A\) is not somewhere dense (in \(X\)) (equivalently if \(A^{-c}= A^{c\circ}\) is dense in \(X\)).
\end{defn}

\speeddicttwo{35}{meagre defn}{countable defn}{denseness defn}
\begin{defn}[Meagreness, \Assumed{35}]\label{\dict{35 ref}}
If \(X\) is a topological space and \(A\subseteq X\), then we say that \(A\) is \textit{meagre} (in \(X\)) if \(A\) is a countable union of nowhere dense sets (or equivalently if \(A\) is disjoint from the intersection of a countable collection of dense open sets).

We say that \(A\) is \textit{comeagre} if \(A^c\) is meagre in \(X\) (or equivalently if \(A\) contains a countable intersection of dense open sets).
\end{defn}

\speeddictthree{39}{baire category theorem}{topology on a metric space defn}{complete metric defn}{meagre defn}
\begin{theorem}[Baire Category Theorem, \Assumed{39}]\label{baire category theorem}
If \(X\) is a topological space which is compatible with a complete metric, then all comeagre subsets of \(X\) are dense in \(X\).
\end{theorem}
There are many known proofs of the Baire Category Theorem, see for example Theorem 8.4 of \cite{kechris2012classical}.

\speeddicttwo{36}{almost containment defn}{types of binary relation defns}{meagre defn}
\begin{defn}[Almost containment, \Assumed{36}]\label{\dict{36 ref}}
If \(X\) is a topological space, then we define a preorder \(\lesssim_X\) on the subsets of \(X\) by
\[A\lesssim_X B \iff \text{ there exists }C\text{ comeagre in }X\text{ such that } A\cap C \subseteq B\cap C.\]
We also say \(A\approx_X B\) if \(A\lesssim_X B\lesssim_X A\). This notation is of particularly use due to the following observations:
\begin{align*}
    A\approx_X B &\iff \text{ there is }C\text{ comeagre in }X\text{ with }C\cap A = C\cap B,\\
     A\approx_X \varnothing&\iff A \text{ is meagre in }X ,\\
   A\approx_X X &\iff  A \text{ is comeagre in }X.
\end{align*}
\end{defn}

\speeddicttwo{37}{meagreness in subspaces rmk}{subspaces}{almost containment defn}
\begin{remark}[Meagreness in subspaces, \Assumed{37}]\label{\dict{37 ref}}
If \(X\) is a topological space and \(A\subseteq B\) are subspaces of \(X\), then \(A\approx_B \varnothing\) implies that \(A\approx_X \varnothing\) (the converse is not always true).
\end{remark}

We conclude this section with a proof of the Kuratowski-Ulam Theorem which we will use when giving a proof that all homomorphisms from groups with ample generics to second countable groups are continuous (Theorem~\ref{ample automatic continuity thm}).

The converse of the Kuratowski-Ulam Theorem also holds for sets which are almost open (Definition~\ref{almost open defn}). 
However we only provide the version of the result which we will be needing.
For more information see for example Theorem 8.41 of \cite{kechris2012classical}.

\speeddictthree{38}{kuratowsiulam}{second countable defn}{product topology defn}{almost containment defn}
\begin{theorem}[Kuratowski-Ulam, \Assumed{38}]\label{\dict{38 ref}}
Suppose that \(X, Y\) are second countable topological spaces and \(C \approx_{X\times Y} X\times Y\).
Then
\[X\approx_X \makeset{x\in X}{\((\{x\} \times Y)\approx_{(\{x\}\times Y)} (\{x\}\times Y)\cap C\)}.\]
\end{theorem}
\begin{proof}
Let \((B_i)_{i\in \N}\) be a countable basis for \(Y\). Let \((U_i)_{i\in \N}\) be a sequence of dense open subsets of \(X\times Y\) such that 
\[C\supseteq \intersection{i\in \N} U_i.\]
For all \(i, j\in \N\) let \(D_{i, j}:= (U_i\cap (X\times B_j))\pi_0\).
As \(\pi_0\) maps open sets to open sets, the sets \(D_{i, j}\) are all open. Moreover as \(\pi_0\) is a continuous surjection and each of the sets \(U_i\cap (X\times B_j)\) is dense in \(X\times B_j\), the sets \(D_{i, j}\) are dense in \(X\). So for all \(i, j\in \N\) we have \(D_{i, j} \approx_X X\) and thus 
\[\intersection{i, j\in \N} D_{i, j} \approx_X X.\]
Whenever \(x\in D_{i, j}\) it follows that \(U_i \cap (X\times B_j) \cap (\{x\}\times Y) \neq \varnothing\). So if \(x\in \intersection{i, j\in \N} D_{i, j}\) then \(U_i \cap (\{x\}\times Y)\) is open and dense in \(\{x\}\times Y\) for all \(i\in \N\). The result follows.
\end{proof}

\section{Structures}
In this section, we introduce the model theoretic view of the structures we will be working with.
This allows us to talk about topological structures and define notation in a more general context. This will be useful as while we are primarily concerned with groups and semigroups we will occasionally need to use other types of structures.

\speeddicttwo{40}{signature defn}{Numbers defns}{product sets and projection maps def}
\begin{defn}[Signature, \Assumed{40}]\label{signature defn}
A \textit{signature} is a 3-tuple \(\sigma=(F_\sigma, R_\sigma, \operatorname{ar}_\sigma)\), where
\begin{enumerate}
    \item \(F_\sigma\) is a set (whose elements are called function symbols or operation symbols).
    \item \(R_\sigma\) is a set (whose elements are called relation symbols) disjoint from \(F_\sigma\).
    \item \(\operatorname{ar}_\sigma: F_\sigma\cup R_\sigma \to \N\) is a function (assigns each symbol an arity).
\end{enumerate}
\end{defn}

\speeddictone{41}{structures defn}{signature defn}
\begin{defn}[Structure, \Assumed{41}]\label{structures defn}
If \(\sigma\) is a signature, then we say that \(\mathbb{M}=(M, \sigma, I)\) is a \(\sigma\)-\textit{structure} if
\begin{enumerate}
    \item \(M\) is a set (called the \textit{universe} of \(\mathbb{M}\)).
    \item \(I:F_\sigma \cup R_\sigma \to \union{n\in \N} \left(M^{M^n} \cup \mathcal{P}(M^n)\right)\) is such that
\begin{align*}
    (F)I\in M^{M^n} &\iff (F\in F_\sigma \text{ and } (F)\operatorname{ar}_\sigma= n) \text{ and }\\
    (R)I\in \mathcal{P}(M^n) &\iff (R\in R_\sigma \text{ and } (R)\operatorname{ar}_\sigma= n).
\end{align*}
\end{enumerate}
The idea being that \(I\) interprets each element of \(F_\sigma\) as an operation and each element of \(R_\sigma\) as a relation.
If \(S\in F_\sigma \cup R_\sigma\) then we will use the notation \(S^\mathbb{M}\) to denote \((S)I\).
We often treat a structure \(\mathbb{M}\) as if it where equal to its universe \(M\) for ease of notation.
We say that a \(\sigma\)-structure is \textit{algebraic} if \(R_{\sigma}= \varnothing\).
\end{defn}

\speeddictfour{42}{structure hom defn}{binary relations defns}{Categories defn}{Products in Categories Defn}{structures defn}
\begin{defn}[Homomorphisms, \Assumed{42}]\label{structure hom defn}
If \(\mathbb{X}\) and \(\mathbb{Y}\) are \(\sigma\)-structures and \(f:\mathbb{X} \to \mathbb{Y}\) is a function, then we say that \(f\) is a \textit{homomorphism} if for all \(F\in F_\sigma\) and \(R\in R_\sigma\) we have
    \[F^\mathbb{X} \circ f = \boldsymbol{f}_{(F)\ar_\sigma}\circ F^\mathbb{Y} \quad \text{ and }\quad(R^\mathbb{X})\boldsymbol{f}_{(R)\ar_\sigma} \subseteq  R^\mathbb{Y}\]
where \(\boldsymbol{f}_n := \langle(\pi_if)_{i\in \{0, 1, \ldots, n-1\}}\rangle_{X^n}\) (recall Definition~\ref{Products in Categories Defn}).
We say that a homomorphism is a \textit{embedding} if it is injective and the containment above can be strengthened to \((R^\mathbb{X})\boldsymbol{f}_{(R)\ar_\sigma} =  R^\mathbb{Y} \cap (\im(f)^n)\).

 We say that a homomorphism is an \textit{isomorphism} if it is a surjective embedding. If there is an isomorphism between structures \(\mathbb{X}\) and \(\mathbb{Y}\), then we say that they are \textit{isomorphic} (denoted \(\mathbb{X}\cong \mathbb{Y}\)). 
\end{defn}

It is routine to verify that for a signature \(\sigma\), the class of \(\sigma\)-structures and homomorphisms forms a category (\ref{Categories defn}), and the notions of isomorphism from Definitions~\ref{Categories defn} and \ref{structure hom defn} coincide.

\speeddictone{43}{semigroup signature defn}{signature defn}
\begin{defn}[Semigroup signature, \Assumed{43}]\label{\dict{43 ref}}
We define \(\sigma_S\) to be the signature \((\{*\},\varnothing,\{(*, 2)\})\), where \(*\) is some fixed symbol. 
The choice of \(*\) is not very important, but for example one could define \(*:= (19,20,01,18)\).
\end{defn}

\speeddictone{44}{semigroup defn}{semigroup signature defn}
\begin{defn}[Semigroups and monoids, \Assumed{44}]\label{semigroup defn}
We say a \(\sigma_S\)-structure \(S\) is a \textit{semigroup} if for all \(a,b,c\in S\) we have
\(((a, b)*^{S},c)*^{S}=(a,(b,c)*^{S})*^{S}\).
If \(S\) is a semigroup, and \(a, b\in S\) then we will often write \(ab\) as an abbreviation of \((a, b)*^{S}\).

If \(S\) is a semigroup and \(s\in S\), then we define the maps \(\rho_s:S\to S\) and \(\lambda_s:S\to S\) by \((t)\rho_s = ts\), and \((t)\lambda_s = st\) for all \(t\in S\).
These maps are dependent on the semigroup \(S\) but when used the intended semigroup is to be inferred from context.

If \(S\) is a semigroup and there is an element \(e\in S\) such that \(ea=a e=a\) for all \(a\in S\), then we say that \(S\) is a \textit{monoid} and \(e\) is an \textit{identity} for \(S\). It is routine to verify that if \(M\) is a monoid then it has a unique identity element, we will denote this element by \(1_M\).
\end{defn}

\speeddictone{45}{inverse semigroup signature defn}{semigroup signature defn}
\begin{defn}[Inverse semigroup signature, \Assumed{45}]\label{\dict{45 ref}}
We define the \(\sigma_I\) to be the signature \((\{*, \iota\}, \varnothing, \{(*, 2), (\iota, 1)\})\) (extending the signature \(\sigma_S\)) where \(\iota\) is some fixed symbol.
The actual value of the symbol \(\iota\) is not very important, but for example one could define \(\iota:= (09,14,22,05,18,19,05)\).
\end{defn}

\speeddictone{46}{inverse semigroup defn}{inverse semigroup signature defn}
\begin{defn}[Inverse semigroup, \Assumed{46}]\label{\dict{46 ref}}
We say a \(\sigma_I\)-structure \(\mathbb{S} =(S, \sigma_I, I)\) is an \textit{inverse semigroup} if \((S, \sigma_S, I\restriction_{\{*\}})\) is a semigroup and for all \(a, b\in \mathbb{S}\) we have
\[ ((a)\iota^\mathbb{S})\iota^\mathbb{S}=a, \quad a\ (a)\iota^\mathbb{S}\ a=a, \quad a\ (a)\iota^\mathbb{S}\ b\ (b)\iota^\mathbb{S}=b\ (b)\iota^\mathbb{S}\ a\ (a)\iota^\mathbb{S}.\]
If \(\mathbb{S}\) is an inverse semigroup and \(a\in \mathbb{S}\), then we will usually denote \((a)\iota^\mathbb{S}\) by \(a^{-1}\) for brevity.
\end{defn}

\speeddictthree{136}{group of units defn}{structures defn}{semigroup defn}{group defn}
\begin{defn}[Units, \Assumed{136}]\label{group of units defn}
If \(S\) is a semigroup and \(g\in S\), then we say that \(g\) is a \textit{unit} if there is an element \(g^{-1}\in S\) (called the \textit{inverse} of \(g\)) such that for all \(s\in S\) we have \(g^{-1}gs=s=sgg^{-1}\).
It is routine to verify that the inverse of any unit is unique.
\end{defn}

\speeddicttwo{47}{group defn}{inverse semigroup defn}{group of units defn}
\begin{defn}[Groups, \Assumed{47}]\label{group defn}
It is routine to verify that if \(S\) is a semigroup, then the set of units \(\mathcal{U}(S)\) of \(S\) is a substructure of \(S\).
We consider \(\mathcal{U}(S)\) to be an inverse semigroup with the added unary operation being the map sending an element to its inverse.

We refer to the inverse semigroup \(\mathcal{U}(S)\) as the \textit{group of units} of \(S\), and we say that an inverse semigroup is a \textit{group} if it is the group of units of a semigroup.
Note that if \(G\) is a group and \(g\in G\), then \(gg^{-1}\) is an identity for \(G\).
\end{defn}

It is important to note that (as defined in this section), we do not consider groups or inverse semigroups to be the same objects as the semigroups obtained from them by removing their inverse operation.
It will be important to keep this distinction in mind as it will have a real impact when we start talking about topological groups and topological inverse semigroups.


\subsection{Topological structures}
\speeddictfive{48}{topological structures defn}{Categories defn}{product topology defn}{structures defn}{structure hom defn}{group defn}
\begin{defn}[Topological structures, \Assumed{48}]\label{topological structures defn}
A topological \(\sigma\)-structure \(\mathbb{S}\) is a pair \((S, \mathcal{T}_\mathbb{S})\), where 
\begin{enumerate}
    \item \(S\) is a \(\sigma\)-structure.
    \item \(\mathcal{T}_{\mathbb{S}}\) is a topology on \(S\).
    \item For all \(F\in F_\sigma\), the function \(F^\mathbb{S}\) is continuous with respect to \(\mathcal{T}_\mathbb{S}\) and the corresponding product topology.
    \item For all \(R\in R_\sigma\) the set \(R^\mathbb{S}\) is closed with respect to \(\mathcal{T}_\mathbb{S}\) and the corresponding product topology.
\end{enumerate}
 In this case we say that \(\mathcal{T}_{\mathbb{S}}\) is \textit{compatible} with \(S\). 
 We will use the established structure terminology and topology terminology for topological structures by ignoring the topological or structural part of the object as required. For example a topological group is a topological structure whose underlying structure is a group, and we consider a topological structure to be discrete if its associated topology is discrete.

We will sometimes think of a topological space as a topological structure in the signature with no symbols, and think of a structure without topology as a discrete structure.

If \(f:\mathbb{S}\to \mathbb{T}\) is a function between topological structures which is both a homeomorphism of spaces and an isomorphism of structures, then we call \(f\) a \textit{topological isomorphism}.
\end{defn}

Topological structures and continuous homomorphisms form a category (\ref{Categories defn}), this is essentially the intersection of the usual categories of structures and topological spaces.

\speeddictone{48}{substructures defn}{topological structures defn}
\begin{defn}[Substructures, \Assumed{48}]\label{\dict{48 ref}}
If \(\mathbb{S}\) is a topological \(\sigma\)-structure and \(\mathbb{M}\subseteq S\) is closed under the operations \(\{F^{\mathbb{S}}:F\in F_\sigma\}\), then we say that \(\mathbb{M}\) is a topological \textit{substructure} of \(\mathbb{S}\), and we view \(\mathbb{M}\) is a topological structure in the signature \(\sigma\) by restricting the operations, relations and topology in the natural fashion.

One should verify that such an object always satisfies the conditions required to be a topological structure.
This follows from the observation that if \(X\) is a topological space, \(A\subseteq X\) and \(V\) is a set, then the topology on \(A^V\) is the same regardless of whether we view it as a subspace of a power of \(X\) or a power of a subspace of \(X\).
\end{defn}

The following example of a topological structure is particularly important as many of the topological semigroups discussed in Part~\ref{semigroups section} are topological substructures of it.

\speeddictfive{168}{topological binary relations}{binary relations defns}{nbhds defn}{subbasis}{semigroup defn}{topological structures defn}
\begin{example}[Binary relation monoids are topological, \Assumed{168}]\label{topological binary relations}
If \(X\) is a set, then we define \(\mathcal{B}_X\) to be the monoid of binary relations from \(X\) to \(X\), with composition as the binary operation (recall Definition~\ref{binary relations defns}).
We define \(\mathcal{TB}_X\) to be the topology \(X\) generated by the sets of the form
\[U_{x, y}:= \makeset{b\in \mathcal{B}_X}{\((x, y)\in b\)}\]
for all \(x, y\in X\). With these definitions, \((\mathcal{B}_X, \mathcal{TB}_X)\) is a topological semigroup and moreover the map \(b\mapsto b^{-1}\) is continuous.
\end{example}
\begin{proof}
Let \(x, y\in X\) be arbitrary. We first need to show that the set
\[V:=\makeset{(a, b)\in \mathcal{B}_X}{\(ab\in U_{x, y}\)}\]
is open. Let \((a, b)\in V\) be arbitrary, to conclude that \(V\) is open we need only show that \(V\) is a neighbourhood of \((a, b)\).

As \(ab\in U_{x, y}\), we have \((x, y)\in ab\). Thus by the definition of \(ab\), there is some \(z\in X\) such that \((x, z)\in a\) and \((z, y)\in b\).
Thus \((a, b)\in U_{x, z}\times U_{z, y}\).
By the definition of composition of binary relations we have \(U_{x, z}U_{y,z}\subseteq U_{x, y}\), thus \(U_{x, z}\times U_{y, z}\subseteq V^\circ\). In particular, \(V\) is a neighbourhood of \((a,b)\).

It remains to show that the map \(b\mapsto b^{-1}\) is continuous. For all \(x, y\in X\) we have
\[U_{x, y}^{-1}= \makeset{b^{-1}\in \mathcal{B}_X}{\((x, y)\in b\)}
= \makeset{b\in \mathcal{B}_X}{\((y, x)\in b\)}
= U_{y, x}.\]
Thus the map \(b\to b^{-1}\) fixes the subbasis for \(\mathcal{TB}_X\) (setwise) and hence is continuous as required.
\end{proof}

\speeddicttwo{186}{full trans def}{product topology defn}{topological binary relations}
\begin{example}[The full transformation monoid, \Assumed{186}]\label{full trans def}
If \(X\) is a set, then we define the \textit{full transformation monoid} of \(X\) to be the set \(X^X\) with composition of functions as the operation.
We also define the \textit{pointwise topology} \(\mathcal{PT}_{X}\) on \(X^X\) to be the product topology obtained by viewing each copy of \(X\) with the discrete topology. The pair \((X^X, \mathcal{PT}_{X})\) is notably a topological subsemigroup of \((\mathcal{B}_X, \mathcal{TB}_X)\) (recall Example~\ref{topological binary relations}),
\end{example}

\speeddictthree{49}{product structures}{Products in Categories Defn}{product topology defn}{topological structures defn}
\begin{defn}[Product structures, \Assumed{49}]\label{product structures}
If \((\mathbb{S}_i)_{i\in I}\) are topological \(\sigma\)-structures, then we view \(\mathbb{P}:=\prod_{i\in I} (\mathbb{S}_i)_{i\in I}\) as a topological \(\sigma\)-structure as follows:

If \(f\in F_\sigma\) has \((f)\ar_\sigma = n\) and \((s_{i, 0})_{i\in I}, (s_{i, 1})_{i\in I}, \ldots (s_{i, n-1})_{i\in I}\in \mathbb{P}\) then
\[((s_{i, 0})_{i\in I}, (s_{i, 1})_{i\in I}, \ldots, (s_{i, n-1})_{i\in I})f^\mathbb{P} = ((s_{i, 0}, s_{i, 1}, \ldots, s_{i, n-1}) f^{\mathbb{S}_i})_{i\in I}.\]

If \(R\in R_\sigma\) has \((R)\ar_\sigma = n\) and \((s_{i, 0})_{i\in I}, (s_{i, 1})_{i\in I}, \ldots (s_{i, n-1})_{i\in I}\in \mathbb{P}\) then
\[((s_{i, 0})_{i\in I}, (s_{i, 1})_{i\in I}, \ldots, (s_{i, n-1})_{i\in I})\in R^{\mathbb{P}} \iff \left((s_{i, 0}, s_{i, 1}, \ldots, s_{i, n-1}) \in R^{\mathbb{S}_i} \text{ for all }i\in I\right).\]

It is routine to verify that the above construction is a product in the category of topological \(\sigma\)-structures and continuous homomorphisms (recall Definition~\ref{Products in Categories Defn}). 
\end{defn}

\speeddicttwo{700}{product poset}{types of binary relation defns}{product structures}
\begin{example}[Product of posets, \Assumed{700}]\label{product poset}
If \((X, \leq)\) is a partially ordered set (which we view as a structure in the signature with one binary relation symbol) and \(n\in \N\), then the relation on \(X^n\) as given in Definition~\ref{product structures} is also a partial order and is defined by
\[(x_0, x_1, \ldots, x_{n-1})\leq (y_0, y_1, \ldots, y_{n-1}) \iff (x_i\leq y_i \text{ for all } i<n)\]
where \(x_0, x_1, \ldots, x_{n-1},y_0, y_1, \ldots, y_{n-1} \in X\) are arbitrary.
\end{example}

\speeddictfour{50}{congruences and quotients defn}{binary relations defns}{types of binary relation defns}{substructures defn}{product structures}
\begin{defn}[Congruences and quotients, \Assumed{50}]\label{\dict{50 ref}}
If \(\mathbb{S}\) is a structure, and \(\sim\) is an equivalence relation on \(\mathbb{S}\), then we say that \(\sim\) is a \textit{congruence} on \(\mathbb{S}\) if \(\sim\) is a substructure of \(\mathbb{S}^2\). In this case we view the set \(Q:=\mathbb{S}/\sim\) as a \(\sigma\)-structure (called the \textit{quotient} of \(\mathbb{S}\) by \(\sim\)) as follows:

If \(f\in F_\sigma\), \(R\in R_\sigma\) and \(s_{0}, s_1, \ldots, s_{n-1}\in \mathbb{S}\) then
\[([s_0]_\sim, [s_1]_\sim, \ldots, [s_{n-1}]_\sim)f^{Q} = [(s_0, s_1,\ldots, s_{n-1})f^{\mathbb{S}}]_\sim,\]
\[([s_0]_\sim, [s_1]_\sim, \ldots, [s_{n-1}]_\sim)\in R^{Q} \iff (([s_0]_\sim\times[s_1]_\sim \times\ldots \times [s_{n-1}]_\sim) \cap R^{\mathbb{S}}) \neq \emptyset.\]
This is well-defined precisely because \(\sim\) is a congruence.

For an important example of this, note that if \(\mathbb{S}, \mathbb{T}\) are \(\sigma\)-structures and \(f:\mathbb{S}\to \mathbb{T}\) is a homomorphism, then \(\ker(f)\) (recall Definition~\ref{binary relations defns}) is a congruence on \(\mathbb{S}\).
Moreover the induced map \(f^*:\mathbb{S}/\operatorname{ker}(f) \to \mathbb{T}\) is an injective homomorphism.
When this induced map \(f^*\) is an isomorphism, we call \(f\) a \textit{quotient map}.
\end{defn}

\subsection{Polish semitopological groups are nice}

In this subsection we display a proof of a key result from the literature:
That comparable Polish topologies compatible with a group are always equal (Theorem~\ref{comparable Polish group topologies theorem}).
This fact is the primary reason we are interested in Polish spaces as opposed to any other class.
While the topological assumptions are quite strong, this result is very useful when investigating the potential topologies on a group. 
While this result does not hold for semigroups in general or even inverse semigroups, as we will see in Part~\ref{semigroups section}, it will still be useful to us in these more general algebraic contexts.

\speeddictfour{169}{semitopological defn}{semigroup defn}{inverse semigroup defn}{topological structures defn}{topological binary relations}
\begin{defn}[Semitopological semigroups, \Assumed{169}]\label{semitopological defn}
We define a \textit{semitopological semigroup} to be a pair \((S, \mathcal{T})\) such that
\begin{enumerate}
    \item \(S\) is a semigroup.
    \item \(\mathcal{T}\) is a topology on \(S\).
    \item For all \(s\in S\), the maps \(\lambda_s\) and \(\rho_s\) are continuous (recall Definition~\ref{semigroup defn}).
\end{enumerate}
Similarly we define a \textit{semitopological inverse semigroup} to be a pair \((I, \mathcal{T})\) such that \((S, \mathcal{T})\) is a semitopological semigroup (where \(S\) the semigroup obtained by removing the unary operation of \(I\)).
In these cases we say that \(\mathcal{T}\) is \textit{semicompatible} with \(S\) or \(I\).

For each \(s\in S\) let \(c_{S, s}:S\to S\) be the constant map with value \(s\).
As
\[\lambda_s = \langle c_{S, s}, \text{id}_S\rangle_{S^2} \circ *^S \quad \text{and}\quad \rho_s = \langle \text{id}_S, c_{S, s}\rangle_{S^2}\circ *^S \quad \text{(recall Definition~\ref{Products in Categories Defn})}\]
it follows that all topologies compatible with \(S\) are also semicompatible with \(S\).
\end{defn}

The contents of Lemma~\ref{nice Polish subspaces lemma}, Lemma~\ref{more nice Polish subspaces lemma} and Corollary~\ref{even more nice Polish subspaces cor} are well-known. 
However how these facts are stated and grouped tends to vary depending on the author's needs. 
For example in \cite{kechris2012classical}, they can be found as (parts of) Proposition 3.3, Theorem 3.11, Lemma 13.2 and Lemma 13.3.

\speeddicttwo{209}{nice Polish subspaces lemma}{subspaces}{Polish space defn}
\begin{lemma}[Nice Polish subspaces, \Assumed{209}]\label{nice Polish subspaces lemma}
If \((X, \mathcal{T})\) is a Polish topological space and \(U\subseteq X\) is open or closed, then the topological space \((U, \mathcal{T}\restriction_{U})\) is Polish.
\end{lemma}
\begin{proof}
First observe that a subspace of a second countable space is always second countable so we need only show that the subspaces are compatible with complete metrics.
Let \(d\) be a complete metric compatible with \((X, \mathcal{T})\).
First suppose that \(U\) is closed. In this case any Cauchy sequence of \((U, d\restriction_{U\times U})\) has a limit in \(X\), which must be an element of \(U\) due to \(U\) being closed. We next suppose that \(U\) is open. We define a map \(d': U\to \mathbb{R}\) by
\[(x)d'= \inf\left(\makeset{(x, v)d}{\(v\in X\backslash U\)}\right).\]
Note that if \(\varepsilon >0\), \(x, y\in U\) and \((x, y)d\leq \varepsilon\) then \(|(x)d'- (y)d'|\leq \varepsilon\), so the map \(d'\) is continuous. 

As \(X\backslash U\) is closed, it follows that \(\im(d')\subseteq (0, \infty)\).
We define a new map \(\delta: U\times U \to \mathbb{R}\) by
\[(x, y)\delta = (x, y)d + \left|\frac{1}{(x)d'} - \frac{1}{(y)d'}\right|.\]
It is routine to verify that \(\delta\) satisfies the first two conditions of being a metric on \(U\).
For the third condition, note that if \(x, y, z\in U\), then
\begin{align*}
 (x, z)\delta&= (x, z)d + \left|\frac{1}{(x)d'} - \frac{1}{(z)d'}\right|\\
 &\leq (x, y)d + (y, z)d + \left|\frac{1}{(x)d'} - \frac{1}{(y)d'}\right| +\left|\frac{1}{(y)d'} - \frac{1}{(z)d'}\right|=(x, y)\delta + (y, z)\delta.
\end{align*}
If \((x_i)_{i\in \N}\) is a sequence in \(U\), then if \((x_i)_{i\in \N}\) is Cauchy with respect to \(\delta\), then \((x_i)_{i\in \N}\) is Cauchy with respect to \(d\) and hence converges with respect to \(\mathcal{T}\restriction_{U}\).
Similarly if \((x_i)_{i\in \N}\) converges to \(x\in U\) with respect to \(\mathcal{T}\restriction_{U}\), then
\[\lim_{i\in \N} (x, x_i)\delta =\lim_{i\in \N} \left((x, x_i)d + \left|\frac{1}{(x)d'} - \frac{1}{(x_i)d'}\right|\right) = 0 .\]
Thus the metric \(\delta\) is complete and compatible with \(\mathcal{T}\restriction_{U}\) as required.
\end{proof}

\speeddictthree{221}{more nice Polish subspaces lemma}{Products in Categories Defn}{product topology defn}{nice Polish subspaces lemma}
\begin{lemma}[More nice Polish spaces, \Assumed{221}]\label{more nice Polish subspaces lemma}
Suppose that \(((X_i, \mathcal{T}_i))_{i\in \N}\) are Polish topological spaces and consider the space \((I, \mathcal{T}_I)\) where \(I\) is the intersection of the sets \(X_i\) and \(\mathcal{T}_I\) is the topology on \(I\) generated by all of the subspace topologies \(\mathcal{T}_i\restriction_{I}\).

The space \(\prod_{i\in \N} X_i\) is also Polish (with the usual product topology).
Moreover, if there is a Hausdorff space \((X, \mathcal{T})\) such that for all \(i\in \N\), we have \(X_i\subseteq X\) and \(\mathcal{T}\restriction_{X_i}\subseteq \mathcal{T}_i\), then \(\mathcal{T}_I\) is Polish.
\end{lemma}
\begin{proof}
For each \(i\in \N\), let \(d_i\) be a metric compatible with \(X_i\). Then for each \(i\in \N\), let \(d_i': X_i \times X_i \to \mathbb{R}\) be defined by \((x, y)d_i' = \min(\{(x, y)d_i, \frac{1}{2^i}\})\).
We now define a metric \(d\) on \(\prod_{i\in \N} X_i\) by
\[(x, y)d = \max_{i\in \N} ((x)\pi_i, (y)\pi_i)d_i'.\]
It is routine to verify that this is a well defined complete metric compatible with the product topology on \(\prod_{i\in \N} X_i\). Thus \(\prod_{i\in \N} X_i\) is Polish. Consider the subspace
\[D:= \makeset{x\in \prod_{i\in \N} X_i}{\((x)\pi_i = (x)\pi_j\)\\ for all \(i, j\in \N\)}=  \makeset{x\in \prod_{i\in \N} X_i}{\((x)\langle \pi_i, \pi_j \rangle_{X^2}\in \text{id}_X\)\\ for all \(i, j\in \N\)}=\intersection{i, j\in \N} (\text{id}_X)\langle \pi_i, \pi_j \rangle_{X^2}^{-1}\]
of \(\prod_{i\in \N} X_i\). As \((X, \mathcal{T})\) is Hausdorff, the space \(D\) is a closed subspace of the Polish space \(\prod_{i\in \N} X_i\).
By Lemma~\ref{nice Polish subspaces lemma}, it follows that \(D\) is a Polish space.
It follows from the definition of \(D\) that \(\langle (\text{id}_I)_{i\in \N}\rangle_{\prod_{i\in \N} X_i}:I \to  D\) is a continuous bijection.
As \(\pi_i\) is the inverse of this map for every choice of \(i\in \N\), it follows that this is a homeomorphism from \(I\) to a Polish space. The result follows.
\end{proof}

\speeddictthree{233}{even more nice Polish subspaces cor}{disjoint union topology defn}{nice Polish subspaces lemma}{more nice Polish subspaces lemma}
\begin{corollary}[Even more nice Polish spaces, \Assumed{221}]\label{even more nice Polish subspaces cor}
Suppose that \((X, \mathcal{T})\) is a Polish topological space and \((V_i)_{i\in \N}\) are closed subsets of \(X\). In this case the subspace \(\intersection{i\in \N} X\backslash V_i\) of \((X, \mathcal{T})\) is Polish, and the topology generated by \(\mathcal{T}\) and all of the sets \((V_i)_{i\in \N}\) is also Polish.
\end{corollary}
\begin{proof}
From Lemma~\ref{nice Polish subspaces lemma} each of the spaces \(X\backslash V_i\) is Polish, so their intersection is also Polish from Lemma~\ref{more nice Polish subspaces lemma}.

Let \(i\in \N\) be arbitrary. By Lemma~\ref{nice Polish subspaces lemma}, each of \(V_i\) and \(X\backslash V_i\) is a Polish subspace of \(X\).
As the topology \(\mathcal{T}_i\) generated by \(\mathcal{T}\) and \(V_i\) is equal to the disjoint union topology of \(V_i\) and \(X\backslash V_i\), it follows that it is Polish. Thus from Lemma~\ref{more nice Polish subspaces lemma}, it follows that the topology generated by all of the topologies \(\mathcal{T}_i\) is also Polish as required.
\end{proof}

\speeddicttwo{203}{almost open defn}{topological space defn}{almost containment defn}
\begin{defn}[Almost open sets, \Assumed{203}]\label{almost open defn}
If \(X\) is a topological space and \(S\subseteq X\), then we say that \(S\) is \textit{almost open} if there is an open set \(U\) such that \(U\approx_X S\) (recall Definition~\ref{almost containment defn}). 
\end{defn}

Note that there is no relation between algebraic \(\sigma\) structures as defined earlier and \(\sigma\)-algebras from the following definition.

\speeddicttwo{204}{sigma algebra defn}{power set defn}{countable defn}
\begin{defn}[\(\sigma\)-algebras, \Assumed{204}]\label{sigma algebra defn}
Suppose that \(X\) is a set and \(\Sigma\subseteq \mathcal{P}(X)\) is such that
\begin{enumerate}
    \item If \(A\in \Sigma\), then \(A^c=X\backslash A\in \Sigma\).
    \item If \(A\subseteq \Sigma\) is countable, then \(\union{} A \in \Sigma\).
    \item The empty set is an element of \(\Sigma\).
\end{enumerate}
In this case we say that \(\Sigma\) is a \textit{\(\sigma\)-algebra} on \(X\). Recall that we allow empty unions (Definition~\ref{union and intersection convensions}) so the third condition is technically redundant.
\end{defn}

\speeddicttwo{208}{almost sigma algebra lemma}{almost open defn}{sigma algebra defn}
\begin{lemma}[cf Proposition 8.22 of \cite{kechris2012classical}, Almost open \(\sigma\)-algebra, \Assumed{208}]\label{almost sigma algebra lemma}
If \(X\) is a topological space, then the almost open subsets of \(X\) form a \(\sigma\)-algebra.
\end{lemma}
\begin{proof}
Let \(A\subseteq X\) be almost open. Let \(U\) be an open subset of \(X\) such that \(U\approx_X A\).
Note that
\[U^c=U^{c\circ} \cup (U^{c}\backslash U^{c\circ})=U^{c\circ} \cup (U^{c}\cap U^{c\circ c}) = U^{c\circ} \cup (U^{c}\cap U^{cc-})=U^{c\circ} \cup (U^{c}\cap U^{-}).\]
As \(U^{c}\cap U^{-}\) is closed and \((U^{c}\cap U^{-})^\circ= U^{c\circ}\cap U^{-\circ} = U^{-c}\cap U^{-\circ}=\varnothing\), it follows that \(U^{c}\cap U^{-}\approx_X \varnothing\).
Thus \(A^c \approx_X U^c=U^{c\circ} \cup (U^{c}\cap U^{-})\approx_X U^{c\circ}\) and so \(A^c\) is almost open.

Suppose that \(A\) is a countable collection of almost open sets.
If \(A\) is empty then \(\union{} A = \varnothing\) is almost open, otherwise let \(A= \makeset{A_i}{\(i\in \N\)}\).
For each \(i\in \N\), let \(U_i\) be open and let \(C_i\) be comeagre such that \(U_i \cap C_i= A_i\cap C_i\). It follows that
\[\union{}A \approx_X \left(\intersection{i\in \N} C_i\right)\cap \left(\union{i\in \N}A_i \right)= \left(\intersection{i\in \N} C_i\right)\cap \left(\union{i\in \N}U_i \right)\approx_X \union{i\in \N}U_i.\]
As \(\union{i\in \N}U_i\) is open, the result follows.
\end{proof}

\speeddicttwo{205}{analytic defn}{continuous maps and homeomorphisms defn}{Polish space defn}
\begin{defn}[Analytic spaces, \Assumed{205}]\label{analytic defn}
We say that a topological space \(X\) is \textit{analytic} if it is the image of a Polish space under a continuous map.
\end{defn}

\speeddictfive{207}{nice baire space lemma}{power set defn}{continuous maps and homeomorphisms defn}{Polish space defn}{full trans def}{more nice Polish subspaces lemma}
\begin{lemma}[cf. Theorem 7.9 of \cite{kechris2012classical}, \(\N^\N\) is a nice space, \Assumed{207}]\label{nice baire space lemma}
The topological space \((\N^\N, \mathcal{PT}_\N)\) from Example~\ref{full trans def} is Polish. 
Moreover if \(X\) is a non-empty Polish topological space, then \(X\) is the image of a continuous map from \((\N^\N, \mathcal{PT}_\N)\).
\end{lemma}
\begin{proof}
The fact that \(\N^\N\) is Polish follows from the fact that \(\N\) is Polish together with Lemma~\ref{more nice Polish subspaces lemma}.

Let \(d\) be a complete metric compatible with \(X\) and let \(B:= \makeset{U_i}{\(i\in \N\)}\) be a countable basis for \(X\) consisting of non-empty sets.
For each \(i\in \N\), we choose a fixed \(d_i\in U_i\).
As \(B\) is a basis for \(X\), it follows that \(D:=\makeset{d_i}{\(i\in \N\)}\) is dense in \(X\) and so for all \(x\in X\), there are sequences of points in \(D\) converging to \(x\).

 We say that a sequence \((x_i)_{i\in \N}\) in \(X\) is strongly Cauchy if for all \(i\in \N\) we have \((x_i, x_{i+1})d\leq \frac{1}{2^i}\).
It is routine to verify that strongly Cauchy sequences are Cauchy.

We now define
\[C:=\makeset{f\in \N^\N}{\((d_{(i)f})_{i\in \N}\) is a strongly Cauchy sequence}.\]
From the definition of a strongly Cauchy sequence, it follows that the subspace \(C\) of \(\N^\N\) is closed.
As strongly Cauchy sequences are Cauchy and \(d\) is complete, it follows that the map \(\phi: C\to X\) given by
\[(f)\phi = \lim_{i\to \infty} d_{(i)f}\]
is well defined. It is routine to verify that the map \(\phi\) is also continuous.
Moreover as \(D\) is dense in \(X\), the map \(\phi\) is surjective as well.

Let \(c:\mathcal{P}(\N^\N)\backslash \{\varnothing\} \to \N^\N\) be such that \((S)c\in S\) for all \(S\in \mathcal{P}(\N^\N)\backslash \{\varnothing\}\).
For each \(f\in \N^\N\backslash C\), we define \(m_f:=\max\left(\makeset{m\in \N}{there exists \(h\in C\) with \(h\restriction_m=f\restriction_m\)}\right)\) (note that this is only well defined because \(C\) is closed). We then define a map \(\psi: \N^\N \to C\) as follows
\[(f)\psi = \left\{\begin{array}{lr}
f     &\text{ if }f\in C  \\
\left(\makeset{g\in C}{\(g\restriction_{m_f}=f\restriction_{m_f}\)}\right)c     & \text{ if }f\not\in C
\end{array}\right\}.\]
This map \(\psi:\N^\N \to C\) is surjective by definition, it is also routine to verify that \(\psi\) is continuous.
It follows that \(\psi\phi\) is a continuous surjection from \(\N^\N\) to \(X\) as required.
\end{proof}

\speeddictfive{206}{nice analytic lemma}{subspaces}{Hausdorff defn}{sigma algebra defn}{analytic defn}{nice baire space lemma}
\begin{lemma}[cf. Theorem 14.7 of \cite{kechris2012classical}, Lusin Separation Theorem, \Assumed{206}]\label{nice analytic lemma}
Suppose that \((X, \mathcal{T})\) is a Hausdorff topological space and \(\Sigma\) is a \(\sigma\)-algebra on \(X\) with \(\mathcal{T}\subseteq \Sigma\).
If \(\{U, V\}\) is a partition of \(X\) into analytic subspaces of \((X, \mathcal{T})\), then \(\{U, V\}\subseteq \Sigma\).
\end{lemma}
\begin{proof}
If \(U=\varnothing\) or \(V=\varnothing\), then the result is clear.
Otherwise both \(U\) and \(V\) are images of non-empty Polish spaces under continuous maps.
Thus from Lemma~\ref{nice baire space lemma}, there are continuous maps \(\phi_U:\N^\N \to X, \phi_V:\N^\N \to X\) with \(\im(\phi_U)= U\) and \(\im(\phi_V) = V\).

We say \(A, B\subseteq X\) are \textit{separated by }\(\Sigma\) if there are \(A', B'\in \Sigma\) such that \(A\subseteq A', B\subseteq B'\) and \(A'\cap B'= \varnothing\). It suffices to show that \(U\) and \(V\) are separated by \(\Sigma\).

\underline{Claim:} Suppose that \((U_{i})_{i\in \N}, (V_i)_{i\in \N}\) are sequences of subsets of \(X\) such that \(\union{i\in \N}U_i\) and \(\union{i\in \N}V_i\) are not separated by \(\Sigma\). In this case there are \(i, j\in \N\) such that \(U_i\) and \(V_j\) are not separated by \(\Sigma\).\\
\underline{Proof of Claim:}
Suppose for a contradiction that \(U_i\) and \(V_j\) are separated by \(\Sigma\) for all \(i, j\in \N\).
For each \(i, j\in \N\) let \(U_{i, j}', V_{i, j}'\in \Sigma\) be disjoint such that \(U_i\subseteq U_{i, j}'\) and \(V_j\subseteq V_{i, j}'\).
We define
\[U':=  \union{i\in \N}\intersection{j\in \N} U_{i, j}' \quad \text{ and }\quad V':=  \union{j\in \N}\intersection{i\in \N} V_{i, j}'.\]
We have that \(U', V'\in \Sigma\) are disjoint, \(\union{i\in \N}U_i \subseteq U' \) and \(\union{i\in \N}V_i \subseteq V' \).
As \(U'\cap V' = \varnothing\), it follows that \(\union{i\in \N}U_i\) and \(\union{i\in \N}V_i\) are separated by \(\Sigma\).
This is a contradiction. \(\diamondsuit\)\\
Suppose for a contradiction are \(U\) and \(V\) are not separated by \(\Sigma\). For each \(n\in \N\), we define \(f_n:\{0, 1, \ldots, n-1\}\to \N\), and \(g_n:\{0, 1, \ldots, n-1\}\to \N\) as follows:
\begin{enumerate}
    \item \(f_0=g_0= \varnothing\).
    \item If \(f_i\) is already defined and the sets 
    \[\left(\makeset{f\in \N^\N}{\(f_i\subseteq f\)}\right)\phi_U\quad \text{ and }\quad \left(\makeset{g\in \N^\N}{\(g_i\subseteq g\)}\right)\phi_V\]
    are not separated by \(\Sigma\), then by the claim choose \(f_{i+1}, g_{i+1}\) such that \(f_i\subseteq f_{i+1}\), \(g_i\subseteq g_{i+1}\) and \(\left(\makeset{f\in \N^\N}{\(f_{i+1}\subseteq f\)}\right)\phi_U\) and \(\left(\makeset{g\in \N^\N}{\(g_{i+1}\subseteq g\)}\right)\phi_V\) are not separated by \(\Sigma\).
\end{enumerate}
We can then define \(f:= \union{i\in \N} f_i\) and \(g:= \union{i\in \N} g_i\).
As \(X\) is Hausdorff and \((f)\phi_U \neq (g)\phi_V\) (\(\phi_U\) and \(\phi_V\) has disjoint images), we can choose open \(U_f\in \nbhd{X}{(f)\phi_U}\) and \(V_g\in \nbhd{X}{(g)\phi_V}\) such that \(U_f\cap V_g = \varnothing\). As \(\phi_U, \phi_V\) are both continuous, there is some \(N\in \N\) such that 
\[(U_f)\phi_U^{-1}\supseteq \makeset{f\in \N^\N}{\(f_N\subseteq f\)}\quad\text{ and }\quad (V_g)\phi_V^{-1}\supseteq \makeset{g\in \N^\N}{\(g_N\subseteq g\)}.\]
As \(U_f, V_g\in \Sigma\), it follows that \(\left(\makeset{f\in \N^\N}{\(f_N\subseteq f\)}\right)\phi_U\) and \(\left(\makeset{g\in \N^\N}{\(g_N\subseteq g\)}\right) \phi_V\) are separated by \(\Sigma\), this is a contradiction.
\end{proof}

\speeddictthree{210}{comparable Polish group topologies theorem}{Polish space defn}{topological structures defn}{semitopological defn}
\begin{theorem}[Comparable Polish group topologies, \Assumed{210}]\label{comparable Polish group topologies theorem}
Suppose that \(G,H\) are groups, and \(\mathcal{T}_G, \mathcal{T}_H\) are Polish topologies semicompatible with \(G, H\) respectively.
If \(\phi:G\to H\) is an isomorphism and \(\phi\) is continuous with respect to \(\mathcal{T}_G\) and \(\mathcal{T}_H\), then \(\phi^{-1}\) is also continuous.
\end{theorem}
\begin{proof}
We need only show that if \((x_i)_{i\in \N}\) is a sequence of elements of \(H\) which converge to some \(x\in H\) with respect to \(\mathcal{T}_H\), and \(S:= \{x\}\cup \makeset{x_i}{\(i\in \N\)}\), then \(\phi^{-1}\restriction_{S}\) is continuous.

Let \(B = \makeset{U_i}{\(i\in \N\)}\) be a countable basis for the topology \(\mathcal{T}_G\).
By Lemma~\ref{nice Polish subspaces lemma}, the sets \(U_i\) and \(X\backslash U_i\) are analytic (with respect to \(\mathcal{T}_G\)) for each \(i\in \N\).
Therefore as \(\phi\) is continuous, it follows that \((U_i)\phi\) and \((X\backslash U_i)\phi\) are analytic with respect to \(\mathcal{T}_H\).
From Lemmas~\ref{almost sigma algebra lemma} and \ref{nice analytic lemma}, it follows that \(U_i\) and \(X\backslash U_i\) are almost open with respect to \(\mathcal{T}_H\). 
Thus there is a comeagre subset \(C\) of \(H\) such that for all \(i\in \N\), the set \(U_i\cap C\) is open in \(\mathcal{T}_H\restriction_{C}\).
In particular \(\phi^{-1}\restriction_{C}\) is continuous.

As \(C\) is comeagre in the semitopological group \((H, \mathcal{T}_H)\), it follows that
\(C\cap Cx^{-1}\cap \left(\intersection{i\in \N} Cx_i^{-1}\right)\)
is also comeagre in \((H, \mathcal{T}_H)\).
So from Theorem~\ref{baire category theorem}, it follows that there is some \(h\in C\cap Cx^{-1}\cap \left(\intersection{i\in \N} Cx_i^{-1}\right)\).
Thus \((S)\lambda_h\subseteq C\). It follows that \(\phi^{-1}\restriction_{S} =\lambda_h\restriction_{S}\circ \phi^{-1}\restriction_{C} \circ \lambda_{(h^{-1})\phi^{-1}}\), and is thus continuous as required.
\end{proof}

\subsection{Words and terms}
In this subsection we introduce words and terms. 
Words are used constantly throughout Part~\ref{nv section} and terms give us access to a large collection of continuous maps to work with given a topological structure (they are also used to define the Zariski topology in Part~\ref{semigroups section}). 

\speeddictfour{97}{words defn}{function sets def}{binary relations defns}{types of binary relation defns}{semigroup defn}
\begin{defn}[Words, \Assumed{97}]\label{words defn}
If \(X\) is a set, then we denote the \textit{free monoid} on \(X\) by \(X^*\). The universe of this monoid is the set
\[X^* := \union{n\in \N} X^n.\]
If \(w\in X^* \cup X^\N\), then we call \(w\) a \textit{word} or \textit{string} in the \textit{alphabet} \(X\).
In this context we will refer to the elements of \(X\) as letters.
We denote by \(|w|\) the \textit{length} of the word \(w\), that is the value \(n\in \N\cup \{\aleph_0\}\) such that \(w\in X^n\) (conveniently this is the same as the cardinality of \(w\) so this notation is unambiguous).
If \(n\in \N\) then we write \(w\restriction_n\) to mean the restriction of the function \(w\) to the set \({\{0, 1, \ldots, n-1\}}\), we call such restrictions the \textit{prefixes} of \(w\). The set \(X^*\cup X^\N\) is partially ordered by the relation \(v\leq w\) if \(v\) is a prefix of \(w\) (this is the same as the \(\subseteq\) relation).

If \(v=(v_0, v_1, \ldots, v_{|v|-1})\in X^*\) and \(w=(w_0, w_1, \ldots)\in X^* \cup X^\N\), then the \textit{concatenation} of \(v\) and \(w\) is defined by
\[v w = (v_0, v_1, \ldots, v_{|v|-1}, w_0, w_1, \ldots).\]

Restricting concatenation to \(X^*\) gives the binary operation of \(X^*\) as a monoid. The identity of \(X^{*}\) is the element of length \(0\) which we call the \textit{empty word} and denote by \(\varepsilon\).
Although technically \(X\neq X^1\) (Definition~\ref{function sets def}), we will often for convenience treat letters and words of length \(1\) as the same objects.

It is routine to verify that concatenation is always \textit{cancellative}, that is if \(v, u,  w\) are finite words then
\[v u = v w \Rightarrow u = w, \quad\text{ and } \quad u v = w v \Rightarrow u  = w.\]
It follows from this that maps \(\rho_w\) and \(\lambda_w\) (using the monoid \(X^*\)) are injective for all \(w\in X^*\).
When infinite words are being used, we extend the domain of \(\lambda_w\) to allow for infinite words (note that \(\lambda_w\) is still injective).
\end{defn}

\speeddicttwo{98}{terms defn}{signature defn}{words defn}
\begin{defn}[Terms of a signature, \Assumed{98}]\label{terms defn}
For this definition we will need sets to represent the symbols comma, open bracket and closed bracket. 
To avoid confusion with the symbols being used for their English function, we use 、to denote the set representing a comma, we use「 to denote the set representing an open bracket and we use 」to done the set representing a closed bracket.
The actual choice of these sets is unimportant as long as they are distinct from each other and anything else they could be confused with. For example one could define
\begin{align*}
    \text{、} &:= (03, 15, 13, 13, 01)\\
    \text{「} &:= (15, 16, 05, 14, 00, 02, 18, 01, 03, 11, 05, 20)\\
    \text{」} &:= (03, 12, 15, 19, 05, 04, 00, 02, 18, 01, 03, 11, 05, 20).
\end{align*}
If \(V\) is a set and \(\sigma\) is a signature, then we inductively define a \textit{term} of \(\sigma\) over the variable set \(V\) to be a string in the alphabet \(\{\text{、}, \text{「}, \text{」}\}\cup F_\sigma \cup V\) which is one of:
\begin{enumerate}
    \item The string \(x\) for some \(x\in V\).
    \item The string 「\(t_0\)、\(t_1\)、\(\ldots\) 、\(t_n\)」\(F\) where \(F\in F_\sigma\) has arity \(n\) and \(t_0, t_1,\ldots, t_n\) are terms.
\end{enumerate}
\end{defn}

\speeddictfour{99}{terms operations defn}{Products in Categories Defn}{product topology defn}{topological structures defn}{terms defn}
\begin{defn}[Term operations, \Assumed{99}]\label{term operations defn}
If \(\mathbb{X}\) is a \(\sigma\)-structure, and \(V\) is a set (called a \textit{variable set}), then for each term \(t\) of \(\sigma\) over \(V\) we will define a corresponding \textit{term operation} \(t_{V}^\mathbb{X}:\mathbb{X}^V \to \mathbb{X}\).
We do this inductively on the length of \(t\) as follows:
\begin{enumerate}
    \item If \(t\in V\), then we define \(t_V^\mathbb{X}\) to be the projection map \(\pi_t\).
    \item If \(t=\)「\(t_0\)、\(t_1\)、\(\ldots\) 、\(t_n\)」\(F\) where \(F\in F_\sigma\) has arity \(n\) and \(t_0, t_1, \ldots, t_n\) are terms, then we define \(t_{V}^\mathbb{X}:\mathbb{X}^V\to \mathbb{X}\) by
    \[t_V^\mathbb{X} = \langle {t_0}_V^\mathbb{X},{t_1}_V^\mathbb{X}, \ldots,{t_{n-1}}_V^\mathbb{X}\rangle_{\mathbb{X}^n} \circ F^\mathbb{X}\quad \text{(recall Definition~\ref{Products in Categories Defn}).}\]
    
\end{enumerate}
It is routine to check that any term of \(\sigma\) over \(V\) is considered exactly once in exactly one of the above cases. Note that by construction, if \(\mathbb{X}\) is a topological structure, then all term operations of \(\mathbb{X}\) over any variable set are continuous (using the appropriate product topologies).
\end{defn}

\speeddictone{600}{term operations in groups and inverse semigroups}{term operations defn}
\begin{example}[Term operations in semigroups and inverse semigroups, \Assumed{600}]\label{term operations in groups and inverse semigroups}
It is routine to verify that if \(S\) is a topological semigroup and \(V\) is a variable set, then the term operations \(t:S^V\to S\) are precisely the maps of the form
\[((s_v)_{v\in V})t = s_{v_0}s_{v_1}s_{v_2}\ldots s_{v_{n-1}}\]
where \(n\in \N\) and \(v_0, v_1, \ldots, v_{n-1}\in V\). In particular these maps are continuous.
Similarly if \(S\) is a topological inverse semigroup and \(V\) is a variable set, then the term operations \(t:S^V\to S\) are precisely the maps of the form
\[((s_v)_{v\in V})t = s_{v_0}^{i_0}s_{v_1}^{i_1}\ldots s_{v_{n-1}}^{i_{n-1}}\]
where \(n\in \N\), \(i_0, i_1, \ldots, i_{n-1}\in \{1, -1\}\) and \(v_0, v_1, \ldots, v_{n-1}\in V\).

\end{example}

\section{Nice actions}
In this section we introduce actions of semigroups on objects from categories. We also provide proofs of the remaining key results from the literature motioned earlier. 
We first show that all homomorphisms from \(\Aut(2^\N)\) or \(\Sym(\N)\) to second countable groups are continuous, and we conclude by proving Rubin's Theorem which given a nice group allows one to construct a canonical action of that group on a nice topological space. 

\speeddictfive{54}{automorphism groups defn}{binary relations defns}{Categories defn}{semigroup defn}{group defn}{topological structures defn}
\begin{defn}[Endomorphisms and automorphisms, \Assumed{54}]\label{automorphism groups defn}

If \(\mathcal{C}\) is a category, \(O\) is an object of \(\mathcal{C}\), and the class of morphisms in \(\mathcal{C}\) with \(O\) as both source and target form a set, then we denote this set by \(\End(O)\) (the category being used will be given by the context).

We give \(\End(O)\) a monoid structure by defining the product of \(f, g\in \End(O)\) to be \((f, g)\circ_{\mathcal{C}}\). We then define \(\Aut(O)\) to be the group of units of \(\End(O)\) (recall Definition~\ref{group defn}).
\end{defn}

\speeddicttwo{83}{group actions defn}{product structures}{automorphism groups defn}
\begin{defn}[Actions, \Assumed{83}]\label{group actions defn}
If \(S\) is a topological semigroup, \(\mathbb{X}\) is a topological structure and \(a:\mathbb{X}\times S \to \mathbb{X}\) is a function, then we say that \(a\) is an \textit{action}
(of \(S\) on \(\mathbb{X}\)) if the map \(\phi_a:S\to \End(\mathbb{X})\) given by
\[(x)((f)\phi_a)=(x,f)a\]
is a well-defined homomorphism of semigroups. We define actions analogously for groups using \(\Aut(\mathbb{X})\) in place of \(\End(\mathbb{X})\). Recall that as per Definition~\ref{topological structures defn}, we consider both structures with no topology and topological spaces with no structure as topological structures. In the case that an action being used is clear from context we will sometimes write \(xf\) instead of \((x, f)a\). We say that the action \(a\) is:
\begin{enumerate}
    \item \textit{Continuous} if the map \(a:\mathbb{X} \times S \to \mathbb{X}\) is continuous.
    \item \textit{Transitive} if for all \(x, y\in \mathbb{X}\), there is \(f\in S\) with \((x, f)a= y\).
    \item \textit{Faithful} if the map \(\phi_a\) is injective.
\end{enumerate}

\end{defn}

\speeddictone{55}{powers of actions defn}{group actions defn}
\begin{defn}[Powers of actions, \Assumed{55}]\label{\dict{55 ref}}
If \(a\) is an action of a group \(G\) on a set \(X\), then we define \(a^n\) to be the action of \(G\) on \(X^n\) given by:
\[((x_0, x_1 \ldots x_{n-1}), g)a^n = ((x_0,g)a, (x_1,g)a, \ldots, (x_{n-1},g)a).\]
This notation always takes precedent over the notation in Definition~\ref{function sets def}.
\end{defn}

\speeddictone{56}{action orbits defn}{group actions defn}
\begin{defn}[Action orbits, \Assumed{56}]\label{\dict{56 ref}}
If \(a\) is an action of a group \(G\) on a topological structure \(\mathbb{X}\), and  \(x\in \mathbb{X}^n\) then we call the set \((\{x\}\times G)a\) the \textit{orbit} of \(x\) under \(a\).
\end{defn}

The following type of action (ample generics) frequently occurs in the literature in the case of a group acting on itself by conjugation.
Indeed we will only be using it in this case.
There are many nice examples of actions with ample generics (see Theorem \ref{ample symmetry} and Theorem \ref{ample canor}) but proving this tends to be quite involved.
\speeddictthree{57}{ample generics defn}{meagre defn}{powers of actions defn}{action orbits defn}
\begin{defn}[Ample generics, \Assumed{57}]\label{\dict{57 ref}}
If \(a\) is an action of a topological group \(G\) on a topological structure \(\mathbb{X}\), then we say that \(a\) has \textit{ample generics} if for all \(n\in \N\) the action \(a^n\) has an orbit comeagre in \(\mathbb{X}^n\).
\end{defn}

\speeddicttwo{86}{centralizer defn}{group defn}{substructures defn}
\begin{defn}[Centralizer and normalizer, \Assumed{86}]\label{centralizer defn}
Recall that two semigroup elements \(f, g\) \textit{commute} if \(fg=gf\). If \(G\) is a group and \(S\subseteq G\), then we define the \textit{centralizer} and \textit{normalizer} of \(S\) in \(G\) by
\[C_G(S) := \makeset{f\in G}{\(f\) commutes with all elements of \(S\)},\]
\[N_G(S) := \makeset{f\in G}{\(fS=Sf\)}\]
respectively. It is routine to verify that \(C_G(S)\) and \(N_G(S)\) are always subgroups of \(G\). In the case that \(S\) is a singleton \(\{f\}\), we will often write \(C_G(f)\) in place of \(C_G(\{f\})\) for brevity.

If \(G\) is a group, \(S\) is a subgroup of \(G\) and \(N_G(S) = G\), then we say that \(S\) is a \textit{normal subgroup} of \(G\) (denoted \(S \normalsubgroup G\)). In this case we define the \textit{quotient} of \(G\) by \(S\) by:
\[G/S := \makeset{gS}{\(g\in G\)}.\]
The set \(G/S\) is viewed as a group, and sets of elements of \(G\) are multiplied in the usual fashion.
\end{defn}

\speeddicttwo{58}{conjugation action defn}{group actions defn}{centralizer defn}
\begin{defn}[Conjugation action, \Assumed{58}]\label{\dict{58 ref}}
If \(G\) is a group, then we define \(\conj_{G}:G^2\to G\) to be the action
\[(g, h)\conj_{G}= h^{-1}g h.\]
We call the group \((G)\phi_{\conj_{G}}\) (recall Definition~\ref{group actions defn}), the \textit{inner automorphism} group of \(G\) (denoted \(\operatorname{Inn}(G)\)).
If \(\psi\in \Aut(G)\), and \(h\in G\), then \(\psi^{-1} \circ((h)\phi_{\conj_{G}})\circ \psi =  ((h)\psi)\phi_{\conj_{G}}\).
In particular \(\operatorname{Inn}(G) \normalsubgroup \operatorname{Aut}(G)\).
We call the group \(\A(G)/\operatorname{Inn}(G)\) the \textit{outer automorphism group} of \(G\) and denote this group by \(\operatorname{Out}(G)\).
\end{defn}

Before proceeding to \(\Sym(\N)\) and \(\Aut(2^\N)\), we give a nice and easy consequence of a group having ample generics (Theorem~\ref{ample commutators thm}). 
This fact seems to be well known but we do not know of a specific reference for it.

\speeddictone{59}{commutator defn}{topological structures defn}
\begin{defn}[Commutators, \Assumed{59}]\label{\dict{59 ref}}
If \(G\) is a topological group, then we define the \textit{commutator map }\([\cdot,\cdot]: G^2\to G\) by
\([f, g]=f^{-1}g^{-1}f g\).
\end{defn}

\speeddictthree{60}{ample commutators thm}{baire category theorem}{ample generics defn}{commutator defn}
\begin{theorem}[Commutator width 1, \Assumed{60}]\label{ample commutators thm}
If \(G\) is a completely metrizable topological group and \(\conj_{G}\) has a comeagre orbit (in particular if it has ample generics), then the commutator map  \([\cdot, \cdot]: G^2\to G\) is surjective.
\end{theorem}
\begin{proof}
Let \(C\) be a comeagre orbit of \(\conj_{G}\) and let \(f\in G\) be arbitrary. Then it follows that \(Cf\) is also comeagre, and thus by the Baire Category Theorem (Theorem~\ref{baire category theorem}) the set \(C\cap Cf\) is dense in \(G\).
It follows that there are \(g,h\in C\) such that \(g=hf\).
Thus \(h^{-1}g=f\). As \(h, g\in C\), there is \(k\in G\) such that \(h=k^{-1}gk\). Thus 
\[f=h^{-1}g=(k^{-1}gk)^{-1}g=k^{-1}g^{-1}kg=[k, g]\]
as required.
\end{proof}

\speeddictthree{63}{symmetric group topology}{subspaces}{topological structures defn}{full trans def}
\begin{defn}[Symmetric groups, \Assumed{63}]\label{symmetric group topology}
If \(X\) is a set, then we define \(\Sym(X)\) to be the group of all bijections from \(X\) to itself (if \(X\in \N\), we treat \(X\) as \(\{0, 1, \ldots, X-1\}\)). We define the \textit{pointwise topology} on \(\Sym(X)\) to be the topology \(\mathcal{TB}_X\restriction_{\Sym(X)}=\mathcal{PT}_X\restriction_{\Sym(X)}\) (recall Examples~\ref{topological binary relations} and \ref{full trans def}).
The collection of sets of the form
\[U_{x, y}:=\makeset{f\in \Sym(X)}{\((x)f= y\)}\]
for \(x, y\in X\) form a subbasis for this topology, and the collection of sets of the form
\[U_{h}:=\makeset{f\in \SN}{ \(h\subseteq f\)}\]
for finite \(h\subseteq X^2\) form a basis.
From Example~\ref{topological binary relations}, we know that this topology makes \(\Sym(X)\) into a topological group.
\end{defn}

\speeddictthree{73}{continuous function space defn}{continuous maps and homeomorphisms defn}{compact defn}{subbasis}
\begin{defn}[Spaces of continuous functions, \Assumed{73}]\label{continuous function space defn}
If \(X, Y\) are topological spaces. Then we define \(C(X, Y)\) to be the set of continuous maps from \(X\) to \(Y\). If \(X=Y\) then we write \(C(X)\) for brevity.
The \textit{compact-open topology} on \(C(X, Y)\) is then the topology generated by the sets of the form
\[\makeset{f\in C(X, Y)}{\((F)f\subseteq U\)}\]
where \(F\subseteq X\) is compact and \(U\subseteq Y\) is open.
\end{defn}

\speeddictthree{74}{compact-open is Polish thm}{closure and interior defn}{Polish space defn}{continuous function space defn}
\begin{theorem}[cf Proposition~1.3.3 of \cite{Mill2001aa}, Polish compact-open topologies, \Assumed{74}]\label{compact-open is Polish thm}
If \(X\) and \(Y\) are compact metrizable spaces, then \(C(X, Y)\) with the compact-open topology is a Polish space. Moreover if \(B_X\) and \(B_Y\) are countable bases for \(X\) and \(Y\) respectively, then the collection of sets of the form
\[B(U_X, U_Y):=\makeset{f\in C(X, Y)}{\((U_X^-)f \subseteq U_Y\)}\]
for \(U_X\in B_X\) and \(U_Y\in B_Y\) generate the topology on \(C(X, Y)\).
Moreover if \(d\) is any metric compatible with \(Y\), then the metric \(d_\infty\) on \(C(X, Y)\) defined by
\[(f, g)d_\infty = \sup_{x\in X} ((x)f, (x)g)d\]
is complete and compatible with the compact-open topology.
\end{theorem}

It is well known (see for example Proposition 4.5 of \cite{kechris2012classical}) that if \(f:(X, d_x) \to (Y, d_y)\) is a continuous map between compact metric spaces and \(\varepsilon>0\), then there is \(\delta>0\) such that 
\[(a, b)d_x<\delta\Rightarrow ((a)f, (b)f)d_y < \varepsilon.\]
As a result we obtain the following corollary to Theorem~\ref{compact-open is Polish thm}.

\speeddictthree{500}{Composition of continuous maps is continuous}{product topology defn}{continuous function space defn}{compact-open is Polish thm}
\begin{corollary}[Composition of continuous maps is continuous, \Assumed{500}]\label{Composition of continuous maps is continuous}
Let \((X, d_x), (Y, d_y), (Z, d_z)\) be compact metric spaces. 
The composition map from \(C(X, Y) \times C(Y, Z) \to C(X, Z)\) is continuous (using the compact-open topologies (Definition~\ref{continuous function space defn}) and the product topology (Definition~\ref{product topology defn})).
\end{corollary}
\begin{proof}
Let \((f, g)\in C(X, Y) \times C(Y, Z)\) and \(\varepsilon > 0\) be arbitrary.
Let \(\delta >0\) be such that 
\[(a, b)d_y<\delta\Rightarrow ((a)g, (b)g)d_x < \frac{\varepsilon}{2}.\]
It follows that if 
\[\sup_{x\in X} ((x)f, (x)f')d_y < \delta \quad \text{ and }\quad \sup_{y\in Y} ((x)g, (x)g')d_z <  \frac{\varepsilon}{2},\]
then
\[\sup_{x\in X} ((x)f' g', (x)f g)d_z \leq \sup_{x\in X} \left(((x)f' g, (x)f g)d_z +  ((x)f' g, (x)f' g')d_z\right)< \varepsilon.\]
The result then follows from Theorem~\ref{compact-open is Polish thm}.
\end{proof}

\speeddictthree{256}{compact-open group is Polish prop}{compactness facts}{even more nice Polish subspaces cor}{Composition of continuous maps is continuous}
\begin{proposition}[Polish homeomorphism groups, \Assumed{256}]\label{compact-open group is Polish prop}
If \(X\) is a compact metrizable topological space, then the group \(\Aut(X)\) of homeomorphisms of \(X\) is a Polish space with the compact-open topology.
\end{proposition}
\begin{proof}
Let \(B\) be a countable basis for \(X\) with \(\varnothing\not\in B\).
From Theorem~\ref{compact-open is Polish thm}, the sets \(B(U, V):= \makeset{f\in C(X)}{\((U^-)f\subseteq V\)}\) for \(U, V\in B\) form a basis for the compact-open topology.

For all \(x, y\in X\) with \(x\neq y\), let 
\[P_{x, y}:=\makeset{(U, V)\in B^2}{there exist \(U',V'\in B\) such that \(x\in U\subseteq U^-\subseteq U'\),\\ \(y\in V\subseteq V^-\subseteq V'\), and \(U'\cap V'=\varnothing\)}\]
Note that (by Remark~\ref{compactness facts} part 9) the set \(P_{x, y}\) is never empty and that the set \(P:=\union{x, y\in X} P_{x, y}\subseteq B^2\) is countable.
From Remark~\ref{compactness facts}, \(\Aut(X)\) consists of the elements of \(C(X)\) which are bijections.
Thus
\[\Aut(X)= \left(\intersection{(U, V)\in P}\union{U', V'\text{open}\\\text{and disjoint}} B(U, U'')\cap B(V, V'')\right) \cap\left( \intersection{V\in B} \union{U\in B} B(U, V)\right).\]
is a countable intersection of sets which are open with respect to the compact-open topology. 
It follows from Corollary~\ref{even more nice Polish subspaces cor} that this subspace of \(C(X)\) is also Polish.
\end{proof}

\subsection{The symmetric group}
In \cite{kechris2007turbulence}, it is stated that ``one can easily show that \(S_\infty\) has ample generics". 
We do not know of a reference for this fact so we give our own proof (Theorem~\ref{ample symmetry}).
As we see below, the key in showing that \(\Sym(\N)\) has ample generics will be viewing elements of \(\Sym(\N)\) as algebraic objects.

\speeddictone{61}{unary algebras defn}{signature defn}
\begin{defn}[Unary algebras, \Assumed{61}]\label{\dict{61 ref}}
If \(n\in \N\), then we define the signature \(\sigma_{u,n}:= (\{0, 1, \ldots, n-1\}, \varnothing, \{(0, 1), (1, 1), \ldots, (n-1, 1)\})\).
That is \(\sigma_{u, n}\) has the \(n\) unary function symbols \(\{0, 1, \ldots, n - 1\}\) and no other symbols. We call a \(\sigma_{u, n}\)-structure an \textit{\(n\)-unary algebra}. Moreover, we say that an \(n\)-unary algebra \(\mathbb{U}\) is \textit{bijective} if for all \(i\in \{0, 1, \ldots, n - 1\}\) the map \(i^\mathbb{U}\) is a permutation of \(\mathbb{U}\).

For example when \(n=2\) a bijective \(n\)-unary algebra consists of a set \(X\) and pair of permutations of \(X\). 
\end{defn}

\speeddictthree{62}{unary algebra components}{types of binary relation defns}{substructures defn}{unary algebras defn}
\begin{defn}[Connected components, \Assumed{62}]\label{\dict{62 ref}}
The \textit{connected components} of an \(n\)-unary algebra \(\mathbb{U}\) are the equivalence classes of the least equivalence relation on \(\mathbb{U}\) containing the set
\[\bigcup_{i<n} i^\mathbb{U}.\]
We say that an \(n\)-unary algebra is \textit{connected} if this relation is all of \(\mathbb{U}\times \mathbb{U}\). In particular, the connected components of a unary algebra partition it into connected substructures.
\end{defn}

\speeddictthree{257}{pointwise group is Polish prop}{even more nice Polish subspaces cor}{nice analytic lemma}{symmetric group topology}
\begin{proposition}[Polish symmetric groups, \Assumed{257}]\label{pointwise group is Polish prop}
The group \(\Sym(\N)\) is a Polish space with the pointwise topology.
\end{proposition}
\begin{proof}
By Lemma~\ref{nice baire space lemma}, the space \(\N^\N\) is Polish.
As 
\[\Sym(\N) =\left(\intersection{a\in \N}\union{b\in \N}\makeset{f\in \N^\N}{\((b, a)\in f\)}\right)\cap \left(\intersection{a\neq b}\union{c\neq d}\makeset{f\in \N^\N}{\((a, c)\in f\) and \((b, d)\in f\)}\right),\]
the result follows from Corollary~\ref{even more nice Polish subspaces cor}.
\end{proof}

\speeddictsix{64}{ample symmetry}{function sets def}{almost containment defn}{ample generics defn}{conjugation action defn}{unary algebra components}{symmetric group topology}
\begin{theorem}[\(\SN\) has ample generics, \Assumed{64}]\label{ample symmetry}
The action \(\conj_{\SN}\) has ample generics (using the pointwise topology).
\end{theorem}
\begin{proof}
Let \(n\in \N\) be fixed, and let 
  \[C := \makeset{y\in \SN^n}{the connected components of the unary algebra \((\N,\sigma_{u, n}, y)\)\\ consist of countably infinitely many copies of every finite\\ bijective connected \(n\)-unary algebra.}\]
We will first show that \(C\) is an orbit of \(\conj_{\SN}^n\).

Note that \(x,y \in \Sym(\N)^n\) are in the same orbit precisely when the corresponding unary algebras \((\N,\sigma_{u,n}, x)\) and \((\N, \sigma_{u,n}, y)\) are isomorphic (a conjugating element of \(\SN\) is the same as an isomorphism).
As the property defining \(C\) is preserved by isomorphisms, it follows that \(C\) is closed under the action of \(\Sym(\N)\).

We now show that all elements of \(C\) are in the same orbit. Suppose that \(x, y\in C\). Let \(\{C_{i, j}: i, j\in \N\}\) be the connected components of \((\N,\sigma_{u,n}, y)\), indexed in such a way that \(C_{i_0, j_0} \cong C_{i_1, j_1}\) if and only if \(i_0 = i_1\).
Similarly let \(\{D_{i, j}: i, j\in \N\}\) be the set of connected components of \((\N,\sigma_{u,n}, x)\), indexed in such that way that \(D_{i_0, j_0} \cong D_{i_1, j_1}\) if and only if \(i_0 = i_1\), and \(C_{i_0, j} \cong D_{i_1, j}\) if and only if \(i_1 = i_2\).

For all \(i, j\in \N\), let \(\phi_{i, j}: C_{i, j} \to D_{i, j}\) be an isomorphism of unary algebras. We now define \(\phi := \union{i, j\in \N} \phi_{i, j}\). It follows that \(\phi: (\N,\sigma_{u, n}, y) \to (\N,\sigma_{u, n}, x)\) is an isomorphism of unary algebras and hence \((y, \phi)\conj_{\SN}^n = x\).

Now that we have identified our orbit, we show that it is comeagre.
Let \(R\) be a countable set whose elements are bijective connected \(n\)-unary algebras and such that each isomorphism class of finite bijective \(n\)-unary algebras has precisely one representative in \(R\). Our description of \(C\) gives us that if we define the sets
\begin{align*}
B_m:= &\makeset{y\in \SN^n}{the connected component of \(m\) in \((\N,\sigma_{u,n},y)\) is finite}\\
B_{\mathbb{A}, m}:=&\makeset{y\in \SN^n}{\((\N,\sigma_{u,n}, y)\) has at least \(m\) connected components\\ isomorphic to \(\mathbb{A}\)}
\end{align*}
for \(m\in \N\) and \(\mathbb{A}\in R\) then
\[C = \left(\bigcap_{m\in \N}B_m\right)\cap \left(\bigcap_{\substack{\mathbb{A}\in R\\ m\in \N}}B_{\mathbb{A}, m}\right).\]
So it suffices to show that all of these sets we are open and dense.
They are open because the properties defining them are determined by only finitely many points.

We must show that they are dense.
Let \(m\in \N\) and \(\mathbb{A}\in R\) be arbitrary. Let 
\(U:=\prod_{i< n}U_i\) be a non-empty basic open set in \(\SN^n\), and for all \(k< n\) let \(f_k\) be finite bijections between subsets of \(\N\) such that \(U_k = \makeset{g\in \SN}{\(f_k\subseteq g\)}\).
Let \[P:= \bigcup_{k\in \{0, 1, \ldots ,n-1\}} \dom(f_k) \cup \im(f_k).\]

We first show that \(B_m\cap U \neq \emptyset\). For all \(k< n\), let \(f_k'\) be a permutation of \(P\cup  \{m\}\) which contains \(f_k\), and let \(g_k\in \SN\) contain \(f_k'\). Then \((g_0, g_1, \ldots, g_{n-1}) \in B_m\cap U\).

It remains to show that \(B_{\mathbb{A},m}\cap U \neq \emptyset\). 
Let \(\phi_{0}, \phi_1, \ldots, \phi_{m-1}:\mathbb{A}\to \N\backslash P\) be injections with disjoint images.
For each \(k\in \{0, 1, \ldots, n-1\}\) let \(g_k\in \SN\) contain the partial bijection
\[f_{k}\cup \union{j\in m} \phi_j^{-1} (k^\mathbb{A}) \phi_j.\]
Then \((g_0, g_1, \ldots, g_{n-1})\in B_{\mathbb{A}, m} \cap U\) as required. 
\end{proof}

\speeddicttwo{65}{Permutations are Commutators}{ample commutators thm}{ample symmetry}
\begin{corollary}[Permutations are commutators, \Assumed{65}]\label{\dict{65 ref}}
Every element of the group \(\SN\) is a commutator.
\end{corollary}

\subsection{Homeomorphisms of Cantor space}
Our main aim in this subsection is to provide an argument that the action \(\conj_{\Aut(2^\N)}\) has ample generics (where \(\Aut(2^\N)\) is the group of homeomorphisms of the Cantor set equipped with the compact-open topology).

The argument we present is that of Aleksandra Kwiatkowska \cite{good_generic_Cantor}, which shows that a generic homeomorphism of \(2^\N\) can be constructed as a sort of limit of a nice class of finite topological structures.
This limit is in the sense of ``Projective Fra\"iss\'e Theory" as introduced by Irwin and Solecki in \cite{projective_fr}.

The conditions defining a projective Fra\"iss\'e class are analogous to those for the usual definition of a Fra\"iss\'e class, except with quotient maps used in place of embeddings (however a projective Fra\"iss\'e class need not be closed under quotients).

\speeddicttwo{51}{projective fraisse class def}{structure hom defn}{congruences and quotients defn}
\begin{defn}[Projective Fra\"iss\'e class, \Assumed{51}]\label{projective fraisse class def}
We say a class \(\mathcal{C}\) of discrete \(\sigma\)-structures is a \textit{projective Fra\"iss\'e class} if it satisfies the following properties:
\begin{enumerate}
    \item \(\mathcal{C}\) is closed under topological isomorphism.
    \item \(\mathcal{C}\) contains countably infinitely many structures up to isomorphism.
    \item \(\mathcal{C}\) consists of finite structures.
    \item If \(\mathbb{D}, \mathbb{E}\in \mathcal{C}\), then there is \(\mathbb{F}\in \mathcal{C}\) and quotient maps \(\phi_D, \phi_E\) from \(\mathbb{F}\) onto \(\mathbb{D}\) and \(\mathbb{E}\) respectively.
    \item If \(\mathbb{C},\mathbb{D}, \mathbb{E}\in \mathcal{C}\) and \(\phi_{D, C}:\mathbb{D}\to \mathbb{C}, \phi_{E,C}:\mathbb{E}\to \mathbb{C}\) are quotient maps, then there is \(\mathbb{F}\in \mathcal{C}\) and quotient maps \(\phi_{F,D}, \phi_{F,E}\) from \(\mathbb{F}\) onto \(\mathbb{D}\) and \(\mathbb{E}\) respectively, such that
    \(\phi_{F, D}\phi_{D, C}= \phi_{F, E}\phi_{E, C}.\)
\end{enumerate}
\end{defn}

\speeddicttwo{500}{finite semigroups are a projective fraisse class example}{semigroup defn}{projective fraisse class def}
\begin{example}[Non-empty finite semigroups, \Assumed{51}]\label{finite semigroups are a projective fraisse class example}
The class of non-empty finite discrete semigroups form a projective Fra\"iss\'e class.
\end{example}
\begin{proof}
The first three conditions of Definition~\ref{projective fraisse class def} are immediate.
We first verify Condition 4.
If \(D\) and \(E\) are finite discrete semigroups, then so is \(F:=D\times E\).
Moreover if \(D, E\) are non-empty then the projection maps \(\pi_0:F \to D\) and \(\pi_1:F\to E\) are surjective.
As semigroups are algebraic structures, it follows that these are quotient maps, thus Condition 4 follows.

It remains to verify Condition 5. Suppose that \(C, D\) and \(E\) are finite discrete semigroups and \(\phi_{D, C}:D\to C, \phi_{E, C}:E\to C\) are quotient maps. 
Let 
\[F:= \makeset{(d, e)\in D\times E}{\((d)\phi_{D, C}= (e)\phi_{E, C}\)}.\]
It is routine to verify that \(F\) is a subsemigroup of \(D\times E\).
If we define \(\phi_{F,D}:F\to D\) and \(\phi_{F,E}:F\to E\) to be the restrictions of the projection maps \(\pi_0\) and \(\pi_1\) respectively to the set \(F\), then we have 
\(\phi_{F, D}\phi_{D, C}= \phi_{F, E}\phi_{E, C}\).

It thus suffices to show that \(\phi_{F,D}\) and \(\phi_{F,E}\) are quotient maps (which for algebraic structures is equivalent to being a surjective homomorphism).
If \(d\in D\) is arbitrary, then we have 
\[\{d\}\times (\{(d)\phi_{D, C}\})\phi_{E, C}^{-1} \subseteq F\]
(by the definition of \(F\)). 
As \(\phi_{E, C}\) is a quotient map, it follows that \((\{(d)\phi_{D, C}\})\phi_{E, C}^{-1}\neq \varnothing\), hence \(\phi_{F, D}\) is surjective.
By the same argument, \(\phi_{F, E}\) is also surjective as required.
\end{proof}

\speeddictthree{52}{projective fraisse limit theorem}{Brouwer' Theorem}{topology on a metric space defn}{projective fraisse class def}
\begin{theorem}[Projective Fra\"iss\'e limit, \Assumed{52}]\label{\dict{52 ref}}
If \(\mathcal{C}\) is a projective Fra\"iss\'e class, then there is a unique (up to topological isomorphism) second countable, compact, metrizable, zero-dimensional topological structure \(\operatorname{Flim}_{\leftarrow}(\mathcal{C})\) satisfying
\begin{enumerate}
    \item For all \(\mathbb{D}\in \mathcal{C}\), there is a continuous quotient map \(\phi_D:\operatorname{Flim}_{\leftarrow}(\mathcal{C})\to \mathbb{D}\).
    \item If \(X\) is a discrete, finite set and \(f: \PFL{\mathcal{C}} \to X\) is a continuous surjection, then there is \(\mathbb{S}_f\in \mathcal{C}\), a continuous quotient map \(\phi_f:\PFL{\mathcal{C}}\to \mathbb{S}_f\) and a surjective map \(s_f: \mathbb{S}_f\to X\) such that \(f=\phi_f s_f\).
    \item If \(\mathbb{A},\mathbb{B}\in \mathcal{C}\), and \(\psi_1:\mathbb{A} \to \mathbb{B},\ \psi_2: \operatorname{Flim}_{\leftarrow}(\mathcal{C})\to \mathbb{B}\) are continuous quotient maps, then there is a continuous quotient map \(\psi_3 :\PFL{\mathcal{C}} \to \mathbb{A}\) such that \(\psi_3\psi_1=\psi_2\).
\end{enumerate}
We call such an object \(\operatorname{Flim}_{\leftarrow}(\mathcal{C})\) a \textbf{projective Fr\"aiss\'e limit} of \(\mathcal{C}\).
\end{theorem}
\begin{proof}
We first show existence.
Let \(S=\makeset{\mathbb{S}_i}{\(i\in \N\)}\) be a countable set of structures in \(\mathcal{C}\), such that every element of \(\mathcal{C}\) is isomorphic to a unique element of \(S\).

Build by the following algorithm a sequence of structures \((\mathbb{P}_{i})_{i\in \N}\) in \(\mathcal{C}\), and a sequence of quotient maps \((\phi_{i+1}:\mathbb{P}_{i+1}\to \mathbb{P}_{i})_{i\in \N}\):
\begin{enumerate}
    \item Choose \(N\subseteq \N\) containing \(0\) such that \(|N|=|N^c|\), define \(i:= 0\), choose \(t: N\to S\) to be a bijection, and define \(\mathbb{P}_0:= \mathbb{S}_0\).
    
    \item Choose \(M\subseteq \N \backslash (N\cup \{0,1, \ldots i-1\})\) such that \(|\N|=|\N \backslash (M\cup N\cup \{0,1, \ldots i-1\})|\) and choose a bijection \(q\) from \(M\) to the set
    \[\union{j<i \\ \mathbb{A},\mathbb{B}\in S}
    \makeset{(\psi_0, \psi_1)\in \Hom(\mathbb{A}, \mathbb{B}) \times \Hom(\mathbb{P}_j, \mathbb{B})}{\(\psi_0, \psi_1\) is are quotient maps}.\]
    
    \item Redefine \(t:= t\cup q\), \(N:= N\cup M\) and \(i:= i+1\).
    
    \item If \((i)t\) is undefined, then define \(\mathbb{P}_i:= \mathbb{P}_{i-1}\) and define \(\phi_i\) to be the identity map.
    If \((i)t\) is a structure, then by Definition~\ref{projective fraisse class def} Condition 4, let \(\mathbb{P}_i\in \mathcal{C}\) be a structure which \((i)t\) and \(\mathbb{P}_{i-1}\) are both quotients of and choose \(\phi_i:\mathbb{P}_i\to \mathbb{P}_{i-1}\) to be a quotient map.
    If \((i)t=(\psi_0, \psi_1)\) is a pair of quotient maps \(\psi_0:\mathbb{A}\to \mathbb{B}\) and \(\psi_1:\mathbb{P}_j\to \mathbb{B}\), then by Definition~\ref{projective fraisse class def} Condition 5, choose \(\mathbb{P}_i\) to be a structure and \(\phi_i:\mathbb{P}_i\to \mathbb{P}_{i-1}\) be a quotient map such that there is a quotient map \(\psi:\mathbb{P}_i\to \mathbb{A}\) satisfying 
    \[\psi\circ \psi_0=\phi_i\circ (\phi_{i-1}\circ \ldots \circ \phi_{j+1}\circ \psi_1).\]
    
\item Go to step 2.
\end{enumerate}
In particular, this sequence has the property that every structure in \(\mathcal{C}\) is a quotient of a structure in the sequence, and if \(\mathbb{A}, \mathbb{B}\in \mathcal{C}\) and \(\psi_1: \mathbb{A}\to \mathbb{B},\ \psi_2:\mathbb{P}_j\to \mathbb{B}\) are quotient maps, then there is \(i\in \N\) and a quotient map from \(\psi_3:\mathbb{P}_i\to \mathbb{A}\) such that \(\psi_3\psi_1=\phi_i\phi_{i-1}\ldots\phi_{j+1}\psi_2\).

We now define \(\mathbb{L}\) to be the topological structure
\[\makeset{(p_i)_{i\in \N}\in \prod_{i\in \N} \mathbb{P}_i}{for all \(i\in \N\) we have \((p_{i+1})\phi_{i+1} = p_i\) }.\]
We next need to verify that \(\mathbb{L}\) has the required properties.
First note that \(\mathbb{L}\) is a closed subspace of \(\prod_{i\in \N} \mathbb{P}_i\) which by Theorem~\ref{Brouwer' Theorem} is a Cantor space, so \(\mathbb{L}\) satisfies the given topological conditions.

Note also that all the projection maps \(\pi_i: \mathbb{L}\to \mathbb{P}_i\) are quotient maps. Thus \(\mathbb{L}\) has the first required property as each element of \(\mathcal{C}\) is a quotient of one of the \(\mathbb{P}_i\).

To check the second condition, suppose that \(X\) is a finite discrete space and \(f:\mathbb{L} \to X\) is a continuous surjection.
The preimage of each point in \(X\) is a clopen and is hence a finite union of sets of the form \((\{a\})\pi_i^{-1}\) where \(a\in \mathbb{P}_i\). It follows, by taking the max \(M\) of all such \(i\), that \(f\) can be factored into maps \(\pi_Ms_f\) as required.

We now check the third condition. Let \(\psi_1:\mathbb{A}\to \mathbb{B}\) and \(\psi_2:\mathbb{L}\to \mathbb{B}\) be continuous quotient maps. As in the above paragraph, we can find \(j\in \N\) and \(\psi_2':\mathbb{P}_j\to \mathbb{B}\) such that \(\psi_2= \pi_j\psi_2'\). Thus by the construction of \((\mathbb{P}_{i})_{i\in \N}\), there is \(i\in \N\) and \(\psi_3':\mathbb{P}_i\to \mathbb{A}\) with \(\psi_3'\psi_1 = \phi_{i}\phi_{i-1}\ldots\phi_{j+1}\psi_2'\). Thus choosing \(\phi_3:= \pi_i\phi_3'\) gives the required map.

It remains to show uniqueness.
Let \(\mathbb{L}_0\) and \(\mathbb{L}_1\) be structures with the listed properties. As \(\mathbb{L}_0\) is zero-dimensional and second countable, we can find a sequence of equivalence relations \((\sim_{0, i})_{i\in \N}\), such that the equivalence classes of these relations are all clopen and they together form a basis for the topology on \(\mathbb{L}_0\).
Be redefining \(\sim_{0, i}\ :=\ \sim_{0, i-1}\cap \sim_{0, i}\), we can assume without loss of generality that each relation in this sequence is contained in the previous one (so the corresponding sequence of partitions is getting finer).
Also by the second condition of the theorem (and refining more if necessary), we can assume that for all \(i\in \N\), the structure \(\mathbb{L}_0/\sim_{0, i}\) is isomorphic to a structure in \(\mathcal{C}\). 
We also define an analogous sequence of equivalence relations \((\sim_{1, i})_{i\in \N}\) on \(\mathbb{L}_1\).

As 
\begin{enumerate}
    \item the classes of \((\sim_{0, i})_{i\in \N}\) form a basis for \(\mathbb{L}_0\),
    \item \(\mathbb{L}_1/\sim_{1, 0}\) is isomorphic to an element of \(\mathcal{C}\), and
    \item\(\mathbb{L}_0\) satisfies the first condition of the theorem,
\end{enumerate}
  we can find \(i_0\in \N\) and a quotient map \(\phi_{0, 0}: \mathbb{L}_0/\sim_{0,i_0} \to \mathbb{L}_1/\sim_{1, 0}\). Similarly using the fact that \(\mathbb{L}_1\) satisfies the third condition of the theorem, we can find \(j_0\in \N\) and \(\phi_{1, 0}:\mathbb{L}_1/\sim_{1, j_0}\to \mathbb{L}_0/\sim_{0,i_0}\) such that \(([x]_{\sim_{1, j_0}})\phi_{1, 0}\phi_{0, 0}= [x]_{\sim_{1, 0}}\) for all \(x\in \mathbb{L}_1\).

Continuing in this fashion we construct strictly increasing sequences \((i_k)_{k\in \N}, (j_k)_{k\in \N}\), and sequences \((\phi_{0, k})_{k\in \N}, (\phi_{1, k})_{k\in \N}\) such that for all \(x_0\in \mathbb{L}_0, x_1\in \mathbb{L}_1\) and \(k\in \N\backslash\{0\}\) we have:
\[([x_0]_{\sim_{0, i_{k+1}}})\phi_{0, k+1}\phi_{1, k}= [x_0]_{\sim_{1, i_{k}}},\quad \text{and}\quad  ([x_1]_{\sim_{1, j_k}})\phi_{1, k}\phi_{0, k}= [x_1]_{\sim_{1, j_{k -1}}}.\]

We define a homomorphism \(\phi_0: \mathbb{L}_0\to \mathbb{L}_1\) by defining \((x_0)\phi_0\) to be the unique element of the set \(\intersection{k\in \N} ([x_0]_{\sim_{0, i_{k+1}}})\phi_{0, k+1}\) (this exists because \(\mathbb{L}_1\) is compact and is unique because \(\mathbb{L}_1\) is Hausdorff). We define \(\phi_2:\mathbb{L}_1\to \mathbb{L}_0\) analogously. By construction \(\phi_0, \phi_1\) are continuous homomorphisms and they are inverses of each other, so the result follows.
\end{proof}
We also give an alternative formulation of the above theorem which will be more useful at times.

\speeddictone{53}{projective fraisse limit theorem clopen sets version}{projective fraisse limit theorem}
\begin{theorem}[Projective Fra\"iss\'e limit 2, \Assumed{53}]\label{projective fraisse limit theorem clopen sets version}
If \(\mathcal{C}\) is a projective Fra\"iss\'e class, then \(\operatorname{Flim}_{\leftarrow}(\mathcal{C})\) is the unique (up to topological isomorphism) compact, second countable, Hausdorff, zero-dimensional topological structure  satisfying:
\begin{enumerate}
    \item For all \(\mathbb{D}\in \mathcal{C}\), there is a partition of \(\operatorname{Flim}_{\leftarrow}(\mathcal{C})\) into clopen sets such that the corresponding quotient structure is isomorphic to \(\mathbb{D}\).
    \item If \(P\) is a finite partition of \(\operatorname{Flim}_{\leftarrow}(\mathcal{C})\) into clopen sets, then there is a partition \(R\) of \(\operatorname{Flim}_{\leftarrow}(\mathcal{C})\) into clopen sets refining \(P\) and such that the quotient structure corresponding to \(R\) is an element of \(\mathcal{C}\).
    \item If \(P_0, P_1\) are finite partitions of \(\operatorname{Flim}_{\leftarrow}(\mathcal{C})\) into clopen sets and \(\phi: P_0 \to P_1\) is an isomorphism of the corresponding quotient structures, then there is a topological isomorphism \(h:\operatorname{Flim}_{\leftarrow}(\mathcal{C}) \to \operatorname{Flim}_{\leftarrow}(\mathcal{C})\) such that \((S)h = (S)\phi\) for all \(S\in P_0\).
\end{enumerate}
\end{theorem}
\begin{proof}
The first two conditions are reformulations of the corresponding conditions in Theorem~\ref{projective fraisse limit theorem} using the fact that for any quotient map \(\phi\) we have
\[\dom(\phi)/\ker(\phi) \cong \im(\phi).\]
We show that the third condition of Theorem~\ref{projective fraisse limit theorem} is equivalent to the third condition of Theorem~\ref{projective fraisse limit theorem clopen sets version}.

((3) of Theorem~\ref{projective fraisse limit theorem} \(\Rightarrow\) (3) of Theorem~\ref{projective fraisse limit theorem clopen sets version}): Let \(P_0, P_1\) be finite partitions of \(\operatorname{Flim}_{\leftarrow}(\mathcal{C})\) into clopen sets. By the second condition, we can assume without loss of generality that the corresponding quotient structures are in \(\mathcal{C}\).
The result follows by revisiting the uniqueness proof from Theorem~\ref{projective fraisse limit theorem} using \(\mathbb{L}_0 = \mathbb{L}_1 = \operatorname{Flim}_{\leftarrow}(\mathcal{C})\) and choosing \(\phi_{0, 0}\) to be the map \(x\to ([x]_{P_0})\phi\) (where \([x]_{P_0}\) is the element of \(P_0\) containing \(x\)).
The desired map \(h\) is then we map \(\phi_0\) constructed in that proof.

((3) of Theorem~\ref{projective fraisse limit theorem clopen sets version} \(\Rightarrow\) (3) of Theorem~\ref{projective fraisse limit theorem}): Let \(\mathbb{A}, \mathbb{B}\in \mathcal{C}\) and \(\phi_1:\mathbb{A}\to \mathbb{B}, \phi_2:\operatorname{Flim}_{\leftarrow}(\mathcal{C})\to \mathbb{B}\) are quotient maps.
By renaming elements of \(\mathbb{A}\) using Condition (1) we can assume without loss of generality that \(\mathbb{A}\) is a quotient of \(\operatorname{Flim}_{\leftarrow}(\mathcal{C})\).
By renaming elements of \(\mathbb{B}\), we can assume without loss of generality that \(\phi_2\) is the map \(x\mapsto [x]_{\ker(\phi_2)}\).
Let \(\psi:\operatorname{Flim}_{\leftarrow}(\mathcal{C})\to \mathbb{A}\) be the quotient map mapping a point to the set containing it.
We have a natural isomorphism \(\phi:\operatorname{Flim}_{\leftarrow}(\mathcal{C})/\ker(\phi_2)\to \operatorname{Flim}_{\leftarrow}(\mathcal{C})/\ker(\psi\phi_1)\) defined by 
\[([x]_{\ker(\phi_2)})\phi = (\{x\phi_2\})(\psi\phi_1)^{-1}.\]
So there is a topological isomorphism \(h:\operatorname{Flim}_{\leftarrow}(\mathcal{C})\to \operatorname{Flim}_{\leftarrow}(\mathcal{C})\) corresponding to this isomorphism. Then defining \(\phi_3 = h\psi\) we have
\[([x]_{\ker(\phi_2)})\phi_3\phi_1 = (\{x\phi_2\})(\psi\phi_1)^{-1} \psi\phi_1 = \{x\phi_2\}\]
for all \(x\in \operatorname{Flim}_{\leftarrow}(\mathcal{C})\), and so \(\phi_3\phi_1 = \phi_2\) as required.
\end{proof}

We now establish the classes of structures whose limits will give our generic homeomorphisms.

\speeddictone{66}{CA signature}{signature defn}
\begin{defn}[Cantor signatures, \Assumed{66}]\label{\dict{66 ref}}
If \(n\in \N\), then we define the signature \(\sigma_{b,n}:= (\varnothing, \{0, 1, \ldots, n-1\}, \{(0, 2), (1, 2), \ldots, (n-1, 2)\})\). That is \(\sigma_{b, n}\) has the \(n\) binary relation symbols \(\{0, 1, \ldots, n-1\}\) and no other symbols.
\end{defn}
\begin{defn}[Amalgamation structures: \ref{binary relations defns}, \ref{structures defn}, \ref{CA signature}]\label{amalgamation structures defn}
For \(n\geq 1\), we define an \textit{\(n\)-Cantor amalgamation structure (\(n\)-CA structure)} to be a non-empty finite \(\sigma_{b, n}\)-structure \(\mathbb{A}\) satisfying the following properties:
\begin{enumerate}
    \item For each \(i< n\), the binary relations \(i^\mathbb{A}\) and \((i^\mathbb{A})^{-1}\) are surjective.
    \item The binary relation 
    \[\de_{\mathbb{A}}:=\makeset{(a, (i, j))\in A\times\left(\{0, 1,\ldots, n-1\}\times \{-1, 1\} \right)}{\(|(\{a\})(i^{\mathbb{A}})^j|\geq 2\)}\]
    is a function with domain \(A\).
    \item If \((a, b)\in i^\mathbb{A}\), then either \((a)\de_{\mathbb{A}}=(i, 1)\) or \((b)\de_{\mathbb{A}} = (i, -1)\).
\end{enumerate}
\end{defn}

It is somewhat difficult to construct natural examples of \(n\)-CA structures, but the following theorem will allow us to construct a large collection of examples.
\speeddictthree{67}{cofinal CA}{product sets and projection maps def}{congruences and quotients defn}{amalgamation structures defn}
\begin{theorem}[cf. Theorem 4.6 of~\cite{good_generic_Cantor}, Cofinal amalgams, \Assumed{67}]\label{\dict{67 ref}}
Suppose that \(\mathbb{S}\) is a finite non-empty \(\sigma_{b, n}\)-structure, and each of \(i^\mathbb{S}\) and \((i^\mathbb{S})^{-1}\) a surjective binary relation for all \(i<n\).
Then there is an \(n\)-CA structure \(Q_{\mathbb{S}}\) which has \(\mathbb{S}\) as a quotient.
\end{theorem}
\begin{proof}
We define \(\{0, 1, \ldots, n-1\}\times \{-1, 1\} \times \{0, 1\}\times \mathbb{S}\) to be the universe of \(Q_\mathbb{S}\). For all \(s\in \mathbb{S}\) and \(i<n\), we choose some \(s_{i, l}, s_{i, r}\in \mathbb{S}\) with \((s_{i,l}. s), (s, s_{i, r})\in i^{\mathbb{S}}\). We then define
\begin{align*}
    i^{Q_{\mathbb{S}}} := &\makeset{((i, 1, 0, s_{i, l}), (j, p, k, s))}{\((j, p, k, s)\in Q_{\mathbb{S}}\)}\\
    &\cup \makeset{( (j, p, k, s),(i, -1, 0, s_{i, r}))}{\((j, p, k, s)\in Q_{\mathbb{S}}\)}\\
    &\cup \makeset{( (i, 1, j, a),(i, -1, k, b))}{\((a, b)\in i^{\mathbb{S}}, j,k\in \{0, 1\}\)}.
\end{align*}
First note that due to the third of the sets in the union above, the map \((a, b, c, d) \to d\) from \(Q_\mathbb{S} \to \mathbb{S}\) is a quotient map.
We now need to check that \(Q_\mathbb{S}\) in an \(n\)-CA structure.
First each of \(i^{Q_{\mathbb{S}}}\) and \((i^{Q_{\mathbb{S}}})^{-1}\) is surjective because of the first two sets in the union.

We need to show that \(\de_{Q_{\mathbb{S}}}\) is a function.
Let \((a, b, c, d)\in Q_{\mathbb{S}}\) be arbitrary.
We must show that there is exactly one pair \((i, j)\in \{0, 1, \ldots, n-1\}\times \{-1, 1\}\) with \(|(a, b, c, d)(i^{Q_{\mathbb{S}}})^{j}|\geq 2\).
By definition of \(i^{Q_{\mathbb{S}}}\) any such pair \((i, j)\) must be equal to the pair \((a, b)\).
Moreover the last set in the union defining \((a^{Q_{\mathbb{S}}})^b\) must pair \((a, b, c, d)\) with at least \(2\) elements.
So \(\de_{Q_{\mathbb{S}}}\) is precisely the function \((a,b,c,d)\to (a,b)\).
From this observation the third condition for being an \(n\)-CA structure is clear.
\end{proof}

Our proof of Lemma~\ref{nca has amalgamation} is quite different from the proof of the analogous statement given in \cite{good_generic_Cantor}, but it is more inline with the theme of this document.
\speeddictfive{68}{nca has amalgamation}{product sets and projection maps def}{Products in Categories Defn}{words defn}{group actions defn}{cofinal CA}
\begin{lemma}[cf. Theorem 4.1 of \cite{good_generic_Cantor}, Amalgams amalgamate, \Assumed{68}]\label{nca has amalgamation}
Suppose that \(\mathbb{A}, \mathbb{B}\), and \(\mathbb{C}\) are \(n\)-CA structures, and \(\phi_\mathbb{B}, \mathbb{B}\to \mathbb{A}, \phi_\mathbb{C}:\mathbb{C}\to \mathbb{A}\) are quotient maps. Then we can find an \(n\)-CA structure \(\mathbb{D}\) and a pair of quotient maps \(\phi_{\mathbb{D}, \mathbb{B}}:\mathbb{D}\to \mathbb{B}\) and \(\phi_{\mathbb{D}, \mathbb{C}}:\mathbb{D}\to \mathbb{C}\) such that 
\[\phi_{\mathbb{D}, \mathbb{C}}\phi_{\mathbb{C}} = \phi_{\mathbb{D}, \mathbb{B}}\phi_{\mathbb{B}}.\]
\end{lemma}
\begin{proof}
First we consider the following substructure of \(\mathbb{B}\times \mathbb{C}\):
\[\mathbb{P}:= \makeset{(b, c)\in \mathbb{B}\times \mathbb{C}}{\((b)\phi_\mathbb{B}= (c)\phi_\mathbb{C}\)}.\]
 Note that the homomorphisms \(\pi_{0}\phi_\mathbb{B}\) and \(\pi_{1}\phi_\mathbb{C}\) from \(\mathbb{B}\times \mathbb{C}\) to \(\mathbb{A}\) are the same when restricted to \(\mathbb{P}\), however \(\mathbb{P}\) might not be an \(n\)-CA structure so we need to do a bit more work.

We define the free group of rank \(n\) by \(F_n\). That is \(F_n\) is the set of all finite words in the alphabet \(\{0, 1, \ldots, n-1\}\times \{-1, 1\}\) (where we consider \((i, 1)\) and \((i, -1)\) to be mutual inverses for each \(i<n\)) with the property that no letter can follow its inverse. The composition is given by concatenation and then removing all instances of a letter followed by its inverse.

We view \(F_n\) as an \(\sigma_{b, n}\)-structure (its Cayley graph) by defining 
\[i^{F_n} = \makeset{(x, y)\in F_{n}}{\(xi=y\)}.\]

Note that this structure is symmetric in that its automorphism group is transitive (the group \(F_n\) acts transitively on it by left multiplication).
Note also that if we have a \(\sigma_{b, n}\)-structure \(\mathbb{S}\), then showing that all the relations \(i^\mathbb{S}\) and \((i^\mathbb{S})^{-1}\) are surjective is equivalent to showing that every \(s\in \mathbb{S}\) is the image of the empty word \(\varepsilon\) under a homomorphism from \(F_n\) to \(\mathbb{S}\).

Consider the following substructure of \(\mathbb{P}\) \[\mathbb{P}^{\operatorname{core}}:= \makeset{p\in \mathbb{P}}{there is a homomorphism \(q:F_n \to \mathbb{P}\) with \((\varepsilon)q= p\)}.\]
By the comment above, the binary relations \(i^{\mathbb{P}^{\operatorname{core}}}\) and \((i^{\mathbb{P}^{\operatorname{core}}})^{-1}\) are surjective for all \(i< n\).

\underline{Claim:} The maps \(\pi_{0}\restriction_{\mathbb{P}^{\operatorname{core}}}\) and \(\pi_{1}\restriction_{\mathbb{P}^{\operatorname{core}}}\) are both quotient maps.\\
\underline{Proof of Claim:} We show that \(\pi_0\restriction_{\mathbb{P}^{\operatorname{core}}}\) is a quotient map, the other case is symmetric.
Let \(i< n\) be arbitrary, and \((a,b)\in i^\mathbb{B}\) be arbitrary. It suffices to show that there is \((p, q)\in i^{\mathbb{P}^{\operatorname{core}}}\) with \(((p)\pi_0, (q)\pi_0) = (a, b)\).
As \(\mathbb{A}\) is an \(n\)-CA structure, we either have \(((a)\phi_{\mathbb{B}})\de_\mathbb{A} = (i, 1)\) or \(((b)\phi_{\mathbb{B}})\de_{\mathbb{A}} = (i, -1)\). We assume without loss of generality that \(((b)\phi_{\mathbb{B}})\de_{\mathbb{A}} = (i, -1)\). There are 2 cases to consider:
\begin{enumerate}
    \item There are finite sequences \(b= b_0, a=b_{-1}, b_{-2}, \ldots, b_{-l}\) in \(\mathbb{B}\) and \(i_{0}, i_{-1}, \ldots, i_{-l}\) in 
    \(\{0,1,\ldots, n-1\}\) such that \((b_{k-1}, b_k)\in i_k^{\mathbb{B}}\) and \(((b_k)\phi_{\mathbb{B}})\de_{\mathbb{A}}= (i_k, -1)\) for each \(-l<k \leq 0\) and \(((b_{-l})\phi_{\mathbb{B}})\de_{\mathbb{A}}= (i_{-l}, 1)\).
    
    \item There are infinite sequences \(b= b_0, a=b_{-1}, b_{-2}, \ldots\) in \(\mathbb{B}\) and \(i_{0}, i_{-1}, i_{-2}, \ldots\) in \(n\) such that \((b_{k-1}, b_k)\in i_k^{\mathbb{B}}\) and \(((b_k)\phi_{\mathbb{B}})\de_{\mathbb{A}}= (i_k, -1)\) for each \(k\in -\N\).
\end{enumerate}
For now we consider the first case. Let \((c_0, c_1)\in i_{-l + 1}^\mathbb{C}\) be such that 
\[((c_0)\phi_{\mathbb{C}}, (c_1)\phi_{\mathbb{C}})= ((b_{-l})\phi_{\mathbb{B}}, (b_{-l+ 1})\phi_{\mathbb{B}}),\] (these must exist as \(\phi_{\mathbb{C}}\) is a quotient map).
Let \(q_{\mathbb{C}}:F_n\to \mathbb{C}\) be any homomorphism which maps \((\varepsilon, i_{-l+1})\) to \((c_0, c_1)\). Similarly, let \(q_{\mathbb{B}}:F_n\to \mathbb{B}\) be any homomorphism which maps \((\varepsilon, i_{-l+1})\) to \((b_{-l}, b_{1-l})\), maps \((i_{-l+1},i_{-l+1}i_{-l+2})\) to \((b_{-l+1}, b_{-l+2})\) and so on. 

Note that \(((b_{-l})\phi_{\mathbb{B}})\de_{\mathbb{A}}= (i_{-l}, 1)\) and \(((b_{-l+1})\phi_{\mathbb{B}})\de_{\mathbb{A}}= (i_{-l+1}, -1)\).
A routine inductive argument using this fact and the fact that \(\mathbb{A}\) is an \(n\)-CA structure shows that there is a unique homomorphism \(q:F_n\to \mathbb{A}\) such that \((\varepsilon)q= (b_{-l})\phi_{\mathbb{B}}\) and \((i_{-l+1})q = (b_{-l+1})\phi_{\mathbb{B}}\).
In particular it follows that \(q_{\mathbb{C}}\phi_{\mathbb{C}}= q_{\mathbb{B}}\phi_{\mathbb{B}}\) as both maps have this property.

The map \(\langle q_{\mathbb{B}}, q_{\mathbb{C}}\rangle_{\mathbb{B}\times\mathbb{C}}: F_n\to \mathbb{B}\times\mathbb{C}\) is a homomorphism (recall Definition~\ref{Products in Categories Defn}).
As \(q_{\mathbb{C}}\phi_{\mathbb{C}}= q_{\mathbb{B}}\phi_{\mathbb{B}}\), it follows that the image of this homomorphism is contained in \(\mathbb{P}\) and hence \(\mathbb{P}^{\operatorname{core}}\), finally the image of the pair \((i_{-l+1}i_{-l+2}\ldots i_{-1},\ i_{-l+1}i_{-l+2}\ldots i_{0})\) under \(\langle q_{\mathbb{B}}, q_{\mathbb{C}}\rangle_{\mathbb{B}\times\mathbb{C}}\) is mapped by \(\pi_{0}\) to \((a, b)\) as required.

We now consider the case that there are infinite sequences \(b= b_0, a=b_{-1}, b_{-2}, \ldots\) in \(\mathbb{B}\) and \(i_{0}, i_{-1}, i_{-2}, \ldots\) in \(n\) such that \((b_{k-1}, b_k)\in i_k^{\mathbb{B}}\) and \(((b_k)\phi_{\mathbb{B}})\de_{\mathbb{A}}= (i_k, -1)\) for each \(k\in -\N\).
As \(\mathbb{B}\) is finite, we can choose \(s, t\in \N\) such that \(-s<-t\) and \(b_{-s} = b_{-t}\).
Let \((c_0, c_1)\in i_{-s+1}^\mathbb{C}\) be such that \(((c_0)\phi_\mathbb{C}, (c_1)\phi_\mathbb{C})= ((b_{-s})\phi_\mathbb{B}, (b_{-s+1})\phi_\mathbb{B})\). We then inductively define a sequence \(c_0, c_1, c_2, \ldots\) as follows:
\begin{enumerate}
    \item Let \(c_0, c_1\) be as before.
    \item If \(c_k\) has already been defined and \((c_k)\phi_\mathbb{C} = (b_{-j})\phi_{\mathbb{B}}\) where \(-s\leq-j<-t\), then we choose \(c_{k+1}\) to be such that \((c_k, c_{k+1})\in i_{-j+1}^\mathbb{C}\).
\end{enumerate}
As the sequence \(c_0, c_1, c_2, \ldots\) is infinite it must repeat, thus we can assume without loss of generality that we chose it to be eventually be periodic, and thus by changing the choice of \(c_0, c_1\), we can assume without loss of generality that there is \(k\in \N\) such that \(c_l = c_{l \text{ mod }k}\) for all \(l\in \N\).

We are now in the position to define maps from \(F_n\) to \(\mathbb{B}\) and \(\mathbb{C}\).
We have established some ``cycles" of note, that is \(c_0, c_1, \ldots, c_k = c_0\) and \(b_{-s}, b_{1-s}, \ldots, b_{-t}= b_{-s}\).
Moreover these cycles are mapped respectively by \(\phi_\mathbb{C}\) and \(\phi_\mathbb{B}\) to some common cycle \(a_0, a_1, \ldots, a_{p}=a_0\) in \(\mathbb{A}\) (hence both \(k\) and \(s-t\) are multiples of \(p\) but are not necessarily equal).
We define a map \(q_{\mathbb{C}}: F_n\to \mathbb{C}\) by defining \((\varepsilon)q_{\mathbb{C}}= c_0\) and inductively extending this homomorphism to all of \(F_n\) as follows:
\begin{enumerate}
    \item  If \((w_0w_1w_2\ldots w_{l-1})q_{\mathbb{C}}\) is in the cycle \(c_0, c_1, \ldots, c_{k}=c_0\) and mapping \(w_0w_1w_2\ldots w_{l}\) to the next/previous element in the cycle would preserve the fact that \(q\) is a homomorphism, then map \(w_0w_1w_2\ldots w_{l}\) to the next/previous element in the cycle.
    \item Otherwise map \(w_0w_1w_2\ldots w_{l}\) to any element of \(\mathbb{C}\) which preserves the fact that \(q_{\mathbb{C}}\) is a homomorphism.
\end{enumerate}
We define a map \(q_{\mathbb{B}}\) similarly using the cycle \(b_{-s}, b_{1-s}, \ldots , b_{-t}=b_{-s}\), and a map \(q_{\mathbb{A}}\) using the cycle \(a_0, a_1, \ldots, a_p=a_0\).
Note that we never make any arbitrary choices in the construction of \(q_{\mathbb{A}}\) as if we ever had a choice it would follow that \(\de_\mathbb{A}\) was not a function.
It follows that \[q_{\mathbb{B}}\phi_{\mathbb{B}}= q_{\mathbb{A}}=q_{\mathbb{C}}\phi_{\mathbb{C}}.\]

Thus \(\langle q_{\mathbb{B}}, q_{\mathbb{C}}\rangle_{\mathbb{B}\times\mathbb{C}}\) is a homomorphism into \(\mathbb{P}^{\operatorname{core}}\).
It is routine to verify that the edge \((a, b)\) is mapped onto by \(q_{\mathbb{B}}\), so again we have the required outcome. \(\diamondsuit\)

By Theorem~\ref{cofinal CA}, let \(\mathbb{D}\) be an \(n\)-CA structure such that there is a quotient map \(\psi:\mathbb{D}\to \mathbb{P}^{\operatorname{core}}\). The result follows by defining \(\phi_{\mathbb{D}, \mathbb{C}}:= \psi \pi_{1}\) and \(\phi_{\mathbb{D}, \mathbb{B}}:= \psi \pi_{0}\).
\end{proof}
\begin{theorem}[We have a Projective \F Class, Assumed Knowledge: \ref{projective fraisse class def}, \ref{product structures}, \ref{cofinal CA}, \ref{nca has amalgamation}]\label{CA is a PFC}
For each \(n\geq 1\), the \(n\)-CA structures form a projective \F class.
\end{theorem}
\begin{proof}
We verify conditions (1)-(5) of Definition~\ref{projective fraisse class def}.
\begin{enumerate}
    \item The defining property of the class is an isomorphism invariant, so the class is closed under isomorphisms.
    \item The class consists of finite structures in a finite signature and so has countably many isomorphism types. It also has infinitely many isomorphism types as by Theorem~\ref{cofinal CA}, it contains structures of arbitrarily large finite cardinality.
    \item The class consists of finite structures by definition.
    \item Suppose that \(\mathbb{B}\) and \(\mathbb{C}\) are \(n\)-CA structures.
    Note that the structure \(\mathbb{B}\times \mathbb{C}\) is a non-empty \(\sigma_{b, n}\)-structure such that each relation and its inverse are surjective, moreover the projection maps \(\pi_{0}\) and \(\pi_{1}\) are quotient maps.
    By Theorem~\ref{cofinal CA} we can find an \(n\)-CA structure \(\mathbb{D}\) such that there is a quotient map \(\psi: \mathbb{D}\to \mathbb{B}\times \mathbb{C}\). 
    The maps \( \psi \pi_{0}\) and \(\psi \pi_{1}\) are the required quotient maps. 
    \item This is immediate from Lemma~\ref{nca has amalgamation}.
\end{enumerate}
\end{proof}

\speeddicttwo{69}{The CA Limit}{projective fraisse limit theorem}{CA is a PFC}
\begin{defn}[The CA limit, \Assumed{69}]\label{69 ref}
For each \(n\geq 1\), define \(n\mathbb{L}\) to be a projective \F limit of the class of \(n\)-CA structures (see Theorem~\ref{projective fraisse limit theorem}).
\end{defn}

Now that for all \(n\in \N\backslash\{0\}\) we have the objects \(n\mathbb{L}\) that we wanted, we need to show that \(n\mathbb{L}\) can be viewed as a generic element of \(\Aut(2^\N)^n\).

\speeddictone{70}{useful intermediate nCA}{cofinal CA}
\begin{lemma}[Splitting points in \(n\)-CA structures, \Assumed{70}]\label{useful intermediate nCA}
Suppose that \(\mathbb{S}\) is an \(n\)-CA structure, \(i< n\), \(p\in \mathbb{S}\), \(k\geq 2\) and \((p)i^{\mathbb{S}}
=\{q_0,q_1,\ldots, q_{k-1}\}\). 
Then there is an \(n\)-CA structure \(\mathbb{S}_2\) and a quotient map \(\phi:\mathbb{S}_2\to \mathbb{S}\) with \((p)\phi^{-1}\geq k\) and such that for each \(p'\in (p)\phi^{-1}\) the set \(((\{p'\})i^{\mathbb{S}_2})\phi\) is a singleton.
\end{lemma}
\begin{proof}
Let \(\mathbb{T}\) denote the structure with universe \((\mathbb{S}\backslash \{p\}) \cup (\{p\}\times \{0, 1, \ldots, k-1\})\), and relations to be defined later (we assume without loss of generality that \(\{p\}\times \{0, 1, \ldots, k-1\}\) is disjoint from \(\mathbb{S}\)).
Define \(\phi_\mathbb{T}:\mathbb{T}\to \mathbb{S}\) by
\[(x)\phi_\mathbb{T}= \left\{\begin{array}{cc}
x   &  \text{ if }  x\in \mathbb{S}\backslash \{p\}\quad\quad\quad\quad\quad\quad\quad\\
p    & \text{ if } x\in \{p\}\times \{0, 1,\ldots, k-1\}
\end{array}\right\}.\]
For \(j\in \{0, 1, \ldots, n-1\}\backslash\{i\}\), we define
\[j^\mathbb{T} := \makeset{(a, b)}{\(((a)\phi_\mathbb{T}, (b)\phi_{\mathbb{T}})\in j^\mathbb{S}\)}\]
and \[i^\mathbb{T} := \makeset{(a, b)}{\(((a)\phi_\mathbb{T}, (b)\phi_{\mathbb{T}})\in i^\mathbb{S}\) and if \(a = (p, l)\) then \(b = q_l\)}.\]
Note that \(\phi_\mathbb{T}\) is a quotient map from \(\mathbb{T}\) to \(\mathbb{S}\) and that each \(j^\mathbb{T}\) is surjective with surjective inverse. Thus by Theorem~\ref{cofinal CA} there is an \(n\)-CA structure \(\mathbb{S}_2\) which quotients onto \(\mathbb{T}\). The composite of these two quotient maps gives the required quotient map \(\phi\).
\end{proof}

\speeddictthree{71}{Cantor linit universe}{Brouwer' Theorem}{projective fraisse limit theorem clopen sets version}{useful intermediate nCA}
\begin{lemma}[Cantor limit is a Cantor space, \Assumed{71}]\label{Cantor linit universe}
For each \(n\geq 1\), the topological structure \(n\mathbb{L}\) is homeomorphic to the Cantor set \(2^\N\).
\end{lemma}
\begin{proof}
By Theorem~\ref{Brouwer' Theorem} it suffices to show that \(n\mathbb{L}\) is non-empty, second countable, zero-dimensional, Hausdorff, compact and has no isolated points. All of these except no isolated points are immediate from Theorem~\ref{projective fraisse limit theorem}. So we show that \(n\mathbb{L}\) has no isolated points.

Suppose for a contradiction that \(p\) is an isolated point of \(n\mathbb{L}\).
Thus \(\{\{p\}, n\mathbb{L}\backslash\{p\}\}\) is a partition of \(n\mathbb{L}\) into clopen sets.
By Theorem~\ref{projective fraisse limit theorem clopen sets version} part (2), there is \(k\in \N\) and a partition \(\mathbb{S}:=\{C_0, C_1, \ldots, C_{k-1}, \{p\}\}\) of \(n\mathbb{L}\) into clopen sets such that the corresponding quotient structure is an \(n\)-CA structure.
By Lemma~\ref{useful intermediate nCA} there is an \(n\)-CA structure \(\mathbb{S}_2\) and a quotient map \(\phi:\mathbb{S}_2\to \mathbb{S}\) with \((\{\{p\}\})\phi^{-1}\geq 2\) (we can use \(\de_\mathbb{S}\) to choose a relation allowing us the guarantee this).

By Theorem~\ref{projective fraisse limit theorem clopen sets version} part (1), we can assume without loss of generality that \(\mathbb{S}_2\) is a quotient structure of \(n\mathbb{L}\) (whose elements are clopen subsets of \(n\mathbb{L}\)).
Thus the map sending \(s\in \mathbb{S}\) to \(\union{}(\{s\})\phi^{-1}\) is an isomorphism from \(\mathbb{S}\) to a quotient structure of \(n\mathbb{L}\).
By Theorem~\ref{projective fraisse limit theorem clopen sets version} part (3), we can find a topological isomorphism \(h:n\mathbb{L} \to n\mathbb{L}\) such that \((s)h= \union{}(\{s\})\phi^{-1}\) for all \(s\in \mathbb{S}\).
Let \(A, B\in (\{\{p\}\})\phi^{-1}\) be distinct.
It follows that \((A)h^{-1}\) and \((B)h^{-1}\) are disjoint non-empty subsets of \(\{p\}\), this is a contradiction.
\end{proof}

\speeddictthree{72}{Cantor limit homeomorphisms}{product sets and projection maps def}{compactness facts}{Cantor linit universe}
\begin{lemma}[Cantor limit relations are homeomorphisms, \Assumed{72}]\label{Cantor limit homeomorphisms}
If \(n\geq 1\) and \(i\in \{0, 1, \ldots, n-1\}\), then \(i^{n\mathbb{L}}\) is a homeomorphism of \(n\mathbb{L}\).
\end{lemma}
\begin{proof}
We first show that \(i^{n\mathbb{L}}\) is a function. Let \(x\in n\mathbb{L}\) be arbitrary, and suppose that \((x, y), (x, z)\in i^{n\mathbb{L}}\). It suffices to show that \(y= z\). Suppose for a contradiction that \(y\neq z\) and let \(\{U_y, U_z\}\) be a partition of \(n\mathbb{L}\) into disjoint clopen subsets of \(n\mathbb{L}\) with \(y\in U_y\), \(z\in U_z\). 
By Theorem~\ref{projective fraisse limit theorem clopen sets version} part (2) there is a finite refinement \(\mathbb{S}\) of this partition into clopen sets such that if we view \(\mathbb{S}\) as quotient structure of \(n\mathbb{L}\), then \(\mathbb{S}\) is an \(n\)-CA structure.

Using the element of \(\mathbb{S}\) containing \(x\) as \(p\), let \(\phi\) and \(\mathbb{S}_2\) be as in Lemma~\ref{useful intermediate nCA}. 
By Theorem~\ref{projective fraisse limit theorem} part (3), there is a continuous quotient map \(\psi:n\mathbb{L} \to \mathbb{S}_2\) such that \(\psi\phi\) is the map sending an element of \(n\mathbb{L}\) to the element of \(\mathbb{S}\) containing it.

Thus by the choice of \(\mathbb{S}_2\), each element of \( (\{[x]_{\ker(\psi)}\})i^{n\mathbb{L}/\ker(\psi)}\) is contained in a common element of \(\mathbb{S}\) (and hence only one element of \(\{U_x, U_y\}\)).
This is a contradiction as \(x\) is related to both \(y\) and \(z\) by \(i^{n\mathbb{L}}\).

We have shown that \(i^{n\mathbb{L}}\) is a function, by symmetry \((i^{n\mathbb{L}})^{-1}\) is also a function and thus \(i^{n\mathbb{L}}\) is a bijection. It remains to show that \(i^{n\mathbb{L}}\) is continuous. As \(n\mathbb{L}\) is a topological structure, we have that \(i^{n\mathbb{L}}\) is closed and hence compact subspace of \(n\mathbb{L}^2\) (Remark~\ref{compactness facts}).
 As \(i^{n\mathbb{L}}\) is a bijection, both the maps \(\pi_0\) and \(\pi_1\) from \(i^{n\mathbb{L}}\) to \(n\mathbb{L}\) are bijections.
 Moreover (as projections) both the maps \(\pi_0\restriction_{i^{n\mathbb{L}}}\), \(\pi_1\restriction_{i^{n\mathbb{L}}}\) are continuous and thus are homeomorphisms as they map between compact Hausdorff spaces.
 As the function \(i^{n\mathbb{L}}\) is equal to the composition of the inverse of \(\pi_0\) (using \(i^{n\mathbb{L}}\) as the domain of \(\pi_0\)) with \(\pi_1\), the result follows.
\end{proof}

\speeddictseven{75}{ample canor}{topological structures defn}{projective fraisse limit theorem clopen sets version}{automorphism groups defn}{ample generics defn}{cofinal CA}{Cantor limit homeomorphisms}{Composition of continuous maps is continuous}
\begin{theorem}[cf. Theorem 4.7 of \cite{good_generic_Cantor}, \(\Aut(2^\N)\) has ample generics, \Assumed{75}]\label{ample canor}
The action \(\conj_{\Aut(2^\N)}\) has ample generics (using the compact-open topology on \(\Aut(2^\N)\)) .
\end{theorem}
\begin{proof}
For all \(f=(f_0, f_1, \ldots, f_{n-1})\in \Aut(2^\N)^n\), let \(\mathbb{S}_f\) denote the \(\sigma_{b, n}\) structure with universe \(2^\N\) and \(i^{\mathbb{S}_f}= f_i\) for all \(i< n\). We show that the set 
\[C :=\makeset{f\in \Aut(2^\N)^n}{\(\mathbb{S}_f \cong n\mathbb{L}\)}\]
is comeagre.

Let \(R= \makeset{\mathbb{R}_i}{\(i\in \N\)}\) be a collection of \(n\)-CA structures such that every \(n\)-CA structure is isomorphic to precisely one element of \(R\).
Let \(\Clo(2^\N)\) denote the collection of clopen subsets of \(2^\N\), let \(\CloP(2^\N)\) denote the set of finite partitions of \(2^\N\) into clopen sets.
Let \(\PS(2^\N)\) denote the set of surjections from elements of \(\CloP(2^\N)\) to elements of \(R\).
Let \(\QR\) denote the set of quotient maps between elements of \(R\).
Note that the sets \(R,\Clo(2^\N), \CloP(2^\N), \QR\) and \(\PS(2^\N)\) are all countable.

By Theorem~\ref{compact-open is Polish thm}, it follows that the sets of the form
\[\makeset{f\in \Aut(2^\N)}{\((U)f = V\)} = B(U, V)\cap B(U^c, V^c) \]
where \(U, V\in \Clo(2^\N)\), form a subbasis for the topology on \(\Aut(2^\N)\).

\underline{Claim:} The sets of the form 
\[U_{\phi}:= \makeset{f\in \Aut(2^\N)^n}{ \(\phi\) is a quotient map from \(\mathbb{S}_f\) to \(\im(\phi)\)}\]
where \(\phi\) can be any surjection from an element of \(\CloP(2^\N)\) to a \(\sigma_{b, n}\)-structure, are a basis of clopen sets for the topology on \(\Aut(2^\N)^n\).\\
\underline{Proof of Claim}:\\
Let \(\phi\) be an arbitrary surjection from a partition of \(2^\N\) into clopen sets to a \(\sigma_{b, n}\)-structure.
To see that \(U_\phi\) is clopen in the compact-open topology, note that it is the intersection of the following two sets
\[ \intersection{i<n\\(a, b)\in i^{\im(\phi)} } \union{U, V\in \Clo(2^\N)\\
U \subseteq \union{} ((a)\phi^{-1})\\
V\subseteq \union{} ((b)\phi^{-1})} \makeset{(f_0, f_1,\ldots, f_{n-1})\in \Aut(2^\N)^n}{\((U)f_i = V\)},\]
\[ \intersection{i<n\\(a, b)\notin i^{\im(\phi)} } \union{ V\in \Clo(2^\N)\\
V\cap \left(\union{} ((b)\phi^{-1})\right) = \varnothing} \makeset{(f_0, f_1,\ldots, f_{n-1})\in \Aut(2^\N)^n}{\((\union{} ((a)\phi^{-1}))f_i = V\)}.\]
We now need to show that every set of the form
\[\intersection{j<k}\makeset{(f_0,f_1, \ldots, f_{n-1})\in \Aut(2^\N)^n}{\((U_j)f_i = V_j\)}\]
for \(i<n\), \(k\in \N\), \(U_0, U_1, \ldots, U_{k-1} V_0, V_1, \ldots, V_{k-1}\in \Clo(2^\N)\) is a union of sets of the form \(U_\phi\).
Let \(f=(f_0,f_1, \ldots, f_{n-1})\) in the above set be arbitrary.
Let \(P\) be a partition of \(2^\N\) into clopen sets such that each of the sets \(U_j\) and \(V_j\) can be expressed as unions of elements of \(P\) and let \(\phi\) be the identity map from \(P\) to \(P\) (where we view the second \(P\) as a quotient structure of \(\mathbb{S}_f\)).
Note that \(f\in U_\phi\) and let \(g= (g_0, g_1,\ldots, g_{n-1})\in U_\phi\) be arbitrary. 
Then \((F)g_i\subseteq V_j\) for all \(j<k\) and \(F\in P\) with \(F\subseteq U_j\). 
Similarly \((F)g_i \cap V_j = \varnothing\) for all \(j<k\) and \(F\in P\) with \(F\cap U_j = \varnothing\).
As \(g\) was arbitrary, it follows that 
\[U_\phi\subseteq \intersection{j<k}\makeset{(f_0,f_1, \ldots, f_{n-1})\in \Aut(2^\N)^n}{\((U_j)f_i = V_j\)}\]
as required.\(\diamondsuit\)

We then define:
\begin{align*}
    U_1 &:= \intersection{i\in \N} \union{\phi\in \PS(2^\N) \\ \im(\phi) = \mathbb{R}_i} U_\phi,\\
    U_2 &:= \intersection{P\in \CloP(2^\N)} \union{\phi\in \PS(2^\N)\\\ker(\phi) \text{ refines }P} U_\phi,\\
    U_3 &:= \intersection{\phi \in \PS(2^\N)\\
    g\in \QR\\
    \im(g) = \im(\phi)}\left((\Aut(2^\N)^n\backslash U_{\phi}) \cup \left(\union{\psi \in \PS(2^\N)\\
    \dom(g) = \im(\psi)\\
    \psi g=\phi} U_{\psi} \right)\right).
\end{align*}
From Theorem~\ref{projective fraisse limit theorem}, it follows that \(C= U_1\cap U_2\cap U_3\).

We have now shown that \(C\) is a countable intersection of open sets.
We need now only show that \(C\) is dense. Let \(\phi\in \PS(2^\N)\) be an arbitrary surjection from a clopen partition of \(2^\N\) to a \(\sigma_{b, n}\)-structure. 
If \(U_\phi \neq \varnothing\), then \(\im(\phi)\) has the property that all its relations are surjective with surjective inverse.
Thus by Theorem~\ref{cofinal CA}, we can assume without loss of generality that \(\im(\phi)\) is an \(n\)-CA structure.

It suffices to show that \(U_\phi \cap C\) is not empty.
By Lemmas~\ref{Cantor linit universe} and \ref{Cantor limit homeomorphisms}, let \(f\in C\) be fixed.
As \(f\in U_1\), there is \(\phi'\in \PS(2^\N)\) such that \(\im(\phi')= \im(\phi)\) and \(f\in U_{\phi'}\).
Let \(q:2^\N \to \im(\phi)\) be defined by \((x)q= (U)\phi'\) where \(U\) is the element of \(\dom(\phi')\) containing \(x\).
We have that \(q:\mathbb{S}_f\to \im(\phi)\) is a continuous quotient map.
Let \(h\) be any homeomorphism of \(2^\N\) such that for all \(p\in \im(\phi)\) we have \(((\{p\})q^{-1})h = \union{}((\{p\})\phi^{-1})\) (Theorem~\ref{projective fraisse limit theorem clopen sets version} (3)).
It follows that \((f, h)\conj_{\Aut(2^\N)}^n\) is an element of \(C\cap U_{\phi}\) as required.
\end{proof}

\subsection{Automatic continuity for groups}\label{ac for groups subsection}
In this subsection we use the ample generics results from the previous two subsection to show that all homomorphisms from either of \(\Sym(\N)\) or \(\Aut(2^\N)\) to second countable groups are continuous.

This is done via a result (Theorem~\ref{ample automatic continuity thm}) of the paper \cite{kechris2007turbulence} of Kechris and Rosendal, which states that any Polish group whose conjugacy action has ample generics will satisfy this condition. 
The results in this subsection are all individually present in some form in the literature but we include them so we can present a self contained proof of Theorem~\ref{automatic continuity examples thm}.

\speeddictfour{77}{countable covers lem}{nbhds defn}{second countable defn}{denseness defn}{topological structures defn}
\begin{lemma}[Countable covers, \Assumed{77}]\label{\dict{77 ref}}
Suppose that \(G\) is a second countable topological group with identity \(1_G\). If \(N\) is a neighbourhood of \(1_G\), then there is a sequence \((g_{i})_{i\in \N}\) of elements of \(G\) such that \(G = \union{i\in \N} N g_i\).
\end{lemma}
\begin{proof}
As \(G\) is a topological group and \(1_G=1_G^{-1}\), it follows that \(A := N\cap N^{-1}\) is a neighbourhood of \(1_G\).
Let \(B\) be a countable basis for \(G\) with \(\varnothing \notin B\).
By the axiom of choice, let \((g_i)_{i\in \N}\) be a sequence of elements of \(G\) such that every element of \(B\) contains \(g_i\) for some \(i\in \N\).
It follows that \(\makeset{g_i}{\(i\in \N\)}\) is a dense subset of \(G\).
It suffices to show that \(G = \union{i\in \N} N g_i\).
Let \(g\in G\) be arbitrary. 
The set \(Ag\) is a neighbourhood of \(g\), so there is \(i\in \N\) such that \(g_i\in Ag\). It follows that \(g\in A^{-1}g_i= Ag_i\subseteq N g_i\) as required.
\end{proof}

The following theorem is often referred to as (part of) Effros' Theorem.
For more information see \cite{effrosvan2004note} for example.
\speeddictseven{78}{microtrans}{closure and interior defn}{topology on a metric space defn}{Polish space defn}{meagre defn}{topological structures defn}{group actions defn}{countable covers lem}
\begin{theorem}[Effros' Theorem, \Assumed{78}]\label{\dict{78 ref}}
Suppose that \(G\) is a Polish group and suppose that \(a:X\times G\to X\) is a continuous and transitive action of \(G\) on a metrizable space \(X\) which is not meagre in itself. 
If \(N\) is a neighbourhood of \(1_G\) in \(G\) and \(x\in X\), then \((x,N)a\) is a neighbourhood of \(x\) in \(X\). 
\end{theorem}
\begin{proof}
Let \(d_G\) be a complete metric compatible with \(G\), and \(d_X\) be a metric compatible with \(X\).

\underline{Claim:} If \(A\) is a neighbourhood of \(1_G\) in \(G\) and \(z\in X\), then \(z\in ((z,A)a)^{-\circ}\).\\
\underline{Proof of Claim:} Let \(A_2\) be a neighbourhood of \(1_G\) in \(G\) such that \(A_2^{-1}=A_2\) and \(A_2^2\subseteq A\).
By Lemma~\ref{countable covers lem}, let \((t_i)_{i\in \N}\) be a sequence of elements of \(G\) such that \(\union{i\in \N}A_2t_i= G\).
It follows that \(\union{i\in \N} (z,A_2t_i)a = X\).
As \(X\) is not meagre in itself, it follows that at least one of the sets \((z,A_2t_i)a\) is somewhere dense and so \((z,A_2)a\) is somewhere dense.
If \(b\in A_2\) is such that \((z,b)a\in ((z,A_2)a)^{-\circ}\), then \(z\in (((z,A_2)a)^{-\circ},b^{-1})a=((z,A_2b^{-1})a)^{-\circ}\subseteq  ((z,A)a)^{-\circ}\) as required.\(\diamondsuit\)

Let \(V\) be an open neighbourhood of \(1_G\) such that \(V^{-1} = V\) and \((V^-)^2\subseteq N\).
By the claim,  we have \(x\in ((x, V)a)^{-\circ}\).
It thus suffices to show that \(((x, V)a)^{-\circ}\subseteq (x,N)a\).

Let \(y\in ((x,V)a)^{-\circ}\) be arbitrary.
As \(V\) is inverse closed, it follows that similarly \(x\in ((y,V)a)^{-\circ}\).
We define sequences \((g_i)_{i\in \N}\) and \((h_i)_{i\in \N}\) in \(G\) inductively as follows:
\begin{enumerate}
    \item Choose \(g_0\in  V\) to be such that 
    \[(x, g_0)a \in B_{d_X}\left(y, \frac{1}{2^0}\right)\cap \left(\left(y,B_{d_G}\left(1_G, \frac{1}{2^0}\right)\cap V\right)a\right)^{-\circ}\]
    (we can do this as \((x,V)a\) is dense in a neighbourhood of \(y\)).
    
    \item Choose \(h_0\in B_{d_G}\left( 1_G, \frac{1}{2^0}\right)\cap V\) to be such that 
    \[(y, h_0)a \in B_{d_X}\left((x, g_0)a, \frac{1}{2^1}\right)\cap \left(\left(x,B_{d_G}\left(g_0, \frac{1}{2^1}\right)\cap V\right)a\right)^{-\circ}\]
     (we can do this as \(\left(y, B_{d_G}\left( 1_G, \frac{1}{2^0}\right)\cap V\right)a\) is dense in a neighbourhood of \((x, g_0)a\)).
     
    \item Suppose that \(g_{i-1}\) and \(h_{i-1}\) are defined. Choose \(g_i\in B_{d_G}\left(g_{i-1}, \frac{1}{2^i}\right)\cap V\) to be such that 
    \[(x, g_i)a \in B_{d_X}\left((y, h_{i-1})a, \frac{1}{2^i}\right)\cap \left(\left(y,B_{d_G}\left(h_{i-1}, \frac{1}{2^i}\right)\cap V\right)a\right)^{-\circ}\]
     (we can do this as \(\left(x, B_{d_G}\left(g_{i-1}, \frac{1}{2^i}\right)\cap V\right)a\) is dense in a neighbourhood of \((y, h_{i-1})a\)).
     
    \item Suppose that \(g_{i}\) and \(h_{i-1}\) are defined. Choose \(h_i\in B_{d_G}\left(h_{i-1}, \frac{1}{2^i}\right)\cap V\) to be such that 
    \[(x, h_i)a \in B_{d_X}\left((x, g_{i})a, \frac{1}{2^{i+1}}\right)\cap \left(\left(x, B_{d_G}\left(g_{i}, \frac{1}{2^{i+1}}\right)\cap V\right)a\right)^{-\circ}\]
     (we can do this as \(\left(y,B_{d_G}\left(h_{i-1}, \frac{1}{2^i}\right)\cap V\right)a\) is dense in a neighbourhood of \((x, g_i)a\)).
\end{enumerate}
By construction, the sequences \((g_i)_{i\in \N}\) and \((h_i)_{i\in \N}\) are Cauchy and so they converge to some \(g, h\in V^-\) respectively.
As our action is continuous, it follows that \(((x, g_i)a)_{i\in \N}\) and \(((y, h_i)a)_{i\in \N}\) converge to \((x, g)a\) and \((y,h)a\) respectively.

For all \(i\in \N\) we have \(((x,g_{i})a, (y,h_i)a)d_X\leq \frac{1}{2^i}\), so we must have that \((x, g)a=(y, h)a\).
Thus \(y=(x, g h^{-1})a\), and so \(y\in (x,V^-(V^{-1})^-)a=(x, (V^-)^2)a\subseteq (x, N)a\) as required.
\end{proof}

\speeddictseven{79}{generic extensions}{nbhds defn}{Polish space defn}{meagreness in subspaces rmk}{kuratowsiulam}{baire category theorem}{ample generics defn}{microtrans}
\begin{lemma}[cf. Lemma 6.6 of \cite{kechris2007turbulence}, Extending generic tuples, \Assumed{79}]\label{\dict{79 ref}}
Suppose that \(G\) is a Polish Group, the action \(\conj_{G}\) has ample generics,  \(A\subseteq G\) is not meagre in \(G\), \((x_0, x_1, \ldots, x_{n-1})\in~G^n\) has a comeagre orbit under \(\conj_G^n\), \(B\subseteq G\) is such that for all comeagre subsets \(C\) of \(G\) the set \(B\cap C \) is dense in \(G\), and \(N\) is a neighbourhood of the identity \(1_G\) of \(G\).

In this case there is \((a, g, b)\in A\times N \times B\) such that
\[((x_0, x_1, \ldots, x_{n-1}, a),g)\conj_G^n = (x_0, x_1, \ldots, x_{n-1}, b)\]
has comeagre orbit under \(\conj_G^{n+1}\).
\end{lemma}
\begin{proof}
For the purposes of this proof, we will say that a element of \(G^n\) is \textit{generic} if its orbit under \(\conj_G^n\) is comeagre.
By Theorems~\ref{baire category theorem} and \ref{kuratowsiulam}, there exists a generic \((z_0, z_1,\ldots, z_{n-1})\in G^n\) such that 
\[G \approx_G \makeset{y\in G}{\((z_0, z_1,\ldots, z_{n-1}, y)\)  is generic}.\]
Let \(h\in G\) be such that \((z_0, z_1,\ldots, z_{n-1})= ((x_0, x_1, \ldots, x_{n-1}),h)\conj_G^n\), we now have:
\begin{align*}
    G &= h G h^{-1}\\
    &\approx_G h\makeset{y\in G}{\((z_0, z_1,\ldots, z_{n-1}, y)\)  is generic}h^{-1}\\
    &= \makeset{h y h^{-1}}{\((z_0, z_1,\ldots, z_{n-1}, y)\)  is generic}\\
    &= \makeset{y\in G}{\((z_0, z_1,\ldots, z_{n-1}, h^{-1}y h)\)  is generic}\\
  &= \makeset{y\in G}{\(((x_0, x_1,\ldots, x_{n-1}, y), h)\conj_G^n\) is generic}\\
  &= \makeset{y\in G}{\((x_0, x_1,\ldots, x_{n-1}, y)\) is generic}.
\end{align*}
Let \(a\in A\) be such that \((x_0, x_1,\ldots, x_{n-1}, a)\) is generic. Let \(G_{x_0, x_1,\ldots, x_{n-1}}\) be the subgroup of elements of \(G\) which commute with each of \(x_0, x_1, \ldots, x_{n-1}\). Note that 
\[\makeset{y\in G}{\((x_0, x_1,\ldots, x_{n-1}, y)\) is generic} = (a, G_{x_0, x_1,\ldots, x_{n-1}})\conj_G.\]
Moreover this set is comeagre and so it is not meagre in \(G\) by Theorem~\ref{baire category theorem}.
In particular the space \((a, G_{x_0, x_1,\ldots, x_{n-1}})\conj_G\) is not meagre in itself by Remark~\ref{meagreness in subspaces rmk}.
As \(G_{x_0, x_1,\ldots, x_{n-1}}\leq G\) is closed, it is Polish and so by Theorem~\ref{microtrans}, we have that \((a, G_{x_0, x_1,\ldots, x_{n-1}}\cap N)\conj_G\) is a neighbourhood of \(a\) in the space \((a, G_{x_0, x_1,\ldots, x_{n-1}})\conj_G\).
By assumption \(B\cap (a, G_{x_0, x_1,\ldots, x_{n-1}})\conj_G\) is dense in \(G\) and so in particular, it follows that there is some \(b\in B\cap (a, G_{x_0, x_1,\ldots, x_{n-1}}\cap N)\conj_G\) as required.
\end{proof}

\speeddictthree{80}{continuous generic Cantor map}{Cantor space defn}{words defn}{generic extensions}
\begin{lemma}[cf. Lemma 6.7 of \cite{kechris2007turbulence}, Generic Cantor map, \Assumed{80}]\label{\dict{80 ref}}
Suppose that \(G\) is a Polish Group, the action \(\conj_{G}\) has ample generics,  \(A\subseteq G\) is not meagre in \(G\), \(B\subseteq G\) is such that for all comeagre subsets \(C\) of \(G\) the set \(B\cap C \) is dense in \(G\).

In this case, there is a continuous map \(a\mapsto h_a\) from \(2^\N \to G\) such that if \(p\in \{0, 1\}^*\), \(a,b\in 2^\N\), \(p0< a\) and \(p1<b\), then \(h_a^{-1}Ah_a \cap h_b^{-1}B h_b \neq \varnothing\) (recall Definition~\ref{words defn}).
\end{lemma}
\begin{proof}
Let \(d\) be a complete metric for \(G\). We define a maps \(w\to h_w\), \(w\to x_w\) from \(\{0, 1\}^*\to G\) inductively to satisfy for following properties:

\begin{enumerate}
    \item If \(w\in \{0, 1\}^n\) then the tuple
    \((x_{w\restriction_{0}},\ x_{w\restriction_{1}},\ x_{w\restriction_{2}},\ \ldots,\ x_{w})\)
    has a generic orbit under \(\conj_G^{n+1}\).
    
    \item \(h_{\varepsilon} = 1_G\).
    
    \item If \(w\in \{0, 1\}^n\) and \((n-1)w = 0\) then \(x_{w}\in A\).
    \item If \(w\in \{0, 1\}^n\) and \((n-1)w = 1\) then \(x_{w}\in B\).
    \item If \(w\in \{0, 1\}^n\) and \((n-1)w = 1\) then \(h_{w}=h_{w\restriction_{n-1}}\).
    \item If \(w_0, w_1\in \{0, 1\}^n\), \((n-1)w_0 = 0\), \((n-1)w_1= 1\) and \(w_0\restriction_{n-1} =w_{1}\restriction_{n-1}\) then \((h_{w_0},h_{w_1})d \leq \frac{1}{2^{n}}\), and moreover there is some \(g\in G\) which commutes with each of \(x_{w_0\restriction_{0}}, x_{w_0\restriction_{1}},\ldots ,x_{w_0\restriction_{n-1}}\) and such that \((x_{w_0},g)\conj_G=x_{w_1}\) and \(h_{w_0} = g h_{w_1}\).
\end{enumerate}
We do this by defining \(h_\varepsilon = 1_G\) and \(x_\varepsilon\) to have comeagre orbit under \(\conj_G\), we then define the rest by induction on \(n\) (the inductive step follows from Lemma~\ref{generic extensions}).

For \(w\in 2^\N\), we define \(h_{w}\) to be the limit of the sequence \((h_{w\restriction_n})_{n\in \N}\) (by construction this sequence is Cauchy and so converges). By construction, if \(a\neq b\in 2^\N\), \(n= \min\makeset{i\in \N}{\((i)a\neq (i)b\)}\) and \((n)a =0\), then \(x_{a\restriction_{n+1}}\in A\), \(x_{b\restriction_{n+1}}\in B\).
Note also that (by conditions (5) and (6)) the sequences \(((x_{a\restriction_{n+1}},h_{a\restriction_k})\conj_G)_{k\geq  n+1}\) and \(((x_{b\restriction_{n+1}},h_{b\restriction_k})\conj_G)_{k\geq  n+1}\) are constant.
Thus:
\begin{align*}
    (x_{a\restriction_{n+1}},h_a)\conj_G&=(x_{a\restriction_{n+1}},h_{a\restriction_{n+1}})\conj_G  \\
    &=(x_{a\restriction_{n+1}},g h_{b\restriction_{n+1}})\conj_G \quad \text{ where }g\text{ is as in Condition }(6)\\
    &=(x_{b\restriction_{n+1}},h_{b\restriction_{n+1}})\conj_G\\
    &=(x_{b\restriction_{n+1}},h_{b})\conj_G.
\end{align*}
In particular this value is an element of \(h_a^{-1}Ah_a\cap h_b^{-1}B h_b\) as required.
\end{proof}

\speeddicttwo{81}{ample automatic continuity thm}{structure hom defn}{continuous generic Cantor map}
\begin{theorem}[cf. Theorem 6.24 of \cite{kechris2007turbulence}, Ample continuity, \Assumed{81}]\label{ample automatic continuity thm}
If \(G\) is a Polish Group and the action \(\conj_{G}\) has ample generics, then every homomorphism from \(G\) to a second countable topological group is continuous.
\end{theorem}
\begin{proof}
Let \(H\) be a second countable topological group. Let \(\phi: G\to H\) be a homomorphism.
We need to show that \(\phi\) is continuous at \(1_G\), (that the preimage of every neighbourhood of \(1_H\) is a neighbourhood of \(1_G\)).
Let \(N\) be a neighbourhood of \(1_H\), let \(V\) be an inverse closed neighbourhood of \(1_H\) with \(V^{20}\subseteq N\).
Define \(A:= (V)\phi^{-1}\), we have that
\[G= \union{h\in (G)\phi} (V h)\phi^{-1}.\]
By Lemma~\ref{countable covers lem} (as \((G)\phi\) is second countable) there is a sequence \((g_i)_{i\in \N}\) with 
\[G= \union{i\in \N} (V (g_i)\phi)\phi^{-1}=\union{i\in \N} Ag_i.\]
Thus \(A\) is not meagre in \(G\).

\underline{Claim:} There is a non-empty open \(U\subseteq G\) such that \(U\cap A^5\) is comeagre in \(U\).\\
\underline{Proof of Claim:} Suppose that \(A^5\) is not comeagre in any non-empty open set \(U\). It follows that \(B:=(A^5)^c\) is not meagre in any non-empty open set. If \(C\) is comeagre in \(G\), then \(C\) is comeagre in all non-empty open subsets of \(G\) so \(B\cap C\) is dense in \(G\).
By Lemma~\ref{continuous generic Cantor map}, let \(w\to h_w\) from \(2^\N\) be a continuous map such that if \(p\in \{0, 1\}^*\), \(a, b\in 2^\N\), \(p0\leq a\) and \(p1\leq b\) then \((h_a^{-1}A{h_a}) \cap (h_b^{-1}B{h_b}) \neq \varnothing\).

If \(h\in G\), then 
\[(h^{-1}A{h}) \cap (h^{-1}B{h})= (h^{-1}A{h}) \cap (h^{-1}((A^5)^c){h})\subseteq (h^{-1}A{h}) \cap (h^{-1}(A^c){h})=\varnothing.\]
It follows that the map \(w\to h_w\) is injective. As \(2^\N\) is uncountable, there must be \(i\in \N\) and \(a\neq b\in 2^\N\) such that \(h_a,h_b\in Ag_i\). 
Assume without loss of generality, that there is \(p\in \{0, 1\}^*\) such that \(p0\leq a\) and \(p1\leq b\).

Note that
\begin{align*}
    (A,{h_a h_b^{-1}})\conj_G&=(A,{h_a g_i^{-1}g_i h_b^{-1}})\conj_G\\
    &= (h_b g_i^{-1})(h_a g_i^{-1})^{-1} A(h_a g_i^{-1}) (h_b g_i^{-1})^{-1}\\
    &\subseteq A^5.
\end{align*}
But by assumption \((B, h_b)\conj_G \cap (A,{h_a})\conj_G \neq \varnothing\), so \(B\cap (A,{h_a h_b^{-1}})\conj_G\neq \varnothing\) which is a contradiction. \(\diamondsuit\)\\

So there is a non-empty open set subset of \(G\) in which \(A^5\) is comeagre. As \(A\) (and thus \(A^5\)) are inverse closed, it follows that \(A^{10}\) is comeagre in an open neighbourhood \(U\) of \(1_G\). Let \(h\in U\) be arbitrary.
We have that \(U\cap Uh\) is an open neighbourhood of \(h\). Thus \(A^{10}\cap A^{10}h\) is comeagre in \(U\cap Uh\). So there are \(h_0, h_1\in A^{10}\) such that \(h_0 =h_1h\). It follows that \(h\in A^{20}\).

As \(h\) was arbitrary we obtain \(U\subseteq A^{20}\subseteq (N)\phi^{-1}\), and \(\phi\) is continuous as required.
\end{proof}

\speeddictthree{252}{automatic continuity examples thm}{ample automatic continuity thm}{compact-open group is Polish prop}{pointwise group is Polish prop}
\begin{theorem}[Ample automatic examples, \Assumed{81}]\label{automatic continuity examples thm}
If \(G\) is one of the following topological groups, then all homomorphisms from \(G\) to second countable topological groups are continuous.
\begin{enumerate}
    \item \(\Sym(\N)\) with the pointwise topology (recall Proposition~\ref{pointwise group is Polish prop}).
    \item \(\Aut(2^\N)\) with the compact-open topology (recall Proposition~\ref{compact-open group is Polish prop}).
\end{enumerate}
\end{theorem}
\begin{proof}
The groups \(\Sym(\N)\) and \(\Aut(2^\N)\) are Polish groups by Propositions \ref{pointwise group is Polish prop} and \ref{compact-open group is Polish prop} respectively.
Moreover, by Theorems \ref{ample symmetry} and \ref{ample canor} respectively we know that \(\conj_{\Sym(\N)}\) and \(\conj_{\Aut(2^\N)}\) have ample generics.
The result then follows from Theorem~\ref{ample automatic continuity thm}.
\end{proof}

\subsection{Rubin's Theorem}
Rubin's Theorem was first proved by Matatyahu Rubin in \cite{Rubin1989reconstruction} Corollary~3.5.
The theorem shows that for a given group, there is at most one way for it to act ``nicely" on a ``nice" topological space (see Theorem~\ref{Rubin's Theorem}).
In particular, the topological space is constructable entirely from the group. 
In this subsection we give a proof of Rubin's Theorem, and conclude with the corollary (Corollary~\ref{Rubin Automorphisms cor}) to Rubin's Theorem which we will require in Part~\ref{nv section}. 
The author of this thesis learned of Rubin's proof via a set of notes taken by Francesco Matucci and James Belk on Rubin's original proof.
What is presented here is a minimalistic version of the proof which avoids all references to boolean algebras and gives an explicit description of the Rubin action of a given group (see Definition~\ref{canonical Rubin action defn}).

\speeddictthree{82}{support and fix defns}{Hausdorff defn}{substructures defn}{group actions defn}
\begin{defn}[Support and fix, \Assumed{82}]\label{\dict{82 ref}}
Suppose that \(G\) is a group, \(a:X\times G\to X\) is a continuous action of \(G\) on a topological space \(X\), and \(g\in G\).
We define
\[\supt_a(g) := \{x\in X:(x,g)a\neq x\},\quad \fix_a(g) := \{x\in X:(x,g)a= x\}.\]
Moreover if \(U\subseteq X\) then we define
\[G_{U, a}:= \{g\in G: \supt_a(g) \subseteq U\}.\]
It is routine to verify that \(G_{U, a}\) is always a subgroup of \(G\), and if \(X\) is Hausdorff, then the set \(\supt_a(g)\) is always open, and the set \(\fix_a(g)\) is always closed.
\end{defn}

\speeddictthree{84}{locally moving defn}{nbhds defn}{denseness defn}{support and fix defns}
\begin{defn}[Locally moving, \Assumed{84}]\label{locally moving defn}
If \(a:X\times G \to X\) is an action of a group \(G\) on a topological structure \(X\), then we say that \(a\) is \textit{locally moving} if for all \(x\in X\) and neighbourhoods \(U\) of \(x\), we have that \(x\in((x,G_{U, a})a)^{-\circ}\).
\end{defn}

The support sets (Definition~\ref{support and fix defns}), are precisely how the topological space will be constructed.
Due to the restrictions we will put on the action, the interplay between the sets of this type will completely determine the topological space we need.

The following lemma (Lemma~\ref{comsupt}), is our means for understanding the action of the group on these sets using the group alone.

\speeddictfour{85}{comsupt}{nbhds defn}{conjugation action defn}{commutator defn}{support and fix defns}
\begin{lemma}[Support tricks, \Assumed{85}]\label{comsupt}
If \(a:X\times G \to X\) is a continuous action of a group \(G\) on a topological space \(X\), and \(f, g\in G\), then
\[\supt_a((f, g)\conj_G)= (\supt_a(f),g)a,\]
\[\supt_a([f, g])\subseteq \supt_a(f) \cup (\supt_a(f),g)a,\quad \text{ and }\quad \supt_a([f, g])\subseteq \supt_a(g) \cup (\supt_a(g),f)a.\]
Moreover if \(X\) is Hausdorff and \(x\in \supt_a(f),\) then there is an open neighbourhood \(U\) of \(x\) such that \(U\subseteq \supt_a(f)\) and \((U, f)a\cap U = \varnothing\).
\end{lemma}
\begin{proof}
The first equality can be seen as follows:
\begin{align*}
   x\in  \supt_a((f, g)\conj_G)&\iff(x,(f, g)\conj_G)a\neq x \\
    &\iff (x,g^{-1}f g)a\neq x\\
    &\iff (x,g^{-1}f)a\neq (x, g^{-1})a\\
    &\iff ((x,g^{-1})a, f)a\neq (x, g^{-1})a\\
    &\iff (x,g^{-1})a \in \supt_a(f)\\
    &\iff x \in (\supt_a(f),g)a.
\end{align*}

We now show the first of the inclusions, the second is proved similarly. From the first equality it follows that if \(x\notin \supt_a(f) \cup (\supt_a(f), g)a\) then \(x\in \fix_a(f)\cap \fix_a((f,g)\conj_G)=\fix_a(f^{-1})\cap \fix_a((f,g)\conj_G)\). 
So in particular, if \(x\notin  \supt_a(f) \cup (\supt_a(f), g)a\) then
\[(x,[f, g])a=(x,f^{-1}(f,g)\conj_G)a=(x, (f,g)\conj_G)a=x,\]
the inclusion follows.

Finally suppose that \(X\) is Hausdorff and \(x\in \supt_a(f)\). Let \(U_0, U_1\) be disjoint open sets containing \(x, (x, f)a\) respectively. As \(X\) is Hausdorff, \(\supt_a(f)\) is an open neighbourhood of both \(x\) and \((x, f)a\). Thus by replacing \(U_0, U_1\) with their intersection with \(\supt_a(f)\), we can assume without loss of generality that \(U_0, U_1\) are both contained in \(\supt_a(f)\). Let \(U:= U_0\cap (U_1, f^{-1})a\). The set \(U\) is the required neighbourhood of \(x\).
\end{proof}

The main tool we still need to establish, before we proceed with the main proof, is a purely algebraic means of describing how the support sets (Definition~\ref{support and fix defns}) compare with each other.
This role is served by the notion of ``Algebraic Disjointness" (as we will see in Theorem~\ref{Algebraic disjointness}).

\speeddicttwo{87}{algebraic disjointness defn}{commutator defn}{centralizer defn}
\begin{defn}[Algebraic disjointness, \Assumed{87}]\label{\dict{87 ref}}
If \(G\) is a group and \(f, g\in G\), then we say that \(f\) is \textit{algebraically disjoint} from \(g\) if for all \(h\in G\backslash C_G(f)\) there are \(h_{0, g}, h_{1,g}\in C_G(g)\) such that \([[h, h_{0,g}],h_{1, g}]\) is a non-identity element of \(C_G(g)\).
\end{defn}

\speeddictfive{88}{Algebraic disjointness}{isolated points defn}{Hausdorff defn}{locally moving defn}{comsupt}{algebraic disjointness defn}
\begin{theorem}[Disjoint vs algebraically disjoint, \Assumed{88}]\label{\dict{88 ref}}
Suppose that \(X\) is a Hausdorff topological space without isolated points, \(G\) is a group of homeomorphisms of \(X\), and the natural action \(a:X\times G\to X\) defined by \((x, g)a= (x)g\) is locally moving.
If \(f, g\in G\) then
\begin{align*}
    \supt_a(f)\cap \supt_a(g) = \varnothing & \Rightarrow f \text{ is algebraically disjoint from }g\\
    &\Rightarrow \supt_a(f) \cap \supt_a(g^{12}) = \varnothing.
\end{align*}
\end{theorem}
\begin{proof}
 First suppose that \(\supt_a(f)\cap \supt_a(g) = \varnothing\).
 Let \(h\in G\backslash C_G(f)\) and let \(p\in \supt_a(f)\cap \supt_a(h)\).
By Lemma~\ref{comsupt}, let \(U_0\) be a neighbourhood of \(p\) contained in \(\supt_a(f)\cap \supt_a(h)\) such that \((U_0, h)a\cap U_0=\varnothing\).
As \(a\) is locally moving and \(X\) is Hausdorff and has no isolated points, we may choose \(h_{0,g}\in G_{U_0, a}\) such that \((p)h_{0,g}\neq p\).
Similarly, let \(U_1\subseteq U_0\) be a neighbourhood of \(p\) such that \((U_1,h_{0,g})a\cap U_1= \varnothing\), and let \(h_{1,g}\in G_{U_1, a}\) be such that \((p)h_{1,g}\neq p\).

By Lemma~\ref{comsupt}, \(\supt_a(h^{-1}h_{0, g}^{-1} h) \cap U_1 = \varnothing\), so the actions of \([h, h_{0,g}]\) and \(h_{0,g}\) agree on \(U_1\). Similarly the actions of \([h_{0,g}, h_{1, g}]\) and \(h_{1,g}\) agree on \(U_1\), it follows that
\[(p)[[h,h_{0,g}], h_{1, g}]=(p)[h_{0,g}, h_{1, g}]=(p)h_{1,g} \neq p.\]
So \([[h,h_{0,g}], h_{1, g}] \neq \operatorname{id}_X\). Also by Lemma \ref{comsupt} we have
\begin{align*}
    \supt_a([[h,h_{0,g}], h_{1, g}])&\subseteq \supt_a(h_{1, g})\cup (\supt_a(h_{1, g}),[h,h_{0,g}])a\\
&=\supt_a(h_{1, g})\cup (\supt_a(h_{1, g}),h_{0,g})a\subseteq U_0 \cup U_1\subseteq \supt_a(f).
\end{align*}
As each of \(h_{0, g}, h_{1, g}\) and \([[h,h_{0,g}], h_{1, g}]\) have supports disjoint from \(g\), they all commute with \(g\) so the first implication follows.

Now suppose that \(f\) is algebraically disjoint from \(g\). Suppose for a contradiction that \(x\in \supt_a(f)\cap \supt_a(g^{12})\). Note that
\[\supt_a(g^{12})\subseteq \bigcap_{1\leq i\leq 4} \supt_a(g^i).\]
It follows from Lemma~\ref{comsupt}, that we can find a neighbourhood \(U\) of \(x\) such that
\[(U)f\cap U= (U)g\cap U = (U)g^2\cap U=(U)g^3\cap U= (U)g^4\cap U = \varnothing.\]
Note that the sets \(\{(U)g^{i}:0\leq i\leq 4\}\) are pairwise disjoint.
As \(G\) is locally moving and \(X\) is Hausdorff without isolated points, we may choose an element \(h\) of \(G_{U, a}\backslash \{\operatorname{id}_X\}\).
By Lemma~\ref{comsupt}, \(\supt_a(h)\cap \supt_a(f^{-1}hf)= \varnothing\) and in particular \(h\in G\backslash C_{G}(f)\).
So by assumption we may find \(h_{0, g}, h_{1, g}\in C_G(g)\) such that \(h_2:=[[h, h_{0, g}], h_{1, g}]\in C_G(g)\backslash \{\operatorname{id}_X\}\).
Let \(P:= \{\operatorname{id}_X, h_{0,g}, h_{1, g}, h_{0,g}h_{1,g}\}\subseteq C_{G}(g)\).
By Lemma \ref{comsupt} (applied twice) we have
\[\supt(h_2)\subseteq \bigcup_{p\in P} (U)p.\]
As \(h_2\neq \operatorname{id}_X\), there is thus some \(y\in U\) and \(k\in P\) such that \(((y)k)h_2\neq (y)k\).
As \(g^{-1}h_2g=h_2\), we have \(\supt_a(h_2)g^i= \supt_a(h_2)\) for all \(i\in \mathbb{Z}\).
Moreover, as the sets \(\{(U)g^{i}:0\leq i\leq 4\}\) are pairwise disjoint, when \(|i-j|\in \{1, 2, 3, 4\}\) we have \((y)kg^{i}= (y)g^{i}k\neq (y)g^{j}k=(y)kg^{j}\).
So 
\[|\{(y)kg^{-i}:0\leq i\leq 4\}| = 5.\]
As \(|P|=4\) and \(\{y k g^{-i}:0\leq i\leq 4\} \subseteq \supt_a(h_2)\subseteq \bigcup_{p\in P} (U)p\), there is some \(p\in P\) and
\(0\leq i<j\leq 4\) such that \(\{(y)kg^{-i}, (y)kg^{-j}\}\subseteq (U)p\).
Thus \((y)k\in (U)pg^i\cap (U)pg^j=(U)g^i p\cap (U)g^j p=((U)g^i\cap (U)g^j)p=\varnothing\).
This is a contradiction.
\end{proof}

\speeddicttwo{89}{non trivial powers}{locally moving defn}{comsupt}
\begin{lemma}[Non-trivial powers, \Assumed{89}]\label{non trivial powers}
Suppose that \(X\) is a Hausdorff topological space without isolated points, \(G\) is a group of homeomorphisms of \(X\), and the natural action \(a:X\times G\to X\) defined by \((x, g)a = (x)g\) is locally moving. 
If \(U\subseteq X\) is open and non-empty, then for all \(n\in \N\backslash\{0\}\) there is \(g\in G_{U, a}\) such that \(g^n\neq \operatorname{id}_X\).
\end{lemma}
\begin{proof}
 We prove by induction on \(n\), that for all \(n\in \N\) there is a non-empty open set \(U_n\) and \(g_n\in G_{U, a}\) such that \(U_n\subseteq U\) and \(\{(U_{n})g_n^i:0\leq i < n\}\) are pairwise disjoint (note that this implies the lemma).
 
 As \(a\) is locally moving, and \(X\) is Hausdorff with no isolated points, it follows that there is \(g_1\in G_{U, a}\backslash \{1_G\}\). The case for \(n=1\) now follows immediately from Lemma~\ref{comsupt}.
 
  Suppose inductively that the claim is true for all \(n<k\). 
  If \(U_{k-1}\not\subseteq \fix_a(g_{k-1}^k)\) then let \(x\in \supt_a(g_{k-1}^k)\cap U_{k-1}\).
  For each \(i\in \{0, 1, 2, \ldots, k\}\) choose a neighbourhood \(V_i\) of \((x)g_{k-1}^{i}\) such that the sets \(V_i\) are pairwise disjoint (note that the points \((x)g_{k-1}^{i}\) are distinct by assumption).
  We can then define \(U_k:= \intersection{i\in \{0, 1, 2, \ldots, k\}} V_i g_{k-1}^{-i}\) (which is notably a neighbourhood of \(x\)) and \(g_k:= g_{k-1}\).
  
 If \(U_{k-1}\subseteq \fix_a(g_{k-1}^k)\), choose some \(h\in G_{U_{k-1}, a}\backslash \{1_G\}\).
 If \(x\in \supt_a(h)\), then we have \((x)(h g_{k-1})^k= (x)h g_{k-1}^k=(x)h\) (the first equality follows from the fact that \((\supt_a(h))g_{k-1}^i\cap \supt_a(h)= \varnothing\) for \(0<i<k\)).
 Moreover for \(0\leq i < k\), we have \((U_{k-1})(h g)^i = (U_{k-1})g^i\).
 Hence \(h g_{k-1}\) was an appropriate candidate for \(g_{k-1}\) and we can return to the case that \(U_{k-1}\not\subseteq \fix_a(g_{k-1}^k)\).
\end{proof}

\speeddictfour{91}{algebraic subbasis defn}{basis defn}{automorphism groups defn}{centralizer defn}{algebraic disjointness defn}
\begin{defn}[Algebraic basis, \Assumed{91}]\label{algebraic subbasis defn}
If \(G\) is a group and \(f\in G\), then we define
\[G_f:= C_G(\{g^{12}:f\text{ is algebraically disjoint from }g\}).\]
Note that if \(\phi\in \Aut(G)\), then \((G_f)\phi= G_{(f)\phi}\). 
Moreover, if \(F\subseteq G\) is finite, then we define \(G_F:= \cap_{f\in F} G_f\).
\end{defn}

\speeddictthree{92}{canonical Rubin action defn}{filters defn}{conjugation action defn}{algebraic subbasis defn}
\begin{defn}[Canonical Rubin action, \Assumed{92}]\label{canonical Rubin action defn}
If \(G\) is a group, then we define the \textit{canonical Rubin space} \((X_G, \mathcal{T}_G)\) and \textit{canonical Rubin action} \(r_G:X_G\times G\to X_G\) of \(G\) as follows:
\begin{enumerate}
    \item Let \(B_G := \makeset{G_F}{\(F\) is a finite subset of \(G\)}\). The set \(B_G\) is partially ordered by containment, and \(G\) acts on \(B_G\) by conjugation.
    
    \item Let \(\mathcal{U}_G := \{\mathcal{F}\subseteq B_G: \mathcal{F} \text{ is an ultrafilter of }B_G\text{ as a poset}\}\) (recall Definition~\ref{filters defn}).
    
    \item If \(\mathcal{F}\in \mathcal{U}_G\) and \(F\) is a finite subset of \(G\), then we say that \(\mathcal{F} \in_G G_F\) if there is some non-identity element \(g\in G\) with \(G_g\leq G_F\) and such that
    \[(\mathcal{F},G_F)\conj_G = \makeset{(b, h)\conj_G}{\(b\in \mathcal{F}, h\in G_F\)}\supseteq \{b\in B_G:b\leq G_g\}.\]
    Note that if \(g\in G\) and \(\mathcal{F}\in_G G_F\), then \((\mathcal{F},g)\conj_G\in_G G_{(F,g)\conj_G}\).
    
    \item We define an equivalence relation \(\sim_G\) on \(\mathcal{U}_G\) by
    \[\mathcal{F}_1 \sim_G \mathcal{F}_2 \iff \{G_F\in B_G: \mathcal{F}_1\in_G G_F\}= \{G_F\in B_G: \mathcal{F}_2\in_G G_F\}.\]
    We also extend the relation \(\in_G\) to \(\mathcal{U}_G/\sim_G\) by \([\mathcal{F}]_{\sim_G}\in_G G_F \iff \mathcal{F}\in_G G_F\).
    
    \item We define \(X_G := \{[\mathcal{F}]_{\sim_G}: \mathcal{F}\in \mathcal{U}_G\text{ such that there is }f\in G\text{ with }\mathcal{F}\in_G G_{\{f\}}=G_f\}\).
    
    \item For each finite \(F\subseteq G\), we define \(G_F': = \makeset{x\in X_G}{\(x\in_G G_F\)}\).
    
    \item We define \(\mathcal{T}_{G}\), to be the topology on \(X_G\) generated by the sets \(G_F'\) for finite \(F\subseteq G\).
    
    \item We define \(r_G:X_G\times G\to X_G\) by \(([\mathcal{F}]_{\sim_G}, g)r_G = [(\mathcal{F}, g)\conj_G]_{\sim_G}\). It is routine to verify that this is indeed a well-defined action.
\end{enumerate}
\end{defn}

\speeddicttwo{93}{conjugate actions defn}{Products in Categories Defn}{group actions defn}
\begin{defn}[Conjugate actions, \Assumed{93}]\label{conjugate actions defn}
Suppose that \(G\) is a group, \(X_0, X_1\) are topological structures and \(a_0:X_0\times G\to X_0, a_1:X_1\times G \to X_1\) are actions of \(G\) on these structures. We say that \(a_0, a_1\) are \textit{topologically conjugate} if there is a topological isomorphism \(\psi:X_0 \to X_1\) such that \(a_0\circ \psi = \langle \pi_0\psi, \pi_1 \rangle_{X_1\times G} \circ a_1\) (recall Definition~\ref{Products in Categories Defn}).
\end{defn}

\speeddictsix{95}{Rubin's Theorem}{compactness facts}{locally compact defn}{structure hom defn}{Algebraic disjointness}{canonical Rubin action defn}{conjugate actions defn}
\begin{theorem}[Rubin's Theorem, \Assumed{95}]\label{Rubin's Theorem}
If \(a:X\times G\to X\) is a continuous, faithful, locally moving action of a group \(G\) on a locally compact, Hausdorff topological space \(X\) without isolated points, then \(a\) is topologically conjugate to the canonical Rubin action \(r_G\).
\end{theorem}
\begin{proof}
Let \(\mathcal{T}\) denote the topology on \(X\) partially ordered by \(\subseteq\).\\
\underline{Claim 0:}
 If \(U, V\subseteq X\) are arbitrary open sets satisfying \(G_{U^{-\circ}, a} \subseteq G_{V^{-\circ}. a}\), then \(U^{-\circ} \subseteq V^{-\circ}\).\\
 \underline{Proof of Claim:}
Let \(U'\) be an arbitrary non-empty open subset of \(U\). As \(X\) is Hausdorff and has no isolated points, and \(a\) is locally moving, it follows that we can find some non-identity \(h\in G_{U', a}\). As
\[h\in G_{U', a}\subseteq G_{U, a}\subseteq G_{U^{-\circ}, a}\subseteq G_{V^{-\circ}, a}\]
it follows that \(\supt_a(h)\subseteq V^{-\circ}\), so \(U'\cap V^{-\circ} \neq \varnothing\), and thus \(U'\cap V \neq \varnothing\).
As \(U'\) was arbitrary, it follows that \(V\) is dense in \(U\), so \(U^-\subseteq V^-\), and \(U^{-\circ}\subseteq V^{-\circ}\).\(\diamondsuit\)\\
\underline{Claim 1:} For all finite \(F\subseteq G\), we have that \(G_F = G_{\intersection{f\in F}\supt_a(f)^{-\circ}, a}\). 
If we then define \(\gamma:B_G \to \mathcal{T}\) by \((G_F)\gamma= \intersection{f\in F}\supt_a(f)^{-\circ}\), then \(\gamma\) is a well-defined embedding of partially ordered sets. Moreover, for all \(f, g\in G\) we have \(((G_f, g)\conj_G)\gamma= ((G_f)\gamma, g)a.\)\\
\underline{Proof of Claim}:\\
As \(G_{\intersection{f\in F}\supt_a(f)^{-\circ}, a}= \intersection{f\in F}G_{\supt_a(f)^{-\circ}, a}\), we need only show the first part of the claim in the case that \(F\) is a singleton \(\{f\}\). 

We first show that \(G_f\subseteq G_{\supt_a(f)^{-\circ}, a}\). Suppose that \(h \not\in G_{\supt(f)^{-\circ}, a}\) is arbitrary, we will show that \(h\notin G_f\).
As \(\supt_a(h)\not\subseteq \supt_a(f)^{-\circ}\), there is a non-empty open set \(U_0\subseteq \supt_a(h)\backslash \supt_a(f)\).
By Lemma~\ref{comsupt}, let \(U_1\) be a non-empty open subset of \(U_0\) such that \((U_1,h)a\cap U_1 = \varnothing\).
By Lemma~\ref{non trivial powers}, let \(g\in G_{U_1, a}\) be such that \(g^{12}\neq 1_G\). Then \(\supt_a(g)\subseteq U_1 \subseteq \supt_a(h)\backslash \supt_a(f)\).
So by Theorem~\ref{Algebraic disjointness}, \(f\) is algebraically disjoint from \(g\).
But \(g^{12}\) and \(h^{-1}g^{12}h\) have disjoint supports, so \(g^{12}\) does not commute with \(h\).
Thus \(h\not\in G_f\).

We next show that \(G_{\supt_a(f)^{-\circ}, a}\subseteq G_f\).
Let \(h\in G_{\supt_a(f)^{-\circ}}\) be arbitrary, and \(g\in G\) be arbitrary such that \(f\) is algebraically disjoint from \(g\).
We show that \(h\) commutes with \(g^{12}\).
By Theorem~\ref{Algebraic disjointness}, we have that \(\supt_a(f)\cap \supt_a(g^{12})= \varnothing\).
As \(\supt_a(f)\) is dense in \(\supt_a(h)\) and \(\supt_a(g^{12})\) is open, it follows that \(\supt_a(h)\cap \supt_a(g^{12})= \varnothing\). Thus \(h\) commutes with \(g^{12}\) as required.

We now show that \(\gamma\) is well-defined. If \(F\subseteq G\) is finite, then \[(G_F)\gamma\subseteq (G_F)\gamma^{-\circ}= \left(\intersection{f\in F}\supt_a(f)^{-\circ}\right)^{-\circ}\subseteq\left(\intersection{f\in F}\supt_a(f)^{-\circ-}\right)^{\circ}=\left(\intersection{f\in F}\supt_a(f)^{-}\right)^{\circ}\subseteq (G_F)\gamma.\]
Thus \((G_F)\gamma\) is of the form in Claim 0. If \(G_{U^{-\circ}, a} = (G_F)\gamma= G_{V^{-\circ}, a} \), then we have both \(U^{-\circ}\subseteq V^{-\circ}\) and \(V^{-\circ}\subseteq U^{-\circ}\), so \(V^{-\circ}= U^{-\circ}\). 
Thus \(\gamma\) is a well-defined function (and is order preserving).

We next show that \(\gamma\) is an embedding.
By the first part of the claim, the map \(\gamma': \mathcal{T}\to \mathcal{P}(G)\) defined by \((U)\gamma'= G_{U, a}\) satisfies \(\gamma\gamma' = \operatorname{id}_{B_G}\) (where \(\operatorname{id}_{B_G}\) is the identity function on \(B_G\)).
It follows that \(\gamma\) is an injective function. Moreover the map \(\gamma'\) preserves containment and hence so does \(\gamma^{-1}\subseteq \gamma'\).

Finally by Lemma~\ref{comsupt}, for all \(f,g\in G\) we have that
\begin{align*}
    ((G_f, g)\conj_G)\gamma& = (G_{g^{-1}f g})\gamma= (\supt_a(g^{-1}f g))^{-\circ}\\
    &=((\supt_a(f),g)a)^{-\circ} = (\supt(f)^{-\circ},g)a= ((G_f)\gamma, g)a.\diamondsuit
\end{align*}

\underline{Claim 2:} The set \(\left(\makeset{G_f}{\(f\in G\)}\right)\gamma\subseteq \im(\gamma)\) is a basis for the topology \(\mathcal{T}\).\\
\underline{Proof of Claim}:\\
Let \(x\in X\) and \(U\) an arbitrary open neighbourhood of \(x\).
Let \(V\) be a compact neighbourhood of \(x\).
It follows that \(U \cap V\) is an open neighbourhood of \(x\) in the compact space \(V\).
By Remark~\ref{compactness facts}, there is a compact neighbourhood \(W\) of \(x\) contained in \(U\cap V\).
In particular \(W\) is closed in \(X\).

Let \(f\in G_{W^\circ, a}\) be such that \((x)f\neq x\). We have
\[x\in \supt_a(f)\subseteq \supt_a(f)^{-\circ}\subseteq W^{\circ-\circ}\subseteq W\subseteq U.\]
As \(\supt_a(f)^{-\circ}= (G_f)\gamma\), the result follows.\(\diamondsuit\)\\
\underline{Claim 3:} If \(\mathcal{F}\in \mathcal{U}_G\) and \(U\) is an open set such that \(G_{U, a}\in B_G\), then
\[\left(\text{there is }x\in U\text{ such that }\nbhd{X}{x}\subseteq ((\mathcal{F})\gamma)^*\right) \iff \mathcal{F}\in_G G_{U,a}.\]
Where \(((\mathcal{F})\gamma)^*:= \makeset{S\in \mathcal{P}(X)}{\(S\) contains an element of \( (\mathcal{F})\gamma\)}\).\\
\underline{Proof of Claim}:\\
\((\Rightarrow):\) Suppose that \(x\in U\) and \(\nbhd{X}{x}\subseteq ((\mathcal{F})\gamma)^*\).
As \(a\) is locally moving, it follows from Claim 2 that there is a non-identity element \(g\in G\) such that \((x,G_{U,a})a\) is dense in \((G_g)\gamma\).
Let \(b\in B_G\) be arbitrary with \(b\leq G_g\). It suffices to show that \(b\in (\mathcal{F}, G_{U,a})\conj_G\).

As \((x,G_{U,a})a\) is dense in \((G_g)\gamma\), there is some \(f_b\in G_{U, a}\) such that \((x,f_b)a\in (b)\gamma\).
It follows that \(((b)\gamma ,f_b^{-1})a\in\nbhd{X}{x}\subseteq ((\mathcal{F})\gamma)^*\).
As \(\mathcal{F}\) is a filter, we have \((\mathcal{F})\gamma = ((\mathcal{F})\gamma)^*\cap (B_G)\gamma\).
Moreover, by Claim 1 we have \(((b, f_b^{-1})\conj_G)\gamma=((b)\gamma, f_b^{-1})a\), so \(((b)\gamma, f_b^{-1})a\in  ((\mathcal{F})\gamma)^*\cap (B_G)\gamma=(\mathcal{F})\gamma\). Thus we have
\begin{align*}
 (b)\gamma\in  ((\mathcal{F})\gamma, f_b)a \Rightarrow (b)\gamma\in ((\mathcal{F},f_b)\conj_G)\gamma\Rightarrow b\in (\mathcal{F},f_b)\conj_G\Rightarrow b\in (\mathcal{F}, G_{U, a})\conj_G
\end{align*}
 as required.

\((\Leftarrow):\) Suppose that \(\mathcal{F}\in_G G_{U, a}\).
Let \(g\in G\) be such that \(G_g \leq G_{U,a}\), and for all \(h\in G\) with \(G_h\leq G_g\), we have that \(G_h\in (\mathcal{F},G_{U,a})\conj_G\).
As shown in the proof of Claim 2, \((G_{g})\gamma\) contains a compact set with non-empty interior, and by Claim 2 we can always choose this set to be the closure of an element of \((B_G)\gamma\), thus by replacing \(G_g\) with a smaller element of \(B_G\) we can assume without loss of generality that \(((G_g)\gamma)^-\) is a compact subset of \(U\).

Let \(h\in G\) be arbitrary such that \(G_h\leq G_g\), and note that \(((G_h)\gamma)^-\) is compact.
There is some \(f_h\in G_{U, a}\) such that \(G_{f_h^{-1}hf_h}=(G_h,{f_h})\conj_G\in \mathcal{F}\).
Let \(k := f_h^{-1}hf_h\) and note that \(G_k\in \mathcal{F}\) and \[((G_k)\gamma)^-=(((G_h)\gamma)^-, f_h)a\subseteq (((G_g)\gamma)^-, f_h)a\subseteq (U,f_h)a=U.\]

By Remark~\ref{compactness facts}, there is some \(x\in  (G_k)\gamma^-\cap \intersection{b\in \mathcal{F}}(b)\gamma^-\).
Let \(N\in \nbhd{X}{x}\) be arbitrary. We need only show that \(N\in ((\mathcal{F})\gamma)^*\). Let \(N':=(G_k)\gamma\cap N\). 
As \(B_G\) and \(\mathcal{F}\) are closed under finite intersections, the set \(\makeset{A\in B_G}{\(A\supseteq b\cap (N')\gamma^{-1}\) for some \(b\in \mathcal{F}\)}\) is a proper filter on \(B_G\) containing \(\mathcal{F}\).
As \(\mathcal{F}\) is an ultrafilter, it follows that \((N')\gamma^{-1}\in \mathcal{F}\). Hence \(N\in ((\mathcal{F})\gamma)^*\) as required.
\(\diamondsuit\). \\
From Claim 3, for all \(x\in X_G\) there is \((x)\psi\in X\) such that \(\nbhd{X}{(x)\psi}\subseteq ((\mathcal{F})\gamma)^*\) for all \(\mathcal{F}\in x\). 
As \(X\) is Hausdorff it follows from Claim 2 that this point is unique (thus we can define a map \(\psi:X_G\to X\) in this fashion).
It also immediate from Claim 3 and the definition of \(\sim_G\) that the function \(\psi\) is injective.
Furthermore, by Claim 2 and the definition of \(\mathcal{T}_G\), a subset \(U\subseteq X_G\) is open if and only if it is the preimage of an open set under \(\psi\).

If \(x\in X\) is arbitrary, then \(\nbhd{X}{x}\) is a proper filter on \(\mathcal{P}(X)\), thus it can be extended to some ultrafilter \(\mathcal{F}\) on \(\mathcal{P}(X)\), thus \(([(\mathcal{F} \cap (B_G)\gamma)\gamma^{-1}]_{\sim_G})\psi=x\). As \(x\) was arbitrary, it follows that \(\psi\) is surjective. We now have that \(\psi:X_G\to X\) is a homeomorphism.

It now suffices to show that for all \(g\in G\), and \(x\in X_G\), we have \((x,g)r_G\psi= ((x)\psi, g)a\). Let \(g\in G\) and \(x\in X_G\) be arbitrary. By Claim 2 and the assumption that \(X\) is Hausdorff, it suffices to show that
\[\{f\in G: (x,g)r_G\psi\in (G_f)\gamma\} =\{f\in G: ((x)\psi, g)a\in (G_f)\gamma\}.\]
This can be seen as follows:
\begin{align*}
    \{f\in G: (x,g)r_G\psi\in (G_f)\gamma\}
    &=\{f\in G: (x,g)r_G\in_G G_{f}\}\\
    &=\{f\in G: x\in_G (G_{f}, g^{-1})\conj_G\}\\
    &=\{f\in G: x\in_G G_{g f g^{-1}}\}\\
    &=\{f\in G: (x)\psi\in (G_{g f g^{-1}})\gamma\}\\
    &=\{f\in G: (x)\psi\in (( G_{f}, g^{-1})\conj_G)\gamma\}\\
   &=\{f\in G: (x)\psi\in ( (G_{f})\gamma, g^{-1})a)\}\\
   &=\{f\in G: ((x)\psi, g)a\in (G_f)\gamma\}.\qedhere
\end{align*}
\end{proof}

Now that we have proved Rubin's Theorem (Theorem~\ref{Rubin's Theorem}), we know that any sufficiently nice action of a group on a sufficiently nice topological space is in fact an algebraic invariant of the group.
This is particularly useful when taking about isomorphisms between groups.
In the following corollary, we show how Rubin's Theorem can help with understanding the automorphisms of a group (this corollary is crucial to Part~\ref{nv section}).

\speeddicttwo{96}{Rubin Automorphisms cor}{centralizer defn}{Rubin's Theorem}
\begin{corollary}[Rubin automorphisms, \Assumed{96}]\label{Rubin Automorphisms cor}
If \(G\) is a group and \(a:X\times G\to X\) is a continuous, faithful, locally moving action of \(G\) on a locally compact, Hausdorff topological space \(X\) without isolated points, then 
\[\Aut(G)\cong N_{\Aut(X)}((G)\phi_a)\]
(where \(\phi_a\) is as in Definition~\ref{group actions defn})
via an isomorphism which restricts to an isomorphism from \(\operatorname{Inn}(G)\) to \((G)\phi_a\).
\end{corollary}
\begin{proof}
We define a map \(n: G\times N_{\Aut(X)}((G)\phi_a) \to G\) of \(N_{\Aut(X)}((G)\phi_a)\) on \(G\) by:
\[(g, f)n = (f^{-1}(g)\phi_a f)\phi_a^{-1}.\]
Note that this is only well-defined because of the choice of the group \(N_{\Aut(X)}((G)\phi_a)\).

It suffices to show that \(n\) is a faithful action of \(N_{\Aut(X)}((G)\phi_a)\) on \(G\) and that the corresponding embedding \(\phi_n: N_{\Aut(X)}((G)\phi_a) \to \Aut(G)\) is surjective.

We first show that \(n\) is an action. Let \(g\in G\) and \(f, h\in N_{\Aut(X)}((G)\phi_a)\) be arbitrary, then
\begin{align*}
    (((g, f)n, h)n &=  (h^{-1}((f^{-1}(g)\phi_a f)\phi_a^{-1})\phi_a h)\phi_a^{-1}\\
     &=  (h^{-1}f^{-1}(g)\phi_a f h)\phi_a^{-1}\\
    &=  ((f h)^{-1}(g)\phi_a f h)\phi_a^{-1}\\
    &=  (g, f h)n.
\end{align*}

We next show that \(\phi_n\) is injective.
Suppose that \(f, h\in N_{\Aut(X)}((G)\phi_a)\) are such that \((f)\phi_n=(h)\phi_n\), we show that \(f = h\).
It suffices to show that \(f h^{-1}\) is the identity function on \(X\).
As \(\phi_n\) is a homomorphism, we know that \((f h^{-1})\phi_n\) is the identity function on \(G\).
Thus by the definition of \(n\), it follows that \(f h^{-1}\) commutes with all elements of \((G)\phi_a\).
Suppose for a contradiction that \(f h^{-1}\) is non-trivial. It follows from Lemma~\ref{comsupt}, that there is a non-empty open subset \(U\) of \(X\) such that \((U)f h^{-1} \cap U = \varnothing\).
As \(X\) is Hausdorff and has no isolated points, and \(a\) is locally moving, it follows that there is some non-identity element \(g\in G_{U, a}\).
By Lemma~\ref{comsupt} it follows that no elements of \(\supt_a(g)\) are fixed by \(((g)\phi_a, f h^{-1})\conj_{\Aut(X)}\).
In particular \(f h^{-1}\) does not commute with \((g)\phi_a\), this is a contradiction.

It remains to show that \(\phi_n\) is surjective.
By Rubin's Theorem (\ref{Rubin's Theorem}) let \(\psi: X_G\to X\) be a homeomorphism such that for all \(x\in X_G\) and \(g\in G\) we have \(((x)\psi, g)a= (x, g)r_G\psi\).
Recall that \(X_G \subseteq \mathcal{P}(\mathcal{P}(\mathcal{P}(G)))\), and as per Definition~\ref{binary relations defns}, any automorphism \(\varphi\) of \(G\) can act on \(X_G\). Thus for each \(\varphi\in \Aut(G)\), we define a map \(\varphi':X_G\to X_G\) by \((x)\varphi' = (x)\varphi\).

For all \(f\in G\), we have \((G_{f}')\varphi' = G_{(f)\varphi}'\), so in particular \(\varphi'\) is a homeomorphism of \(X_G\).
We now choose \(\varphi\in \Aut(G)\) to be arbitrary, it suffices to show that \(\varphi\in \im(\phi_n)\).
For all \(g\in G\), we define an automorphism \(c_g\) of \(G\) by \((f)c_g= (f, g)\conj_G\).
Note that for all \(g\in G\) we have \(\varphi^{-1}c_g\varphi = c_{(g)\varphi}\), and the condition on \(\psi\) together with the definition of \(r_G\) gives that \(\psi ((g)\phi_a) \psi^{-1} = c_g'\) for all \(g\in G\).

Let \(g\in G\) be arbitrary, it suffices to show that \((g)\varphi= (g)((\psi^{-1}\varphi'\psi)\phi_n)\). This can be seen as follows:
\begin{align*}
    (g)((\psi^{-1}\varphi'\psi)\phi_n) &= (g, \psi^{-1}\varphi'\psi)n \\
     &=  ((\psi^{-1}\varphi'\psi)^{-1}((g)\phi_a)  (\psi^{-1}\varphi'\psi))\phi_a^{-1}\\
     &=  (\psi^{-1}(\varphi')^{-1}(\psi((g)\phi_a)  \psi^{-1})\varphi'\psi)\phi_a^{-1}\\
     &=  (\psi^{-1}(\varphi')^{-1}(c_g')\varphi'\psi)\phi_a^{-1}\\
     &=  (\psi^{-1}(\varphi^{-1}c_g\varphi)'\psi)\phi_a^{-1}\\
     &=  (\psi^{-1}c_{(g)\varphi}'\psi)\phi_a^{-1}\\
     &=(g)\varphi.
\end{align*}
\end{proof}

\part{Semigroups And Clones}\label{semigroups section}
The results of this part are (for the most part) the results from \cite{AcPaper1} by J. Jonušas, Z. Mesyan, J. D. Mitchell, M. Morayne, Y. Péresse and the author of this document.
The other authors have kindly given permission for the inclusion of these results here.

In this part we discuss various natural ways of defining topologies on arbitrary semigroups, and investigate how these topologies interact (see Proposition~\ref{topology comparison proposition}).
Many of the definitions are minimal/maximal up to some topological condition so by understanding them we gain a better understanding of all the topologies compatible with the semigroups satisfying the required topological conditions.

The topological conditions we are particularly interested are those of being Hausdorff, second countable or Polish.
Hausdorffness (Definition~\ref{Hausdorff defn}) is a useful and popular separation condition which forces the topologies in question to be ``reasonably big" (in particular it allows us to exclude indiscrete/trivial topologies).
Given that we will be primarily concerned with topological spaces of continuum cardinality, second countableness (Definition~\ref{second countable defn}) similarly forces the topologies in question to be ``reasonably small"  (in particular it allows us to exclude discrete topologies).
Polishness (Definition~\ref{Polish space defn}) is a much stronger restriction, but it turns out to be somewhat common for nice semigroups to admit Polish topologies.
Polishness is also very useful as it implies the previous two conditions and gives us access to Theorem~\ref{comparable Polish group topologies theorem}.

In this part we also consider the notion of a (topological) clone (for examples of work on topological clones in the literature see \cite{Behrisch2017aa,Bodirsky2017aa, Pech:2018aa}).
In the same sense that groups are a generalisation of permutation groups and semigroups are a generalisation of transformation semigroups, clones are a generalisation of operation clones (where here an ``operation" is a finitary operation on a set in the sense of Definition~\ref{structures defn}).

For example the set of real polynomials in finitely many variables forms a clone, or more generally the set of finite arity term operations for any algebraic structure (Definition~\ref{term operations defn}).
Also if \(X\) is any object in a category with finite products (for example a graph, a semigroup, or a topological space) then the set of polymorphisms of \(X\) (morphisms between finite powers of \(X\)) form a clone.

The elements of a clone are ``combined" in the same manner as term operations (see Definition~\ref{term operations defn}).
Clones give us another reasonable way of comparing algebraic structures, as if we have two algebraic structures such that every operation of one is a term operation of the other, it is reasonable to think of them as equivalent although this does not imply that they are isomorphic.
Although term operations are defined in such a way that an operation can only output a single element, it is sometimes useful to think of an \(n\)-tuple of \(m\)-ary operations \((t_0, t_1, \ldots, t_{n-1})\) on a set \(X\) as the single object \(\langle t_0, t_1, \ldots, t_{n-1} \rangle_{X^n}\) (recall Definition~\ref{Products in Categories Defn}).
In particular, this reduces the way operations are combined to the usual composition of maps.

We describe the results of the topological constructions when applied to specific examples of semigroups and use this to show various uniqueness/non-uniqueness results about the possible topologies on the semigroups and clones.
Specifically, we consider the following semigroups and clones:
\begin{enumerate}
    \item The binary relation monoid \(\mathcal{B}_\N\) (see Example~\ref{topological binary relations} and Theorem~\ref{main binary relation theorem}).
    
    \item The full transformation monoid \(\N^\N\) (see Example~\ref{full trans def} and Theorem~\ref{main N^N theorem}).
    
    \item The partial function monoid \(\mathcal{P}_\N\) (see Example~\ref{partial function monoid defn} and Theorem~\ref{main N^N theorem}).
    
    \item The symmetric inverse monoid \(\mathcal{I}_\N\) (see Example~\ref{symmetric inverse monoids defn} and Theorem~\ref{main inverse monoid theorem}).
    
    \item The injective function monoid \(\inj(\N)\) (see Example~\ref{topological injective function monoids defn}, Theorem~\ref{much Polish Theorem} and Theorem~\ref{main inj theorem}).
    
    \item The continuous function monoid of the Hilbert cube \(C([0, 1]^\N)\) (See Proposition~\ref{compact-open is Polish thm} and Theorem~\ref{hilber cube main theorem}).
    
    \item The continuous function monoid of the Cantor space \(C(2^\N)\) (See Proposition~\ref{compact-open is Polish thm} and Theorem~\ref{Cantor space main theorem}).
    
    \item The endomorphism monoid of the countably infinite atomless boolean algebra \(B_\infty\) (see Theorem~\ref{main boolean algebra theorem}).
    
    \item The clone of all maps between finite powers of the set \(\N\) (see Corollary~\ref{full func clone cor}).
    
    \item The polymorphism clone of the Hilbert cube \([0, 1]^\N\) (see Corollary~\ref{Hilbert clone cor}).
    
    \item The polymorphism clone of the Cantor set \(2^\N\) (see Corollary~\ref{Cantor clone cor}).
    
    \item The polymorphism clone of the countably infinite atomless boolean algebra \(B_\infty\) (see Corollary~\ref{boolean clone cor}).
\end{enumerate}
Another related topic of interest which has received much attention in the literature (particularly for groups) is that of ``automatic continuity" (see \cite{Mann2016aa,Paolini2020,Tsankov2013aa,Yaacov2010aa} for example). 

As seen earlier in Theorem~\ref{automatic continuity examples thm}, it is sometimes possible to conclude that a homomorphism between semigroups is continuous with little information about the spaces in question.
In particular when we know that a large class of maps from a topological space are continuous, this allows us to use the topology on the spaces as an ``upper bound" of sorts.
In the case of Polish groups, such a bound can often allow us to immediately identity a unique Polish topology compatible with the group by using Theorem~\ref{comparable Polish group topologies theorem}.

As we see throughout this part, while the same approach (see Theorem~\ref{main inverse monoid theorem} for example) does not apply to arbitrary semigroups, such continuity results still give us a great deal of information about the potential topologies compatible with the semigroups.

\section{Natural topologies for arbitrary semigroups}
In this section we introduce various natural ways of defining a topology on an arbitrary semigroup.

Every semigroup has a corresponding ``dual semigroup" (Definition~\ref{dual semigroup defn}).
This semigroup may not be isomorphic to the original semigroup (notably it is in the case of an inverse semigroup) although for our purposes, these objects are essentially the same.

\speeddicttwo{174}{anti-homomorphisms defn}{Products in Categories Defn}{semigroup defn}
\begin{defn}[Anti-homomorphisms, \Assumed{174}]\label{anti-homomorphisms defn}
Suppose that \(S, T\) are semigroups and \(\phi:S\to T\) is a function, then we say that \(\phi\) is an \textit{anti-homomorphism} if \(*^S \circ \phi = \langle \pi_1\phi, \pi_0\phi\rangle_{S^2}\circ *^T\) (recall Definition~\ref{semigroup defn}).
That is for all \(a, b\in S\) we have \((ab)\phi = (b)\phi(a)\phi\).
If \(\phi\) is also a bijection, then we say that \(\phi\) is a \textit{anti-isomorphism} and that \(S\) and \(T\) are \textit{anti-isomorphic}.
\end{defn}

\speeddicttwo{175}{dual semigroup defn}{semigroup defn}{anti-homomorphisms defn}
\begin{defn}[Dual semigroups, \Assumed{175}]\label{dual semigroup defn}
If \(S\) is a semigroup then we define the \textit{dual semigroup} \(S^\dagger\) of \(S\) to be the semigroup with the same universe and such that for all \(s, t\in S\) we have \((s, t)*^{S^\dagger}  = (t, s)*^{S}\). By definition \(S^{\dagger}\) is anti-isomorphic to \(S\) via the identity map.
\end{defn}

It will be useful going forward to be able to make semigroup topologies from other semigroup topologies without needing to reprove continuity, we give a few such ways here.

\speeddictsix{173}{pull back topologies prop}{binary relations defns}{Products in Categories Defn}{topological structures defn}{semitopological defn}{anti-homomorphisms defn}{dual semigroup defn}
\begin{proposition}[Pulling back topologies, \Assumed{173}]\label{pull back topologies prop}
Suppose that \(S,T\) are algebraic structures and \(\phi: S\to T\) is a homomorphism or anti-homomorphism and \(\mathcal{T}\) is a topology on the set \(T\).
\begin{enumerate}
    \item If \((T, \mathcal{T})\) is a semitopological semigroup, then \((S, (\mathcal{T})\phi^{-1})\) is a semitopological semigroup.
    \item If \((T, \mathcal{T})\) is a topological algebraic structure, then \((S, (\mathcal{T})\phi^{-1})\) is a topological structure.
\end{enumerate}
\end{proposition}
\begin{proof}
We first consider the case that \(\phi\) is an anti-homomorphism of semigroups. Note that the map \(\phi:(S, (\mathcal{T})\phi^{-1}) \to (T, \mathcal{T})\) is continuous by definition. If \(s\in S\) is arbitrary, then we have 
\[\phi^{-1}\lambda_s^{-1} = (\lambda_s\phi)^{-1} = (\phi\rho_{(s)\phi})^{-1}\quad \text{and }\quad \phi^{-1}\lambda_s^{-1} = (\rho_s\phi)^{-1} = (\phi\lambda_{(s)\phi})^{-1}\]
(here we are using composition of binary relations).
In particular if \(T\) is a semitopological semigroup and \(U\) is an open subset of \(T\), then
\[((U)\phi^{-1})\lambda_s^{-1} = (U)(\phi\rho_{(s)\phi})^{-1} \quad \text{and} \quad ((U)\phi^{-1})\rho_s^{-1} = (U)(\phi\lambda_{(s)\phi})^{-1}\]
are open sets, so \(S\) is also a semitopological semigroup as well. Similarly if \(T\) is a topological semigroup then
\[((U)\phi^{-1})(*^{S})^{-1}= (U) (\langle \pi_1\phi, \pi_0\phi\rangle_{S^2}\circ *^T)^{-1}\]
is open and thus \(S\) is also a topological semigroup.

We now consider the case that \(\phi\) is a homomorphism of semigroups. As the identity function is an anti-isomorphism from \(T\) to \(T^\dagger\), it follows that \((T^\dagger, \mathcal{T})\) is a (semi)topological semigroup if and only if \((T, \mathcal{T})\) is as well.
If \(\phi:S\to T\) is a homomorphism, then \(\phi\) is an anti-homomorphism from \(S\) to \((T^\dagger, \mathcal{T})\) and so the result follows from the first case.

If \(\phi\) is a homomorphism of arbitrary algebraic structures, then for any open subset \(U\) of \(T\) and any \(n\)-ary function symbol \(F\) from the signature of \(S\) we have
\[((U)\phi^{-1})(F^{S})^{-1}= (U) (\langle \pi_0\phi, \pi_1\phi, \ldots, \pi_{n-1}\phi\rangle_{S^m}\circ F^T)^{-1}.\]
As before, it follows that \(F^S:S^m\to S\) is continuous as required.
\end{proof}

\speeddictone{178}{combining topologies lemma}{topological structures defn}
\begin{lemma}[Combining topologies, \Assumed{178}]\label{combining topologies lemma}
If \(\mathbb{S}\) is an algebraic structure, \(I\) is a set, and \(\makeset{\mathcal{T}_i}{\(i\in I\)}\) are topologies compatible with \(\mathbb{S}\), then the topology generated by \(\union{i\in I}\mathcal{T}_i\) is compatible with \(\mathbb{S}\).
\end{lemma}
\begin{proof}
Let \(\mathcal{T}\) be the topology generated by \(\union{i\in I}\mathcal{T}_i\). Let \(U\in \union{i\in I}\mathcal{T}_i\) be arbitrary and let \(f:\mathbb{S}^n \to \mathbb{S}\) be an arbitrary operation of \(\mathbb{S}\). It suffices to show that \((U)f^{-1}\) is open with respect to \(\mathcal{T}\).

Let \(i\in I\) be such that \(U\in \mathcal{T}_i\). As \(\mathcal{T}_i\) is compatible with \(\mathbb{S}\), it follows that \((U)f^{-1}\) is an open subset of \(\mathbb{S}^n\) using the product topology obtained from using the topology \(\mathcal{T}_i\) on \(\mathbb{S}\).
As this topology on \(\mathbb{S}^n\) is contained in the topology that we are using, the result follows.
\end{proof}

\speeddicttwo{177}{somewhat nice topologies lemma}{frechet defn}{semitopological defn}
\begin{lemma}[Intersecting somewhat nice topologies, \Assumed{177}]\label{somewhat nice topologies lemma}
Suppose that \(S\) is a semigroup, \(I\) is a set and \(\makeset{\mathcal{T}_i}{\(i\in I\)}\) is a collection of Fréchet topologies semicompatible with \(S\). In this case \(\mathcal{T}:= \intersection{i\in I} \mathcal{T}_i\) is also Fréchet and semicompatible with \(S\).
\end{lemma}
\begin{proof}
Let \(s\in S\) and \(U\in \mathcal{T}\) be arbitrary. To show that \((S, \mathcal{T})\) is a semitopological semigroup, it suffices to show that \((U)\lambda_s^{-1}\) and \((U)\rho_s^{-1}\) are open with respect to \( \mathcal{T}\).
As \(U\) is open in all of the topologies \(\mathcal{T}_i\) and these topologies are semicompatible with \(S\), it follows these two sets are also open in all of the topologies \(\mathcal{T}_i\). Thus they are open in \(\mathcal{T}\) by definition.

If \(s\in S\) is arbitrary, then \(S\backslash\{s\}\) is an element of all Fréchet topologies on \(S\).
As \(\mathcal{T}\) is an intersection of Fréchet topologies, it follows that \(S\backslash\{s\}\in \mathcal{T}\) as well and thus \(\mathcal{T}\) is Fréchet.
\end{proof}

The first of the topologies we introduce (the minimum Fr\'echet topology) is very useful due to its simple definition, however when considered on a group it is less interesting as it always results in the cofinite topology.
\speeddictthree{172}{minimum topology defn}{types of binary relation defns}{frechet defn}{semitopological defn}
\begin{defn}[Minimum topologies, \Assumed{172}]\label{minimum topology defn}
If \(S\) is an (inverse) semigroup, then we define the \textit{minimum Fr\'echet topology} of \(S\) by
\[\mathcal{MF}(S) := \intersection{} \makeset{\mathcal{T}}{\((S, \mathcal{T})\) is a Fr\'echet semitopological (inverse) semigroup}\]
(by Lemma~\ref{somewhat nice topologies lemma} this is in the minimum, with respect to containment of topologies, Fr\'echet topology semicompatible with \(S\)).
\end{defn}

\speeddicttwo{171}{Markov topology defn}{frechet defn}{topological structures defn}
\begin{defn}[Markov topologies, \Assumed{171}]\label{Markov topology defn}
If \(\mathbb{S}\) is an algebraic structure, then we define the \textit{Fr\'echet-Markov topology} on \(\mathbb{S}\) by
\[\mathcal{FM}(\mathbb{S}):= \intersection{} \makeset{\mathcal{T}}{\((\mathbb{S}, \mathcal{T})\) is a Fr\'echet topological structure},\]
and the \textit{Hausdorff-Markov topology} on \(\mathbb{S}\) by
\[\mathcal{HM}(\mathbb{S}):= \intersection{} \makeset{\mathcal{T}}{\((\mathbb{S}, \mathcal{T})\) is a Hausdorff topological structure}.\]
\end{defn}

In contrast to the minimum Fr\'echet topology being semi-compatible, the Markov topologies on a structure need not be compatible with the structure.

\speeddictfour{176}{nice minimum topologies prop}{pull back topologies prop}{somewhat nice topologies lemma}{minimum topology defn}{Markov topology defn}
\begin{proposition}[Small topologies are nice, \Assumed{176}]\label{nice minimum topologies prop}
If \(S\) is a semigroup, then \(\mathcal{MF}(S)\), \(\mathcal{FM}(S)\) and \(\mathcal{HM}(S)\) are Fréchet and semicompatible with \(S\). Moreover if \(\phi\) is an automorphism or anti-automorphism of \(S\) as a discrete semigroup, then \(\phi\) is continuous with respect to each of \(\mathcal{MF}(S)\), \(\mathcal{FM}(S)\) and \(\mathcal{HM}(S)\).
\end{proposition}
\begin{proof}
The first statement follows from Lemma~\ref{somewhat nice topologies lemma}.
Let \(\phi\) be an arbitrary automorphism or anti-automorphism of \(S\).
Let \(\mathcal{T}_0, \mathcal{T}_1\) and \(\mathcal{T}_2\) be \(\mathcal{MF}(S)\), \(\mathcal{FM}(S)\) and \(\mathcal{HM}(S)\) respectively. Similarly let \(P_0, P_1\) and \(P_2\) be the properties of being a Fr\'echet semitopological semigroup, Fr\'echet topological semigroup and Hausdorff topological semigroup respectively.

As \(\phi\) is a bijection it follows that if \(\mathcal{T}\) is a Fréchet/Hausdorff topology on the set \(S\), then so is \((\mathcal{T})\phi^{-1}\). So for each \(i\in \{0, 1, 2\}\), it follows from Proposition~\ref{pull back topologies prop} that
\begin{align*}
    (\mathcal{T}_i)\phi &= \left(\intersection{} \makeset{\mathcal{T}}{\((S, \mathcal{T})\) satisfies \(P_i\)}\right)\phi=\intersection{} \makeset{(\mathcal{T})\phi}{\((S, \mathcal{T})\) satisfies \(P_i\)}\\
    &=\intersection{} \makeset{\mathcal{T}}{\((S, (\mathcal{T})\phi^{-1})\) satisfies \(P_i\)}=\intersection{} \makeset{\mathcal{T}}{\((S, \mathcal{T})\) satisfies \(P_i\)}=\mathcal{T}_i.
\end{align*}
So \(\phi\) is continuous as required.
\end{proof}

\speeddictthree{170}{Zariski topology defn}{subbasis}{topological structures defn}{term operations defn}
\begin{defn}[Zariski topology, \Assumed{170}]\label{Zariski topology defn}
Let \(\mathbb{S}\) be an algebraic \(\sigma\)-structure. We say that a subset \(E\subseteq \mathbb{S}\) is \textit{elementary algebraic} if there are \(n, m\in \N\backslash\{0\}\), \(s_1, s_2, \ldots, s_{n-1}, t_1, t_2, \ldots, t_{m-1} \in \mathbb{S}\), terms \(a, b\) of \(\sigma\) over \(\{0, 1, \ldots, n-1\}\) and \(\{0, 1, \ldots, m-1\}\) respectively such that
\[E = \makeset{s\in \mathbb{S}}{\((s, s_1, \ldots, s_{n-1})a_{\{0, 1, \ldots, n-1\}}^{\mathbb{S}} = (s, t_1,\ldots, t_{m-1})b_{\{0, 1, \ldots, m-1\}}^{\mathbb{S}}\)}\]
(recall Definition~\ref{term operations defn}).
We then define the \textit{Zariski topology} \(\mathcal{Z}(\mathbb{S})\) of \(\mathbb{S}\) to be the topology on \(\mathbb{S}\) generated by the sets of the form \(\mathbb{S}\backslash E\) where \(E\) is elementary algebraic.
\end{defn}

For semigroups in particular, there is a rather limited scope for term operations, in fact every term operation is of the form
\[(s_0, s_1, s_2, \ldots, s_{n-1}) \to s_{i_0}s_{i_1}\ldots s_{i_{k-1}}\]
for some \(n, k\in \N\) and \(i_0, i_1,\ldots i_{k-1}< n\).
Thus the elementary algebraic sets for the Zariski topology can be thought of as the ``solution sets" to pairs of ``polynomials".

\speeddictfour{179}{Zariski topologies are nice}{semigroup signature defn}{anti-homomorphisms defn}{dual semigroup defn}{Zariski topology defn}
\begin{proposition}[Zariski topologies are nice, \Assumed{179}]\label{Zariski topologies are nice}
If \(S\) is a semigroup, then \(\mathcal{Z}(S)\) is a Fréchet topology semicompatible with \(S\). Moreover if \(\phi:S\to T\) is an isomorphism or anti-isomorphism between \(S\) and \(T\) as discrete semigroups, then \(\phi\) is continuous with respect to \(\mathcal{Z}(S)\) and \(\mathcal{Z}(T)\). 
\end{proposition}
\begin{proof}
Let \(a\), \(b\) be arbitrary terms of \(\sigma_S\) over the variable sets \(\{0,1 , \ldots, n-1\}\) and \(\{0, 1, \ldots, m-1\}\) respectively (recall Definition~\ref{semigroup signature defn}). Moreover let \(u, s_1, s_2, \ldots, s_{n-1}\), \(t_1, t_2, \ldots, t_{m-1}\in S\) be arbitrary.
To conclude that \(\mathcal{Z}(S)\) is semicompatible with \(S\) it suffices to show that the sets
\[\left(\makeset{s\in S}{\((s, s_1, \ldots, s_{n-1})a_{\{0, 1,\ldots, n-1\}}^S = (s, t_1,\ldots, t_{m-1})b_{\{0, 1, \ldots, m-1\}}^S\)}\right)\lambda_u^{-1},\]
\[\left(\makeset{s\in S}{\((s, s_1, \ldots, s_{n-1})a_{\{0, 1, \ldots, n-1\}}^S = (s, t_1,\ldots, t_{m-1})b_{\{0, 1, \ldots, m-1\}}^S\)}\right)\rho_u^{-1}\]
are closed in \(\mathcal{Z}(S)\). 
We show that the first of these sets is closed, the second follows by a symmetric argument.
We define term operations \(a':S^{n+1}\to S\) and \(b':S^{m+1}\to S\) by
\begin{align*}
    (x_0, x_1, \ldots, x_n)a'&= ((x_0, x_n)*^S, x_1, \ldots, x_{n-1})a_{\{0, 1, \ldots, n-1\}}^S,\\
    (x_0, x_1, \ldots, x_m)b'&= ((x_0, x_m)*^S, x_1, \ldots, x_{m-1})b_{\{0, 1, \ldots, m-1\}}^S.
\end{align*}
We now have
\begin{align*}
    &\left(\makeset{s\in S}{\((s, s_1, \ldots, s_{n-1})a_{\{0, 1, \ldots, n-1\}}^S = (s, t_1,\ldots, t_{m-1})b_{\{0, 1, \ldots, m-1\}}^S\)}\right)\lambda_u^{-1}\\
    =&\makeset{s\in S}{\(((s)\lambda_u, s_1, \ldots, s_{n-1})a_{\{0, 1, \ldots, n-1\}}^S = ((s)\lambda_u, t_1,\ldots, t_{m-1})b_{\{0, 1, \ldots, m-1\}}^S\)}\\
    =&\makeset{s\in S}{\(((u, s)*^S, s_1, \ldots, s_{n-1})a_{\{0, 1, \ldots, n-1\}}^S = ((u, s)*^S, t_1,\ldots, t_{m-1})b_{\{0, 1, \ldots, m-1\}}^S\)}\\
    =&\makeset{s\in S}{\((s, s_1, \ldots, s_{n-1}, u)a' = (s, t_1,\ldots, t_{m-1}, u)b'\)}.
\end{align*}
So the required set is elementary algebraic and hence closed.
To see that \(\mathcal{Z}\) is Fréchet, note that the projection maps \(\pi_0, \pi_1:S^2\to S\) are term operations, thus for all \(x\in S\) the set \(\{x\}= \makeset{s\in S}{\((s, x)\pi_0 = (s, x)\pi_1\)}\) is elementary algebraic and hence closed in \(\mathcal{Z}(S)\).

If \(\phi:S\to T\) is an isomorphism, then 
\[\left(\makeset{s\in S}{\((s, s_1, \ldots, s_{n-1})a_{\{0, 1, \ldots, n-1\}}^S = (s, t_1,\ldots, t_{m-1})b_{\{0, 1, \ldots, m-1\}}^S\)}\right)\phi\]
\[=\makeset{t\in T}{\((t, (s_1)\phi, \ldots, (s_{n-1})\phi)a_{\{0, 1, \ldots, n-1\}}^T = (t, (t_1)\phi,\ldots, (t_{m-1})\phi)b_{\{0, 1, \ldots, m-1\}}^T\)}.\]
In particular \(\phi\) and \(\phi^{-1}\) map elementary algebraic sets to elementary algebraic sets, so \(\phi\) is a homeomorphism of the Zariski topologies.

If \(\phi:S\to T\) is an anti-isomorphism, then \(\phi\) is an isomorphism from \(S^\dagger\) to \(T\).
So \(\phi:(S, \mathcal{Z}(S^\dagger)) \to (T, \mathcal{Z}(T))\) is a homeomorphism.
As the operations of \(S\) and \(S^\dagger\) are term operations of each other, it follows that they have the same term operations and thus \(\mathcal{Z}(S)= \mathcal{Z}(S^\dagger)\).
Thus \(\phi\) is continuous with respect to \(\mathcal{Z}(S)\) and \(\mathcal{Z}(T)\) as required.
\end{proof}

The ``second countable continuity topology'' defined in the following definition is primarily of interest due to Proposition~\ref{ac topology is nice}.
In this part we will see many examples of semigroups for which this topology is second countable and is thus maximum among the second countable topologies compatible with the semigroup.

\speeddictthree{180}{ac topology defn}{second countable defn}{topological structures defn}{combining topologies lemma}
\begin{defn}[Second countable continuity topology, \Assumed{180}]\label{ac topology defn}
If \(\mathbb{S}\) is an algebraic structure, then we define the \textit{second countable continuity topology} \(\mathcal{SCC}(\mathbb{S})\) of \(\mathbb{S}\) to be the topology generated by the set
\[\union{}\makeset{\mathcal{T}}{\((\mathbb{S}, \mathcal{T})\) is a second countable topological structure}\]
(by Lemma~\ref{combining topologies lemma}, this topology is compatible with \(\mathbb{S}\)).
\end{defn}

\speeddictthree{181}{ac topology is nice}{group defn}{dual semigroup defn}{ac topology defn}
\begin{proposition}[Second countable continuity topologies are nice, \Assumed{181}]\label{ac topology is nice}
Suppose that \((S, \mathcal{T})\) is a topological semigroup. The following are equivalent:
\begin{enumerate}
    \item \(\mathcal{SCC}(S)\subseteq \mathcal{T}\).
    \item If \(\phi:(S, \mathcal{T}) \to T\) is a homomorphism, where \(T\) is a second countable topological semigroup, then \(\phi\) is continuous.
    \item \(\mathcal{SCC}(S^\dagger)\subseteq \mathcal{T}\).
    \item If \(\phi:(S, \mathcal{T}) \to T\) is an anti-homomorphism, where \(T\) is a second countable topological semigroup, then \(\phi\) is continuous.
\end{enumerate}
Suppose further that there is an inverse semigroup \(I\) such that \(S\) is the semigroup obtained by removing the unary operation of \(I\), and that \((I, \mathcal{T})\) is a topological inverse semigroup.
In this case \(\mathcal{SCC}(S) = \mathcal{SCC}(I)\) and the following condition is equivalent to the above conditions:
\begin{enumerate}
\setcounter{enumi}{4}
    \item If \(\phi: (I, \mathcal{T})\to T\) is a homomorphism, where \(T\) is a second countable topological inverse semigroup, then \(\phi\) is continuous.
\end{enumerate}
\end{proposition}
\begin{proof}
\((1\Rightarrow 2):\) Let \(\phi:(S, \mathcal{T}) \to T\) be a homomorphism, where \(T\) is a second countable topological semigroup.
Let \(\mathcal{U}\) be the topology on \(T\). 
By Proposition~\ref{pull back topologies prop}, the topology \((\mathcal{U})\phi^{-1}\) is compatible with \(S\).
If \(B\) is a countable basis for \(\mathcal{U}\), then \((B)\phi^{-1}\) is a countable basis for \((\mathcal{U})\phi^{-1}\), so \((\mathcal{U})\phi^{-1}\) is second countable. By the definition of \(\mathcal{SCC}(S)\), it follows that \((\mathcal{U})\phi^{-1}\subseteq \mathcal{SCC}(S)\), and hence \(\phi\) is continuous.

\((2\Rightarrow 1):\) Let \(\mathcal{U}\) be an arbitrary second countable topology compatible with \(S\).
The identity map \(\text{id}_S:(S, \mathcal{T}) \to (S, \mathcal{U})\) is a homomorphism, so by assumption it is continuous.
Thus \(\mathcal{U}\subseteq \mathcal{T}\). As \(\mathcal{U}\) was arbitrary and such topologies generate \(\mathcal{SCC}(S)\), it follows that \(\mathcal{SCC}(S)\subseteq \mathcal{T}\).

\((1\iff 3):\) The topology \(\mathcal{SCC}(S)\) is generated by the set
\[\union{}\makeset{\mathcal{U}}{\((S, \mathcal{U})\) is a second countable topological semigroup}\]
and the topology \(\mathcal{SCC}(S^\dagger)\) is generated by the set
\[\union{}\makeset{\mathcal{U}}{\((S^\dagger, \mathcal{U})\) is a second countable topological semigroup}.\]
By Proposition~\ref{pull back topologies prop} (using the anti-isomorphism \(\text{id}_S\)) it follows that these sets are the same and hence \(\mathcal{SCC}(S^\dagger) = \mathcal{SCC}(S)\).

\((3\iff 4):\) Note that \(\phi:S\to T\) is an anti-homomorphism if and only if \(\phi:S^\dagger\to T\) is a homomorphism.
Thus by replacing \(S\) with \(S^\dagger\), it follows that  \((1\iff 2) \iff (3\iff 4)\).
As we have already shown \((1\iff 2)\), the result follows.

\((\mathcal{SCC}(S)= \mathcal{SCC}(I)):\) Note that the map \(x\to x^{-1}\) is an anti-isomorphism from \(I\) to itself.
In particular if \(\mathcal{T}\) is a topology compatible with \(S\), then \(\mathcal{T}^{-1}\) is also a topology compatible with \(S\) (by Proposition~\ref{pull back topologies prop}).
Thus by Lemma~\ref{combining topologies lemma}, the topology generated by \(\mathcal{T}\cup \mathcal{T}^{-1}\) is also compatible with \(S\).
As the set \(\mathcal{T}\cup \mathcal{T}^{-1}\) is inverse closed, it follows that the topology generated by this set is compatible with \(I\) as well.

We have shown that if \(\mathcal{T}\) is a topology compatible with \(S\), then the topology generated by \(\mathcal{T} \cup \mathcal{T}^{-1}\) is compatible with \(I\).
If \(B\) is a countable basis for \(\mathcal{B}\), then \(\makeset{\cap F}{finite \(F\subseteq B\cup B^{-1}\)}\) is a countable basis for the topology generated by \(\mathcal{T} \cup \mathcal{T}^{-1}\). Thus every second countable topology compatible with \(S\) is contained in a second countable topology compatible with \(I\).
The result then follows from the definitions of \(\mathcal{SCC}(I)\) and \(\mathcal{SCC}(S)\).

\((1\iff 5):\) This proof is identical to the proof of \((1\Rightarrow 2)\) together with the proof of \((2\Rightarrow 1)\) if we replace \(S\) with \(I\).
\end{proof}

\speeddicttwo{191}{sym(N) has SCC cor}{automatic continuity examples thm}{ac topology is nice}
\begin{corollary}[\(\Sym(\N)\) continuity, \Assumed{191}]\label{sym(N) has SCC cor}
The topology \(\mathcal{SCC}(\Sym(\N))\) is the pointwise topology on \(\Sym(\N)\).
\end{corollary}
\begin{proof}
From Proposition~\ref{ac topology is nice} together with Theorem~\ref{automatic continuity examples thm}, we have that \(\mathcal{SCC}(\Sym(\N))\) is contained in the pointwise topology.
As every second countable topology compatible with \(\SN\) is contained in \(\mathcal{SCC}(\Sym(\N))\), we also have the reverse inclusion.
\end{proof}

\speeddictthree{182}{Hausdorff Zariski lemma}{Products in Categories Defn}{Hausdorff defn}{Zariski topology defn}
\begin{lemma}[Hausdorff Zariski Comparison, \Assumed{182}]\label{Hausdorff Zariski lemma}
If \(\mathbb{S}\) is an algebraic \(\sigma\)-structure and \(\mathcal{T}\) is a Hausdorff topology compatible with \(\mathbb{S}\), then \(\mathcal{Z}(\mathbb{S})\subseteq \mathcal{T}\).
\end{lemma}
\begin{proof}
Let \(a, b\) be arbitrary terms of \(\sigma\) over \(\{0, 1, \ldots, n-1\}\) and \(\{0, 1, \ldots, m-1\}\) respectively. Suppose that \(s_1, s_2, \ldots, s_{n-1}\), \(t_1, t_2\ldots, t_{m-1}\in \mathbb{S}\) are arbitrary and \(c_{s,1}, c_{s, 2}, \ldots, c_{s, n-1},\) \(c_{t, 1}, c_{t, 2}, \ldots, c_{t, m-1}:\mathbb{S}\to \mathbb{S}\) are the constant maps with these values. 
We define continuous maps \(a':\mathbb{S}\to \mathbb{S}\) and \(b':\mathbb{S}\to \mathbb{S}\) as follows (recall Definition~\ref{Products in Categories Defn})
\[a' := \langle \text{id}_{\mathbb{S}}, c_{s, 1}, \ldots, c_{s, n-1}\rangle_{\mathbb{S}^n} \circ a_{\{0, 1, \ldots n-1\}}^S,\]
\[ b' := \langle \text{id}_{\mathbb{S}}, c_{t, 1}, \ldots, c_{t, m-1}\rangle_{\mathbb{S}^m} \circ b_{\{0, 1, \ldots m-1\}}^S.\]
Note that by construction the maps \(a', b'\) are continuous with respect to \(\mathcal{T}\) and
\begin{align*}
    E&:= \makeset{s\in \mathbb{S}}{\((s,s_1, \ldots, s_{n-1})a_{\{0, 1, \ldots n-1\}}^S= (s, t_1, \ldots, t_{m-1})b_{\{0, 1, \ldots n-1\}}^S\)}\\
    &=\makeset{s\in \mathbb{S}}{\((s)a' = (s)b'\)}=\makeset{s\in \mathbb{S}}{\((s)\langle a', b'\rangle_{\mathbb{S}^2}\in \text{id}_{\mathbb{S}}\)}= (\text{id}_{\mathbb{S}})\langle a', b'\rangle_{\mathbb{S}^2}^{-1}
\end{align*}
is an arbitrary elementary algebraic subset of \(\mathbb{S}\).
As \(\mathcal{T}\) is Hausdorff, it follows that \(\text{id}_{\mathbb{S}}\) is a closed subset of \(\mathbb{S}^2\) and hence \(E\) is closed as required.
\end{proof}

\speeddictsix{183}{topology comparison proposition}{minimum topology defn}{Markov topology defn}{nice minimum topologies prop}{Zariski topologies are nice}{ac topology defn}{Hausdorff Zariski lemma}
\begin{proposition}[Comparing topologies, \Assumed{183}]\label{topology comparison proposition}
If \(S\) is a semigroup, then we have the following containments:\newline

\vspace*{10pt}~\hspace*{135pt}\begin{rotate}{00}\(\mathcal{MF}(S)\)

\begin{rotate}{-45}\(\quad\subseteq\quad
\)\begin{rotate}{45}\(\mathcal{FM}(S)
\)\begin{rotate}{45}\(\quad\subseteq\quad
\)\begin{rotate}{-45}\(\mathcal{HM}(S).
\)\end{rotate}\end{rotate}\end{rotate}\end{rotate}

 \begin{rotate}{45}\(\quad\subseteq\quad
\)\begin{rotate}{-45} \(\mathcal{Z}(S)
\) \begin{rotate}{-45}\(\quad\subseteq\quad
\)\end{rotate}\end{rotate}\end{rotate}

\end{rotate}

\vspace*{20pt}
\noindent Moreover if \(S\) is compatible with some second countable Hausdorff topology \(\mathcal{T}\), then \(\mathcal{HM}(S)\subseteq \mathcal{T}\subseteq \mathcal{SCC}(S)\).
\end{proposition}
\begin{proof}
\((\mathcal{MF}(S) \subseteq \mathcal{Z}(S) \text{ and }\mathcal{MF}(S) \subseteq \mathcal{FM}(S)):\) By Propositions~\ref{nice minimum topologies prop} and \ref{Zariski topologies are nice}, both of the topologies \(\mathcal{Z}(S) \text{ and } \mathcal{FM}(S)\) are Fréchet and semicompatible with \(S\).
The containments then follow from the definition of \(\mathcal{MF}(S)\).

\((\mathcal{Z}(S)\subseteq \mathcal{HM}(S)):\) By the definition of \(\mathcal{HM}(S)\), it suffices to show that \(\mathcal{Z}(S)\) is contained in every Hausdorff topology compatible with \(S\). This is immediate from Lemma~\ref{Hausdorff Zariski lemma}.

\((\mathcal{FM}(S)\subseteq \mathcal{HM}(S)):\) As all Hausdorff spaces are Fréchet, the topologies in the intersection defining \(\mathcal{HM}(S)\) are included in the intersection defining the topology \(\mathcal{FM}(S)\).

\((\mathcal{HM}(S)\subseteq \mathcal{T}\subseteq  \mathcal{SCC}(S)):\) This is immediate from the definitions of \(\mathcal{HM}(S)\) and \(\mathcal{SCC}(S)\).
\end{proof}

\section{The small index property and property \textbf{W}}

We next introduce the ``right/left small index properties" for topological semigroups and explain how they interact with the second countable continuity topology mentioned earlier (see Corollary~\ref{SCC vs small index cor}).
These notions generalise the notation of the ``small index property" for topological groups which says that all countable index subgroups of a group are open.
This property has been studied a lot in the literature (see for example \cite{Dixon1986aa,Evans1986aa,Herwig1998aa,Hodges1993ab,Truss1989aa}). 
It is routine to verify that each of the right and left small index properties coincide with the usual small index property when applied to a group.

\speeddicttwo{184}{one-sided congruences}{types of binary relation defns}{semigroup defn}
\begin{defn}[One-sided congruences, \Assumed{184}]\label{one-sided congruences}
Let \(S\) be a semigroup. We say that an equivalence relation \(\sim\) on \(S\) is a \textit{right congruence} if for all \((s,t)\in\ \sim\) and \(u\in S\) we have \((s u, t u)\in\ \sim\). Similarly we say that \(\sim\) is a \textit{left congruence} if for all \((s,t)\in\ \sim\) and \(u\in S\) we have \((us, u t)\in\ \sim\).
\end{defn}

\speeddicttwo{185}{small index property defn}{topological structures defn}{one-sided congruences}
\begin{defn}[Small index properties, \Assumed{185}]\label{small index property defn}
If \(S\) is a semitopological semigroup, then we say that \(S\) has the \textit{right small index property} if all countable index right congruences on \(S\) are open subsets of \(S^2\) (recall Definition~\ref{types of binary relation defns}).
We define the \textit{left small index property} analogously.
\end{defn}

\speeddictfive{187}{small index property prop}{group actions defn}{dual semigroup defn}{pull back topologies prop}{full trans def}{small index property defn}
\begin{proposition}[Small index condition equivalence, \Assumed{187}]\label{small index property prop}
If \((S, \mathcal{T})\) is a semitopological semigroup, then the following are equivalent:
\begin{enumerate}
    \item All homomorphisms \(\psi:(S, \mathcal{T}) \to (\N^\N, \mathcal{PT}_{\N})\) are continuous,
    \item \((S, \mathcal{T})\) has the right small index property.
\end{enumerate}
Similarly the following are equivalent (to each other not to the first two points):
\begin{enumerate}
\setcounter{enumi}{2}
    \item All anti-homomorphisms \(\psi:(S, \mathcal{T}) \to (\N^\N, \mathcal{PT}_{\N})\) are continuous,
    \item \((S, \mathcal{T})\) has the left small index property.
\end{enumerate}
\end{proposition}
\begin{proof}
\((1\Rightarrow 2):\) Let \(\sim\) be an arbitrary countable index right congruence on \(S\).
We assume without loss of generality that \((S/\sim) \cap \N = \varnothing\). 
Let \(\Phi: (S/\sim) \cup \N \to \N\) be a bijection.
We define an action \(a\) of \(S\) on the set \(\N\) by
\[(p, s)a = \left\{\begin{array}{lr}
([ts]_{\sim})\Phi           &\text{ if } p = ([t]_{\sim})\Phi \\
([s]_{\sim})\Phi    &\text{ if } p\in (\N)\Phi
\end{array}\right\}.\]
Note that the map \(\phi_a:(S, \mathcal{T})\to (\N^\N, \mathcal{PT}_\N)\) defined by \((i)((s)\phi_a)= (i, s)a\) is a semigroup homomorphism. By assumption, it follows that \(\phi_a\) is continuous. Let \(s\in S\) be arbitrary and consider the open subset
\[U_{(1)\Phi, ([s]_{\sim})\Phi}:= \makeset{f\in \N^\N}{\(((1)\Phi)f =  ([s]_{\sim})\Phi\)}\]
of \(\N^\N\). As \(\phi_a\) is continuous, it follows that
\begin{align*}
    (U_{(1)\Phi, ([s]_{\sim})\Phi})\phi_a^{-1}&=\makeset{t\in S}{\(((1)\Phi)((f)\phi_a) =  ([s]_{\sim})\Phi\)}\\
    &=\makeset{t\in S}{\(((1)\Phi,s)a =  ([s]_{\sim})\Phi\)}=[s]_{\sim}
\end{align*}
is an open subset of \(S\).
As \(\sim = \union{s\in S} \left([s]_{\sim} \times [s]_{\sim}\right)\),  it follows that \(\sim\) is open as required.

\((2\Rightarrow 1):\) Let \(\psi:(S, \mathcal{T}) \to (\N^\N, \mathcal{PT}_{\N})\) be a homomorphism, let \(a, b\in \N\) be arbitrary and let
\[U_{a, b}:= \makeset{f\in \N^\N}{\((a, b)\in f\)}.\]
It suffices to show that \((U_{a, b})\psi^{-1}\) is an open subset of \(S\).
We define an equivalence relation \(\sim_a\) on \(S\) by 
\[\sim_a := \makeset{(s, t)\in S}{\((a)((s)\psi) = (a)((s)\psi)\)}.\]
As \(\phi\) is a homomorphism, it follows that \(\sim_a\) is a right congruence.
Note also that 
\[S/\sim_a = \makeset{\makeset{s\in S}{\((a)((s)\phi) = b\)}}{\(b\in \N\)}\]
and so \(\sim_a\) has countable index. Thus \(\sim_a\) is open and \((U_{a, b})\phi^{-1}\) is an equivalence class of \(\sim_a\).
If \(s\in S\), then \([s]_{\sim_a} = ((S\times \{s\})\cap \sim_a)\pi_0\). As \((S\times \{s\})\cap \sim_a\) is an open subset of \(S\times \{s\}\) and the projection map \(\pi_0:S\times \{s\} \to S\) is a homeomorphism, it follows that all equivalence classes of \(\sim_a\) are open. The result follows.

\((3\iff 4):\) The identity map \(\text{id}_S:(S, \mathcal{T}) \to (S^\dagger, \mathcal{T})\) is a homeomorphism of spaces and an anti-isomorphism of semigroups. Moreover by Proposition~\ref{pull back topologies prop}, the pair \((S^\dagger, \mathcal{T})\) is a topological semigroup.
By the definition of \(S^\dagger\), it follows that \(\sim\) is a left congruence on \(S\) if and only if \(\sim\) is a right congruence on \(S^\dagger\). Moreover \(\psi:S\to \N^\N\) is an anti-homomorphism if and only if \(\phi:S^\dagger \to \N^\N\) is a homomorphism. 

Thus we need only show that \((S^\dagger, \mathcal{T})\) having the right small index property is equivalent to all homomorphisms \(\psi: (S^\dagger, \mathcal{T})\to (\N^\N, \mathcal{PT}_\N)\) being continuous. This follows from the equivalence \((1\iff 2)\).
\end{proof}

\speeddicttwo{188}{SCC vs small index cor}{ac topology is nice}{small index property prop}
\begin{corollary}[Automatic continuity vs the small index property, \Assumed{188}]\label{SCC vs small index cor}
If \(S\) is a semigroup, then \((S, \mathcal{SCC}(S))\) has both the right and left small index properties.
\end{corollary}
\begin{proof}
As the topological semigroup \((\N^\N, \mathcal{PT}_{\N})\) is second countable, this follows from Propositions~\ref{ac topology is nice} and \ref{small index property prop}.
\end{proof}

We now introduce ``property \textbf{W}" (our version of what is called ``property \textbf{X}" in the joint paper \cite{AcPaper1}), which is a key tool that we use to lift automatic continuity type results from groups to semigroups in which they are ``prominently" contained (see Lemma~\ref{property W is nice lemma}).

\speeddictthree{189}{property W defn}{nbhds defn}{topological structures defn}{term operations defn}
\begin{defn}[Property \textbf{W}, \Assumed{189}]\label{property W defn}
Suppose that \(\mathbb{S}\) is a topological algebraic \(\sigma\)-structure and \(A\subseteq \mathbb{S}\).
We say that \(\mathbb{S}\) has \textit{property \textbf{W} with respect to \(A\)} if:
\begin{enumerate}
    \item For all \(s\in \mathbb{S}\) there is \(t_s\in A\), \(t_{1, s}, t_{2, s}, \ldots, t_{m -1, s}\in \mathbb{S}\) and a term \(a\) of \(\sigma\) over \(\{0, 1, \ldots, m-1\}\) (for some \(m\in \N\backslash\{0\}\)) such that for all \(N\in \nbhd{A}{t_s}\) we have
\[(t_s, t_{1,s}, t_{2, s}, \ldots, t_{m-1, s}){a}_{\{0, 1, \ldots, m-1\}}^\mathbb{S} =s\quad \text{and}\]
\[(N, t_{1,s}, t_{2, s}, \ldots, t_{m-1, s}){a}_{\{0, 1, \ldots, m-1\}}^\mathbb{S} \in \nbhd{S}{s}.\]
\end{enumerate}
\end{defn}

Every topological structure has property \textbf{W} with respect to itself but more interesting examples tend to be quite technical to verify. 
Later on we will be showing that many semigroups have property \textbf{W} with respect to their groups of units (see Theorems~\ref{BX property W}, \ref{X^X property W}, \ref{partial property W}, \ref{inverse property W lemma}, \ref{injective property w lemma}, \ref{Hilbert property W lemma}, and \ref{Cantor property W lemma}).

\speeddictthree{190}{property W is nice lemma}{subspaces}{structure hom defn}{property W defn}
\begin{lemma}[Property \textbf{W} is nice, \Assumed{190}]\label{property W is nice lemma}
Suppose that \(\mathbb{S}, \mathbb{T}\) are topological \(\sigma\)-structures and \(\mathbb{S}\) has property \textbf{W} with respect to \(A\subseteq \mathbb{S}\).
If \(\phi:\mathbb{S} \to \mathbb{T}\) is a homomorphism and \(\phi\restriction_{A}:A\to \mathbb{T}\) is continuous then \(\phi\) is continuous.
\end{lemma}
\begin{proof}
Let \(A, \mathbb{S}, \mathbb{T}, \phi:\mathbb{S}\to \mathbb{T}\) be as hypothesised.
Let \(U\subseteq \mathbb{T}\) be open. We need to show that \((U)\phi^{-1}\) is open. Let \(s\in (U)\phi^{-1}\) be arbitrary, it suffices to show that \((U)\phi^{-1}\in \nbhd{\mathbb{S}}{s}\).

Let \(m\in \N\backslash \{0\}\), \(t_s\in A\), \(t_{1, s}, t_{2,s}, \ldots, t_{m -1}\in \mathbb{S}\), and a term \(a\in (\{\text{、}, \text{「}, \text{」}\}\cup F_\sigma \cup \{0, 1, \ldots, m-1\})^*\) be as in the definition of property \textbf{W}. As \(U\) is open, it follows that 
\[V:= \makeset{t\in \mathbb{T}}{\((t, (t_{1,s})\phi, \ldots, (t_{m-1, s})\phi){a}_{\{0, 1, \ldots, m-1\}}^\mathbb{T}\in U\)}\]
is an open subset of \(\mathbb{T}\). Moreover as \((t_s, t_{1,s}, t_{2, s}, \ldots, t_{m-1, s}){a}_{\{0, 1, \ldots, m-1\}}^\mathbb{S} = s\) and \(\phi\) is a homomorphism, it follows that \((t_s)\phi \in V\).

As \(V\in \nbhd{\mathbb{T}}{(t_s)\phi}\) and \(\phi\restriction_{A}\) is continuous, it follows that \((V)\phi^{-1}\cap A \in \nbhd{A}{t_s}\). Thus by property \textbf{W}, it follows that 
\[((V)\phi^{-1}\cap A, t_{1,s}, t_{2, s}, \ldots, t_{m-1, s}){a}_{\{0, 1, \ldots, m-1\}}^\mathbb{S} \in \nbhd{\mathbb{S}}{s}.\]
By the definition of \(V\) we have
\[((V)\phi^{-1}\cap A, t_{1,s}, t_{2, s}, \ldots, t_{m-1, s}){a}_{\{0, 1, \ldots, m-1\}}^\mathbb{S} \subseteq (U)\phi^{-1}\]
so \((U)\phi^{-1}\in \nbhd{\mathbb{S}}{s}\) as required.
\end{proof}

\speeddictone{192}{property W is transitive lemma}{property W defn}
\begin{lemma}[Property \textbf{W} is transitive lemma, \Assumed{192}]\label{property W is transitive lemma}
Suppose that \(\mathbb{S}\) is a topological algebraic \(\sigma\)-structure, \(\mathbb{T}\) is a substructure of \(\mathbb{S}\) and \(A\subseteq \mathbb{T}\). If \(\mathbb{S}\) has property \textbf{W} with respect to \(\mathbb{T}\), and \(\mathbb{T}\) has property \textbf{W} with respect to \(A\), then \(\mathbb{S}\) has property \textbf{W} with respect to \(A\).
\end{lemma}
\begin{proof}
Let \(s\in \mathbb{S}\) be arbitrary. 
Let \(t_s\in \mathbb{T}\), \(t_{1, s}, \ldots, t_{m-1, s}\in \mathbb{S}\), \(m\in \N\backslash\{0\}\) and a term \(a\) of \(\sigma\) over \(\{0, 1, \ldots, m-1\}\) be as in the definition of property \textbf{W} with respect to \(\mathbb{T}\). 
Using the fact that \(\mathbb{T}\) has property \textbf{W} with respect to \(A\), let \(t_{t_s}\in A\), \(t_{1, t_s}, \ldots, t_{m-1, t_s}\in \mathbb{S}\), \(n\in \N\backslash\{0\}\) and a term \(b\) of \(\sigma\) over \(\{0, 1, \ldots, n-1\}\) be such that for all \(N\in \nbhd{A}{t_{t_s}}\) we have
\[(t_{t_s}, t_{1,t_s}, t_{2, t_s}, \ldots, t_{m-1, t_s}){b}_{\{0, 1, \ldots, n-1\}}^\mathbb{T} =s\quad \text{and}\]
\[(N, t_{1,s}, t_{2, s}, \ldots, t_{m-1, s}){b}_{\{0, 1, \ldots, n-1\}}^\mathbb{T} \in \nbhd{\mathbb{T}}{s}.\]
Let \(c:\mathbb{S}^{n+m-1}\to \mathbb{S}\) be the term operation such that \((x_0, x_1, \ldots, x_{n+m-2})c\) is defined to be 
\[((x_0, x_{1}, x_{2}, \ldots, x_{n-1}){b}_{\{0, 1, \ldots, n-1\}}^\mathbb{S}, x_{n}, x_{n+1}, \ldots, x_{n+m-2}){a}_{\{0, 1, \ldots, n-1\}}^\mathbb{S}.\]
We now have
\begin{align*}
&(t_{t_s}, t_{1,t_s}, t_{2, t_s}, \ldots, t_{n-1, t_s}, t_{1,s}, t_{2, s}, \ldots, t_{m-1, s})c\\
    =&((t_{t_s}, t_{1,t_s}, t_{2, t_s}, \ldots, t_{n-1, t_s}){b}_{\{0, 1, \ldots, n-1\}}^\mathbb{S}, t_{1,s}, t_{2, s}, \ldots, t_{m-1, s}){a}_{\{0, 1, \ldots, n-1\}}^\mathbb{S}\\
    =&(t_s, t_{1,s}, t_{2, s}, \ldots, t_{m-1, s}){a}_{\{0, 1, \ldots, n-1\}}^\mathbb{S}=s.
\end{align*}
Similarly if \(N\in \nbhd{A}{t_{t_s}}\), then
\[(N, t_{1,t_s}, t_{2, t_s}, \ldots, t_{n-1, t_s}, t_{1,s}, t_{2, s}, \ldots, t_{m-1, s})c\in \nbhd{\mathbb{S}}{s}\]
and so \(\mathbb{S}\) has property \textbf{W} with respect to \(A\) as required.
\end{proof}

\section{Topological clones}
In this section we introduce two ways of viewing topological clones and explain why they are actually equivalent.
We also prove a proposition (Proposition~\ref{topological clones from topological semigroups}) which will we be using later to transfer topological semigroup results to the corresponding clones.
\speeddicttwo{236}{Abstract Clones defn}{Categories defn}{Products in Categories Defn}
\begin{defn}[Abstract clones, \Assumed{236}]\label{Abstract Clones defn}
An \textit{abstract clone} is a pair \((\mathcal{C}, \operatorname{Pow}_\mathcal{C})\) where
\begin{enumerate}
    \item  \(\mathcal{C}\) is a small category (recall Definition~\ref{Categories defn}).
    \item \(\operatorname{Pow}_\mathcal{C}\) is a bijection from the object set of \(\mathcal{C}\) to \(\N\) (the \textit{power map}).
    \item For each \(n\in \N\), there exist morphisms \(\pi_0, \pi_1, \ldots, \pi_{n-1}\) of \(\mathcal{C}\)  such that  \[((n)\operatorname{Pow}_\mathcal{C}^{-1}, (\pi_0, \pi_1, \ldots, \pi_{n-1}))\]
    is a product of \(n\) copies of \((1)\operatorname{Pow}_\mathcal{C}^{-1}\) (recall Definition~\ref{Products in Categories Defn}).
\end{enumerate}
For ease of notation, we often treat a clone as if it where equal to its underlying small category.
\end{defn}

Note that in the previous definition, the choice of morphisms \(\pi_0, \pi_1, \ldots, \pi_{n-1}\) is usually not unique. 
However we will always be using the ``obvious" choice when such a choice exists.

\speeddictthree{237}{topological abstract clones defn}{continuous maps and homeomorphisms defn}{topological structures defn}{Abstract Clones defn}
\begin{defn}[Topological abstract clones, \Assumed{237}]\label{topological abstract clones defn}
A \textit{topological abstract clone} is a triple \((\mathcal{C}, \mathcal{T}_{\mathcal{C},\mathcal{O}}, \mathcal{T}_{\mathcal{C},\mathcal{M}})\), where
\begin{enumerate}
    \item \(\mathcal{C}\) is an abstract clone.
    \item \(\mathcal{T}_{\mathcal{C},\mathcal{O}}\) is a topology on the object set \(\mathcal{O}_\mathcal{C}\) of \(\mathcal{C}\).
    \item \(\mathcal{T}_{\mathcal{C},\mathcal{M}}\) is a topology on the morphism set \(\mathcal{O}_\mathcal{M}\) of \(\mathcal{C}\).
    \item The power map \(\operatorname{Pow}_\mathcal{C}:(\mathcal{O}_\mathcal{C}, \mathcal{T}_{\mathcal{C}, \mathcal{O}}) \to \N\) is continuous (or equivalently \(\mathcal{T}_{\mathcal{C}, \mathcal{O}}\) is discrete).
    \item The source and target maps \(s_{\mathcal{C}},t_{\mathcal{C}}:(\mathcal{M}_{\mathcal{C}}, \mathcal{T}_{\mathcal{C},\mathcal{M}}) \to (\mathcal{O}_{\mathcal{C}}, \mathcal{T}_{\mathcal{C},\mathcal{O}})\) are continuous,
    \item The composition map \(\circ_{\mathcal{C}}:\operatorname{Comp}_{\mathcal{C}}\to \mathcal{M}_{\mathcal{C}}\) is continuous (where \(\mathcal{M}_{\mathcal{C}}\) has the topology \(\mathcal{T}_{\mathcal{C}, \mathcal{M}}\) and \(\operatorname{Comp}_{\mathcal{C}}\) is viewed as a subspace of the product space \(\mathcal{M}_{\mathcal{C}}\times \mathcal{M}_{\mathcal{C}}\)).
    \item For each \(i, j\in \N\), the bijection \[\phi_{i,j}:\left(\makeset{f\in \mathcal{M}_\mathcal{C}}{\((f)s_{\mathcal{C}}\operatorname{Pow}_{\mathcal{C}}=i, (f)t_{\mathcal{C}}\operatorname{Pow}_{\mathcal{C}}=1\)}\right)^j \to \makeset{f\in \mathcal{M}_\mathcal{C}}{\((f)s_{\mathcal{C}}\operatorname{Pow}_{\mathcal{C}}=i, (f)t_{\mathcal{C}}\operatorname{Pow}_{\mathcal{C}}=j\)}\]
 defined by  \((f_0, f_1, \ldots, f_{j-1})\phi_{i, j} = \langle f_0, f_1, \ldots, f_{j-1}\rangle_{X^j}\)
 is continuous.
\end{enumerate}
In this case we say that \(\mathcal{C}, \mathcal{T}_{\mathcal{C}, \mathcal{O}}\) and \(\mathcal{T}_{\mathcal{C}, \mathcal{M}}\) are compatible. If \(P\) is a topological property, then we will describe \((\mathcal{C}, \mathcal{T}_{\mathcal{C},\mathcal{O}}, \mathcal{T}_{\mathcal{C},\mathcal{M}})\) as having this property if both of \(\mathcal{T}_{\mathcal{C}, \mathcal{O}}\) and \(\mathcal{T}_{\mathcal{C}, \mathcal{M}}\) do.
\end{defn}

\speeddicttwo{238}{operation clones defn}{function sets def}{Products in Categories Defn}
\begin{defn}[Operation clones, \Assumed{238}]\label{operation clones defn}
An \textit{operation clone} (over a set \(X\)) is a collection of functions from finite powers of \(X\) to \(X\) such that
\begin{enumerate}
    \item For all \(n\in \N\), the projection maps \(\pi_0, \pi_1, \ldots, \pi_{n-1}:X^n \to X\) are elements of \(\mathcal{C}\).
    \item If \(f_0, f_1, \ldots, f_{n-1}\in \mathcal{C}\) have the same domain and \(g\in \mathcal{C}\) has domain \(X^n\), then \[\langle f_0, f_1, \ldots, f_{n-1} \rangle_{X^n}\circ g\in \mathcal{C}.\]
\end{enumerate}
If \(f\in \mathcal{C}\), then we define the \textit{arity} \((f)\ar\) of \(f\) to be the unique \(n\in \N\) such that \(\dom(f) = X^n\). 
We also denote the set \(\makeset{f\in \mathcal{C}}{\(f\) has arity \(n\)}\) by \(\mathcal{C}^{(n)}\).
\end{defn}

\speeddictthree{239}{topological Operation clones defn}{continuous maps and homeomorphisms defn}{topological structures defn}{Abstract Clones defn}

\begin{defn}[Topological operation clones, \Assumed{239}]\label{topological Operation clones defn}
A \textit{topological operation clone} is a pair \((\mathcal{C}, \mathcal{T}_{\mathcal{C}})\), where:
\begin{enumerate}
    \item \(\mathcal{C}\) is an operation clone.
    \item \(\mathcal{T}_{\mathcal{C}}\) is a topology on the set \(\mathcal{C}\).
    \item For all \(m, n\in \N\), the map \(\circ_{m, n}:(\mathcal{C}^{(m)})^n \times \mathcal{C}^{(n)} \to \mathcal{C}^{(m)}\) defined by
\[((f_0, f_1, \ldots, f_{n-1}), g)\circ_{m, n} = \langle f_0, f_1, \ldots, f_{n-1} \rangle_{X^n}\circ g\in \mathcal{C}\]
is continuous.
\end{enumerate}
\end{defn}

\speeddictthree{240}{topological clones are top clones prop}{Important Categories}{topological abstract clones defn}{topological Operation clones defn}
\begin{proposition}[Topological clones are topological clones, \Assumed{240}]\label{topological clones are top clones prop}
If \((\mathcal{C}, \mathcal{T}_{\mathcal{C}})\) is a topological operation clone over a non-empty set \(X\), then the object \((\mathcal{C}^*, \mathcal{T}_{\mathcal{C}^*, \mathcal{O}}, \mathcal{T}_{\mathcal{C}^*, \mathcal{M}})\) defined as follows is a topological abstract clone.
\begin{enumerate}
    \item Let \(\mathcal{C}^*\) be the subcategory of the category of sets and functions with object set \(\makeset{X^n}{\(n\in \N\)}\) and morphism set consisting of the maps
    \[\langle f_0, f_1, \ldots, f_{n-1}\rangle_{X^n}\]
    where \(n, m\in \N\) and \(f_0, f_1, \ldots, f_{n-1}\in \mathcal{C}^{(m)}\) are arbitrary.
    \item Let \(\mathcal{T}_{\mathcal{C}^*, \mathcal{O}}\) be the discrete topology on \(\mathcal{O}_{\mathcal{C}^*}\).
    \item Let \(\mathcal{T}_{\mathcal{C}^*, \mathcal{M}}\) be the topology on \(\mathcal{M}_{\mathcal{C}^*}\) generated by the sets of the forms
    \[A_{n, m}:=\makeset{f\in \mathcal{M}_{\mathcal{C}^*}}{\((f)s_{\mathcal{C}^*}\operatorname{Pow}_{\mathcal{C}^*}=n,(f)t_{\mathcal{C}^*}\operatorname{Pow}_{\mathcal{C}^*}=m\)},\]
    \[V_{n, U}:=\makeset{f\in \mathcal{M}_{\mathcal{C}^*}}{\(f\circ \pi_n \in U\)}\]
    where \(n, m\in \N\) and \(U\in \mathcal{T}_{\mathcal{C}}\) are arbitrary.
\end{enumerate}
Moreover, the original topological operation clone is then precisely the subspace of \(\mathcal{M}_{\mathcal{C}^*}\) consisting of morphisms with target \(X\), and (up to topological isomorphism) \((\mathcal{C}^*, \mathcal{T}_{\mathcal{C}^*, \mathcal{O}}, \mathcal{T}_{\mathcal{C}^*, \mathcal{M}})\) is the only topological abstract clone with this property.
\end{proposition}
\begin{proof}
We need to verify the conditions of Definition~\ref{topological abstract clones defn}.
The first five conditions are immediate from the definition. We next show Condition 7. For each \(i, j\in \N\), let \(\phi_{i,j}:(A_{i, 1})^j \to A_{i, j}\) be the bijection defined by \((f_0, f_1, \ldots, f_{j-1})\phi_{i, j} = \langle f_0, f_1, \ldots, f_{j-1}\rangle_{X^j}\).
We need to show that \(\phi_{i, j}\) is continuous with respect to \(\mathcal{T}_{C^*, \mathcal{M}}\).
Let \(n\in \N\) and \(U\in \mathcal{T}_{\mathcal{C}}\) be arbitrary. We have
\begin{align*}
    (V_{n, U})\phi_{i, j}^{-1}&=  \left(\makeset{f\in A_{i, j}}{\(f\circ \pi_n \in U\)}\right)\phi_{i, j}^{-1}\\
    &=  \makeset{f\in (A_{i, 1})^j}{\(((f)\phi_{i, j})\circ \pi_n \in U\)}\\
 &=  \makeset{f\in (A_{i, 1})^j}{\((f) \pi_n \in U \)}\\
 &=  \makeset{f\in (A_{i, 1})^j}{\((f) \pi_n \in V_{0, U}\cap A_{i, 1}\)}\\
  &=  (V_{0, U}\cap A_{i, 1})\pi_n^{-1}.
\end{align*}
So indeed \(\phi_{i, j}\) is continuous. Moreover, as \(\phi_{i, j}\) is a bijection and the topology on \(A_{i, j}^j\) is generated by the sets of the form \((V_{0, U}\cap A_{i, 1})\pi_n^{-1}=(U)\pi_n^{-1}\) for \(n\in \N\) and \(U\in \mathcal{T}_{\mathcal{C}}\), it follows from the above equality that \(\phi_{i, j}\) is a homeomorphism.

It remains to show Condition \(6\). As the sets \((A_{n, i}\times A_{i, m})_{i, n, m\in \N}\) are a partition of \(\operatorname{Comp}_{\mathcal{C}^*}\) into open sets, it suffices to show that for all \(n, m, i\in \N\), the map \(\circ_{\mathcal{C}^*}\restriction_{A_{n, i}\times A_{i, m}}\) is continuous.
Let \(i, n, m\in \N\) be arbitrary and let \((f, g)\in A_{n, i}\times A_{i, m}\).
We have
\begin{align*}
    (f, g)\circ_{\mathcal{C}^*}&= f\circ g\\
  &= \langle f\circ \pi_0, f\circ \pi_1,\ldots, f\circ \pi_{n-1}\rangle_{X^n}\circ g\\
   &= ((f\circ \pi_0, f\circ \pi_1,\ldots, f\circ \pi_{n-1}),g)\circ_{m, n} \\
     &= ((f)\phi_{m, n}^{-1},g)\circ_{m, n}.
\end{align*}
As \(\phi_{m,n}^{-1}\) and \(\circ_{m, n}\) are continuous, it follows that \(\circ_{C^*}\) is also continuous as required.
\end{proof}

\speeddictsix{241}{topological clones from topological semigroups}{disjoint union topology defn}{semigroup defn}{topological structures defn}{nice Polish subspaces lemma}{Hausdorff Zariski lemma}{topological abstract clones defn}
\begin{proposition}[Corresponding clone and semigroup topologies, \Assumed{241}]\label{topological clones from topological semigroups}
Suppose that \(\mathcal{C}\) is an abstract clone, and \(M\) is the endomorphism monoid of the object \((1)\operatorname{Pow}_{\mathcal{C}}^{-1}\) in the category \(\mathcal{C}\). Suppose further that the objects \((1)\operatorname{Pow}_{\mathcal{C}}^{-1}\) and \((2)\operatorname{Pow}_{\mathcal{C}}^{-1}\) are isomorphic in the category \(\mathcal{C}\).

If \(\mathcal{S}\) is a topology compatible with the semigroup \(M\), then there is at most one pair of topologies \((\mathcal{T}_{\mathcal{C}, \mathcal{O}, \mathcal{S}},\mathcal{T}_{\mathcal{C}, \mathcal{M}, \mathcal{S}})\) compatible with \(\mathcal{C}\) such that \(\mathcal{T}_{\mathcal{C}, \mathcal{M}, \mathcal{S}}\restriction_{M} = \mathcal{S}\) (recall Definition~\ref{subspaces}). 
Moreover the correspondence \(\mathcal{S} \leftrightarrow (\mathcal{T}_{\mathcal{C}, \mathcal{O}, \mathcal{S}},\mathcal{T}_{\mathcal{C}, \mathcal{M}, \mathcal{S}})\) preserves containments of topologies as well as the properties of being Hausdorff, second countable or Polish.
\end{proposition}
\begin{proof}
For all possible \(n\in \N\backslash \{0\}\), choose \(\psi_{1,n}:(1)\operatorname{Pow}_{\mathcal{C}}^{-1}\to (n)\operatorname{Pow}_{\mathcal{C}}^{-1}\) to be an isomorphism
(by assumption the map \(\psi_{1, 2}\) is defined). For all \(n\in \N\backslash \{0\}\) the morphism
\[\langle \pi_0\circ_{\mathcal{C}}\psi_{1, 2}\circ_{\mathcal{C}} \pi_0,\ \pi_0\circ_{\mathcal{C}}\psi_{1, 2}\circ_{\mathcal{C}} \pi_1,\ \pi_1,\ \pi_2,\ \ldots,\ \pi_{n-1}\rangle_{(n+1)\operatorname{Pow}_{\mathcal{C}}^{-1}}\]
is an isomorphism from \((n)\operatorname{Pow}_{\mathcal{C}}^{-1}\) to \((n+1)\operatorname{Pow}_{\mathcal{C}}^{-1}\). So the morphisms \(\psi_{1, n}\) are defined for all \(n\in \N\backslash\{0\}\).
    For all \(n\in \N\), let \(\iota_n\) be the unique morphism \(\langle \rangle_{(0)\operatorname{Pow}_{\mathcal{C}}^{-1}}\) from \((n)\operatorname{Pow}_{\mathcal{C}}^{-1}\) to \((0)\operatorname{Pow}_{\mathcal{C}}^{-1}\)

Suppose that \((\mathcal{T}_{\mathcal{C}, \mathcal{O}, \mathcal{S}},\mathcal{T}_{\mathcal{C}, \mathcal{M}, \mathcal{S}}), (\mathcal{T}_{\mathcal{C}, \mathcal{O}, \mathcal{S}}',\mathcal{T}_{\mathcal{C}, \mathcal{M}, \mathcal{S}}')\) are pairs of topologies compatible with \(\mathcal{C}\) and such that \(\mathcal{T}_{\mathcal{C}, \mathcal{M}, \mathcal{S}}\restriction_{M} =\mathcal{T}_{\mathcal{C}, \mathcal{M}, \mathcal{S}}\restriction_{M}' = \mathcal{S}\).  
We need to show that \((\mathcal{T}_{\mathcal{C}, \mathcal{O}, \mathcal{S}}, \mathcal{T}_{\mathcal{C}, \mathcal{M}, \mathcal{S}})=(\mathcal{T}_{\mathcal{C}, \mathcal{O}, \mathcal{S}}', \mathcal{T}_{\mathcal{C}, \mathcal{M}, \mathcal{S}}')\).
 As \(\mathcal{T}_{\mathcal{C}, \mathcal{O}, \mathcal{S}}\) and \(\mathcal{T}_{\mathcal{C}, \mathcal{O}, \mathcal{S}}'\) are both discrete, they must be equal.
 It remains to show that \(\mathcal{T}_{\mathcal{C}, \mathcal{M}, \mathcal{S}}'= \mathcal{T}_{\mathcal{C}, \mathcal{M}, \mathcal{S}}\).
  For all \(i, j\in \N\), let
    \[A_{i, j}:=\makeset{f\in \mathcal{M}_{\mathcal{C}}}{\((f)s_{\mathcal{C}}\operatorname{Pow}_{\mathcal{C}}=i,(f)t_{\mathcal{C}}\operatorname{Pow}_{\mathcal{C}}=j\)}.\]
    
    As \(s_\mathcal{C}, t_{\mathcal{C}}\) and \(\operatorname{Pow}_{\mathcal{C}}\) are all assumed to be continuous, the sets \(A_{i, j}\) are open with respect to both \(\mathcal{T}_{\mathcal{C}, \mathcal{M}, \mathcal{S}}\) and \(\mathcal{T}_{\mathcal{C}, \mathcal{M}, \mathcal{S}}'\).  
    Note that the sets \(A_{i, j}\) partition \(\mathcal{M}_{\mathcal{C}}\) into sets which are open with respect to both of the topologies \(\mathcal{T}_{\mathcal{C}, \mathcal{M}, \mathcal{S}}'\) and \(\mathcal{T}_{\mathcal{C}, \mathcal{M}, \mathcal{S}}\).
    Let \(i, j\in \N\) be arbitrary.
    To conclude that \((\mathcal{T}_{\mathcal{C}, \mathcal{O}, \mathcal{S}}, \mathcal{T}_{\mathcal{C}, \mathcal{M}, \mathcal{S}})=(\mathcal{T}_{\mathcal{C}, \mathcal{O}, \mathcal{S}}', \mathcal{T}_{\mathcal{C}, \mathcal{M}, \mathcal{S}}')\), it suffices to show that the topologies \(\mathcal{T}_{\mathcal{C}, \mathcal{M}, \mathcal{S}}'\restriction_{A_{i, j}}\) and \(\mathcal{T}_{\mathcal{C}, \mathcal{M}, \mathcal{S}}\restriction_{A_{i, j}}\) are equal.
    
    In the case that \(i=j=1\), both of these topologies are \(\mathcal{S}\) and in particular they are equal.
    If neither of \(i, j\) are \(0\), then since composition is continuous with respect to both of \(\mathcal{T}_{\mathcal{C}, \mathcal{M}, \mathcal{S}}, \mathcal{T}_{\mathcal{C}, \mathcal{M}, \mathcal{S}}'\) and \(\psi_{1, i}^{-1}, \psi_{1, j}\) are both invertible, it follows that 
    \[\mathcal{T}_{\mathcal{C}, \mathcal{M}, \mathcal{S}}\restriction_{A_{i, j}} =\psi_{1, i}^{-1} \circ_{\mathcal{C}} \mathcal{S}\circ_{\mathcal{C}} \psi_{1,j}= \mathcal{T}_{\mathcal{C}, \mathcal{M}, \mathcal{S}}'\restriction_{A_{i, j}}.\]
    
    If \(j= 0\), then \(A_{i, j}\) is the singleton \(\{\iota_i\}=\{\langle \rangle_{(0)\operatorname{Pow}_\mathcal{C}^{-1}}\}\), so again \(\mathcal{T}_{\mathcal{C}, \mathcal{M}, \mathcal{S}}'\restriction_{A_{i, j}}=\mathcal{T}_{\mathcal{C}, \mathcal{M}, \mathcal{S}}\restriction_{A_{i, j}}\).
    It remains to consider the case that \(i=0\) and \(j\neq 0\). If \(A_{0, j} = \varnothing\), then the result is clear, otherwise let \(z\in A_{0, j}\) be fixed.
    
    Let \(\phi_j: A_{0, j} \to A_{1, j}\) and \(\phi_j':\im(\phi_j) \to A_{0, j}\) be defined by
    \[(f)\phi_j= \iota_1 \circ_{\mathcal{C}} f  \quad \text{ and }\quad (f)\phi_j'= z \circ_{\mathcal{C}} \psi_{1, j}^{-1} \circ_{\mathcal{C}} f .\]
    By definition, both of these maps are continuous with respect to both of \(\mathcal{T}_{\mathcal{C}, \mathcal{M}, \mathcal{S}}\) and \(\mathcal{T}_{\mathcal{C}, \mathcal{M}, \mathcal{S}}'\).
    If \(f\in A_{i, j}\) is arbitrary, then \((f)\phi_j\phi_j'=z \circ_{\mathcal{C}} \psi_{1, j}^{-1}\circ_{\mathcal{C}}\iota_1 \circ_{\mathcal{C}} f\). 
    As \(z \circ_{\mathcal{C}} \psi_{1, j}^{-1}\circ_{\mathcal{C}}\iota_1\) has both source and target \((0)\operatorname{Pow}_{\mathcal{C}}^{-1}\), it follows that this is \(\iota_0\) (which is also the identity element of the object).
    Thus \((f)\phi_j\phi_j'= f\).
    Moreover if \((f)\phi_j\in \im(\phi_j)\) is arbitrary, then \(((f)\phi_j)\phi_j'\phi_j=((f)\phi_j\phi_j')\phi_j=(f)\phi_j\). 
    Thus \(\phi_j\) is a homeomorphism from \(A_{0, j}\) to \(\im(\phi_j)\). 
    As \(\im(\phi_j)\) is contained in \(A_{1, j}\) and we have already shown that \(\mathcal{T}_{\mathcal{C}, \mathcal{M}, \mathcal{S}}\restriction_{A_{1, j}}=\mathcal{T}_{\mathcal{C}, \mathcal{M}, \mathcal{S}}'\restriction_{A_{1, j}}\), it follows that
     \[\mathcal{T}_{\mathcal{C}, \mathcal{M}, \mathcal{S}}'\restriction_{A_{i, j}} =(\mathcal{T}_{\mathcal{C}, \mathcal{M}, \mathcal{S}}'\restriction_{\im(\phi_j)})\phi_j'=(\mathcal{T}_{\mathcal{C}, \mathcal{M}, \mathcal{S}}\restriction_{\im(\phi_j)})\phi_j'= \mathcal{T}_{\mathcal{C}, \mathcal{M}, \mathcal{S}}\restriction_{A_{i, j}}.\]
     
     We have now shown that the topology \(\mathcal{T}_{\mathcal{C}, \mathcal{M}, \mathcal{S}}\) is a disjoint union of countably many topological spaces which are homeomorphic to specific subspaces of \((M, \mathcal{S})\) (either \(M\) or \(\im(\phi_j\psi_{1, j}^{-1})\) for some \(j\)). 
     Thus the correspondence preserves containment as well as the properties of being second countable and Hausdorff.
     To show that Polishness is preserved, it suffices to show that if \(\mathcal{S}\) is Polish and \(j\in \N\backslash \{0\}\), then so is \(\mathcal{S}\restriction_{\im\left(\phi_j\psi_{1, j}^{-1}\right)}\).
     
     We show that \(\im\left(\phi_j\right) \circ_{\mathcal{C}} \psi_{1, j}^{-1}\) is closed (and is thus Polish from Lemma~\ref{nice Polish subspaces lemma}).
     As \(\mathcal{S}\) is Hausdorff and thus \(\mathcal{Z}(M)\subseteq \mathcal{S}\) (Lemma~\ref{Hausdorff Zariski lemma}), the set
     \[RO:=\intersection{g\in M}\makeset{f\in M}{\(g f=f\)}\]
     is closed. 
     We need only show that \(\im\left(\phi_j\right) \circ_{\mathcal{C}} \psi_{1, j}^{-1}= RO\).
     If \((f)\phi_j \in \im\left(\phi_j\right), g\in M\) are arbitrary, then
     \[g\circ_{\mathcal{C}}(f)\phi_j \circ_{\mathcal{C}} \psi_{1, j}^{-1}=(g\circ_{\mathcal{C}}\iota_1) \circ_{\mathcal{C}} f \circ_{\mathcal{C}} \psi_{1, j}^{-1}=\iota_1 \circ_{\mathcal{C}} f \circ_{\mathcal{C}} \psi_{1, j}^{-1}=(f)\phi_j \circ_{\mathcal{C}} \psi_{1, j}^{-1}.\]
     So indeed \((f)\phi_j \circ_{\mathcal{C}} \psi_{1, j}^{-1}\in RO\). 
     If \(f\in RO\) is arbitrary, then
     \[f = (\iota_1 \circ_{\mathcal{C}} z) \circ_{\mathcal{C}} f=(\iota_1 \circ_{\mathcal{C}} z) \circ_{\mathcal{C}} f \circ_\mathcal{C} (\psi_{1, j} \circ_\mathcal{C} \psi_{1, j}^{-1})=(z \circ_{\mathcal{C}} f \circ_\mathcal{C} \psi_{1, j})\phi_j \circ_{\mathcal{C}} \psi_{1, j}^{-1}.\]
     So \(f\in \im(\phi_j) \circ_{\mathcal{C}} \psi_{1, j}^{-1}\) as required.
\end{proof}

\section{Investigating important semigroups and clones}\label{big section}

In this section we investigate the topologies compatible with various important examples of semigroups using the tools introduced in the previous sections.

\subsection{Binary relations}

In this subsection we explore the binary relation monoids (see Example~\ref{topological binary relations}).
These monoids contain many of the monoids we will be studying later.

\speeddictthree{199}{BX property W}{nbhds defn}{topological binary relations}{property W defn}
\begin{lemma}[Property \textbf{W} for \(\mathcal{B}_X\), \Assumed{199}]\label{BX property W}
If \(X\) is an infinite set, then the topological semigroup \((\mathcal{B}_X, \mathcal{TB}_X)\) (recall Example~\ref{topological binary relations}) has property \textbf{W} with respect to \(\Sym(X)\).
\end{lemma}
\begin{proof}
Let \(s\in \mathcal{B}_X\) be arbitrary.
It suffices to find \(f_s, g_s\in \mathcal{B}_X\) and \(h_s\in \Sym(X)\), such that \(f_sh_s g_s = s\) and for all \(N\in \nbhd{\Sym(X)}{h_s}\) we have \(f_s N g_s\in \nbhd{\mathcal{B}_X}{s}\).

We first fix a bijection \(\Phi:X\to X\times X \times \{0, 1\}\) (this must exist because \(X\) is infinite and \(\{0, 1\}\) is finite).
We then define a binary relation \(f\) from \(X\) to \(X\times X \times \{0, 1\}\) by \(f:= \makeset{(x, (x, y, 0))}{\(x, y\in X\)}\).
We then define
\[f_s :=  f \Phi^{-1} \quad\text{ and }\quad g_s := \Phi f^{-1}.\]
\underline{Claim:} Suppose that \(k\in \mathcal{B}_{X}\) and \(a\) is a bijection between finite subsets of \(X\times X\times \{0, 1\}\) such that \(f a f^{-1}\subseteq k\).
In this case there is some \(t_{k, a}\in \Sym(X\times X \times \{0, 1\})\) such that \(a\subseteq t_{k, a}\) and \(k=ft_{k, a}f^{-1}\).\\
\underline{Proof of Claim:}
Let \(B:=(\dom(a))\pi_0 \cup (\dom(a))\pi_1 \cup (\im(a))\pi_0 \cup (\im(a))\pi_1\), \(G:= X\backslash B\) and \(\phi:X\to G\) be a bijection. 
For each \(x\in X\), let 
\[t_x: \{x\}\times (\{x\})k\phi \times \{0\} \to (\{x\})k \times (\{x\})\phi \times \{0\}\] be the bijection defined by
\[(x, (y)\phi, 0)t_x = (y, (x)\phi, 0).\]
Note that for each \(x\in X\) we have \(\dom(t_x)\cap \dom(a)= \varnothing=\im(t_x)\cap \im(a)\) and \(ft_x f^{-1} = k\cap (\{x\} \times X)\). Moreover if \(x\neq y\), then \(\dom(t_x)\cap \dom(t_y)= \varnothing =\im(t_x)\cap \im(t_y)\). Thus \(t:= \union{x\in X}t_x\) is a bijection between a subsets of \(X\times X\times \{0\}\) and \(f t f^{-1}= k\).

 We now define \(R_d := (X\times X\times \{0\}) \backslash (\dom(t) \cup \dom(a))\) and \(R_i:= (X\times X\times \{0\}) \backslash (\im(t)\cup \im(a))\).
 Let \(b\in G\) be fixed.
 As \(|R_d|, |R_i|\leq |X|\), we can define injections \(u_d:R_d\to G\times \{b\} \times \{1\}\) and \(u_i:R_i\to G\times \{b\} \times \{1\}\).
 We have that \(|G\times (G\backslash \{b\})\times \{1\}|= |X|\), moreover \(G\times (G\backslash \{b\})\times \{1\}\) is disjoint from \[\dom(a)\cup \im(a) \cup \dom(t) \cup \im(t) \cup \dom(u_i)\cup \im(u_i)\cup \dom(u_d) \cup \im(u_d).\]
So there is a bijection \(u\) from the set 
\[(X\times X\times \{0, 1\})\backslash (\dom(a)\cup \dom(u_d) \cup \im(u_i) \cup \dom(t))\] to the set
\[(X\times X\times \{0, 1\})\backslash (\im(a)\cup \im(u_d) \cup \dom(u_i) \cup \im(t)).\]

In particular the sets 
\[\{\dom(u), \dom(u_d), \im(u_i), \dom(t), \dom(a)\} \text{ and } \{\im(u), \im(u_d), \dom(u_i), \im(t), \im(a)\}\] are both partitions of \(X\times X\times \{0, 1\}\).
It follows that \(t_{k, a}:= u\cup u_d\cup u_i^{-1}\cup t \cup a\) is a bijection from \(X\times X\times \{0, 1\}\) to itself.
By definition \(a\subseteq t_{k, a}\), so we need only show that \(ft_{k, a}f^{-1}=k\). This can be seen as follows:
\begin{align*}
    ft_{k, a}f^{-1} &= (f u f^{-1}) \cup (f u_d f^{-1})\cup (f u_i^{-1}f^{-1})\cup (f t f^{-1})\cup (f a f^{-1})\\
&= (\varnothing) \cup (\varnothing)\cup (\varnothing)\cup (k)\cup (f a f^{-1})\\
&=k.\diamondsuit
\end{align*}

By the claim (using \(a=\varnothing\)), we can choose \(t_s\in \Sym(X\times X \times \{0, 1\})\) such that \(f t_s f^{-1} = k\).
So if we define \(h_s:= \Phi t_s \Phi^{-1}\), then \(f_sh_s g_s=s\).

Let \(N\in \nbhd{\Sym(X)}{h_s}\) be arbitrary, it suffices to show that \(f_s N g_s\in \nbhd{\mathcal{B}_X}{s}\).
As \(N\) is a neighbourhood of \(h_s\), there is some finite \(a\subseteq t_s\) such that
\[\Phi\makeset{t\in \Sym(X\times X\times \{0, 1\})}{\(a\subseteq t\)}\Phi^{-1} \subseteq N.\]
We define \(U:=\makeset{k\in \mathcal{B}_X}{\(f a f^{-1}\subseteq k\)}\).
As \(f a f^{-1}\) is finite, the set \(U\) is an open subset of \(\mathcal{B}_X\). Moreover \(f a f^{-1} \subseteq ft_sf^{-1} = s\), so \(U\in \nbhd{\mathcal{B}_X}{s}\).
It therefore suffices to show that \(U\subseteq f_s N g_s\). Let \(k\in U\) be arbitrary. By the claim let \(t_{k, a}\in \Sym(X\times X\times \{0, 1\})\) be such that \(f t_{k, a}f^{-1} = k\) and \(a\subseteq t_{k, a}\).
By the definition of \(a\), it follows that \(\Phi t_{k, a}\Phi^{-1}\in N\). Thus
\[k=ft_{k, a} f^{-1}=f_s\Phi t_{k, a}\Phi^{-1} g_s \in f_s N g_s\]
as required.
\end{proof}

\speeddictthree{200}{Hausdorff binary relations}{Hausdorff defn}{topological binary relations}{nice Hausdorff binary relations}
\begin{defn}[Hausdorff binary relation topology, \Assumed{200}]\label{Hausdorff binary relations}
If \(X\) is a set, then we define the topology \(\mathcal{HTB}_X\) on the set \(\mathcal{B}_X\) to be the topology generated by the sets of the form
\[U_{Y, Z} := \makeset{f\in \mathcal{B}_X}{\((Y\times Z) \cap f\neq \varnothing\)}\quad \text{ and }\quad V_{Y, Z} := \makeset{f\in \mathcal{B}_X}{\((Y\times Z) \cap f= \varnothing\)}\]
for all \(Y, Z\subseteq X\) (We show in Lemma~\ref{nice Hausdorff binary relations} that this is a Hausdorff semigroup topology).
\end{defn}

\speeddictfour{201}{nice Hausdorff binary relations}{subspaces}{nbhds defn}{minimum topology defn}{Hausdorff binary relations}
\begin{lemma}[Hausdorff binary relation topology is nice, \Assumed{201}]\label{nice Hausdorff binary relations}
If \(X\) is a set, then the topology \(\mathcal{HTB}_X\) from Definition~\ref{Hausdorff binary relations} is a Hausdorff topology compatible with \(\mathcal{B}_X\). Moreover \(\mathcal{HTB}_X= \mathcal{MF}(\mathcal{B}_X)\).
\end{lemma}
\begin{proof}
We first show that \(\mathcal{HTB}_X\) is Hausdorff.
Let \(f, g\in \mathcal{B}_X\) be arbitrary with \(f\neq g\).
We can assume without loss of generality that there is some pair \((x, y)\in f\backslash g\).
It follows that \(f\in U_{\{x\}, \{y\}}\) and \(g\in \mathcal{B}_X\backslash U_{\{x\}, \{y\}}= V_{\{x\}, \{y\}}\).
As \(f\) and \(g\) were arbitrary, it follows that \(\mathcal{HTB}_X\) is Hausdorff.

We next show that \(\mathcal{HTB}_X\) is compatible with \(\mathcal{B}_X\). Let \(Y, Z\subseteq X\) be arbitrary. We first show that \((U_{Y, Z})(*^{\mathcal{B}_X})^{-1}\) is open.
Let \((f, g)\in (U_{Y, Z})(*^{\mathcal{B}_X})^{-1}\) be arbitrary.
Note that \((Y\times Z)\cap f g \neq \varnothing\iff (Y)f\cap (Z)g^{-1} \neq \varnothing\).
So there is some \(p\in (Y)f\cap (Z)g^{-1}\).
In particular we have
\[(f, g)\in U_{Y, \{p\}} \times U_{\{p\}, Z} \subseteq (U_{Y, Z})(*^{\mathcal{B}_X})^{-1}.\]
So \((U_{Y, Z})(*^{\mathcal{B}_X})^{-1}\in \nbhd{\mathcal{B}_X^2}{(f,g)}\) and hence \((U_{Y, Z})(*^{\mathcal{B}_X})^{-1}\) is open. 
We next show that the set \((V_{Y, Z})(*^{\mathcal{B}_X})^{-1}\) is open.
Let \((f, g)\in (V_{Y, Z})(*^{\mathcal{B}_X})^{-1}\) be arbitrary.
As before \((Y\times Z)\cap f g = \varnothing \iff (Y)f\cap (Z)g^{-1} = \varnothing\). So we have
\[(f, g)\in V_{Y, X\backslash (Y)f} \times V_{X\backslash (Z)g^{-1}, Z} \subseteq (V_{Y, Z})(*^{\mathcal{B}_X})^{-1}.\]
So \((V_{Y, Z})(*^{\mathcal{B}_X})^{-1}\in \nbhd{\mathcal{B}_X^2}{(f,g)}\) and hence \((V_{Y, Z})(*^{\mathcal{B}_X})^{-1}\) is open.

It remains to show that \(\mathcal{HTB}_X= \mathcal{MF}(\mathcal{B}_X)\).
As \(\mathcal{HTB}_X\) is Fréchet and compatible with \(\mathcal{B}_X\), it follows from the definition of \(\mathcal{MF}(\mathcal{B}_X)\) that \(\mathcal{MF}(\mathcal{B}_X)\subseteq \mathcal{HTB}_X\).
To show the reverse inclusion, let \(\mathcal{T}\) be an arbitrary Fréchet topology semicompatible with \(\mathcal{B}_X\) and let \(Y, Z\subseteq X\) be arbitrary.

It suffices to show that \(U_{Y, Z}\) and \(V_{Y, Z}\) are open in \(\mathcal{T}\).
As \(\mathcal{T}\) is Fréchet, the subspace \[\im(\lambda_{Y\times Y}\circ\rho_{Z\times Z}) = \{\varnothing, Y \times Z\}\]
of \((\mathcal{B}_X, \mathcal{T})\) is discrete.
Moreover as \(\mathcal{T}\) is semicompatible with \(\mathcal{B}_X\), it follows that \((\{\{\varnothing\}, \{Y \times Z\}\})(\lambda_{Y\times Y}\circ\rho_{Z\times Z})^{-1}\) is a partition of \(\mathcal{B}_X\) into open sets. As
\[(\{\{\varnothing\}, \{Y \times Z\}\})(\lambda_{Y\times Y}\circ \rho_{Z\times Z})^{-1}= \{V_{Y\times Z}, U_{Y\times Z}\},\]
the result follows.
\end{proof}

\speeddictseven{202}{main binary relation theorem}{subbasis}{ac topology is nice}{sym(N) has SCC cor}{topology comparison proposition}{property W is nice lemma}{BX property W}{nice Hausdorff binary relations}
\begin{theorem}[The topologies of \(\mathcal{B}_\N\), \Assumed{202}]\label{main binary relation theorem}
Every homomorphism from \((\mathcal{B}_\N, \mathcal{TB}_\N)\) (recall Example~\ref{topological binary relations}) to a second countable topological semigroup is continuous and there are no second countable Fréchet topologies compatible with \(\mathcal{B}_\N\).
Moreover we have
\[\mathcal{TB}_\N = \mathcal{SCC}(\mathcal{B}_\N) \subsetneq \mathcal{MF}(\mathcal{B}_\N)= \mathcal{Z}(\mathcal{B}_\N)  =  \mathcal{FM}(\mathcal{B}_\N) = \mathcal{HM}(\mathcal{B}_\N)=\mathcal{HTB}_\N .\]
\end{theorem}
\begin{proof}
From Lemma~\ref{nice Hausdorff binary relations}, we have that \(\mathcal{HTB}_\N = \mathcal{MF}(\mathcal{B}_\N)\) is Hausdorff and compatible with \(\mathcal{B}_\N\).
Thus from the definition of \(\mathcal{HM}(\mathcal{B}_\N)\) we have \(\mathcal{HM}(\mathcal{B}_\N)\subseteq \mathcal{HTB}_\N=\mathcal{MF}(\mathcal{B}_\N)\).
Thus from Proposition~\ref{topology comparison proposition}, we have 
\[\mathcal{MF}(\mathcal{B}_\N)= \mathcal{Z}(\mathcal{B}_\N)  =  \mathcal{FM}(\mathcal{B}_\N) = \mathcal{HM}(\mathcal{B}_\N)=\mathcal{HTB}_\N.\]
If \(x, y\in \N\) are arbitrary, then the set \(U_{x, y}\) from Example~\ref{topological binary relations} coincides with the set \(U_{\{x\}, \{y\}}\) from Definition~\ref{Hausdorff binary relations}.
Thus \(\mathcal{TB}_\N\subseteq \mathcal{HTB}_\N\). Moreover, the only neighbourhood of \(\varnothing\) in the topology \(\mathcal{TB}_\N\) is the entire space, thus \(\mathcal{TB}_\N\) is not Hausdorff and so \(\mathcal{TB}_\N\subsetneq \mathcal{HTB}_\N\).

It remains to show that \(\mathcal{TB}_\N = \mathcal{SCC}(\mathcal{B}_\N)\) (recall Proposition~\ref{ac topology is nice}).
By definition, the topology \(\mathcal{TB}_\N\) is second countable.
As \(\mathcal{TB}_\N\) is compatible with \(\mathcal{B}_\N\), it follows from the definition of \(\mathcal{SCC}(\mathcal{B}_\N)\) that \(\mathcal{TB}_\N \subseteq \mathcal{SCC}(\mathcal{B}_\N)\).
Let \(\mathcal{T}\) be an arbitrary second countable topology compatible with \(\mathcal{B}_\N\).
It suffices to show that \(\mathcal{T}\subseteq \mathcal{TB}_\N\).

By Corollary~\ref{sym(N) has SCC cor}, we have \(\mathcal{SCC}(\Sym(\N)) = \mathcal{TB}_\N\restriction_{\Sym(\N)}\).
Thus \(\mathcal{T}\restriction_{\Sym(\N)}\subseteq \mathcal{TB}_{\N}\restriction_{\Sym(\N)}\). By applying Lemma~\ref{property W is nice lemma} to the identity homomorphism from \((\mathcal{B}_\N, \mathcal{TB}_\N)\) to \((\mathcal{B}_\N, \mathcal{T}_\N)\), it suffices to show that \(\mathcal{B}_X\) has property \textbf{W} with respect to \(\Sym(\N)\). This is shown in Lemma~\ref{BX property W}.
\end{proof}

Sets and binary relations form a category in a natural fashion.
This category even has products, although the products of this category are very different form the usual products in the categories of sets and functions. 
There is thus a natural clone of binary relations to consider using Proposition~\ref{topological clones from topological semigroups},
however from Theorem~\ref{main binary relation theorem} there is no second countable Fréchet topology compatible with \(\mathcal{B}_\N\), so the same is true of such a clone.

\subsection{(Partial) transformations}
In this subsection we explore the topologies of the transformation and partial transformation monoids (see Examples~\ref{full trans def} and \ref{partial function monoid defn}).

\speeddictone{193}{partial function monoid defn}{topological binary relations}
\begin{example}[Partial function monoid, \Assumed{193}]\label{partial function monoid defn}
If \(X\) is a set then we define the \textit{partial function monoid} \(\mathcal{P}_X\) to be the submonoid of \(\mathcal{B}_X\) (recall Example~\ref{topological binary relations}) consisting of those relations which are functions from subsets of \(X\) to \(X\).

We define the \textit{partial pointwise topology} \(\mathcal{PPT}_X\) on \(\mathcal{P}_X\) to be the topology generated by sets of the form
\[U_{x, y}:= \makeset{f\in \mathcal{P}_X}{\((x, y)\in f\)} \quad \text{ or } \quad V_x:= \makeset{f\in \mathcal{P}_X}{\(x \not\in \dom(f)\)}\]
for \(x, y\in X\). We show in Lemma~\ref{topological partial functions lemma} that \((\mathcal{P}_X, \mathcal{PPT}_X)\) is a topological semigroup.
\end{example}

\speeddictthree{194}{topological partial functions lemma}{subspaces}{Polish space defn}{partial function monoid defn}
\begin{lemma}[Topological partial functions, \Assumed{194}]\label{topological partial functions lemma}
The pair \((\mathcal{P}_X, \mathcal{PPT}_X)\) from Example~\ref{partial function monoid defn} is a topological semigroup.
\end{lemma}
\begin{proof}
Let \(x,y\in X\) be arbitrary. It suffices to show that the sets
\[(U_{x, y})(*^{\mathcal{P}_X})^{-1} \quad \text{ and }\quad (V_{x})(*^{\mathcal{P}_X})^{-1}\]
are open subsets of \(\mathcal{P}_X^2\) (using the topology \(\mathcal{PPT}_X\)). As \((\mathcal{P}_X, \mathcal{TB}_X\restriction_{\mathcal{P}_X})\) is a topological subsemigroup of \((\mathcal{B}_X, \mathcal{TB}_X)\) and \(\mathcal{TB}_X\restriction_{\mathcal{P}_X}\subseteq \mathcal{PPT}_X\) (recall Example~\ref{topological binary relations} and Definition~\ref{subspaces}), it follows that \((U_{x, y})(*^{\mathcal{P}_X})^{-1}\) is open.

It remains to show that \((V_{x})(*^{\mathcal{P}_X})^{-1}\) is open. Note that if \(a, b\in \mathcal{P}_X\) are arbitrary, then
\(x\not\in \dom(ab) \iff (x\notin \dom(f) \text{ or }(x)f\notin \dom(g))\). In other words
\[(V_{x})(*^{\mathcal{P}_X})^{-1} = (V_x\times \mathcal{P}_X)\cup \left(\union{y\in X} \left( U_{x, y}\times V_y\right)\right).\]
So \((V_{x})(*^{\mathcal{P}_X})^{-1}\) is open as required.
\end{proof}

In this document we will only be using the following lemma for the semigroups \(\mathcal{P}_X\) and \(X^X\) (to establish the ``lower bounds" in Theorem~\ref{main N^N theorem}), but it also applies to many other potentially interesting semigroups (\(\End(\mathbb{Q}, \leq)\) for example)).

\speeddictthree{195}{partial minimal T1 lemma}{subspaces}{minimum topology defn}{partial function monoid defn}
\begin{lemma}[Minimum partial function topologies, \Assumed{195}]\label{partial minimal T1 lemma}
Suppose that \(X\) is a set and \(S\) is a subsemigroup of \(\mathcal{P}_X\) such that
\begin{enumerate}
    \item For all \(x\in X\) the constant map \(c_{X, x}:= X\times \{x\}\) is an element of \(S\).
    \item For all \(x\in X\) there is some \(f_x\in S\) such that \((\{x\})f_x^{-1} = \{x\}\) and \(|\im(f_x)|\) is finite.
\end{enumerate}
In this case \(\mathcal{MF}(S) = \mathcal{PPT}_X\restriction_{S}\).
\end{lemma}
\begin{proof}
First note that by Lemma~\ref{topological partial functions lemma}, the topology \(\mathcal{PPT}_X\) is compatible with \(S\). Thus by the definition of \(\mathcal{MF}(S)\) we have \(\mathcal{MF}(S) \subseteq \mathcal{PPT}_X\restriction_{S}\).
It remains to show the reverse inclusion.

Let \(\mathcal{T}\) be an arbitrary Fréchet topology semicompatible with \(S\) and let \(x, y\in X\) be arbitrary.
It suffices to show that the sets \(U_{x, y}, V_x\) from Example~\ref{partial function monoid defn} are open in \(\mathcal{T}\).
Note that if \(g\in S\) then
\begin{align*}
(x, y)\in g&\iff c_{X, x}g f_{y} =c_{X, y} \iff c_{X, x}g f_{y} \notin \{\varnothing\} \cup \makeset{c_{X, z}}{\(z\in \im(f_y)\backslash\{y\}\)},\\
x\not\in\dom(g)&\iff c_{X, x}g c_{X, x} \neq c_{X, x} \iff c_{X, x}g c_{X, x} \in \{\varnothing\}.
\end{align*}
In other words 
\[U_{x, y} = \mathcal{P}_X\backslash \left(\left( \{\varnothing\} \cup \makeset{c_{X, z}}{\(z\in \im(f_y)\backslash\{y\}\)}\right)\lambda_{c_{X, x}}^{-1}\rho_{f_y}^{-1}\right),\]
\[V_x= (\{\varnothing\})\lambda_{c_{X, x}}^{-1}\rho_{c_{X, x}}^{-1}.\]

As the topology \(\mathcal{T}\) is Fréchet, the finite sets \(\{\varnothing\}\) and \(\{\varnothing\} \cup \makeset{c_{X, z}}{\(z\in \im(f_y)\backslash\{y\}\)}\) are closed with respect to \(\mathcal{T}\).
As \(\mathcal{T}\) is semicompatible with \(S\), it follows from the above equalities that \(U_{x, y}\) is open with respect to \(\mathcal{T}\).
\end{proof}

We next prove the lemmas for establishing the ``upper bounds" in Theorem~\ref{main N^N theorem}.

\speeddictthree{196}{X^X property W}{nbhds defn}{full trans def}{property W defn}
\begin{lemma}[Property \textbf{W} for \(X^X\), \Assumed{196}]\label{X^X property W}
If \(X\) is a set, then the topological semigroup \((X^X, \mathcal{PT}_X)\) has property \textbf{W} with respect to \(\Sym(X)\).
\end{lemma}
\begin{proof}
Let \(s\in X^X\) be arbitrary.
It suffices to find \(f_s, g_s\in X^X\) and \(h_s\in \Sym(X)\), such that \(f_sh_sg_s = s\) and for all \(N\in \nbhd{\Sym(X)}{h_s}\) we have \(f_s N g_s\in \nbhd{X^X}{s}\).

If \(X\) is finite, then \(X^X\) is discrete and so we can choose \(f_s=s, h_s=1_{X^X}, g_s= 1_{X^X}\).
If \(X\) is infinite then we fix a bijection \(\Phi:X\to X\times X\) and \(b\in X\). 

We then define \(f_s, g_s\in X^X\) by
\[(x)f_s = (x, b)\Phi^{-1} \quad \text{ and }\quad (x)g_s = (x)\Phi\pi_0.\]
\underline{Claim:} Suppose that \(k\in X^X\) and \(a\) is a bijection between finite subsets of \(X\times X\) such that whenever \((x, b)\in \dom(a)\) we have \((x, b)a\pi_0 = (x)k\).
In this case there is some \(t_{k, a}\in \Sym(X\times X)\) such that \(a\subseteq t_{k, a}\) and \(f_s\Phi t_{k, a} \Phi^{-1} g_s = k\).\\
\underline{Proof of Claim:} Let \(M\subseteq X\) be such that \(|M|= |X\backslash M|\) and let \(\phi:X\to M\backslash ((\im(a))\pi_0 \cup (\im(a))\pi_1)\) be an injection.

Define \(t': (X\times \{b\})\backslash \dom(a) \to X\times M\) by \((x, b)t' = ((x)k, (x)\phi)\).
As \(t'\pi_1\) is injective, it follows that \(t'\) is injective. Moreover
\[|X|= |X\times X|\geq |(X\times X)\backslash \dom(t')|\geq |X\times (X\backslash \{b\})| \geq|X|,\]
\[|X|=|X\times X|\geq |(X\times X)\backslash \im(t')|\geq |X\times (X\backslash M)| \geq|X|.\]
So \(|(X\times X)\backslash \dom(t')|=|(X\times X)\backslash \im(t')|=|X|\). As \(a\) is finite, there is a bijection
\[u: (X\times X)\backslash (\dom(t')\cup \dom(a))\to (X\times X)\backslash (\im(t')\cup \im(a)).\]
By construction each of \(t', a, u\) is injection and the sets
\[\{\dom(t'), \dom(a), \dom(u)\}\quad \text{ and }\quad \{\im(t'), \im(a), \im(u)\}\]
are partitions of \(X\times X\). Thus \(t_{k, a} := t'\cup a\cup u\) is an element of \(\Sym(X\times X)\).
Moreover by definition \(a\subseteq t_{k, a}\) and if \(x\in X\) then either \((x, b)\in \dom(a)\) and 
\[(x)f_s\Phi t_{k, a} \Phi^{-1} g_s = (x, b)t_{k, a}\pi_0 = (x, b)a\pi_0 = (x)k,\]
or \((x, b)\not\in \dom(a)\) and
\[(x)f_s\Phi t_{k, a} \Phi^{-1} g_s = (x, b)t_{k, a}\pi_0 = ((x)k, (x)\phi)\pi_0 = (x)k.\diamondsuit\]

By the claim (using \(a=\varnothing\)), we can choose \(t_s\in \Sym(X\times X)\) such that \(f_s\Phi t_s \Phi^{-1} g_s = k\). We then define \(h_s:= \Phi t_s \Phi^{-1}\).

Let \(N\in \nbhd{\Sym(X)}{h_s}\) be arbitrary, it suffices to show that \(f_s N g_s\in \nbhd{X^X}{s}\).
As \(N\) is a neighbourhood of \(h_s\), there is some finite \(a\subseteq t_s\) such that
\[\Phi\makeset{t\in \Sym(X\times X)}{\(a\subseteq t\)}\Phi^{-1} \subseteq N.\]
We define
\[U:=\makeset{k\in X^X}{whenever \((x, b)\in \dom(a)\) we have \((x)k=(x, b)a\pi_0\)}.\]
As \(a\) is finite, the set \(U\) is an open subset of \(X^X\). Moreover if \((x, b)\in \dom(a)\), then \((x)s=(x)f_sh_sg_s= (x, b)a\pi_0\) and so \(U\in \nbhd{X^X}{s}\).
It therefore suffices to show that \(U\subseteq f_s N g_s\).

Let \(k\in U\) be arbitrary. By the claim let \(t_{k, a}\in \Sym(X\times X)\) be such that \(f_s\Phi t_{k, a}\Phi^{-1} g_s = k\) and \(a\subseteq t_{k, a}\).
By the definition of \(a\), it follows that \(\Phi t_{k, a}\Phi^{-1}\in N\). Thus
\[k= f_s\Phi t_{k, a}\Phi^{-1} g_s \in f_s N g_s\]
as required.
\end{proof}

\speeddictthree{197}{partial property W}{nbhds defn}{property W defn}{partial function monoid defn}
\begin{lemma}[Property \textbf{W} for \(\mathcal{P}_X\), \Assumed{197}]\label{partial property W}
If \(X\) is a set, then the topological semigroup \((\mathcal{P}_X, \mathcal{PPT}_X)\) has property \textbf{W} with respect to \(X^X\).
\end{lemma}
\begin{proof}
Let \(s\in X^X\) be arbitrary.
It suffices to find \(g_s\in \mathcal{P}_X\) and \(h_s\in X^X\), such that \(h_s g_s = s\) and for all \(N\in \nbhd{X^X}{h_s}\) we have \(N g_s\in \nbhd{\mathcal{P}_X}{s}\).

If \(X\) is finite, then \(\mathcal{P}_X\) is discrete and so we can choose \(h_s=1_{\mathcal{P}_X}, g_s= s\).
If \(X\) is infinite then fix some \(p\in X\) and let \(g_s:X\backslash\{p\} \to X \) be a bijection. 

\underline{Claim:} Suppose that \(k\in \mathcal{P}_X\) and \(a\) is a function between finite subsets of \(X\) such that \(a g_s \subseteq k\) and \(
(\{p\})a^{-1} \cap \dom(k)= \varnothing\).
In this case there is some \(h_{k, a}\in X^X\) such that \(a\subseteq h_{k, a}\) and \(h_{k, a}g_s = k\).\\
\underline{Proof of Claim:} Let \(b: \dom(k)\backslash \dom(a) \to X\) and \(c: X\backslash (\dom(k) \cup \dom(a)) \to X\) be defined by 
\[(x)b = (x)kg_s^{-1} \quad\text{ and }\quad (x)c = p.\] 
As \(\{\dom(a), \dom(b), \dom(c)\}\) is a partition of \(X\), it follows that \(h_{k, a}:= a\cup b\cup c \in X^X\).

Note that \(\dom(a g_s)= \dom(a)\backslash (\{p\})a^{-1})\) and so by the hypothesis of the claim we have \(\dom(b) = \dom(k) \backslash \dom(a g_s)\).
It remains to show that \(h_{k, a}g_s = k\). There are \(4\) cases to consider:
\begin{enumerate}
    \item If \(x\in \dom(k)\cap \dom(a g_s)\) then
    \((x)h_{k, a}g_s = (x)a g_s= (x)k.\)
    \item If \(x\in \dom(k)\backslash \dom(a g_s)\) then
        \((x)h_{k, a}g_s = (x)kg_s^{-1}g_s= (x)k.\)
    \item If \(x\in \dom(a g_s)\backslash \dom(k)\) then
        (as \(a g_s\subseteq k\)) we have that \(x\in \varnothing\) and hence \(x\notin \dom(h_{k, a}g_s)\).
    \item If \(x\in X\backslash (\dom(a g_s)\cup \dom(k))\) then \(x\not\in \dom(a g_s) \cup \dom(b g_s) =\dom(a g_s) \cup \dom(b g_s)\cup \dom(cg_s) = \dom(h_{k, a}g_s)\).
\end{enumerate}
So \(h_{k, a}g_s = k\) as required.\(\diamondsuit\)\\
By the claim (using \(s=k\) and \(\varnothing=a\)) let \(h_s\in X^X\) be such that \(h_s g_s = s\).
Let \(N\in \nbhd{X^X}{h_s}\) be arbitrary.
It suffices to show that \(N g_s\in \nbhd{\mathcal{P}_X}{s}\). By the choice of \(N\), there is a finite \(a\subseteq h_s\) such that \(N\supseteq \makeset{h\in X^X}{\(a\subseteq h\)}.\) We define
\[U:= \makeset{k\in \mathcal{P}_\N}{\(a g_s\subseteq k\) and \((\{p\})a^{-1} \cap \dom(k)= \varnothing\)}.\]
Note that the set \(U\) is open in \(\mathcal{PPT}_X\). Moreover we have \(a g_s\subseteq h_s g_s= s\), and if \(x\in (\{p\})a^{-1}\) then \((a)h_s = p\) and hence \(x\notin \dom(h_s g_s)\).
Thus \(U\in \nbhd{\mathcal{P}_X}{s}\).
It thus suffices to show that \(U\subseteq \makeset{h\in X^X}{\(a\subseteq h\)}g_s\subseteq N g_s\).
This is immediate from the claim together with the definition of \(a\), so the result follows.
\end{proof}

\speeddictseven{198}{main N^N theorem}{automatic continuity examples thm}{topology comparison proposition}{property W is nice lemma}{property W is transitive lemma}{partial minimal T1 lemma}{X^X property W}{partial property W}
\begin{theorem}[The topologies of \(\N^\N\) and \(\mathcal{P}_\N\), \Assumed{198}]\label{main N^N theorem}
The topology \(\mathcal{PT}_\N\) is the only second countable Fréchet topology compatible with the semigroup \(\N^\N\) (recall Example~\ref{full trans def}).
Similarly the topology \(\mathcal{PPT}_\N\) is the only second countable Fréchet topology compatible with the semigroup \(\mathcal{P}_\N\) (Recall Example~\ref{partial function monoid defn}).
Moreover all semigroup homomorphisms from either of \((\N^\N,\mathcal{PT}_\N)\) or \((\mathcal{P}_\N, \mathcal{PPT}_\N)\) to second countable semigroups are continuous,
\[\mathcal{PT}_\N = \mathcal{Z}(\N^\N)  = \mathcal{MF}(\N^\N)= \mathcal{FM}(\N^\N) = \mathcal{HM}(\N^\N)= \mathcal{SCC}(\N^\N),\]
\[\mathcal{PPT}_\N = \mathcal{Z}(\mathcal{P}_\N)  = \mathcal{MF}(\mathcal{P}_\N)= \mathcal{FM}(\mathcal{P}_\N) = \mathcal{HM}(\mathcal{P}_\N)= \mathcal{SCC}(\mathcal{P}_\N),\]
and both of these topologies are Polish.
\end{theorem}
\begin{proof}
Let \(d: \mathcal{P}_X^2  \to \mathbb{R}\) be defined by
\[(f, g )d = \inf\left(\makeset{\frac{1}{2^n}}{\(f\restriction_{\{0, 1, \ldots, n-1\}}= g\restriction_{\{0, 1, \ldots, n-1\}}, n\in\N\)}\right).\]

It is routine to verify that this is a complete metric compatible with \(\mathcal{PPT}_X\), and it restricts to a complete metric compatible with \(\mathcal{PT}_X\) so these topologies are Polish.
Let \(S\in \{\N^\N, \mathcal{P}_\N\}\) be arbitrary.
Let \(\mathcal{T}\) be an arbitrary second countable Fréchet topology compatible with \(S\).
By Lemma~\ref{partial minimal T1 lemma} the topology \(\mathcal{T}\) contains \(\mathcal{PPT}_{\N}\restriction_{S}\), so \(\mathcal{T}\) is Hausdorff.
Thus from Proposition~\ref{topology comparison proposition}, we have\newline

\vspace*{10pt}~\hspace*{70pt}\begin{rotate}{00}\(\mathcal{MF}(S)\)

\begin{rotate}{-45}\(\quad\subseteq\quad
\)\begin{rotate}{45}\(\mathcal{FM}(S)
\)\begin{rotate}{45}\(\quad\subseteq\quad
\)\begin{rotate}{-45}\(\mathcal{HM}(S)\subseteq \mathcal{T} \subseteq \mathcal{SCC}(S).\)\end{rotate}\end{rotate}\end{rotate}\end{rotate}

 \begin{rotate}{45}\(\quad\subseteq\quad
\)\begin{rotate}{-45} \(\mathcal{Z}(S)
\) \begin{rotate}{-45}\(\quad\subseteq\quad
\)\end{rotate}\end{rotate}\end{rotate}

\end{rotate}

\vspace*{20pt}
It suffices to show that all of these containments are equalities.
Thus we need only show that \(\mathcal{SCC}(S)\subseteq \mathcal{MF}(S)\).

By Lemma~\ref{partial minimal T1 lemma}, we have that \(\mathcal{MF}(S) = \mathcal{PPT}_\N\restriction_{S}\) (so if \(S=\N^\N\) then \(\mathcal{MF}(S) = \mathcal{PT}_{\N}\)).
Thus it suffices to show that \(\mathcal{SCC}(S) \subseteq \mathcal{PPT}_\N\restriction_{S}\).
Let \(\phi\) be an arbitrary semigroup homomorphism from \((\mathcal{S}, \mathcal{PPT}_\N\restriction_{S})\) to a second countable topological semigroup.
From Proposition~\ref{ac topology is nice}, it suffices to show that \(\phi\) is continuous.

From Proposition~\ref{ac topology is nice} and Theorem~\ref{automatic continuity examples thm}, it follows that \(\phi\restriction_{\SN}\) is continuous.
From Lemma~\ref{property W is nice lemma}, it suffices to show that \(S\) has property \textbf{W} with respect to \(\Sym(\N)\).
If \(S= \N^\N\), then this follows from Lemma~\ref{X^X property W} and if \(S=\mathcal{P}_\N\), then this follows from Lemmas~\ref{X^X property W}, \ref{property W is transitive lemma} and \ref{partial property W}.
\end{proof}

We have now proved the main result of this subsection, and we can use Proposition~\ref{topological clones from topological semigroups} to extend it to the corresponding clones.
The category of sets and functions, and the category of sets and partial functions are both nice enough for the proposition to apply.
However in the category of sets and partial functions the products and projections are a bit unusual (the projection maps are not defined everywhere anymore and a product of no sets no longer contains the empty tuple).
As a result we only state the clone corollary for the full function clone but for those interested, the argument for the partial function clone is the same.

\speeddictthree{242}{full func clone cor}{topological abstract clones defn}{topological clones from topological semigroups}{main N^N theorem}
\begin{corollary}[The topologies of the full function clone, \Assumed{242}]\label{full func clone cor}
Let \(\mathcal{C}\) be the subcategory of sets and functions with the object set
\(\mathcal{O}_\mathcal{C}:= \makeset{\N^i}{\(i\in \N\)}\), and the morphism set \(\mathcal{M}_\mathcal{C}\) consisting of the all functions between these objects.
Let \(\mathcal{P}_\mathcal{O}\) be the discrete topology on \(\mathcal{O}_{\mathcal{C}}\) and let \(\mathcal{P}_\mathcal{M}\) be the topology on \(\mathcal{M}_\mathcal{C}\) generated by the sets of the form
\[A_{i, j} := (\N^{j})^{\N^i}\quad \text{ and } \quad U_{x, y}:= \makeset{f\in \mathcal{M}_{\mathcal{C}}}{\((x, y)\in f\)}\]
for all \(i, j\in \N\), \(x\in \N^i\), and \(y\in \N^j\).

The triple \((\mathcal{C}, \mathcal{P}_{\mathcal{O}}, \mathcal{P}_{M})\) is a Polish topological abstract clone.
Moreover if \((\mathcal{T}_{\mathcal{O}}, \mathcal{T}_{\mathcal{M}})\) are compatible with the clone \(\mathcal{C}\) then
\begin{enumerate}
    \item If \(\mathcal{T}_{\mathcal{M}}\) is Hausdorff then \(\mathcal{T}_{\mathcal{M}}\supseteq \mathcal{P}_{\mathcal{M}}\).
    \item If \(\mathcal{T}_{\mathcal{M}}\) is second countable then \(\mathcal{T}_{\mathcal{M}}\subseteq \mathcal{P}_{\mathcal{M}}\).
\end{enumerate}
\end{corollary}
\begin{proof}
We first show that \((\mathcal{C}, \mathcal{P}_{\mathcal{O}}, \mathcal{P}_{M})\) is a topological abstract clone.
We need to verify the seven conditions in Definition~\ref{topological abstract clones defn}.
The first five conditions are immediate from the definition. 
We next show that \(\circ_{\mathcal{C}}\) is continuous. 
As the sets \((A_{n, i}\times A_{i, m})_{i, n, m\in \N}\) are a partition of \(\operatorname{Comp}_{\mathcal{C}}\) into open sets, it suffices to show that for all \(n, m, i\in \N\), the map
\(\circ_{\mathcal{C}}\restriction_{A_{n, i}\times A_{i, m}}\) is continuous.
For all \(j,k \in \N\), \(x\in \N^i\), \(y\in \N^j\) the set \((A_{j, k})\circ_{\mathcal{C}}\restriction_{A_{n, i}\times A_{i, m}}^{-1}\) is either empty or \(A_{n, i}\times A_{i, m}\) and is thus open, and 
\[(U_{x, y})\circ_{\mathcal{C}}\restriction_{A_{n, i}\times A_{i, m}}^{-1} = 
\union{l\in \N}\union{z\in  \N^l}
U_{x, z} \times U_{z, y}\]
which is also open.

It remains to show that for all \(i, j\in \N\), the bijection \(\phi_{i, j}:A_{i, 1}^j \to A_{i, j}\) from Definition~\ref{topological abstract clones defn} is continuous.
This follows from the observation that if \(x\in \N^i\) and \(y\in \N^j\), then
\[(U_{x, y})\phi_{i, j}^{-1}= \prod_{k\in \{0, 1, \ldots, j-1\}} U_{x, (y)\pi_k}.\]
Thus \((\mathcal{C}, \mathcal{P}_{\mathcal{O}}, \mathcal{P}_{M})\) is a topological abstract clone.
From Theorem~\ref{main N^N theorem}, it follows that \(\mathcal{P}_\mathcal{M}\restriction_{\N^\N}\) is Polish and the required containments hold when all topologies are restricted to \(\N^\N\).
Thus the result follows from Proposition~\ref{topological clones from topological semigroups}.
\end{proof}

\subsection{Partial bijections}
In this subsection we explore the topologies on the symmetric inverse monoids (see Theorem~\ref{main inverse monoid theorem}), these monoids are the inverse semigroup analogues of the symmetric groups.

\speeddicttwo{211}{symmetric inverse monoids defn}{inverse semigroup defn}{partial function monoid defn}
\begin{example}[Symmetric inverse monoids, \Assumed{211}]\label{symmetric inverse monoids defn}
If \(X\) is a set, then we define the \textit{symmetric inverse monoid} \(\mathcal{I}_X\) of \(X\) to be the inverse semigroup with universe consisting of the bijections between subsets of \(X\), and composition and inversion of binary relations as the operations.
Note that the universe of \(\mathcal{I}_X\) is precisely \(\mathcal{P}_X\cap \mathcal{P}_X^{-1}\).
\end{example}

\speeddictthree{212}{topological symmetric inverse monoids defn}{subspaces}{partial function monoid defn}{symmetric inverse monoids defn}
\begin{defn}[Topologies on symmetric inverse monoids, \Assumed{212}]\label{topological symmetric inverse monoids defn}
If \(X\) is a set, then we define the following topologies on the set \(\mathcal{I}_X\)
\begin{enumerate}
    \item The \textit{complemented pointwise topology} \(\mathcal{CPT}_X\) generated by the sets of the form
    \[\makeset{f\in \mathcal{I}_X}{\((x, y)\in f\)} \text{ and }\makeset{f\in \mathcal{I}_X}{\((x, y)\in (X\times X)\backslash f\)}\]
    for all \(x, y\in X\).
    \item The \textit{left pointwise topology} \(\mathcal{LPT}_X := \mathcal{PPT}_X\restriction_{\mathcal{I}_X}\).
  \item The \textit{right pointwise topology} \(\mathcal{RPT}_X :=\mathcal{LPT}_X^{-1}\),
   \item The \textit{inverse pointwise topology} \(\mathcal{IPT}_X\), generated by  \(\mathcal{LPT}_X\cup \mathcal{RPT}_X\).
\end{enumerate}
\end{defn}

\speeddictfour{213}{complemented pointwise topology is nice lemma}{Hausdorff defn}{compactness facts}{minimum topology defn}{topological symmetric inverse monoids defn}
\begin{lemma}[The complemented pointwise topology is nice, \Assumed{213}]\label{complemented pointwise topology is nice lemma}
If \(X\) is a set, then \(\mathcal{CPT}_X= \mathcal{MF}(\mathcal{I}_X)\). Moreover, this topology is compact and Hausdorff.
\end{lemma}
\begin{proof}
We first show that this space is compact and Hausdorff.
By Remark~\ref{compactness facts}, if we view \(\{0, 1\}\) as a discrete space, then \(\{0,1\}^{X\times X}\) is a compact space (using the product topology). For each \(f\in \mathcal{B}_X\), we define \(\chi_f\in \{0,1\}^{X\times X}\) by:
\[(a, b)\chi_f:= \left\{\begin{array}{lr}
 1    & \text{ if } (a, b)\in f \\
0    & \text{ if }(a, b)\in (X\times X)\backslash f \\
\end{array}\right\}.\]
Let \(\phi: \mathcal{B}_X\to \{0, 1\}^{X\times X}\) be the bijection \(f\to \chi_f\).
From the definition of \(\mathcal{CPT}_X\), it follows that \(\phi\restriction_{\mathcal{I}_X}\) is a homeomorphism from \((\mathcal{I}_X, \mathcal{CPT}_X)\) to \((\mathcal{I}_X)\phi\).

Thus by Remark~\ref{compactness facts}, to conclude that \(\mathcal{CPT}_X\) is compact and Hausdorff, it suffices to show that \((\mathcal{I}_X)\phi\) is a closed subset of \(\{0, 1\}^{X\times X}\). Note that
\begin{align*}
     (\mathcal{I}_X)\phi   &= \left(\makeset{f\in \mathcal{B}_X}{for all \(x\in X\) we have \(|(\{x\}) f|\leq 1\) and \(|(\{x\})f^{-1}|\leq 1\)}\right)\phi\\
  &= \makeset{\alpha\in \{0, 1\}^{X\times X}}{for all \(x\in X\) we have \(\left|\makeset{y\in X}{\((x, y)\alpha = 1\)}\right|\leq 1\)\\
   and \(\left|\makeset{y\in X}{\((y, x)\alpha = 1\)}\right|\leq 1\)}.
\end{align*}
As the complement of the set above is
\[\makeset{\alpha\in \{0, 1\}^{X\times X}}{there are \(x, y, z\in X\) such that \(y\neq z\) and either\\ \((x, y)\alpha=(x, z)\alpha = 1\) or
   \((y, x)\alpha=(z, x)\alpha = 1\)},\]
it follows that \((\mathcal{I}_X)\phi\) is closed.

We next show that \(\mathcal{CPT}_X = \mathcal{MF}(\mathcal{I}_X)\).
To show \(\mathcal{CPT}_X \supseteq \mathcal{MF}(\mathcal{I}_X)\), we need to show that \(\mathcal{CPT}_X\) is semicompatible with \(\mathcal{I}_X\) (we have already shown that it is Fréchet).
Let \(f\in \mathcal{I}_X\) and \(x, y\in X\) be arbitrary.
If \(y\not\in \im(f)\), then
\[\left(\makeset{g\in \mathcal{I}_X}{\((x, y)\in g\)}\right)\rho_f^{-1}= \varnothing \text{ and }\left(\makeset{g\in \mathcal{I}_X}{\((x, y)\in (X\times X)\backslash g\)}\right)\rho_f^{-1}= \mathcal{I}_X\]
are both open in \(\mathcal{CPT}_X\). If \(y\in \im(f)\), then we have
\[\left(\makeset{g\in \mathcal{I}_X}{\((x, y)\in g\)}\right)\rho_f^{-1}= \makeset{g\in \mathcal{I}_X}{\((x, (y)f^{-1})\in g\)},\]
\[\left(\makeset{g\in \mathcal{I}_X}{\((x, y)\in (X\times X)\backslash g\)}\right)\rho_f^{-1}=\makeset{g\in \mathcal{I}_X}{\((x, (y)f^{-1})\in (X\times X)\backslash g\)}.\]
Again these sets are both open in \(\mathcal{CPT}_X\), so \(\rho_f\) is continuous. By symmetry, it follows that \(\lambda_f\) is continuous as well and thus \(\mathcal{CPT}_X\) is semicompatible with \(\mathcal{I}_X\).
It remains to show that \(\mathcal{CPT}_X \subseteq \mathcal{MF}(\mathcal{I}_X)\).

Let \(\mathcal{T}\) be an arbitrary Fréchet topology semicompatible with \(\mathcal{I}_X\).
It suffices to show that \(\mathcal{CPT}_X\subseteq \mathcal{T}\).
Let \(x, y\in X\) be arbitrary. As \(\mathcal{T}\) is Fréchet, it follows that \(\im(\lambda_{\{(x, x)\}}\circ \rho_{\{(y, y)\}}) = \{\varnothing, \{(x, y)\}\}\) is a discrete subspace of \((\mathcal{I}_X, \mathcal{T})\).
Thus
\[(\{\{\varnothing\}, \{\{(x, y)\}\}\})(\lambda_{\{(x, x)\}}\circ \rho_{\{(y, y)\}})^{-1}\]
is a partition of \(\mathcal{I}_X\) into two sets which are open with respect to \(\mathcal{T}\). 
This partition is precisely
\[\left\{\makeset{f\in \mathcal{I}_X}{\((x, y)\in (X\times X)\backslash f\)}, \makeset{f\in \mathcal{I}_X}{\((x, y)\in f\)}\right\}\]
so indeed \(\mathcal{T}\supseteq \mathcal{CPT}_X\) as required.
\end{proof}

\speeddictsix{214}{left/right pointwise topologies are nice lemma}{Hausdorff defn}{pull back topologies prop}{Markov topology defn}{topological partial functions lemma}{topological symmetric inverse monoids defn}{complemented pointwise topology is nice lemma}
\begin{lemma}[The left/right pointwise topologies are nice, \Assumed{214}]\label{left/right pointwise topologies are nice lemma}
If \(X\) is a set and \(S\) is the semigroup obtained by removing the unary operation of \(\mathcal{I}_X\), then every Fréchet topology compatible with \(S\) contains either \(\mathcal{LPT}_X\) or \(\mathcal{RPT}_X\).
Moreover these topologies are compatible with \(S\) and \(\mathcal{FM}(S)= \mathcal{HM}(S) = \mathcal{LPT}_X\cap \mathcal{RPT}_X\).
\end{lemma}
\begin{proof}
If \(X\) is finite then the result is clear as all of these topologies are discrete, so we assume without loss of generality that \(X\) is infinite.
As \(\mathcal{PPT}_X\) is a Hausdorff semigroup topology on \(\mathcal{P}_X\) (Lemma~\ref{topological partial functions lemma}), it follows that \(\mathcal{LPT}_X\) is Hausdorff and compatible with \(S\).
Moreover as the inversion map is an anti-isomorphism of \(S\), it follows from Proposition~\ref{pull back topologies prop} that \(\mathcal{RPT}_X\) is also Hausdorff and compatible with \(S\).
Thus by the definitions of \(\mathcal{FM}(S)\) and \(\mathcal{HM}(S)\), it follows that \(\mathcal{FM}(S)\subseteq \mathcal{HM}(S)\subseteq \mathcal{LPT}_X\cap \mathcal{RPT}_X\).

To conclude the second part of the lemma, it therefore suffices to show that \(\mathcal{LPT}_X\cap \mathcal{RPT}_X\subseteq \mathcal{FM}(S)\).
This follows from the first part of the lemma which we will now show.

Let \(\mathcal{T}\) be an arbitrary Fréchet topology compatible with \(S\) and let \(x\in X\) be fixed. Recall that by Proposition~\ref{pull back topologies prop}, the topology \(\mathcal{T}^{-1}\) is also a Fréchet and compatible with \(S\).
By Lemma~\ref{complemented pointwise topology is nice lemma}, we know that \(\mathcal{CPT}_X\subseteq \mathcal{T}\).
In particular the set \(V:=\makeset{f\in S}{\((x, x)\not\in f\)}\) is open with respect to \(\mathcal{T}\).
As \(\varnothing \circ \varnothing = \varnothing \in V\) and \(\mathcal{T}\) is compatible with \(S\), it follows that there is an open neighbourhood \(U\) of \(\varnothing\) (in \(\mathcal{T}\)) such that \(UU\subseteq V\).

For each \(W\subseteq S\), we define 
\[A_W:=X\backslash \union{a\in W}(\{x\})a,\quad \text{ and }\quad B_W:= X\backslash\union{b\in W}(\{x\})b^{-1}.\]
As \(UU\subseteq V\), the sets \(\union{a\in U}(\{x\})a\) and \(\union{b\in U}(\{x\})b^{-1}\) are disjoint, and so \(X=A_{U} \cup B_U\).
As \(A_U= B_{U^{-1}}\) and \(B_U= A_{U^{-1}}\), it follows that either \(|B_U|= |X|\) or \(|B_{U^{-1}}|= |X|\). Let \(U'\in \{U, U^{-1}\}\) be such that \(|B_{U'}|=X\).

Let \(B\subseteq B_{U'}\) be such that \(|B|=|X|=|X\backslash B|\), and let \(\phi':B\to (X\backslash B)\) be a bijection.
It follows that \(\phi:= \phi'\cup {\phi'}^{-1}\in \Sym(X)\subseteq \mathcal{I}_X\), and \(\phi=\phi^{-1}\).

If \(f\in U'\), then by the choice of \(B\) we have \((\{x\})f^{-1} \cap B = \varnothing\). So if \(f= \phi g\in U'\cap (U')\lambda_{\phi}\), then \((\{x\})f^{-1}= (\{x\})g^{-1}\phi = \varnothing\).

As \(\lambda_{\phi}= \lambda_{\phi}^{-1}\), the set \(U'':=U'\cap (U')\lambda_{\phi}\) is open with respect to either \(\mathcal{T}\) or \(\mathcal{T}^{-1}\).
As shown above, we also have \(\varnothing\in U''\subseteq \makeset{f\in S}{\(x\not\in \im(f)\)}\).
Thus if \(f\in S\) is arbitrary, then 
\[x\notin \im(f) \Rightarrow f\circ \{(x,x)\}= \varnothing \Rightarrow  (f)\rho_{\{(x, x)\}} \in U'' \Rightarrow x\notin \im(f).\]
It follows that \((U'')\rho_{\{(x, x)\}}^{-1} = \makeset{f\in S}{\(x\not\in \im(f)\)}\), thus this set is either open in \(\mathcal{T}\) or \(\mathcal{T}^{-1}\).
If \(y\in X\) is arbitrary and \(f\in \Sym(X)\) is such that \(f^{-1}=f\) and \((x)f= y\), then 
\[\left(\makeset{f\in S}{\(x\not\in \im(f)\)}\right) \rho_f^{-1}=\left(\makeset{f\in S}{\(x\not\in \im(f)\)}\right) \rho_f = \makeset{f\in S}{\(y\not\in \im(f)\)}.\]
So we either have \(\mathcal{RPT}_X\subseteq \mathcal{T}\) or \(\mathcal{RPT}_X\subseteq \mathcal{T}^{-1}\). As \(\mathcal{RPT}_X\subseteq \mathcal{T}^{-1}\) if and only if \(\mathcal{LPT}_X\subseteq \mathcal{T}\), the result follows.
\end{proof}

\speeddictfour{215}{inverse pointwise topologies are nice lemma}{combining topologies lemma}{Markov topology defn}{topological symmetric inverse monoids defn}{left/right pointwise topologies are nice lemma}
\begin{lemma}[The inverse pointwise topologies are nice, \Assumed{215}]\label{inverse pointwise topologies are nice lemma}
If \(X\) is a set then \(\mathcal{FM}(\mathcal{I}_X)= \mathcal{HM}(\mathcal{I}_X) = \mathcal{IPT}_X\) and this topology is compatible with \(\mathcal{I}_X\).
\end{lemma}
\begin{proof}
Let \(S\) be the semigroup obtained by removing the unary operation of \(\mathcal{I}_X\).
By Lemma~\ref{left/right pointwise topologies are nice lemma}, both \(\mathcal{LPT}_X\) and \(\mathcal{RPT}_X\) are compatible with \(S\).
So by Lemma~\ref{combining topologies lemma}, it follows that \(\mathcal{IPT}_X\) is compatible with \(S\).
By the definition of \(\mathcal{IPT}_X\), the inversion map is also continuous with respect to this topology, so \(\mathcal{IPT}_X\) is compatible with \(\mathcal{I}_X\) as well.

From Lemma~\ref{left/right pointwise topologies are nice lemma}, it follows that every Fréchet topology compatible with \(\mathcal{I}_X\) contains either \(\mathcal{LPT}_X\) or \(\mathcal{RPT}_X\).
As inversion is required to be continuous and these topologies are mapped to each other by the inverse map, it follows that every Fréchet topology compatible with \(\mathcal{I}_X\) contains \(\mathcal{IPT}_X\).

Thus \(\mathcal{IPT}_X \subseteq \mathcal{FM}(\mathcal{I}_X)\subseteq \mathcal{HM}(\mathcal{I}_X) \subseteq \mathcal{IPT}_X \), so we have equality throughout as required.
\end{proof}

\speeddictthree{216}{inverse property W lemma}{property W defn}{property W is nice lemma}{topological symmetric inverse monoids defn}
\begin{lemma}[Property \textbf{W} for \(\mathcal{I}_X\), \Assumed{216}]\label{inverse property W lemma}
If \(X\) is a set and \(S\) is the semigroup obtained by removing the unary operation from \(\mathcal{I}_X\), then \((S, \mathcal{IPT}_X)\) has property \textbf{W} with respect to \(\Sym(X)\).
\end{lemma}
\begin{proof}
It suffices to show for all \(s\in S\), that there are \(f_s,g_s\in S\) and \(h_s\in \Sym(X)\) such that for all \(N\in \nbhd{\Sym(X)}{h_s}\) we have \(f_s N g_s\in\nbhd{S}{s}\).
If \(X\) is finite then \(S\) is discrete, so we can choose \(f_s=h_s = 1_S, g_s=s\).
Thus we can assume that \(X\) is infinite.

Let \(s\in S\) be arbitrary. Let \(\{M_0, M_1\}\) be a partition of \(X\) such that \(|M_0|=|M_1|=X\), and choose \(f_s:X\to M_0\) to be a bijection. We then define \(g_s:= f_s^{-1}\).

\noindent\underline{Claim:} Suppose that \(k\in S\) and \(a\) is a bijection between finite subsets of \(X\) satisfying:
\begin{enumerate}
    \item \(f_sag_s \subseteq  k\).
    \item \(\dom(f_s a) \cap \dom(k) \subseteq \dom(f_sag_s)\).
    \item \(\im(a g_s) \cap \im(k) \subseteq \im(f_sag_s)\).
\end{enumerate}
In this case there is some \(h_{k, a}\in \Sym(X)\) such that
\(a\subseteq h_{k, a}\) and \(f_sh_{k, a}g_s = k\).\\
\underline{Proof of Claim:}  If \(x\in (\dom(k))f_s\cap \dom(a)\), then \((x)f_s^{-1}\in \dom(f_s a)\cap \dom(k)\subseteq \dom(f_sag_s)\).
Thus \((x)f_s^{-1}kg_s^{-1}=(x)f_s^{-1}f_sag_s g_s^{-1}=(x)a\). By a symmetric argument, if \(x\in (\im(k))g_s^{-1}\cap \im(a)\), then \((x)(f_s^{-1}kg_s^{-1})^{-1}=(x)a^{-1}\).
Thus if we define \(b:=f_s^{-1}kg_s^{-1}\backslash (\dom(a) \times \im(a))\), then \(b=f_s^{-1}kg_s^{-1}\backslash a\).
In particular 
\[\dom(b)=(\dom(k))f_s\backslash \dom(a)\quad \text{ and }\quad \im(b)=(\im(k))g_s^{-1}\backslash \im(a).\]

As \(\im(b)\subseteq M_0\), and \(a\) is finite, it follows that \(|M_1\backslash (\im(a)\cup \im(b))|= |X|\). Thus we can choose an injection \(c:M_0\backslash (\dom(a)\cup \dom(b)) \to (M_1\backslash (\im(a)\cup \im(b)))\) such that \(|X\backslash (\im(a)\cup \im(b)\cup \im(c))|=|X|\).

Similarly, as \(\dom(a)\) is finite and \(\dom(b)\cup \dom(c)\subseteq M_0\), we can choose a bijection \(d:X\backslash (\dom(a)\cup \dom(b)\cup \dom(c))\to X\backslash (\im(a)\cup \im(b)\cup \im(c))\).

As each of \(a, b, c, d\) is a bijection between subsets of \(X\) and 
\[\{\dom(a), \dom(b), \dom(c), \dom(d)\}\quad \text{ and }\quad \{\im(a), \im(b), \im(c), \im(d)\}\]
are partitions of \(X\), it follows that \(h_{k,a}:= a\cup b\cup c\cup d \in \Sym(X)\).

By definition \(a\subseteq h_{k, a}\), it remains to show that \(f_sh_{k, a}g_s= k\). Let \(x\in X\) be arbitrary, there are 5 cases to consider:
\begin{enumerate}
    \item If \((x)f_s\in \dom(a)\) and \((x)f_s a\in \dom(g_s)\), then \(x\in \dom(f_s a g_s)\) and hence \[(x)f_sh_{k, a}g_s=(x)f_sag_s=(x)k.\]
    \item  If \((x)f_s\in \dom(a)\), \((x)f_s a\not\in \dom(g_s)\) and \(x\in \dom(k)\), then 
    \[x\in (\dom(f_s a)\cap \dom(k))\backslash \dom(f_sag_s)= \varnothing\]
    a contradiction.
    \item  If \((x)f_s\in \dom(a)\), \((x)f_s a\not\in \dom(g_s)\) and \(x\not\in \dom(k)\), then 
    \[(\{x\})k = \varnothing = (\{(x)f_s a\})g_s=(\{x\})f_sh_{k, a}g_s.\]
    \item If \((x)f_s\in  \dom(b)\), then \(x\in \dom(f_s b g_s)\) and hence 
    \[(x)f_sh_{k, a}g_s=(x)f_s b g_s=(x)f_sf_s^{-1}kg_s^{-1}g_s=(x)k.\]
    \item If \((x)f_s\in  \dom(c)\), then \((x)f_s\notin \dom(a) \cup ((\dom(k))f_s\backslash \dom(a))\). Thus \((\{x\})k= \varnothing\). Also by the choice of \(c\), we have
    \[(\{x\})f_sh_{k, a}g_s=(\{x\})f_s c g_s=(\{(x)f_s c\}) g_s = \varnothing.\diamondsuit\]
\end{enumerate}
From the claim, using \(a= \varnothing\) and \(s=k\), let \(h_s\in \Sym(X)\) be such that \(f_sh_s g_s=s\). Let \(N\in \nbhd{\Sym(X)}{s}\)
be arbitrary, we need only show that \(f_s N g_s\in \nbhd{S}{s}\).

Let \(a\) be a finite bijection between subsets of \(X\) such that \(a\subseteq h_s\) and \(\makeset{h\in \Sym(X)}{\(a\subseteq h\)}\subseteq N\). We now define
\[U:= \makeset{k\in S}{\(f_sag_s \subseteq  k\),
     \(\dom(f_s a) \cap \dom(k) \subseteq \dom(f_sag_s)\),\\ and
     \(\im(a g_s) \cap \im(k) \subseteq \im(f_sag_s)\).}\]
Note that \(U\) is open with respect to \(\mathcal{IPT}_X\). We first show that \(s\in U\):
\begin{enumerate}
    \item By the definitions of \(a\) and \(h_s\) we have \(f_sag_s\subseteq f_sh_s g_s=s\).
    \item If \(x\in \dom(f_s a)\cap \dom(k)\), then \((x)k=(x)f_sh_s g_s=(x)f_sag_s\). So \(x\in \dom(f_sag_s)\).
    \item If \(x\in \im(a g_s)\cap \im(k)\), then \((x)k^{-1}=(x)(f_sh_s g_s)^{-1}=(x)(f_sag_s)^{-1}\). So \(x\in \im(f_sag_s)\).
\end{enumerate}
Thus \(U\in \nbhd{S}{s}\). We need only show that \(U\subseteq f_s N g_s\).
Let \(k\in U\) be arbitrary.
By the claim, there is \(h_{k, a}\in \Sym(X)\) such that \(a\subseteq h_{k, a}\) and \(f_sh_{k,a}g_s = k\).
By the definition of \(a\), it follows that \(h_{k, a}\in N\).
Thus \(k=f_sh_{k, a}g_s\in f_s N g_s\) as required.
\end{proof}

\speeddictseven{217}{main inverse monoid theorem}{even more nice Polish subspaces cor}{automatic continuity examples thm}{topology comparison proposition}{main N^N theorem}{complemented pointwise topology is nice lemma}{inverse pointwise topologies are nice lemma}{inverse property W lemma}
\begin{theorem}[The topologies of \(\mathcal{I}_X\), \Assumed{217}]\label{main inverse monoid theorem}
The topology \(\mathcal{IPT}_\N\) is Polish and is the only second countable Fréchet topology compatible with the inverse semigroup \(\mathcal{I}_\N\).
Moreover 
\[\mathcal{CPT}_\N=\mathcal{MF}(\mathcal{I}_\N)\subsetneq   \mathcal{FM}(\mathcal{I}_\N) = \mathcal{HM}(\mathcal{I}_\N)= \mathcal{SCC}(\mathcal{I}_\N)=\mathcal{IPT}_\N.\]
If \(S\) is the semigroup obtained by removing the unary operation of \(\mathcal{I}_\N\), then 
\[\mathcal{CPT}_\N=\mathcal{MF}(S)\subsetneq  \mathcal{LPT}_\N\cap \mathcal{RPT}_\N = \mathcal{FM}(S) = \mathcal{HM}(S)\subsetneq \mathcal{SCC}(S)=\mathcal{IPT}_\N\]
and all homomorphisms from \((S, \mathcal{IPT}_\N)\) to second countable topological semigroups are continuous.

\end{theorem}

\begin{proof}
From Theorem~\ref{main N^N theorem}, the space \((\mathcal{P}_\N, \mathcal{PPT}_\N)\) is Polish.
As 
\[\mathcal{P}_\N\backslash \mathcal{I}_\N = \makeset{f\in \mathcal{P}_\N}{there are \(x, y\in \N\) with \(x\neq y\) and \((x)f=(y)f\)},\]
it follows that \(\mathcal{I}_\N\) is closed with respect to \(\mathcal{PPT}_\N\) and is hence \(\mathcal{LPT}_\N\) is Polish by Lemma~\ref{nice Polish subspaces lemma}.
As \(\mathcal{IPT}_\N\) is generated by \(\mathcal{LPT}_\N\) and the sets of the form
\[\makeset{f\in \mathcal{I}_\N}{\(x\notin \im(f)\)}\]
for \(x\in \N\) (which are closed with respect to \(\mathcal{LPT}_\N\)) it follows from Corollary~\ref{even more nice Polish subspaces cor} that \(\mathcal{IPT}_\N\) is Polish.

By Proposition~\ref{ac topology is nice}, we have \(\mathcal{SCC}(\mathcal{I}_\N) = \mathcal{SCC}(S)\).
Moreover as \(\mathcal{IPT}_\N\) is second countable and compatible with \(S\) (Lemma~\ref{inverse pointwise topologies are nice lemma}), it follows from the definition of \(\mathcal{SCC}(S)\), that \(\mathcal{IPT}_\N\subseteq \mathcal{SCC}(S)\).

Let \(\phi\) be an arbitrary homomorphism from \((S, \mathcal{IPT}_\N)\) to a second countable topological inverse semigroup.
To conclude that \(\mathcal{IPT}_\N\supseteq \mathcal{SCC}(S)\), it suffices (by Proposition~\ref{ac topology is nice}) to show that \(\phi\) is continuous.
By Theorem~\ref{automatic continuity examples thm}, we have that \(\phi\restriction_{\Sym(\N)}\) is continuous.
Thus from Lemma~\ref{property W is nice lemma} it suffices to show that \(S\) has property \textbf{W} with respect to \(\Sym(\N)\).
This is shown in Lemma~\ref{inverse property W lemma}.

We now have \(\mathcal{IPT}_\N=\mathcal{SCC}(S)=\mathcal{SCC}(\mathcal{I}_\N)\). By Lemma~\ref{inverse pointwise topologies are nice lemma}, it follows that this topology coincides with both of \(\mathcal{FM}(\mathcal{I}_\N)\) and \(\mathcal{HM}(\mathcal{I}_\N)\) as well.
As every Fréchet second countable topology compatible with \(\mathcal{I}_\N\), contains \(\mathcal{FM}(\mathcal{I}_\N)\) and is contained in \(\mathcal{SCC}(\mathcal{I}_\N)\), they all coincide.
From Lemma~\ref{left/right pointwise topologies are nice lemma} we have \(\mathcal{LPT}_\N\cap \mathcal{RPT}_\N = \mathcal{FM}(S)= \mathcal{HM}(S)\), and from Lemma~\ref{complemented pointwise topology is nice lemma} we have \(\mathcal{CPT}_\N = \mathcal{MF}(\mathcal{I}_X)= \mathcal{MF}(S)\).

It remains to show the containments. The non-strict versions of the containments are immediate from Proposition~\ref{topology comparison proposition}. So we need only show strictness of the containments.
We start by showing that \(\mathcal{IPT}_\N\neq \mathcal{LPT}_\N\cap\mathcal{RPT}_\N\).
The set \(U:=\makeset{f\in \mathcal{I}_\N}{\(0\notin \dom(f)\)}\) is open with respect to \(\mathcal{LPT}_\N\). Moreover \(\varnothing\in U\). If \(V\in \mathcal{RPT}_\N\) is arbitrary such that \(\varnothing\in V\), then by the definition of \(\mathcal{RPT}_\N\), there is some finite \(F\subseteq \N\) such that \(\makeset{f\in \mathcal{I}_X}{\(F\cap \im(f)= \varnothing\)}\subseteq V\).
In particular if \(n\in \N\backslash F\), then \(\{(0, n)\}\in V\backslash U\).
Thus \(V\neq U\) and \(U\in \mathcal{LPT}_\N\backslash \mathcal{RPT}_\N\subseteq \mathcal{IPT}_\N\backslash (\mathcal{LPT}_\N\cap \mathcal{RPT}_\N)\).

We finally show \(\mathcal{CPT}_\N \neq \mathcal{LPT}_\N\cap\mathcal{RPT}_\N\).
Let
\(V:= \makeset{\{(0, n), (n, 0)\}}{\(n\in \N\)}\)
and note that
\[\mathcal{I}_\N\backslash V = \makeset{f\in \mathcal{I}_X}{either \(0\notin \dom(f)\), there is \((a, b)\in ((\N\backslash \{0\})\times (\N\backslash \{0\}))\)\\
with \((a, b)\in f\) or there is \(n\not\in \dom(f)\) with \((0, n)\in f\)}.\]
So \(\mathcal{I}_\N\backslash V\in \mathcal{LPT}_\N\). Thus \(\mathcal{I}_\N\backslash V=(\mathcal{I}_\N\backslash V)^{-1}\in (\mathcal{LPT}_\N)^{-1}= \mathcal{RPT}_\N\). So \(\mathcal{I}_\N\backslash V\) is a neighbourhood of \(\varnothing\) with respect to \(\mathcal{LPT}_\N\cap \mathcal{RPT}_\N\). 
However every neighbourhood of \(\varnothing\) with respect to \(\mathcal{CPT}_\N\) contains a set of the form
\[U_F:=\makeset{f\in \mathcal{I}_\N}{\(F\cap f= \varnothing\)}\]
where \(F\subseteq \N\times \N\) is finite. As all of the sets \(U_F\) intersect \(V\), it follows that \(\mathcal{I}_\N\backslash V\) is not a neighbourhood of \(\varnothing\) with respect to \(\mathcal{CPT}_\N\) so \(\mathcal{CPT}_\N\neq \mathcal{LPT}_\N\cap \mathcal{RPT}_\N\) as required.
\end{proof}

We have now established the main result of this subsection (Theorem~\ref{main inverse monoid theorem}).
The category of partial bijections does not have product objects, so we have no clones in this subsection.
However, we can use this result to distinguish the left/right small index properties discussed earlier.

\speeddicttwo{218}{non symmetric index properties example}{SCC vs small index cor}{main inverse monoid theorem}
\begin{example}[The small index properties are distinct, \Assumed{218}]\label{non symmetric index properties example}
If \(S\) is the semigroup obtained by removing the unary operation of \(\mathcal{I}_\N\), then the topological semigroup \((S, \mathcal{LPT}_\N)\) has the right small index property but not the left small index property.
\end{example}
\begin{proof}
We first show that \((S,\mathcal{LPT}_\N)\) has the right small index property.
By Theorem~\ref{main inverse monoid theorem}, we have that \(\mathcal{IPT}_\N=\mathcal{SCC}(S)\).
Thus by Corollary~\ref{SCC vs small index cor}, the topological semigroup \((S, \mathcal{IPT}_\N)\) has the right small index property.
Let \(\sim\) be an arbitrary right congruence on \(S\), and let \(s\in S\) be arbitrary.
To conclude that \((S, \mathcal{LPT}_\N)\) has the right small index property, it suffices to show that \([s]_\sim\) is a neighbourhood of \(s\) with respect to the topology \(\mathcal{LPT}_\N\).

As \((S, \mathcal{IPT}_\N)\) has the right small index property, the set \([s]_\sim\) is open in \(\mathcal{IPT}_\N\).
Thus there are finite \(A, B\subseteq  X\) and a finite bijection \(a\) between subsets of \(\N\) such that if
\[V:= \makeset{f\in S}{\(A\cap \dom(f)= \varnothing\), \(B\cap \im(f)= \varnothing\) and \(a\subseteq f\)},\]
then \(s\in V\subseteq [s]_\sim\).
Let \(U:=  \makeset{f\in S}{\(A\cap \dom(f)= \varnothing\) and \(a\subseteq f\)}\).
By construction, \(U\) is a neighbourhood of \(s\) with respect to \(\mathcal{LPT}_\N\). 
Let \(u\in U\) be arbitrary.
To conclude that \((S, \mathcal{LPT}_\N)\) has the right small index property, it suffices to show that \(u\in [s]_\sim\).
Let \(g: \N\to \N\backslash B\) be a bijection such that \(g\restriction_{\im(a)}\) is the identity map.
Note that \(u g, s g\in V\subseteq [s]_\sim\), so \(u g \sim s g\).
As \(\sim\) is a right congruence, it follows that
\(s=s g g^{-1}\sim u g g^{-1}=u\).
So \(u\sim s\) as required.

It remains to show that \((S, \mathcal{LPT}_\N)\) does not have the left small index property.
We define a left congruence \(\sim\) on \(S\) by
\[\sim := \makeset{(f, g)\in S\times S}{\((\{0\})f^{-1} = (\{0\})g^{-1}\)}.\]
Note that \([\varnothing]_{\sim} = \makeset{f\in S}{\(0\notin \dom(f)\)}\).
We show in the proof of Theorem~\ref{main inverse monoid theorem} that this set is not open with respect to \(\mathcal{LPT}_\N\).
Thus \(\sim\) is not open with respect to \(\mathcal{LPT}_\N\) and \((S, \mathcal{LPT}_\N)\) does not have the left small index property.
\end{proof}

\subsection{Injective functions}
 In this subsection we discuss injective function monoids (which notably contain all embedding monoids of structures).
 While most of the monoids we explore in this document only admit one compatible Polish topology, in the case of \(\inj(\N)\) we show that are infinitely many (see Theorems~\ref{much Polish Theorem} and \ref{main inj theorem}).
 Also unlike the previous subsections, we are unable to give nice descriptions of \(\mathcal{MF}(\inj(\N))\) or \(\mathcal{FM}(\inj(\N))\).

\speeddictthree{231}{topological injective function monoids defn}{subspaces}{full trans def}{topological symmetric inverse monoids defn}
\begin{example}[Topologies on the injective function monoid, \Assumed{231}]\label{topological injective function monoids defn}
We define the \textit{injective function monoid} \(\inj(\N)\) to be the monoid of injective functions from \(\N\) to itself with composition of binary relations as the operation.
We will consider the following topologies on this semigroup:
\begin{enumerate}
    \item The topology \(\mathcal{PT}_\N\restriction_{\inj(\N)}\) generated by the sets of the form
    \[U_{x, y}:=\makeset{f\in \inj(\N)}{\((x, y)\in f\)}\]
    for all \(x, y\in X\).
    \item The topology \(\mathcal{IPT}_\N\restriction_{\inj(\N)}\) generated by \(\mathcal{PT}_\N\restriction_{\inj(\N)}\), as well as the sets of the form
    \[V_{x}:=\makeset{f\in \inj(\N)}{\(x\not\in \im{(f)}\)}\]
    for all \(x, y\in X\).
    \item The \textit{coimage pointwise topology} \(\mathcal{CIPT}_\N\) generated by \(\mathcal{IPT}_\N\restriction_{\inj(\N)}\) as well as the sets of the form
     \[W_n:=\makeset{f\in \inj(\N)}{\(|\N\backslash \im(f)|= n\)}\]
     for each \(n\in \N \cup \{\omega\}\).
\end{enumerate}
\end{example}

\speeddictsix{232}{much Polish Theorem}{more nice Polish subspaces lemma}{pull back topologies prop}{combining topologies lemma}{main N^N theorem}{main inverse monoid theorem}{topological injective function monoids defn}
\begin{theorem}[Much Polish, \Assumed{232}]\label{much Polish Theorem}
There is a collection \(\makeset{\mathcal{T}_i}{\(i\leq \omega + 1\)}\) of Polish topologies compatible with \(\inj(\N)\) such that
\[\mathcal{PT}_\N\restriction_{\inj(\N)}\subsetneq \mathcal{IPT}_\N\restriction_{\inj(\N)} = \mathcal{T}_0\subsetneq\mathcal{T}_1\subsetneq\mathcal{T}_2\subsetneq\mathcal{T}_3 \ldots \subsetneq \mathcal{T}_{\omega} \subsetneq \mathcal{T}_{\omega + 1}= \mathcal{CIPT}_\N.\]
Moreover the topology \(\mathcal{PT}_\N\restriction_{\inj(\N)}\) is Polish.
\end{theorem}
\begin{proof}
As \(\mathcal{PT}_\N\) is compatible with \(\N^\N\) and \(\mathcal{IPT}_\N\) is compatible with \(\mathcal{I}_\N\), both of these topologies are compatible with \(\inj(\N)\).
From Theorems~\ref{main N^N theorem} and \ref{main inverse monoid theorem}, the topologies \(\mathcal{PT}_\N\) and \(\mathcal{IPT}_\N\) are Polish.
As 
\[\N^\N\backslash \inj(\N) = \makeset{f\in \N^\N}{there are \(x, y\in \N\) with \(x\neq y\) and \((x)f= (y)f\)},\]
\[\mathcal{I}_\N\backslash \inj(\N) = \makeset{f\in \mathcal{I}_\N}{there is \(x\in \N\backslash \dom(f)\)},\]
the set \(\inj(\N)\) is closed in each of these Polish spaces. Thus by Lemma~\ref{nice Polish subspaces lemma}, the topologies \(\mathcal{PT}_\N\restriction_{\inj(\N)}\) and \(\mathcal{IPT}_\N\restriction_{\inj(\N)}\) are Polish. As the set \(\makeset{f\in \inj(\N)}{\(0\not\in \im(f)\)}\) is not open with respect to \(\mathcal{PT}_\N\restriction_{\inj(\N)}\), these topologies are also distinct.

For each \(i\leq \omega+ 1\), let \(\mathcal{T}_i\) be the topology generated by \(\mathcal{IPT}_\N\restriction_{\inj(\N)}\) as well as the sets \(\makeset{W_j}{\(j< i\)}\) (\(W_i\) is as in Example~\ref{topological injective function monoids defn}). 

We need to show that these topologies are Polish, compatible with \(\inj(\N)\), and distinct. We first show that they are Polish.
Suppose for a contradiction that there is \(n\in \N\cup \{\omega, \omega + 1\}\) such that \(\mathcal{T}_n\) is not Polish, and let \(n\) be the smallest such value. There are \(3\) cases to consider:
\begin{enumerate}
    \item If \(n= 0\), then \(\mathcal{T}_n\) is defined to be \(\mathcal{IPT}_\N\restriction_{\inj(\N)}\) which we have already shown is Polish.
    \item If \(n = \omega\), then \(\mathcal{T}_\omega\) is the topology generated by the topologies \((\mathcal{T}_i)_{i<\omega}\) which are all Polish. From Lemma~\ref{more nice Polish subspaces lemma}, it follows that \(\mathcal{T}_\omega\) is Polish as well.
    \item If \(n= m+1\) for some \(m\in \N\cup \{\omega\}\), then \(\mathcal{T}_m\) is Polish and
    \[\inj(\N)\backslash W_m=\left(\union{i<m} W_i\right) \cup\left(\union{A\subseteq\N\\n<|A|<\infty} \left(\intersection{a\in A} \makeset{f\in \inj(\N)}{\(a\not\in \im(f)\)}\right)\right).\]
    So \(W_m\) is closed with respect to \(\mathcal{T}_m\). As \(\mathcal{T}_{n}\) is generated by \(\mathcal{T}_m\) and \(W_m\), it follows from Corollary~\ref{even more nice Polish subspaces cor} that \(\mathcal{T}_m\) is Polish as well.
\end{enumerate}

We next show that they are distinct. It suffices to show that if \(n\in \N\cup \{\omega\}\), then \(W_n\) is not open with respect to \(\mathcal{T}_{n}\). Let \(f\in W_n\) and suppose that \(U\in \mathcal{T}_n\) is such that \(f\in U\).
As \(f\) is not an element of any of the sets \(W_m\) with \(m\neq n\), it follows that there is are finite \(a\subseteq f\) and \(B\subseteq \N\backslash \im(a)\) such that 
\[U\supseteq\makeset{g\in \inj(\N)}{\(a\subseteq g\) and \(\im(g)\cap B = \varnothing\)}\]
If \(b:\N\backslash\dom(a) \to \N\backslash (\im(a)\cup B)\) is any bijection, then \(a\cup b \in U\backslash W_m\). 
In particular \(U\neq W_n\). 
As \(U\) as arbitrary, it follows that \(W_n\not\in \mathcal{T}_n\).

It remains to show that all of these topologies are compatible with \(\inj(\N)\). For each \(n\in \N\cup \{\omega\}\), let \((\N_n, \mathcal{D}_n)\) be the discrete topological semigroup with universe \(\makeset{i\in \N}{\(i< n\)} \cup \{n\}\) and binary operation \(+_n\) defined by \(i+_n j = \min(n, i+j)\).
Define \(\phi_n: \inj(\N) \to \N_n\) by \((f)\phi_n = \min(n, |\N\backslash \im(f)|)\).
If \(f, g\in \inj(\N)\), then \(\N\backslash \im(f g)= (\N\backslash\im(f))g \cup (\N\backslash \im(g))\), so each map \(\phi_n\) is a semigroup homomorphism.
From Proposition~\ref{pull back topologies prop}, each of the topologies \((\mathcal{D}_n)\phi_n^{-1}\) is compatible with \(\inj(\N)\). 
Thus for each \(i\in \N\), it follows from Lemma~\ref{combining topologies lemma} that the topology generated by \(\mathcal{IPT}_\N\restriction_{\inj(\N)}\) and \((\mathcal{D}_n)\phi_n^{-1}\) is compatible with \(\inj(\N)\).

As the topology generated by \(\mathcal{IPT}_\N\restriction_{\inj(\N)}\) and \((\mathcal{D}_n)\phi_n^{-1}\) is the topology \(\mathcal{T}_{n+1}\), it remains only to show that \(\mathcal{T}_\omega\) is compatible with \(\inj(\N)\). As \(\mathcal{T}_\omega\) is generated by the topologies \(\mathcal{T}_n\) for \(n\in \N\), this follows from Lemma~\ref{combining topologies lemma}.
\end{proof}

The following property \textbf{W} proof is significantly easier that the previous ones we have seen so far, and unlike the previous ones we actually use different term operations for different elements.

\speeddicttwo{234}{injective property w lemma}{property W defn}{much Polish Theorem}
\begin{lemma}[Property \textbf{W} for \(\inj(\N)\), \Assumed{234}]\label{injective property w lemma}
The topological semigroup \((\inj(\N), \mathcal{CIPT}_\N)\) has property \textbf{W} with respect to \(\Sym(\N)\).
\end{lemma}
\begin{proof}
It suffices to show that for all \(s\in \inj(\N)\), there is some \(f_s\in \inj(\N)\) and \(h_s\in \Sym(\N)\) such that \(f_sh_s= s\) and for all \(N\in \nbhd{\Sym(\N)}{h_s}\) we have \(f_s N\in \nbhd{\inj(\N)}{s}\).

Let \(s\in \inj(\N)\) be arbitrary. We define \(f_s:= s\) and \(h_s:= 1_{\inj(\N)}\). Let \(N\in \nbhd{\Sym(\N)}{h_s}\) be arbitrary.
As \(f_sh_s=s1_{\inj(\N)}= s\), it suffices to show that \(f_s N\in\nbhd{\inj(\N)}{s} \).

As \(N\in \nbhd{\Sym(\N)}{1_{\inj(\N)}}\), there is some finite \(A\subseteq N\) such that 
\[N\subseteq \makeset{h\in \Sym(\N)}{\((a)h= a\) for all \(a\in A\)}.\]
Let \(U\) be the neighbourhood 
\[\makeset{g\in \inj(\N)}{\((\{a\})g^{-1}= (\{a\})f^{-1}\) for all \(a\in A\) and \(|\N\backslash\im(g)|= |\N\backslash \im(f)|\)}\]
of \(f\), and let \(g\in U\) be arbitrary. It suffices to show that \(g\in f_s N\). We define \(h':= f^{-1}g\). Note that \(h'\) is a bijection from \(\im(f)\) to \(\im(g)\) and \(f h'=g\). 

By the choice of \(g\), we have that \(\im(g)\cap A= \im(f)\cap A\) and \(|\N\backslash \im(g)|=|\N\backslash \im(f)|\). Thus we can choose a bijection \(b:\N\backslash (\im(f)\cup A)\to\N\backslash (\im(g)\cup A)\). Let \(h:= h'\cup b \cup \makeset{(a, a)}{\(a\in A\backslash \im(f)\)}\). As \(h\) is a union of bijections whose domains and images each partition \(\N\), it follows that \(h\in \Sym(\N)\). As \(h'\subseteq h\), we also have that \(f h=f h'=g\).

It therefore suffices to show that \(h\in N\). By the definition of \(A\), we need only show that \(h\) fixes all elements of \(A\). Let \(a\in A\) be arbitrary.
If \(a\notin\im(f)\), then \((a, a)\in h\) by construction. If \(a\in \im(f)\), then \((a)h=(a)h'=(a)f^{-1}g\). As \((\{a\})g^{-1}= (\{a\})f^{-1}=\{(a)f^{-1}\}\), it follows that \((a)h=(a)f^{-1}g=a\) as required.
\end{proof}

\speeddictfive{235}{main inj theorem}{sym(N) has SCC cor}{topology comparison proposition}{property W is nice lemma}{much Polish Theorem}{injective property w lemma}
\begin{theorem}[The topologies of \(\inj(\N)\), \Assumed{235}]\label{main inj theorem}
All homomorphisms from the topological semigroup \((\inj(\N),\mathcal{CIPT}_\N)\) to second countable topological semigroups are continuous and
\[\mathcal{PT}\restriction_{\inj(\N)}=\mathcal{Z}(\inj(\N))=\mathcal{HM}(\inj(\N))\subsetneq \mathcal{SCC}(\inj(\N))= \mathcal{CIPT}_\N.\]
Moreover, the above topologies are Polish.
\end{theorem}
\begin{proof}
From Theorem~\ref{much Polish Theorem}, it follows that \(\mathcal{PT}\restriction_{\inj(\N)}\subsetneq \mathcal{CIPT}_\N\) and both of these topologies are Polish.
As the topologies are Polish it follows also (Proposition~\ref{topology comparison proposition}) that
\[\mathcal{Z}(\inj(\N))=\mathcal{HM}(\inj(\N))\subseteq \mathcal{PT}\restriction_{\inj(\N)}\subsetneq \mathcal{CIPT}_\N\subseteq \mathcal{SCC}(\inj(\N)).\]
From Corollary~\ref{sym(N) has SCC cor}, we know that 
\[\mathcal{SCC}(\inj(\N))\restriction_{\Sym(\N)}\subseteq \mathcal{PT}\restriction_{\Sym(\N)}=\mathcal{CIPT}_\N\restriction_{\Sym(\N)}.\]
Thus from Lemmas~\ref{property W is nice lemma} and \ref{injective property w lemma}, it follows that \(\mathcal{SCC}(\inj(\N)) \subseteq \mathcal{CIPT}_\N\). It remains to show that \(\mathcal{PT}\restriction_{\inj(\N)}\subseteq \mathcal{Z}(\inj(\N))\).

Let \(x, y\in \N\) be arbitrary. We need only show that
\(\makeset{h\in \inj(\N)}{\((x)h\neq y\)}\)
is closed with respect to \(\mathcal{Z}(\inj(\N))\). Let \(f_0, f_1, g_0, g_1\in \inj(\N)\) be such that 
\[\im(f_0)\cap \im(f_1) = \{x\}\quad \text{ and }\quad(n)g_0 =(n)g_1 \iff n\neq y.\]
We then define
\[V_0:= \makeset{h\in \inj(\N)}{\(f_0h g_0=f_0h g_1\)}, \quad V_1:= \makeset{h\in \inj(\N)}{\(f_1 h g_0=f_1 h g_1\)}.\]
Note that both of these sets are elementary algebraic and hence closed with respect to \(\mathcal{Z}(\inj(\N))\).
From the choice of \(g_0, g_1\) we have
\[f_0 h g_0=f_0 h g_1 \iff (n)g_0 = (n)g_0 \text{ for all }n\in \im(f_0h)\iff y\not\in \im(f_0h),\]
\[f_1h g_0=f_1h g_1 \iff (n)g_0 = (n)g_0 \text{ for all }n\in \im(f_1h)\iff y\not\in \im(f_1h).\]
So \(V_0= \makeset{h\in \inj(\N)}{\(y\not\in \im(f_0h)\)}\) and \(V_1= \makeset{h\in \inj(\N)}{\(y\not\in \im(f_1h)\)}\). 
In particular, the set
\[V_0\cup V_1 =\makeset{h\in \inj(\N)}{\(y\not\in \im(f_0h)\cap \im(f_1h)\)}\]
is closed with respect to \(\mathcal{Z}(\inj(\N))\). 
It therefore suffices to show that \(y\not\in \im(f_0h)\cap \im(f_1h)\iff (x)h \neq y\).
As \(h\) is injective, we have
\[\im(f_0h)\cap \im(f_1h)=\im(f_0)h\cap \im(f_1)h=(\im(f_0)\cap \im(f_1))h.\]
From the choice of \(f_0\) and \(f_1\), it follows that
\begin{align*}
    y\not\in \im(f_0h)\cap \im(f_1h)&\iff y\not\in (\im(f_0)\cap \im(f_1))h\\
    &\iff y\not\in (\{x\})h\\
     &\iff (x)h\neq y
\end{align*}
as required.
\end{proof}

\subsection{The Hilbert cube}\label{hilbert subsection}
In this subsection we explore the continuous transformations of the Hilbert cube \([0, 1]^\N\) (see Theorem~\ref{hilber cube main theorem} and Corollary~\ref{Hilbert clone cor}).
Unlike the previous subsections, the main theorem of this subsection has no automatic continuity component.
However if one were to be established for the group \(\Aut([0, 1]^\N)\), we could extend it to \(C([0, 1]^\N)\) similarly to the previous subsections.

Similarly to the proof of Lemma~\ref{partial minimal T1 lemma}, we will think of the set of constant elements of a continuous function monoid as a copy of the space we are acting on.
This gives us a means of algebraically reconstructing our desired topology (the compact-open topology).
With this in mind the following pair of lemmas will be useful (in both this subsection and the Cantor space subsection).
\speeddicttwo{219}{closed constants lemma}{binary relations defns}{Zariski topology defn}
\begin{lemma}[Closed constants, \Assumed{219}]\label{closed constants lemma}
If \(S\) is a semigroup of functions with composition as the operation, then the set \(C\) of constant elements of \(S\) is closed with respect to \(\mathcal{Z}(S)\).
\end{lemma}
\begin{proof}
If \(C= \varnothing\), then the result is clear.
Otherwise let \(c\in C\) be fixed.
By the definition of \(\mathcal{Z}(S)\), the set
\[RO:=\intersection{g\in S}\makeset{f\in S}{\(g f=f\)}\]
is closed with respect to \(\mathcal{Z}(S)\).
It therefore suffices to show that \(RO=C\).
As \(C\) consists of constant maps, it follows that \(C\subseteq RO\).
Moreover if \(f\in RO\) be arbitrary, then \(cf=f\). As \(c\) is a constant map, it follows that \(cf\) is a constant map as well. Thus \(f=cf\in C\) as required.
\end{proof}

\speeddicttwo{220}{Containing the compact-open topology}{Hausdorff defn}{Composition of continuous maps is continuous}
\begin{lemma}[Containing the compact-open topology, \Assumed{220}]\label{Containing the compact-open topology}
If \(X\) and \(Y\) are topological spaces, \(\mathcal{T}\) is a topology on the set \(C(X, Y)\) and the function \(a:X\times (C(X, Y),\mathcal{T})\to Y\) defined by \((x, f)a = (x)f\) is continuous, then \(\mathcal{T}\) contains the compact-open topology (Recall Definition~\ref{continuous function space defn}).
\end{lemma}
\begin{proof}
Let \(f\in C(X, Y)\), compact \(C\subseteq X\) and open \(U\subseteq Y\) be arbitrary such that \((C)f\subseteq U\).
It suffices to show that 
\[V:= \makeset{g\in C(X, Y)}{\((C)g\subseteq U\)}\]
is a neighbourhood of \(f\) with respect to \(\mathcal{T}\).

As \(a\) is continuous, it follows that the set \((U)a^{-1}\) is open in \(X\times (C(X, Y), \mathcal{T})\).
Thus we can choose an index set \(I\), a collection of open subsets \((U_i)_{i\in I}\) of \(X\), and a collection of open subsets \((W_i)_{i\in I}\) of \((C(X, Y), \mathcal{T})\) such that
\[(U)a^{-1} = \union{i\in I} U_i \times W_i.\]
As \(C\) is compact and \(C\times \{f\}\subseteq (U)a^{-1}\), there is a finite set \(F\subseteq I\), such that 
\[C\times \{f\}\subseteq \union{i\in F} U_i \times W_i\subseteq (U)a^{-1}.\]
We can assume without loss of generality that \(f\in W_i\) for all \(i\in F\). 
Let \(W := \intersection{i\in F} W_i\), and note that \(W\) is a neighbourhood of \(f\) with respect to \(\mathcal{T}\).
By the definition of \(W\), we have \(C\times W\subseteq (U)a^{-1}\). So \(f\in W\subseteq V\), and \(V\) is a neighbourhood of \(f\) as required.
\end{proof}

The proof of the following lemma involves a somewhat unusual application of Theorem~\ref{comparable Polish group topologies theorem}, in which we show that two spaces are homeomorphic by giving them a group structure.

\speeddictfive{222}{continuous Hilbert cube action lemma}{nice Polish subspaces lemma}{even more nice Polish subspaces cor}{comparable Polish group topologies theorem}{topology comparison proposition}{closed constants lemma}
\begin{lemma}[Continuous action on \({[0, 1]}^\N\) ,\Assumed{222}]\label{continuous Hilbert cube action lemma}
Suppose that \(\mathcal{T}\) is a Polish topology compatible with the semigroup \(C([0, 1]^\N)\). In this case the map \(a:[0, 1]^\N \times (C([0, 1]^\N), \mathcal{T}) \to [0, 1]^\N\) defined by \((x, f)a = (x)f\) is continuous.
\end{lemma}
\begin{proof}
Let \(C\) denote the set of constant maps from \([0, 1]^\N\) to itself.
By Lemma~\ref{closed constants lemma}, the set \(C\) is closed with respect to \(\mathcal{Z}(C([0, 1]^\N))\).
So by Proposition~\ref{topology comparison proposition}, it follows that \(C\) is closed with respect to \(\mathcal{T}\). Thus by Lemma~\ref{nice Polish subspaces lemma}, the topology \(\mathcal{T}\restriction_{C}\) is Polish.

Let \(I:C\to [0, 1]^\N\) be the bijection sending a constant map to the unique point in its image.
Note that \(a = \langle \pi_0I^{-1}, \pi_1 \rangle_{C \times C([0, 1]^\N)} \circ *^{C([0, 1]^\N)} \circ I.\)
Thus to show that \(a\) is continuous, it suffices to show that \(I:C\to [0, 1]^\N\) is a homeomorphism.

\underline{Claim:} The map \(I\) is continuous.\\
\underline{Proof of claim:}
Let \(a, b\in \mathbb{R}\), and \(i\in \N\) be arbitrary with \(a<b\). It suffices to show that the set
\[U:=\makeset{f\in C}{\(((f)I)\pi_i\in (a, b)\)}\]
is open with respect to \(\mathcal{T}\restriction_{C}\). For each \(j\in \N \backslash \{i\}\), define \(g_j \in C([0, 1]^\N, [0, 1])\) to be the constant map with value \(0\). 
We then define \(g_i\in C([0, 1]^\N, [0, 1])\) to be the map with
\[(f)g_i= \left\{\begin{array}{lr}
(f)\pi_i     & \text{ if } (f)\pi_i\in (a, b)\\
 a    & \text{ if }(f)\pi_i\leq a\\
  b    & \text{ if }(f)\pi_i\geq b.
\end{array}\right\}\]
The map \(g:= \langle (g_j)_{j\in \N}\rangle_{[0, 1]^\N}\) is continuous, and is thus an element of \(C([0, 1]^\N)\).
Note that 
\[U  = \makeset{f\in C}{\(((f)\rho_g)I\pi_i\not \in\{a, b\}\)}=C \backslash ((\{a, b\})(I\pi_i)^{-1} \cap \im(\rho_g)) \rho_g^{-1}.\]
As the map \(\rho_g\) is continuous and the set \((\{a, b\})(I\pi_i)^{-1} \cap \im(\rho_g)\) has size two (and is hence closed), it follows that \(U\) is open with respect to \(\mathcal{T}\).\(\diamondsuit\)\\
It remains to show that \(I^{-1}\) is continuous. For each \(i\in \N\), we define \(f_i, g_i:[0, 1]\to [0, 1]\) by
\[(x)f_i = \frac{x}{2} + \frac{1}{4},\quad \text{ and }\quad(x)g_i = \left\{\begin{array}{cc}
  (x)f_i^{-1}   & \text{ if }x\in [\frac{1}{4}, \frac{3}{4}] \\
1     & \text{ if }x\in [\frac{3}{4}, 1]\\
0 & \text{ if } x\in [0, \frac{1}{4}]
\end{array}\right\}.\]
We then define \(f, g\in C([0, 1]^\N)\) by \(f=\langle (\pi_if_i)_{i\in \N}\rangle_{[0, 1]^\N}\) and \(g=\langle (\pi_i g_i)_{i\in \N}\rangle_{[0, 1]^\N}\).
In particular \(f g=\langle (\pi_if_i g_i)_{i\in \N}\rangle_{[0, 1]^\N}=\langle (\pi_i)_{i\in \N}\rangle_{[0, 1]^\N}\) is the identity function, and \(\im(f) = [\frac{1}{4}, \frac{3}{4}]^\N\subseteq (0, 1)^\N\).
As \(I^{-1}= f \circ I^{-1}\restriction_{(0, 1)^\N} \circ \rho_g\), it suffices to show that \(I\restriction_{(0, 1)^\N}\) is continuous.

For each \(i\in \N\), let \(\phi_i: (0, 1) \to \mathbb{R}\) be an order preserving homeomorphism. 
We then define \(\phi\in C((0, 1)^\N, \mathbb{R}^\N)\) to be \(\langle (\pi_i\phi_i)_{i\in \N} \rangle_{\mathbb{R}^\N}\).
This allows us to define a group structure on \(((0, 1)^\N)I^{-1}\) by \(xy = ((x)I\phi + (y)I\phi)(I\phi)^{-1}\) and \(x^{-1}= (-(x)I\phi)(I\phi)^{-1}\) for all \(x, y\in (0, 1)^\N\).
By definition (and the claim), the map \(I\phi\) is a continuous group isomorphism from \(((0, 1)^\N)I^{-1}\) to \(\mathbb{R}^\N\).
As \(I^{-1}\restriction_{(0, 1)^\N}= \phi(I\phi)^{-1}\) and \(\phi\) is continuous, it suffices to show that \((I\phi)^{-1}\) is continuous.
Thus by Theorem~\ref{comparable Polish group topologies theorem}, we need only show that \(\mathcal{T}\restriction_{((0, 1)^\N)I^{-1}}\) is Polish and semicompatible with the group structure on \(((0, 1)^\N)I^{-1}\).

The set \((0, 1)^\N\) is a countable intersection of open subsets of \([0, 1]^\N\), so by the continuity of \(I\), the set \(((0, 1)^\N)I^{-1}\) is a countable intersection of open subsets of \(C\).
As \(C\) is Polish, it follows from Corollary~\ref{even more nice Polish subspaces cor}, that \(((0, 1)^\N)I^{-1}\) is Polish.

It remains to show that \(\mathcal{T}\restriction_{((0, 1)^\N)I^{-1}}\) is semicompatible with the group structure on \(((0, 1)^\N)I^{-1}\).
Let \(h\in ((0, 1)^\N)I^{-1}\) be arbitrary.
As \(((0, 1)^\N)I^{-1}\) is commutative, it suffices to show that \(\rho_h\) is continuous with respect to \(\mathcal{T}\restriction_{((0, 1)^\N)I^{-1}}\) (here \(\rho_h\) is right multiplication by \(h\) in the group \(((0, 1)^\N)I^{-1}\)).
Let \(x\in ((0, 1)^\N)I^{-1}\) be arbitrary.
We have
\begin{align*}
    (x)\rho_h= ((x)I\phi + (h)I\phi)(I\phi)^{-1}= (((x)I\phi)\rho_{(h)I\phi})(I\phi)^{-1}= (x)I\phi\rho_{(h)I\phi}\phi^{-1}I^{-1}.
\end{align*}
Let \(h^*:= \phi\rho_{(h)I\phi}\phi^{-1}\).
For each \(i\in \N\), define \(h_i^*:(0, 1) \to (0, 1)\) to be the map \(\phi_i\rho_{(h)I\pi_i\phi_i}\phi_i^{-1}\), and note that \(h^*=\langle (\pi_i h_i^*)_{i\in \N}\rangle_{[0, 1]^\N}\). 
Thus for each \(i\in \N\), the map \(h_i':=h_i^* \cup \{(0, 0), (1, 1)\}\) is an order isomorphism (and hence homeomorphism) of \([0, 1]\).
We can thus define \(h'\in C([0, 1]^\N)\) by \(h':= \langle (h_i')_{i\in \N}\rangle_{[0, 1]^\N}\).
As \(h^*\subseteq h'\), we have \((x)\rho_h= (x)I h' I^{-1}\).
As \(x\in C\) and \(h'\in C([0, 1]^\N)\), it follows that
\[(x)\rho_h=I\phi\rho_{(h)I\phi}\phi^{-1}I^{-1}=I h^*I^{-1} = (x)\rho_{h'}\]
where \(\rho_h\) is using the group operation on \(((0, 1)^\N)I^{-1}\) and \(\rho_{h'}\) is using the semigroup operation on \(C([0, 1]^\N)\).
As \(\rho_{h'}\) is continuous by assumption, it follows that \(\rho_h\) is also continuous as required.
\end{proof}

As in the previous subsections, we need to establish property \textbf{W} for \(C([0, 1]^\N)\).
Most of the work required to establish this is in the proof of the following theorem from \cite{Mill2001aa}.
The proof is somewhat lengthy and only tangentially related to the theme of this document so we do not include it here.

\speeddictzero{223}{extending Hilbert homeomorphisms theorem}
\begin{theorem}[cf. Theorem 5.2.4 of~\cite{Mill2001aa}, Extending Hilbert homeomorphisms]\label{extending Hilbert homeomorphisms theorem}
Let \(d: [0, 1]^\N \times [0, 1]^\N \to \mathbb{R}\) be defined by
\[(x, y)d= \sum_{n\in \N}\frac{|(x)\pi_n- (y)\pi_n|}{2^{n+1}}.\]
If \(\varepsilon>0\), \(C_0, C_1\subseteq (0, 1)^\N\) are compact and \(h':C_0 \to C_1\) is a homeomorphism with 
\[\sup_{(x, y)\in h'}(x, y)d < \varepsilon,\] then there is \(h\in \Aut([0, 1]^\N)\) such that \(h'\subseteq h\) and \(\sup_{(x, y)\in h}(x, y)d < \varepsilon\).
\end{theorem}

\speeddictsix{224}{Hilbert property W lemma}{Numbers defns}{Products in Categories Defn}{product topology defn}{Composition of continuous maps is continuous}{property W defn}{extending Hilbert homeomorphisms theorem}
\begin{theorem}[Property \textbf{W} for \(C({[0, 1]}^\N)\), \Assumed{224}]\label{Hilbert property W lemma}
The topological semigroup \(C([0, 1]^\N)\) has property \textbf{W} with respect to \(\Aut([0, 1]^\N)\).
\end{theorem}
\begin{proof}
Let \(d\) be the metric on \([0, 1]^\N\) given in Theorem~\ref{extending Hilbert homeomorphisms theorem}, and let \(d_\infty\) be the metric on \(C([0, 1]^\N)\) given in Theorem~\ref{compact-open is Polish thm}.
It suffices to show that for all \(s\in C([0, 1]^\N)\), there are \(f_s, g_s\in C([0, 1]^\N)\) and \(h_s\in \Aut([0, 1]^\N)\) such that \(f_sh_s g_s= s\) and for all \(N\in \nbhd{\Aut([0 , 1]^\N)}{h_s}\), we have \(f_s N g_s\in \nbhd{C([0, 1]^\N)}{s}\).

Let \(s\in C([0, 1]^\N)\) be arbitrary.
Let \(f', g':[0, 1] \to [0, 1]\) be defined by
\[(x)f' = \frac{x}{2} + \frac{1}{4},\quad \text{ and }\quad(x)g' = \left\{\begin{array}{cc}
  (x)f'^{-1}   & \text{ if }x\in [\frac{1}{4}, \frac{3}{4}] \\
1     & \text{ if }x\in [\frac{3}{4}, 1]\\
0 & \text{ if }x\in [0, \frac{1}{4}]
\end{array}\right\}.\]
We define \(f_s, g_s\in C([0, 1]^\N)\) by \(f_s:= \langle (\pi_{\lfloor i/2 \rfloor}f')_{i\in \N}\rangle_{[0, 1]^\N}\), \(g_s:= \langle (\pi_{2i}g')_{i\in \N}\rangle_{[0, 1]^\N}\).
Note that \(f_s g_s= \langle (\pi_{i}f' g')_{i\in \N}\rangle_{[0, 1]^\N}=\langle (\pi_{i})_{i\in \N}\rangle_{[0, 1]^\N}\) is the identity map.

For each \(i\in 2\N\), we define \(h_{i}':\im(f_s) \to [\frac{1}{4}, \frac{3}{4}]\) by \((x)h_{i}'  = (x)g_s s\pi_{i/2} f'\).
For each \(i\in 2\N +1\), we define \(h_i':=\pi_i\).
We then define \(h_s': \im(f_s) \to [\frac{1}{4}, \frac{3}{4}]^\N\) by \(h_s':= \langle (h_{i}')_{i\in \N}\rangle_{[0, 1]^\N}\).

Note that for each \(i\in \N\) and \(x\in \im(f_s)\), we have \((x)h_s'\pi_{2i+1}=(x)\pi_{2i+1}= (x)\pi_{2i}\), so \(h_s\) is injective.
Moreover the sets \(\im(f_s), \im(h_s')\) are both contained in \([\frac{1}{4}, \frac{3}{4}]^\N\) and as they are continuous images of compact spaces, they are both compact.
Thus \(h_s'\) is a homeomorphism from \(\im(f_s)\) to \(\im(h_s)\) (Remark~\ref{compactness facts}) and by Theorem~\ref{extending Hilbert homeomorphisms theorem}, we can choose \(h_s\in \Aut([0, 1]^\N)\) with \(h_s'\subseteq h_s\).
We now have
\begin{align*}
    f_sh_s g_s&= f_s\langle (h_{i}')_{i\in \N}\rangle_{[0, 1]^\N} \langle (\pi_{2i}g')_{i\in \N}\rangle_{[0, 1]^\N}= f_s\langle (h_{2i}'g')_{i\in \N}\rangle_{[0, 1]^\N}\\
    &= f_s g_s\langle (s\pi_{i}f' g')_{i\in \N}\rangle_{[0, 1]^\N}= \langle (s\pi_{i})_{i\in \N}\rangle_{[0, 1]^\N}=s.
\end{align*}

Let \(N\in \nbhd{\Aut([0, 1]^\N)}{h_s}\) be arbitrary.
It suffices to show that \(f_s N g_s\in \nbhd{C([0, 1]^\N)}{s}\).
Let \(\varepsilon > 0\) be such that for all \(h\in \Aut([0, 1]^\N)\) with \((h_s, h)d_\infty < \varepsilon\) we have \(h\in N\), and let \(k\in C([0, 1]^\N)\) be arbitrary with \((k, s)d_{\infty} < \varepsilon\).
It suffices to show that \(k\in f_s N g_s\).

For each \(i\in 2\N\), we define \(h_{i, k}':\im(f_s) \to [\frac{1}{4}, \frac{3}{4}]\) by \((x)h_{i, k}'  = (x)g_s k\pi_{i/2} f'\).
For each \(i\in 2\N +1\), we define \(h_{i, k}':=\pi_i\).
We then define \(h_k': \im(f_s) \to [\frac{1}{4}, \frac{3}{4}]^\N\) by \(h_k':= \langle (h_{i, k}')_{i\in \N}\rangle_{[0, 1]^\N}\).
As was the case with \(h_s'\), the sets \(\im(f), \im(h_{k}')\) are both compact subsets of \([\frac{1}{4}, \frac{3}{4}]^\N\) and \(h_k'\) is a homeomorphism between these sets.

Note that \(h_s'^{-1}h_k'= \makeset{((y)h_s',(y)h_k')}{\(y\in \im(f)\)}\).
For all \(y\in \im(f)\), we have
\begin{align*}
     ((y)h_{s}', (y)h_{k}')d&=\sum_{n\in \N} \frac{|(y)h_{s}'\pi_n-(y)h_{k}'\pi_n|}{2^{n+1}}=\sum_{n\in \N} \frac{|(y)h_{n}'-(y)h_{n,k}'|}{2^{n+1}}\\
    &=\sum_{n\in 2\N} \frac{|(y)h_{n}'-(y)h_{n,k}'|}{2^{n+1}}=\sum_{n\in 2\N} \frac{|(y)g_s s\pi_{n/2}f'-(y)g_s k\pi_{n/2}f'|}{2^{n+1}}\\
    &\leq\sum_{n\in \N} \frac{|(y)g_s s\pi_{n}f'-(y)g_s k\pi_{n}f'|}{2^{n+1}}
 \leq\sum_{n\in \N} \frac{|(y)g_s s\pi_{n}-(y)g_s k\pi_{n}|}{2^{n+1}}\\
 &= ((y)g_s s, (y)g_s k)d\leq (s, k)d_\infty
\end{align*}
Thus \(\sup_{(x, y)\in h_s'^{-1}h_k'}(x, y)d < \varepsilon\).
By Theorem~\ref{extending Hilbert homeomorphisms theorem}, we can choose \(t\in \Aut([0, 1]^\N)\) such that we have \(\sup_{(x, y)\in t}(x, y)d < \varepsilon\) and \({h_s'}^{-1}h_k'\subseteq t\).
Let \(h_k:= h_st\). Note that
\begin{align*}
    f_sh_kg_s&=f_sh_s' t g_s=f_sh_k' g_s= f_s\langle (h_{i, k}')_{i\in \N}\rangle_{[0, 1]^\N} \langle (\pi_{2i}g')_{i\in \N}\rangle_{[0, 1]^\N}\\
    &= f_s\langle (h_{2i, k}'g')_{i\in \N}\rangle_{[0, 1]^\N}= f_s g_s\langle (k\pi_{i}f' g')_{i\in \N}\rangle_{[0, 1]^\N}= \langle (k\pi_{i})_{i\in \N}\rangle_{[0, 1]^\N}=k.
\end{align*}
It therefore suffices to show that \(h_k\in N\).
By the definition of \(\varepsilon\), it therefore suffices to show that \((h_k, h_s)d_\infty < \varepsilon\).
From the choice of \(t\), we have
\[(h_k, h_s)d_\infty= (h_st, h_s)d_\infty= \sup_{x\in [0, 1]^\N} (((x)h_s)t, (x)h_s)d=\sup_{x\in [0, 1]^\N} ((x)t, x)d< \varepsilon\]
so the result follows.
\end{proof}

\speeddictseven{225}{hilber cube main theorem}{comparable Polish group topologies theorem}{ac topology defn}{property W is nice lemma}{Containing the compact-open topology}{continuous Hilbert cube action lemma}{Hilbert property W lemma}{compact-open group is Polish prop}
\begin{theorem}[The Polish topologies of \(C({[0, 1]^\N})\), \Assumed{225}]\label{hilber cube main theorem}
The compact-open topology is the only Polish topology compatible with the semigroup \(C([0, 1]^\N)\).
\end{theorem}
\begin{proof}
Let \(\mathcal{T}\) be a Polish topology compatible with \(C([0, 1]^\N)\), and let \(\mathcal{T}_{co}\) be the compact-open topology. As \(\mathcal{T}\) is Polish, it follows from Lemmas~\ref{continuous Hilbert cube action lemma} and \ref{Containing the compact-open topology} that \(\mathcal{T}\supseteq \mathcal{T}_{co}\).

As \(\mathcal{T}\supseteq \mathcal{T}_{co}\) and both of them are Polish, it follows from the argument given in Proposition~\ref{compact-open group is Polish prop} that both of \(\mathcal{T}\restriction_{\Aut([0, 1]^\N)}\) and \(\mathcal{T}_{co}\restriction_{\Aut([0, 1]^\N)}\) are Polish.

As \(\mathcal{T}\restriction_{\Aut([0, 1]^\N)}\supseteq \mathcal{T}_{co}\restriction_{\Aut([0, 1]^\N)}\), it follows from Theorem~\ref{comparable Polish group topologies theorem} that \(\mathcal{T}\restriction_{\Aut([0, 1]^\N)}= \mathcal{T}_{co}\restriction_{\Aut([0, 1]^\N)}\) (using the identity isomorphism).
Again using the identity isomorphism, it follows from Lemma~\ref{property W is nice lemma} that if \(C([0, 1]^\N)\) has property \textbf{W} with respect to \(\Aut([0, 1]^\N)\) then \(\mathcal{T}= \mathcal{T}_{co}\).
This was shown in Lemma~\ref{Hilbert property W lemma}.
\end{proof}

\speeddictfour{244}{Hilbert clone cor}{Composition of continuous maps is continuous}{topological abstract clones defn}{topological clones from topological semigroups}{hilber cube main theorem}
\begin{corollary}[The topologies of the Hilbert clone, \Assumed{244}]\label{Hilbert clone cor}
Let \(\mathcal{C}\) be the subcategory of topological spaces and continuous functions with the object set
\(\mathcal{O}_\mathcal{C}:= \makeset{([0, 1]^\N)^i}{\(i\in \N\)}\), and the morphism set \(\mathcal{M}_\mathcal{C}\) consisting of the all continuous maps between these objects. Let \(\mathcal{P}_\mathcal{O}\) be the discrete topology on \(\mathcal{O}_{\mathcal{C}}\) and let \(\mathcal{P}_\mathcal{M}\) be the topology on \(\mathcal{M}_\mathcal{C}\) generated by the sets of the form
\[A_{i, j} := C(([0, 1]^\N)^i, ([0, 1]^\N)^j)\quad \text{ and } \quad B(V, U):= \makeset{f\in \mathcal{M}_{\mathcal{C}}}{\((V)f\subseteq U\)}\]
for all \(i, j\in \N\), compact \(V\subseteq ([0, 1]^\N)^i\), and open \(U\subseteq ([0, 1]^\N)^j\).

The triple \((\mathcal{C}, \mathcal{P}_{\mathcal{O}}, \mathcal{P}_{M})\) is the a Polish abstract clone, and these are the only Polish topologies compatible with \(\mathcal{C}\).
\end{corollary}

\begin{proof}
We first show that \((\mathcal{C}, \mathcal{P}_{\mathcal{O}}, \mathcal{P}_{M})\) is a topological abstract clone.
We need to verify the seven conditions in Definition~\ref{topological abstract clones defn}.
The first five conditions are immediate from the definition. 
We next show that \(\circ_{\mathcal{C}}\) is continuous.

As the sets \((A_{n, i}\times A_{i, m})_{i, n, m\in \N}\) are a partition of \(\operatorname{Comp}_{\mathcal{C}}\) into open sets, it suffices to show that for all \(n, m, i\in \N\), the map
\(\circ_{\mathcal{C}}\restriction_{A_{n, i}\times A_{i, m}}\) is continuous.
For all \(j,k \in \N\),  the set \((A_{j, k})\circ_{\mathcal{C}}\restriction_{A_{n, i}\times A_{i, m}}^{-1}\) is either empty or \(A_{n, i}\times A_{i, m}\) and is thus open.
Moreover if \(V\subseteq ([0, 1]^\N)^j\) is compact and \(U\subseteq ([0, 1]^\N)^k\) is open, then \((B(V, U))\circ_{\mathcal{C}}\restriction_{A_{n, i}\times A_{i, m}}^{-1}\) is open by Corollary~\ref{Composition of continuous maps is continuous}.

It remains to show that for all \(i, j\in \N\), the bijection \(\phi_{i, j}:A_{i, 1}^j \to A_{i, j}\) from Definition~\ref{topological abstract clones defn} is continuous.
From the definition of the product topology on \(([0, 1]^\N)^j\), the open sets \(U\subseteq ([0, 1]^\N)^j\) such that \(U =\prod_{k\in \{0, 1, \ldots, j-1\}} (U)\pi_k \), form a basis for \(([0, 1]^\N)^j\).
Thus from Theorem~\ref{compact-open is Polish thm}, the sets \(B(V, U)\) where \(V\subseteq ([0, 1]^\N)^i\) is compact and \(U\subseteq ([0, 1]^\N)^j\) is open with \(U =\prod_{k\in \{0, 1, \ldots, j-1\}} (U)\pi_k \), generate the topology on \(A_{i, j}\).

If \(V\subseteq ([0, 1]^\N)^i\) is compact, \(U\subseteq ([0, 1]^\N)^j\) is open and \(U= \prod_{k\in \{0, 1, \ldots, j-1\}} (U)\pi_k\), then
\[(B(V, U))\phi_{i, j}^{-1}= \prod_{k\in \{0, 1, \ldots, j-1\}} B(V, (U)\pi_k)\]
so \(\phi_{i, j}\) is indeed continuous as required.

Thus \((\mathcal{C}, \mathcal{P}_{\mathcal{O}}, \mathcal{P}_{M})\) is a topological abstract clone.
The result now follows from Theorem~\ref{hilber cube main theorem} together with Proposition~\ref{topological clones from topological semigroups}.
\end{proof}

\subsection{The Cantor space}
In this subsection we are concerned with continuous transformations of the Cantor space \(2^\N\) (see Theorem~\ref{Cantor space main theorem} and Corollary~\ref{Cantor clone cor}).
The arguments follow a similar path to those in the Hilbert cube subsection but we can now make use of Theorem~\ref{automatic continuity examples thm}.

We know of no Cantor space analogue of Theorem~\ref{extending Hilbert homeomorphisms theorem} in the literature so we provide our own (see Lemma~\ref{extending Cantor homeomorphisms theorem}).
The proof of Lemma~\ref{extending Cantor homeomorphisms theorem} is based on the proof of 
Theorem 5.2.4 in~\cite{Mill2001aa} but is significantly shorter due to the fact that \(2^\N\) can be naturally viewed as a topological group.

\speeddictfour{226}{extending Cantor homeomorphisms theorem}{Products in Categories Defn}{compactness facts}{topological structures defn}{product structures}
\begin{lemma}[Extending \(2^\N\) homeomorphisms, \Assumed{226}]\label{extending Cantor homeomorphisms theorem}
Let \(d: 2^\N \times 2^\N \to \mathbb{R}\) be defined by
\[(x, y)d= \inf \makeset{\frac{1}{2^n}}{\(n\in \N,\ x\restriction_{n}= y\restriction_n\)}.\]
Let \(S := \makeset{x\in 2^\N}{\((x)\pi_i= 0\) for all \(i\in 2\N+1\)}\).
If \(\varepsilon>0\), \(C_0, C_1\subseteq S\) are compact and \(h':C_0 \to C_1\) is a homeomorphism with \(\sup_{(x, y)\in h'}(x, y)d < \varepsilon\), then there is \(h\in \Aut(2^\N)\) such that \(h'\subseteq h\) and \(\sup_{(x, y)\in h}(x, y)d < \varepsilon\).
\end{lemma}
\begin{proof}
We view \{0,1\} as a discrete topological group with operation addition modulo \(2\), and we view \(2^\N\) as a power of this group (we'll denote the operation with \(+\)). 
Let \(\varepsilon>0\), \(C_0, C_1\subseteq S\) be closed and \(h':C_0 \to C_1\) be a homeomorphism with \(\sup_{(x, y)\in h'}(x, y)d < \varepsilon\).
Let \(n\in 2\N\) be such that \(\frac{1}{2^n} < \varepsilon\), and define \(\psi_0: 2^\N \to \{0, 1\}^{n/2}\), \(\psi_1, \psi_2:2^\N \to 2^\N\) by
\[(x)\psi_0=  \langle \pi_1, \pi_3,  \ldots, \pi_{n-1} \rangle_{\{0, 1\}^{n/2}},\quad  (x)\psi_1=  \langle (\pi_{2i})_{i\in \N}\rangle_{2^\N},\]
\[(x)\psi_2=  \langle (\pi_{n+ 2i+ 1})_{i\in \N}\rangle_{2^\N}.\]
We then define \(\psi:= \langle \psi_0, \psi_1, \psi_2 \rangle_{\{0, 1\}^{n/2} \times 2^\N \times 2^\N}\), \(C_0':= C_0\psi_1\) and \(C_1':=C_1\psi_1\) (note that \(\psi\) is a homeomorphism).
As \(C_0, C_1\subseteq S\), the map \(\psi_1\) is injective on these sets, thus \(h'':=\psi_1^{-1}h' \psi_1\)
is a homeomorphism from \(C_0'\) to \(C_1'\).

Let \(c:\mathcal{P}(2^\N)\backslash \{\varnothing\} \to 2^\N \) be such that \((A)c\in A\) for all \(A\in \dom(c)\).
For each \(i\in \{0, 1\}\), let \(t_i: 2^\N \to C_i'\) be defined by
\[(x)t_i= \left(\makeset{y\in C_i'}{\((x, y)d\leq (x, z)d\) for all \(z\in C_i'\)}\right)c\]
(this is well defined because \(C_i\) is compact).
Note that for each \(i\in \{0, 1\}\), \(t_i\restriction_{C_i}\) is the identity function, \(\im(t_i)= C_i\) and for all \(x, y\in 2^\N\) we have \(((x)t_i, (y)t_i)d\leq (x, y)d\). Let \(\phi_0, \phi_1, \phi_2:\{0, 1\}^{n/2}\times \{0,1\}^\N \times \{0,1\}^\N \to \{0, 1\}^{n/2}\times \{0,1\}^\N \times \{0,1\}^\N\) be defined by
\[(x, y, z)\phi_0= (x,y,  y+z), \quad (x, y, z)\phi_1= (x, y-(z)t_0+(z)t_0h'', z),\]
\[(x, y, z)\phi_2= (x, y, z- (y)t_1{h''}^{-1}).\]
It is routine to verify that each of these maps is a homeomorphism.

Finally, let \(h:=  \psi \phi_0\phi_1\phi_2 \psi^{-1}\in \Aut(2^\N).\)
We now show that \(h\) has the required properties. If \(x\in C_0\), then
\begin{align*}
    (x)h&= ((x)\psi_0, (x)\psi_1, 000\ldots)\phi_0\phi_1\phi_2\psi^{-1}\\
&= ((x)\psi_0, (x)\psi_1, (x)\psi_1)\phi_1\phi_2\psi^{-1}\\
&= ((x)\psi_0, (x)\psi_1- (x)\psi_1t_0+ (x)\psi_1t_0h'',  (x)\psi_1)\phi_2\psi^{-1}
\end{align*}
As \((x)\psi_1\in C_0'\), it follows that
\begin{align*}
    (x)h&=  ((x)\psi_0, (x)\psi_1- (x)\psi_1t_0+ (x)\psi_1t_0h'',  (x)\psi_1)\phi_2\psi^{-1}\\
    &=  ((x)\psi_0, (x)\psi_1- (x)\psi_1+ (x)\psi_1h'',  (x)\psi_1)\phi_2\psi^{-1}\\
    &=  ((x)\psi_0, (x)\psi_1h'',  (x)\psi_1)\phi_2\psi^{-1}\\
    &=  ((x)\psi_0, (x)\psi_1h'',  (x)\psi_1-(x)\psi_1)\psi^{-1}\\
    &=  ((x)\psi_0, (x)\psi_1h'', 000\ldots)\psi^{-1}= (x)h'.
\end{align*}
It remains to show that if \(x\in 2^\N\) is arbitrary, then \((x, (x)h)d< \varepsilon\). Let \(n_\varepsilon \in \N\) be such that if \(f, g\in C(2^\N)\) are arbitrary then \((f, g)d< \varepsilon \) if and only if \(f, g\) share a prefix of length \(n_\varepsilon\). Let \(i< n_\varepsilon\) be arbitrary. We show that \((x)h\pi_i = (x)\pi_i\).

By the definition of \(h\), there is some \(z'\in C_0'\) such that
\[((x)h\psi_0, (x)h\psi_1)= ((x)\psi_0, (x)\psi_1 - z'+(z')h'').\]
Let \(z\in C_0\) be such that \(z'= (z)\psi_1\). Note that \(n_\varepsilon \leq n\). Thus if \(i\) is odd, then 
\[(x)\pi_i= (x)\psi_0\pi_{(i-1)/2}= (x)h\psi_0\pi_{(i-1)/2}= (x)h\pi_i.\]
By assumption \((z, (z)h')d < \varepsilon\). So if \(i<n_\varepsilon\) is even, then
\[(z')\pi_{i/2}=((z)\psi_1)\pi_{i/2}=(z)\pi_i = (z)h'\pi_i=(z)h'\psi_1\pi_{i/2}=(z)\psi_1h''\pi_{i/2}=(z')h''\pi_{i/2}.\]
So \((z')\pi_{i/2}-(z')h''\pi_{i/2}= 0\) and hence
\[(x)\pi_i= (x)\psi_1 \pi_{i/2}= (x)\psi_1 \pi_{i/2}+(z')\pi_{i/2}-(z')h''\pi_{i/2}= (x)h\psi_1\pi_{i/2}= (x)h\pi_{i}.\]
\end{proof}

The following proof in very similar to the proof of Lemma~\ref{Hilbert property W lemma}, but adjusted to make use of Lemma~\ref{extending Cantor homeomorphisms theorem} instead of Theorem~\ref{extending Hilbert homeomorphisms theorem}.
\speeddictsix{227}{Cantor property W lemma}{Numbers defns}{Products in Categories Defn}{product topology defn}{Composition of continuous maps is continuous}{property W defn}{extending Cantor homeomorphisms theorem}
\begin{lemma}[Property \textbf{W} for \(C(2^\N)\), \Assumed{227}]\label{Cantor property W lemma}
The topological semigroup \(C(2^\N)\) with the compact-open topology has property \textbf{W} with respect to \(\Aut(2^\N)\).
\end{lemma}
\begin{proof}
Let \(d\) be the metric on \(2^\N\) given in Theorem~\ref{extending Cantor homeomorphisms theorem}, and let \(d_\infty\) be the metric on \(C(2^\N)\) given in Theorem~\ref{compact-open is Polish thm}.
It suffices to show that for all \(s\in C(2^\N)\), there are \(f_s, g_s\in C(2^\N)\) and \(h_s\in \Aut(2^\N)\) such that \(f_sh_s g_s= s\) and for all \(N\in \nbhd{\Aut(\{0 , 1\}^\N)}{h_s}\), we have \(f_s N g_s\in \nbhd{C(2^\N)}{s}\).

Let \(s\in C(2^\N)\) be arbitrary. For each \(i\in \N\), we define \(f_i: 2^\N \to \{0,1\}\) by
\[(x)f_i= \left\{\begin{array}{lr}
0     &  \text{ if }i\in 2\N + 1\\
(x)\pi_{i/4}& \text{ if }i\in 4\N\\
(x)\pi_{(i-2)/4}& \text{ if }i\in 4\N+2
\end{array}\right\}.\]
We then define \(f_s, g_s\in C(2^\N)\) by 
\[f_s:= \langle (f_i)_{i\in \N}\rangle_{2^\N}\quad\text{ and} \quad g_s:= \langle (\pi_{4i})_{i\in \N}\rangle_{2^\N}.\]
Note that \(f_s g_s=\langle (f_{4i})_{i\in \N}\rangle_{2^\N} =\langle (\pi_{i})_{i\in \N}\rangle_{2^\N}\) is the identity map.
Moreover, \(\im(f_s)\subseteq S\) (where \(S\) is as in Theorem~\ref{extending Cantor homeomorphisms theorem}).

For each \(i\in 4\N\), we define \(h_{i}':\im(f_s) \to \{0,1\}\) by \((x)h_{i}'  = (x)g_s s\pi_{i/4}\).
For each \(i\in \N\backslash 4\N\), we define \(h_i':=\pi_i\).
We then define \(h_s': \im(f_s) \to S\) by \(h_s':= \langle (h_{i}')_{i\in \N}\rangle_{2^\N}\).

Note that for each \(i\in \N\) and \(x\in \im(f_s)\), we have \((x)h_s'\pi_{4i+2}=(x)\pi_{4i+2}= (x)\pi_{4i}\), so \(h_s\) is injective.
Moreover the sets \(\im(f_s), \im(h_s')\) are both contained in \(S\) and they are both compact (see Remark~\ref{compactness facts}).
Thus \(h_s'\) is a homeomorphism from \(\im(f_s)\) to \(\im(h_s)\) and by Theorem~\ref{extending Cantor homeomorphisms theorem}, we can choose \(h_s\in \Aut(2^\N)\) with \(h_s'\subseteq h_s\).
We now have
\begin{align*}
    f_sh_s g_s&= f_s\langle (h_{i}')_{i\in \N}\rangle_{2^\N} \langle (\pi_{4i})_{i\in \N}\rangle_{2^\N}= f_s\langle (h_{4i}')_{i\in \N}\rangle_{2^\N}\\
    &= f_s g_s\langle (s\pi_{i})_{i\in \N}\rangle_{2^\N}= \langle (s\pi_{i})_{i\in \N}\rangle_{2^\N}=s.
\end{align*}

Let \(N\in \nbhd{\Aut(2^\N)}{h_s}\) be arbitrary.
It suffices to show that \(f_s N g_s\in \nbhd{C(2^\N)}{s}\).
Let \(\varepsilon > 0\) be such that for all \(h\in \Aut(2^\N)\) with \((h_s, h)d_\infty < \varepsilon\) we have \(h\in N\), and let \(k\in C(2^\N)\) be arbitrary with \((k, s)d_{\infty} < \varepsilon\).
It suffices to show that \(k\in f_s N g_s\).

For each \(i\in 4\N\), we define \(h_{i, k}':\im(f_s) \to \{0, 1\}\) by \((x)h_{i, k}'  = (x)g_s k\pi_{i/4} \).
For each \(i\in \N\backslash4\N\), we define \(h_{i, k}':=\pi_i\).
We then define \(h_k': \im(f_s) \to 2^\N\) by \(h_k':= \langle (h_{i, k}')_{i\in \N}\rangle_{2^\N}\).
As was the case with \(h_s'\), the sets \(\im(f_s), \im(h_{k}')\) are both compact subsets of \(S\) and \(h_k'\) is a homeomorphism between these sets.

Note that \({h_s'}^{-1}h_k'= \makeset{((y)h_{s}', (y)h_{k}')}{\(y\in \im(f)\)}\).
For all \(y\in \im(f)\) with \((y)h_{s}'\neq (y)h_{k}'\), we have
\begin{align*}
    ((y)h_{s}', (y)h_{k}')d&=\inf{\left(\makeset{\frac{1}{2^n}}{\(n\in \N,\ (y)h_{s}'\restriction_{n}= (y)h_{k}'\restriction_{n}\)}\right)}\\
    &=\max{\left(\makeset{\frac{1}{2^n}}{\(n\in \N,\ (y)h_{s}'\pi_n\neq (y)h_{k}'\pi_n\))}\right)}\\
     &= \max\left(\makeset{\frac{1}{2^n}}{\(n \in 4\N, \ (y)g_s s\pi_{n/4}\neq (y)g_s k\pi_{n/4}\)}\right)\\
    &= \max\left(\makeset{\frac{1}{2^{4n}}\in \N}{\(n\in \N, (y)g_s s\pi_{n}\neq (y)g_s k\pi_{n}\)}\right)\\
        &< \max\left(\makeset{\frac{1}{2^{n}}\in \N}{\(n\in \N, (y)g_s s\pi_{n}\neq (y)g_s k\pi_{n}\)}\right)\\
 &= ((y)g_s s, (y)g_s k)d\leq (s, k)d_\infty<\varepsilon.
\end{align*}
Thus \(\sup\limits_{\substack{(x, y)\in h_s'^{-1}h_k'}}(x, y)d < \varepsilon\).
By Theorem~\ref{extending Cantor homeomorphisms theorem}, we can choose \(t\in \Aut(2^\N)\) such that we have \(\sup_{(x, y)\in t}(x, y)d < \varepsilon\) and \({h_s'}^{-1}h_k'\subseteq t\).
Let \(h_k:= h_st\). Note that
\begin{align*}
    f_sh_kg_s&=f_sh_s' t g_s =f_sh_k' g_s=f_s\langle (h_{i, k}')_{i\in \N}\rangle_{2^\N} \langle (\pi_{4i})_{i\in \N}\rangle_{2^\N}= f_s\langle (h_{4i, k}')_{i\in \N}\rangle_{2^\N}\\
    &= f_s g_s\langle (k\pi_{i})_{i\in \N}\rangle_{2^\N}= \langle (k\pi_{i})_{i\in \N}\rangle_{2^\N}=k.
\end{align*}
It therefore suffices to show that \(h_k\in N\).
By the definition of \(\varepsilon\), it therefore suffices to show that \((h_k, h_s)d_\infty < \varepsilon\).
From the choice of \(t\), we have
\[(h_k, h_s)d_\infty= (h_st, h_s)d_\infty= \sup_{x\in 2^\N} (((x)h_s)t, (x)h_s)d=\sup_{x\in 2^\N} ((x)t, x)d< \varepsilon\]
so the result follows.
\end{proof}

\speeddictsix{229}{Cantor has SCC lemma}{Composition of continuous maps is continuous}{compact-open group is Polish prop}{automatic continuity examples thm}{ac topology is nice}{property W is nice lemma}{Cantor property W lemma}
\begin{lemma}[Cantor continuity, \Assumed{229}]\label{Cantor has SCC lemma}
The topologies \(\mathcal{SCC}(\Aut(2^\N))\) and \(\mathcal{SCC}(C(2^\N))\) are the compact-open topologies on each of these sets.
Moreover these topologies are Polish
\end{lemma}
\begin{proof}
Let \(\mathcal{T}_{co}\) denote the compact-open topology on \(C(2^\N)\).
From Theorem~\ref{compact-open is Polish thm}, this topology is Polish and compatible with \(C(2^\N)\). As \(\mathcal{T}_{co}\) is second countable, it follows that 
\(\mathcal{SCC}(C(2^\N))\supseteq \mathcal{T}_{co}\) and \(\mathcal{SCC}(\Aut(2^\N))\supseteq \mathcal{T}_{co}\restriction_{\Aut(2^\N)}\).

By Theorem~\ref{automatic continuity examples thm} every homomorphism from \((\Aut(2^\N), \mathcal{T}_{co})\) to a second countable group (and hence inverse semigroup) is continuous.
From Proposition~\ref{ac topology is nice}, it follows that every homomorphism from the semigroup obtained by removing the unary operation of  \(\Aut(2^\N)\) to a second countable semigroup is continuous.

By Lemma~\ref{Cantor property W lemma}, \(C(2^\N)\) has property \textbf{W} with respect to \(\Aut(2^\N)\).
Thus from Lemma~\ref{property W is nice lemma}, all homomorphisms from \((C(2^\N), \mathcal{T}_{co})\) to second countable topological semigroups are continuous. So from Proposition~\ref{ac topology is nice} we also have \[\mathcal{SCC}(C(2^\N))\subseteq \mathcal{T}_{co}\quad\text{ and }\quad \mathcal{SCC}(\Aut(2^\N))\subseteq \mathcal{T}_{co}\restriction_{\Aut(2^\N)}.\]
\end{proof}

The following lemma is the Cantor space analogue or Lemma~\ref{continuous Hilbert cube action lemma}.
However in this case we can use a different proof with weaker assumptions due to the existence Lemma~\ref{Cantor has SCC lemma}.
\speeddictone{228}{continuous Cantor action lemma}{Cantor has SCC lemma}
\begin{lemma}[Continuous action on \(2^\N\), \Assumed{228}]\label{continuous Cantor action lemma}
Suppose that \(\mathcal{T}\) is a second countable Hausdorff topology compatible with the semigroup \(C(2^\N)\). In this case the map \(a:2^\N \times (C(2^\N), \mathcal{T}) \to 2^\N\) defined by \((x, f)a= (x)f\) is continuous.
\end{lemma}
\begin{proof}
Let \(C\) denote the set of constant maps from \(2^\N\) to itself and note that \(\mathcal{T}\restriction_{C}\) is second countable and Hausdorff. Let \(I:C\to 2^\N\) be the bijection sending a constant map to the unique point in its image.
Note that \(a = \langle \pi_0I^{-1}, \pi_1 \rangle_{C \times C(2^\N)} \circ *^{C(2^\N)} \circ I.\)
Thus to show that \(a\) is continuous, it suffices to show that \(I:C\to 2^\N\) is a homeomorphism. Let \(\mathcal{C}\) denote the topology \[\makeset{U\subseteq C}{\((U)I\) is open in \(2^\N\)}\]
on \(C\).
We will show that \(\mathcal{T}\restriction_C= \mathcal{C}\).
If \(U\subseteq 2^\N\) is open, then \((U)I^{-1} =\makeset{f\in C}{\((2^\N)f\subseteq U\)}\). So \(\mathcal{C}= \mathcal{T}_{co}\restriction_{C}\).
From Lemma~\ref{Cantor has SCC lemma}, it follows that
\[\mathcal{T}\restriction_C\subseteq \mathcal{SCC}(C(2^\N))\restriction_C =\mathcal{T}_{co}\restriction_C=\mathcal{C}.\]

Note that \(\mathcal{T}\restriction_C\subseteq  \mathcal{C}\), \(\mathcal{T}\restriction_C\) is Hausdorff and \(\mathcal{C}\) is compact.
Thus from Remark~\ref{compactness facts} we have \(\mathcal{T}\restriction_C= \mathcal{C}\) as required.
\end{proof}

\speeddictthree{230}{Cantor space main theorem}{Containing the compact-open topology}{Cantor has SCC lemma}{continuous Cantor action lemma}
\begin{theorem}[The topologies of \(C(2^\N)\), \Assumed{230}]\label{Cantor space main theorem}
The compact-open topology is the only second countable Hausdorff topology compatible with \(C(2^\N)\).
Moreover \(\mathcal{SCC}(C(2^\N))\) is the compact-open topology, and all homomorphisms from \(C(2^\N)\) to second countable topological semigroups are continuous.
\end{theorem}
\begin{proof}
Let \(\mathcal{T}\) be an arbitrary second countable Hausdorff topology compatible with \(C(2^\N)\), and let \(\mathcal{T}_{co}\) be the compact-open topology.
The equality \(\mathcal{SCC}(C(2^\N))=\mathcal{T}_{co}\) follows from Lemma~\ref{Cantor has SCC lemma}.
It remains to show that \(\mathcal{T}=\mathcal{T}_{co}\). By the definition of \(\mathcal{SCC}(C(2^\N))\), it follows that \(\mathcal{T}\subseteq \mathcal{SCC}(C(2^\N))=\mathcal{T}_{co}\).
Moreover, from Lemmas~\ref{continuous Cantor action lemma} and \ref{Containing the compact-open topology}, we also have \(\mathcal{T}\supseteq \mathcal{T}_{co}\). The result follows.
\end{proof}

\speeddictfive{243}{Cantor clone cor}{Composition of continuous maps is continuous}{topological abstract clones defn}{topological clones from topological semigroups}{Hilbert clone cor}{Cantor space main theorem}
\begin{corollary}[The topologies of the Cantor clone, \Assumed{243}]\label{Cantor clone cor}
Let \(\mathcal{C}\) be the subcategory of topological spaces and continuous functions with the object set
\(\mathcal{O}_\mathcal{C}:= \makeset{(2^\N)^i}{\(i\in \N\)}\), and the morphism set \(\mathcal{M}_\mathcal{C}\) consisting of the all continuous maps between these objects. Let \(\mathcal{P}_\mathcal{O}\) be the discrete topology on \(\mathcal{O}_{\mathcal{C}}\) and let \(\mathcal{P}_\mathcal{M}\) be the topology on \(\mathcal{M}_\mathcal{C}\) generated by the sets of the form
\[A_{i, j} := C((2^\N)^i, (2^\N)^j)\quad \text{ and } \quad B(V, U):= \makeset{f\in \mathcal{M}_{\mathcal{C}}}{\((V)f\subseteq U\)}\]
for all \(i, j\in \N\), compact \(V\subseteq (2^\N)^i\), and open \(U\subseteq (2^\N)^j\).

The triple \((\mathcal{C}, \mathcal{P}_{\mathcal{O}}, \mathcal{P}_{M})\) is a Polish topological abstract clone.
Moreover if \((\mathcal{T}_{\mathcal{O}}, \mathcal{T}_{\mathcal{M}})\) are compatible with the clone \(\mathcal{C}\) then
\begin{enumerate}
    \item If \(\mathcal{T}_{\mathcal{M}}\) is Hausdorff and second countable then \(\mathcal{T}_{\mathcal{M}}= \mathcal{P}_{\mathcal{M}}\).
    \item If \(\mathcal{T}_{\mathcal{M}}\) is second countable then \(\mathcal{T}_{\mathcal{M}}\subseteq \mathcal{P}_{\mathcal{M}}\).
\end{enumerate}
\end{corollary}
\begin{proof}
The proof of this corollary is the same as the proof of Corollary~\ref{Hilbert clone cor}, but using Theorem~\ref{Cantor space main theorem} instead of Theorem~\ref{hilber cube main theorem}. 
\end{proof}

\subsection{Boolean algebras}\label{boolean subsection}
In this subsection we discuss homomorphisms between boolean algebras.
In particular we are interested in \(B_\infty\) the boolean algebra of clopen subsets of the Cantor space.
This boolean algebra is of note as up to isomorphism it is the only countably infinite boolean algebra with no atoms (minimal non-zero elements).
It is also the Fr\"aiss\'e limit of the class of finite non-trivial boolean algebras. 
The results of this subsection are corollaries of Cantor space subsection using the Stone Duality Theorem (\ref{stone duality theorem}) and Proposition~\ref{topological clones from topological semigroups}.

\speeddictone{245}{boolean algbra signature defn}{signature defn}
\begin{defn}[Boolean algebra signature, \Assumed{245}]\label{boolean algbra signature defn}
We define the \(\sigma_B\) to be the signature
\[\sigma_B:= (\{ \wedge,\vee, \neg, \mathbf{0}, \mathbf{1}\}, \varnothing, \{ (\wedge, 2), (\vee, 2), (\neg, 1), (\mathbf{0}, 0), (\mathbf{1}, 0)\}).\]
The actual value of the symbols \(\wedge, \vee, \neg, \mathbf{0}\) and \(\mathbf{1}\) are not very important, but for completeness we'll define
\[\wedge:= (01, 14, 04), \quad\vee:= (15, 18 ),  \quad \neg:= (14, 15, 20),\]
\[ \mathbf{0}:= (26, 05, 18, 15), \quad \mathbf{1}:= (15, 14, 05).\]
\end{defn}

\speeddicttwo{246}{boolean algebra defn}{structures defn}{boolean algbra signature defn}
\begin{defn}[Boolean algebra, \Assumed{246}]\label{boolean algebra defn}
If \(\mathbb{B}\) is \(\sigma_B\)-structure and \(a, b\in \mathbb{B}\), then we will often write \(a\wedge b, a\vee b, \neg a, \mathbf{0}\) and \(\mathbf{1}\) instead of \((a, b)\wedge^\mathbb{B}, (a, b)\vee^\mathbb{B}, (a)\neg^\mathbb{B}, ()\mathbf{0}^\mathbb{B}\), and \(()\mathbf{1}^\mathbb{B}\) respectively.

We say a \(\sigma_B\)-structure \(\mathbb{B}\) is a \textit{boolean algebra} if for all \(a,b,c\in \mathbb{B}\) we have the following:
\begin{equation*}
    \begin{split}
        a\wedge (b \wedge c) &= (a\wedge b)\wedge c,\\
a\wedge b  &= b\wedge a,\\
a \wedge \mathbf{1}&= a,\\
a \vee (b \wedge c)&= (a \vee b) \wedge (a \vee c),\\
a\vee (a \wedge b) &= a,\\
a\vee \neg a &= \mathbf{1},\\
\neg(a\vee b) &=\neg a \wedge \neg b,
    \end{split}
    \quad\quad
    \begin{split}
          a\vee (b \vee c) &= (a\vee b)\vee c,\\
 a\vee b &= b\vee a,\\
 \quad a\vee \mathbf{0} &= a,\\
 a \wedge (b \vee c) &= (a \wedge b) \vee (a \wedge c),\\
 a\wedge (a \vee b) &= a,\\
 a \wedge \neg a &= \mathbf{0},\\
  \neg(a\wedge b) &=\neg a \vee \neg b.
    \end{split}
\end{equation*}
\end{defn}
It is worth noting that many of the assumptions in the above definitions can be inferred from each other, but the choice of which conditions constitute the definition varies between authors.

From Theorem~\ref{stone duality theorem}, it will follow that every boolean algebra is isomorphic to a boolean algebra of sets (with the operations of intersection, union, complement, empty set, and universe) so these are the key examples to keep in mind.

\speeddictone{248}{two element boolean algebra defn}{boolean algebra defn}
\begin{example}[The two element boolean algebra, \Assumed{248}]\label{two element boolean algebra defn}
If one considers \(0\) as representing ``false" and \(1\) as representing ``true", then \(\{0, 1\}\) is a boolean algebra when given the operations of ``and", ``or", ``not", ``false", and ``true".
More concretely if \(a, b \in \{0, 1\}\) then
\[a\wedge b := \min(a, b),\ a\vee b := \max(a, b),\ \neg a := 1-a,\ \mathbf{0}:=0,\ \text{ and }\ \mathbf{1}:=1\]
\end{example}

\speeddictthree{249}{stone spaces defn}{product topology defn}{structure hom defn}{two element boolean algebra defn}
\begin{defn}[Stone spaces, \Assumed{249}]\label{stone spaces defn}
We say that a topological space \(X\) is a \textit{Stone space} if it is zero-dimensional, Hausdorff and compact. If \(\mathbb{B}\) is a boolean algebra then let \(\operatorname{SSpace}(\mathbb{B})\) denote the set of homomorphisms from \(\mathbb{B}\) to the two element boolean algebra from Example~\ref{two element boolean algebra defn}.
We view \(\operatorname{SSpace}(\mathbb{B})\) as a topological space equipped with the pointwise topology (\(\{0,1\}\) has the discrete topology).
\end{defn}

\speeddictsix{247}{stone duality theorem}{Hausdorff defn}{compact defn}{disjoint union topology defn}{zero-dimensional defn}{dual semigroup defn}{stone spaces defn}
\begin{theorem}[cf. Theorem 34 of \cite{givant2008introduction}, Stone Duality Theorem, \Assumed{247}]\label{stone duality theorem}
If \(X\) is a Stone space, then the collection \(\Clo(X)\) of clopen subsets of \(X\) form a boolean algebra with the following operations.
If \(A, B\in \Clo(X)\) are arbitrary, then
\[A\vee B= A\cup B, \quad A\wedge B= A\cap B,  \quad \neg A= X\backslash A,\quad \mathbf{0}= \varnothing, \quad \mathbf{1}= X.\]

Conversely, if \(\mathbb{B}\) is a boolean algebra, then \(\operatorname{SSpace}(\mathbb{B})\) (recall Definition~\ref{stone spaces defn}) is the only Stone space \(X\) (up to homeomorphism) such that \(\mathbb{B}\cong \Clo(X)\). Moreover if \(X\) is a Stone space and \(\mathbb{B}\) is a boolean algebra, then
\[C(X)\cong \End(\Clo(X))^\dagger \quad \text{ and }\quad  \End(\mathbb{B})\cong C(\operatorname{SSpace}(\mathbb{B}))^\dagger \text{(recall Definition~\ref{dual semigroup defn})}.\]

Additionally if \(X\), \(Y\) are disjoint Stone spaces, then \(X\cup Y\) is a Stone space and \(\Clo(X \cup Y) \cong \Clo(X) \times \Clo(Y)\) .
\end{theorem}

\speeddictfive{250}{main boolean algebra theorem}{pull back topologies prop}{Cantor space main theorem}{stone duality theorem}{full trans def}{automorphism groups defn}
\begin{theorem}[The topologies of \(\End(B_\infty)\), \Assumed{250}]\label{main boolean algebra theorem}
Let \(B_\infty\) denote the boolean algebra \(\Clo(2^\N)\).
The pointwise topology \(\mathcal{PT}_{B_{\infty}}\restriction_{\End(B_{\infty})}\) (recall Example~\ref{full trans def}) is the only second countable Hausdorff topology compatible with \(\End(B_{\infty})\).
Moreover this topology is Polish, equal to \(\mathcal{SCC}(\End(B_{\infty}))\), and all homomorphisms from \((\End(B_\infty), \mathcal{PT}_{B_{\infty}}\restriction_{\End(B_{\infty})})\) to second countable semigorups are continuous.
\end{theorem}
\begin{proof}
From Theorem~\ref{stone duality theorem}, we know that \(\End(B_{\infty}) \cong C(2^\N)^\dagger\).
Thus from Proposition~\ref{pull back topologies prop}, there is a (bijective and topological property preserving) correspondence between the topologies compatible with \(\End(B_{\infty})\), and those compatible with \(C(2^\N)\).
Thus from Theorem~\ref{Cantor space main theorem} there is a unique second countable Hausdorff topology compatible with \(\End(B_{\infty})\), and this topology coincides with \(\mathcal{SCC}(\End(B_{\infty}))\).
This topology is also Polish from Lemma~\ref{Cantor has SCC lemma}.
As \(\mathcal{PT}_{B_{\infty}}\restriction_{\End(B_{\infty})}\) is second countable, Hausdorff and compatible with \(\End(B_{\infty})\), the result follows.
\end{proof}

\speeddictthree{251}{boolean clone cor}{topological abstract clones defn}{topological clones from topological semigroups}{main boolean algebra theorem}
\begin{corollary}[The topologies of the \(B_\infty\) polymorphism clone, \Assumed{242}]\label{boolean clone cor}
Let \(B_{\infty}\) be a boolean algebra, with universe \(\N\), which is isomorphic to \(\Clo(2^\N)\).
Let \(\mathcal{C}\) be the subcategory of boolean algebras and homomorphisms with the object set
\(\mathcal{O}_\mathcal{C}:= \makeset{B_{\infty}^i}{\(i\in \N\)}\), and the morphism set \(\mathcal{M}_\mathcal{C}\) consisting of the all homomorphisms between these objects.
We view \(\mathcal{O}_{\mathcal{C}}\) and \(\mathcal{P}_{\mathcal{C}}\) as subspaces of the topological spaces of the same names given in Theorem~\ref{full func clone cor}.
Let \(\mathcal{P}_\mathcal{O}\) and \(\mathcal{M}_\mathcal{C}\) be the topologies on these sets respectively.

The triple \((\mathcal{C}, \mathcal{P}_{\mathcal{O}}, \mathcal{P}_{M})\) is a Polish topological abstract clone. 
Moreover if \((\mathcal{T}_{\mathcal{O}}, \mathcal{T}_{\mathcal{M}})\) are compatible with the clone \(\mathcal{C}\) then
\begin{enumerate}
    \item If \(\mathcal{T}_{\mathcal{M}}\) is Hausdorff and second countable then \(\mathcal{T}_{\mathcal{M}}= \mathcal{P}_{\mathcal{M}}\).
    \item If \(\mathcal{T}_{\mathcal{M}}\) is second countable then \(\mathcal{T}_{\mathcal{M}}\subseteq \mathcal{P}_{\mathcal{M}}\).
\end{enumerate}
\end{corollary}
\begin{proof}
First note that \((\mathcal{C}, \mathcal{P}_{\mathcal{O}}, \mathcal{P}_{M})\) is a topological abstract clone as it is contained in the clone from Theorem~\ref{full func clone cor}.

Note that \(2^\N\) is isomorphic to its disjoint union with itself (\(2\times 2^\N\cong 2^\N\)).
Thus from Theorem~\ref{stone duality theorem} it follows that \(B_{\infty}\cong B_{\infty}^2\).
Thus the result follows from Proposition~\ref{topological clones from topological semigroups} together with Theorem~\ref{main boolean algebra theorem}.
\end{proof}

\part{Automorphisms of \(dV_n\)}\label{nv section}
The content of this part is a more detailed version of the Paper \cite{elliott2020description} by the same author as this document. 

The groups \(dV_n\) include both the Brin-Thompson groups \(nV\) \((=nV_2)\), and the Higman-Thompson groups \(V_n\) \((=1V_n)\).
They also overlap with the class of groups \(G_{n, r}\) \((G_{n, 1}= 1V_n)\).
The groups \(nV\) \cite{Brin2004, Brin2005, Bleak2010, hennig2011,Belk2016, quick2019, lawson2019higher}, as well as automorphisms of Thompson groups \cite{Brin_AutF, Brin_1998,bleak2016,OlukoyaAutTnr} have each been well researched in the literature.

The paper \cite{bleak2016} of C. Bleak, P. Cameron, Y. Maissel, A. Navas, and F. Olukoya gives a description of the groups \(\Aut(G_{n, r})\) and \(\Out(G_{n, r})\) using Rubin's Theorem and transducers.
Here we give an analogous description of the groups \(\Aut(dV_n)\) and \(\Out(dV_n)\) (see Theorem~\ref{autnV}) and use it to give an embedding of \(\Out(dV_n)\) into \(\Out(V_n){{{\wr}}} \Sym(d)\) (see Theorem~\ref{main theorem 1}).
Moreover in the case that \(n=2\), our embedding is actually an isomorphism (see Theorem~\ref{main theorem 2}).
This proves a conjecture made by Nathan Barker in 2012: The groups \(\Out(dV)\) and \(\Out(V) \wr \Sym(d)\) are isomorphic.

The groups \(nV\) for \(n\geq 2\) have also been described with transducers before but not in the same manner done here (in \cite{Belk2016} it is shown that they embed in the rational group).
We discuss this alternative transducer viewpoint more in Section~\ref{rationality sec}.

For our representation of \(dV_n\) we extend the notion of transducer established by Grigorchuk,  Nekrashevich, and Sushchanskii \cite{GNS2000}, and then identify a class of transducers which is appropriate for representing continuous transformations of \(n\)-dimensional Cantor spaces (see Theorem~\ref{transducerable = continuous theorem}).

\section{Higher dimensional words}
In this section we define the groups of interest and introduce the basic concepts surrounding muti-dimensional words (these will be used heavily thought this part).

\speeddictthree{100}{cones words and Cantor spaces defn}{Brouwer' Theorem}{product structures}{words defn}
\begin{defn}[Higher dimensional words, \Assumed{100}]\label{cones words and Cantor spaces defn}
If \(n\geq 2\), then we define \(X_n := \{0, 1, \ldots, n-1\}\) and \(\mathfrak{C}_n:= \{0, 1, \ldots , n-1\}^\N\).
We think of the elements of \(\mathfrak{C}_n\) as \textit{infinite words} over the alphabet \(X_n\) (see Definition~\ref{words defn}) and we give them the product topology using the discrete topology on \(X_n\).
It follows from Theorem~\ref{Brouwer' Theorem}, that all of these are Cantor spaces.
For each \(i\in \{0,1, \ldots,d-1\}\) and \(x\in X_n\), we define \(x_{d, i}\in (X_n^*)^d\) by \((i)x_{d,i} =x\) and \((j)x_{d,i} = \varepsilon\) for all \(j\neq i\).
Similarly, we define \(\varepsilon_d\in (X_n^*)^d\) to be the constant tuple with value \(\varepsilon\).

If \(d\in \N\backslash \{0\}\), then we extend the prefix relation \(\leq\) (from Definition~\ref{words defn}) on \(\{0, 1, \ldots, n-1\}^* \cup \mathfrak{C}_n\) to the set \((\{0, 1, \ldots, n-1\}^* \cup \mathfrak{C}_n)^d\) in the natural fashion (as in Example~\ref{product poset}).
We also extend the concatenation operation to the domain \((X_n^*)^d \times (X_n^* \cup \mathfrak{C}_n)^d\) coordinatewise. For \(w\in (X_n^*)^d\), then we also extend \(\lambda_w\) to the domain \((X_n^* \cup \mathfrak{C}_n)^d\) to be concatenation with \(w\) on the left (note that this map is injective).

Note that if \(w\in (X_n^*)^d\) then
\[w\mathfrak{C}_n^d= (\mathfrak{C}_n^d)\lambda_w= \makeset{x\in  \mathfrak{C}_n^d}{\( w\leq x\)}.\]

It is routine to verify that these sets are clopen, and the collection of all such sets is a basis for \(\mathfrak{C}_n^d\). Such basic open sets will be referred to as \textit{cones}.
As the cones form a basis of clopen sets for \(\mathfrak{C}_n^d\) and this space is compact, it follows that the clopen subsets of \(\mathfrak{C}_n^d\) are precisely the finite unions of cones.
\end{defn}

\speeddicttwo{101}{prefix codes and dVn defn}{automorphism groups defn}{cones words and Cantor spaces defn}
\begin{defn}[Prefix codes and \(dV_n\), \Assumed{101}]\label{prefix codes and dVn defn}
Suppose that \(d\in \mathbb{N}\backslash \{0\}\) and \(n\in \mathbb{N}\backslash \{0, 1\}\). We say that a subset \(F\subseteq (X_n^*)^d\) is a \(\textit{complete prefix code}\) if \(\makeset{w\mathfrak{C}_n^d}{\(w\in F\)}\) is a partition of \(\mathfrak{C}_n^d\) into clopen sets. Note that as \(\mathfrak{C}_n^d\) is compact, all complete prefix codes will be finite.

If \(F_1, F_2\) are complete prefix codes of \(\mathfrak{C}_n^d\), and \(\phi: F_1 \to F_2\) is a bijection, then we define the \textit{prefix exchange map} \(f_\phi:~\mathfrak{C}_n^d~\to~\mathfrak{C}_n^d\) by
\[f_{\phi} := \union{w\in F_1} (\lambda_w^{-1} \circ \lambda_{(w)\phi})\restriction_{w\mathfrak{C}_n^d}.\]
That is, if \(w\in F_1\) and \(x\in \mathfrak{C}_n^d\), then \((w x)f_{\phi}=((w)\phi) (x)\).

It is routine to verify that such prefix exchange maps are always homeomorphisms, and moreover the set of prefix exchange maps is closed under composition and inversion. We define \(dV_n\) to be the subgroup of \(\Aut(\mathfrak{C}_n^d)\) consisting all the prefix exchange maps. If \(n=2\) or \(d= 1\) we will often omit them in the notation \(dV_n\).
\end{defn}

\speeddictone{102}{anti-chain sizes}{prefix codes and dVn defn}
\begin{lemma}[Complete prefix code sizes, \Assumed{102}]\label{anti-chain sizes}
Suppose that \(m, d\geq 1\) and \(n\geq 2\). 
There is a complete prefix code for \(\mathfrak{C}_n^d\) of cardinality \(m\) if and only if \(m\in  (1 + (n-1)\mathbb{N})\).
\end{lemma}
\begin{proof}
\((\Leftarrow):\) We show by induction on \(k\), that for all \(k\in \N\) there is a complete prefix code of \(\mathfrak{C}_n^d\) of size \(1 + (n-1)k\).
If \(k = 0\), then \(1+ (n-1)k= 1\) and we can choose the complete prefix code \(\{\varepsilon_d\}\).
Suppose that \(k>0\) and the claim holds for smaller \(k\).
Let \(F\) be a complete prefix code with \(|F| = 1 + (n-1)(k-1)\).
Choose \(w\in F\) and define \(F' := (F\backslash \{w\}) \cup \makeset{w x_{d,0}}{\(x\in X_n\)}\), then \(F'\) is a complete prefix code with \(|F'| = 1 + (n-1)k\).

\((\Rightarrow):\) Let \(F\) be a complete prefix code for \(\mathfrak{C}_n^d\).
Let \(k := \max\makeset{|(w)\pi_i|}{\(w\in F, i\in \{0, 1,\ldots, d-1\}\)}\).
Suppose that there is \(w\in F\) and \(i\in \{0, 1, \ldots, d-1\}\) such that \(|(w)\pi_i| \neq k\).
It follows that \(|(w)\pi_i| <k\).
If we then define \(F' := (F\backslash \{w\}) \cup \makeset{w x_{d, i}}{\(x\in X_n\)}\), then \(|F'| -|F|=n-1\) and 
\[\max\makeset{|(w)\pi_i|}{\(w\in F\)} = \max\makeset{|(w)\pi_i|}{\(w\in F'\)}.\]

Thus we can assume without loss of generality that \(|(w)\pi_i| = k\) for all \(w\in F\) and \(i\in \{0, 1, \ldots, d-1\}\). As \(F\) is a complete prefix code, it follows that \(F  = (X_n^k)^d\). In particular \(|F| = n^{k d}\in 1 + (n-1)\N\) as required.
\end{proof}

\speeddictone{103}{scary anti-chains}{anti-chain sizes}
\begin{remark}[Scary prefix codes, \Assumed{103}]\label{scary anti-chains}
In the proof of Lemma~\ref{anti-chain sizes}, we construct complete prefix codes of the required sizes by starting at the trivial prefix code \(\{\varepsilon_{d}\}\), and sequentially refining it by replacing one element with \(n\) new elements.
It can be shown than in the case \(d = 1\), all complete prefix codes can be found this way. However this is not true in general.
For example consider the following complete prefix code when \(n=2\) and \(d=3\):
\[\{(\varepsilon, 0, 0), (0, 1 , \varepsilon), (1, \varepsilon, 1), (0, 0, 1), (1, 1, 0)\}.\]
\end{remark}

\section{Generalizing the transducers of Grigorchuk,  Nekrashevich, and Sushchanskii}
In this section we introduce the type of transducers we will be using throughout this part.
The transducers defined by Grigorchuk, Nekrashevich, and Sushchanskii (which we shorten to GNS) in \cite{GNS2000}, are machines which ``read" letters and then ``transition" between various ``states" and ``write" letters accordingly.
These machines can then read words by reading each letter in turn.

This can be thought of as assigning to each letter of an alphabet \(X_n\), a transformation of a state set, and a word to write for each state.
This assignment is then extended to all elements of \(X_n^*\) via its ``freeness" property.
Our definition has the view of reading/writing elements of a semigroup, but we do not restrict to free monoids.

\speeddictfour{108}{transducers defn}{Numbers defns}{Categories defn}{semigroup defn}{group actions defn}
\begin{defn}[Transducers, \Assumed{108}]\label{transducers defn}
We say that \(T := (Q_T, D_T, R_T, \boldsymbol{\pi}_T, \boldsymbol{\lambda}_T)\) is a transducer if:
\begin{enumerate}
    \item \(Q_T\) is a set (called the set of states).
    \item \(D_T\) is a semigroup (called the domain semigroup).
    \item \(R_T\) is a semigroup (called the range semigroup).
    \item \(\boldsymbol{\pi}_T: Q_T\times D_T \to Q_T\) is an action of \(D_T\) on the discrete set \(Q_T\) (called the transition function).
    \item \(\boldsymbol{\lambda}_T : Q_T\times D_T \to R_T\) is a function with the property that for all \(q\in Q_T\) and \(s,t\in D_T\) we have \[( q,st)\boldsymbol{\lambda}_T= (q,s)\boldsymbol{\lambda}_T(( q,s)\boldsymbol{\pi}_T,t)\boldsymbol{\lambda}_T\text{ (called the output function).}\] 
\end{enumerate}
\end{defn}

\speeddictone{109}{Transducers vs Homomorphims}{transducers defn}
\begin{remark}[Transducers vs homomorphims, \Assumed{109}]\label{Transducers vs Homomorphims}
Defining a one state transducer with domain \(D\) and range \(R\) is equivalent to defining a semigroup homomorphism from \(D\) to \(R\).
In this sense transducers can be thought of as generalisations of semigroup homomorphisms.
\end{remark}

Note that our definition of a transducer consists only of a set, two semigroups and some maps between them.
It is thus easy and natural to view our transducers as objects of a category.
This viewpoint will be key to many of the arguments later in the part.

\speeddictfour{110}{transducers homomorphisms defn}{Products in Categories Defn}{semigroup defn}{congruences and quotients defn}{transducers defn}
\begin{defn}[Transducer homomorphisms, \Assumed{110}]\label{transducers homomorphisms defn}
Let \(A, B\) be transducers. We say that \(\phi\) is a \textit{transducer homomorphism} from \(A\) to \(B\) (written \(\phi:A \to B\)), if \(\phi\) is a 3-tuple \((\phi_Q, \phi_D, \phi_R)\) with the following properties: 
\begin{enumerate}
    \item \(\phi_{R}:R_{A}\to R_{B}\) is a semigroup homomorphism.
    \item \(\phi_{D}:D_{A}\to D_{B}\) is a semigroup homomorphism.
    \item \(\phi_Q:Q_{A} \to Q_B\) is a function, such that 
    \[\langle \boldsymbol{\pi}_{A}, \boldsymbol{\lambda}_{A}  \rangle_{Q_A\times R_A}\langle \pi_0\phi_Q, \pi_1\phi_R\rangle_{Q_B\times R_B} =  \langle \pi_0\phi_Q, \pi_1\phi_D\rangle_{Q_B\times D_B}\langle \boldsymbol{\pi}_{B}, \boldsymbol{\lambda}_{B}  \rangle_{Q_B\times R_B}\]
    (recall Definition~\ref{Products in Categories Defn}). Or equivalently, for all \(q\in Q_{A}\) and \(s\in D_A\) we have
    \[(q, s)\boldsymbol{\pi}_{A}\phi_Q = ((q)\phi_Q, (s)\phi_D)\boldsymbol{\pi}_{B} \quad \text{and}\quad (q, s)\boldsymbol{\lambda}_{A}\phi_R = ((q)\phi_Q, (s)\phi_D) \boldsymbol{\lambda}_{B}.\]
\end{enumerate}
 If the maps \(\phi_D, \phi_R\) are identity maps, then we say that \(\phi\) is \textit{strong}. 
 We compose transducer homomorphisms component-wise.
It is routine to verify that transducers and transducer homomorphisms form a category.
Moreover transducers and strong transducer homomorphisms form a subcategory of this category.

We say that a transducer homomorphism \(\phi\) is a \textit{quotient map} if each of \(\phi_{Q}, \phi_D, \phi_R\) is surjective (so each of them is a quotient map as in Definition~\ref{congruences and quotients defn}).
We say that \(\phi:A\to B\) is a \textit{transducer isomorphism} if it is an isomorphism in the category of transducers and transducer homomorphisms, or equivalently if each of the maps \(\phi_{Q}, \phi_D, \phi_R\) is a bijection.

We say that two transducers \(A, B\) are \textit{isomorphic} (denoted \(A\cong B\)) if there is a transducer isomorphism \(\phi:A\to B\). Similarly we say that  \(A, B\) are \textit{strongly isomorphic} (denoted \(A\cong_S B\)) if there is a strong transducer isomorphism \(\phi:A\to B\)
\end{defn}

The following notion of a ``minimal transducer" corresponds to the notion of ``combining equivalent states" from GNS.
The GNS notion of being ``reduced" is the notion discussed in Theorem~\ref{unique minimal thm}.
\speeddicttwo{111}{minimal transducer defn}{types of binary relation defns}{transducers homomorphisms defn}
\begin{defn}[Minimal transducers, \Assumed{111}]\label{minimal transducer defn}
If \(T\) is a transducer and \(\lambda_{r}\) injective for all \(r\in R_T\), then we define its \textit{minimal transducer} \(M_T\) to be \((Q_T/\sim_{M_T}, D_T, R_T, \boldsymbol{\pi}_{M_T}, \boldsymbol{\lambda}_{M_T})\) where \(\sim_{M_T}, \pi_{M_T}\) and \(\lambda_{M_T}\) are defined by:
\begin{enumerate}
    \item \(\sim_{M_T}\) is the equivalence relation 
    \[\makeset{(p, q)\in Q_T^2}{\((p, s)\boldsymbol{\lambda}_{T} = (q, s)\boldsymbol{\lambda}_{T}\) for all \(s\in D_T\)};\]
    
    \item If \(q\in Q_T\), \(s\in D_T\) then \(([q]_{\sim_{M_T}}, s)\boldsymbol{\pi}_{M_T} = [(q,s)\boldsymbol{\pi}_T]_{\sim_{M_T}}\);
    
    \item If \(q\in Q_T\), \(s\in D_T\) then \(([q]_{\sim_{M_T}}, s)\boldsymbol{\lambda}_{M_T} = (q,s)\boldsymbol{\lambda}_T\).
\end{enumerate}
Moreover, we define a strong quotient map \(q_T:T\to M_T\), by \((p){q_T}_Q= [p]_{\sim_{M_T}}\). We will justify these definitions in Lemma~\ref{minimal transducers are valid lemma}.
\end{defn}

\speeddictone{112}{minimal transducers are valid lemma}{minimal transducer defn}
\begin{lemma}[Minimal transducers are valid, \Assumed{112}]\label{minimal transducers are valid lemma}
The objects defined in Definition~\ref{minimal transducer defn} are well-defined and have the stated properties.
Moreover if \(A\) is a transducer, \(\lambda_r\) is injective for all \(r\in R_A\), and \(\phi:A\to B\) is a strong quotient map, then there is a strong quotient map \(\psi:B\to M_A\) with \(q_A = \phi\psi\).
\end{lemma}
\begin{proof}
Let \(T\) be as in Definition~\ref{minimal transducer defn}. We first need to show that \(\boldsymbol{\pi}_{M_T}\) is well-defined. Let \(p, q\in Q_T\) and \(s\in D_T\) be such that \(p\sim_{M_T} q\).
We need to show that \((p,s)\boldsymbol{\pi}_T \sim_{M_T} (q,s)\boldsymbol{\pi}_T\) and \((p,s)\boldsymbol{\lambda}_T = (q,s)\boldsymbol{\lambda}_T\). The equality \((p,s)\boldsymbol{\lambda}_T = (q,s)\boldsymbol{\lambda}_T\) is immediate from the definition of \(\sim_{M_T}\). To show that \((p,s)\boldsymbol{\pi}_T \sim_{M_T} (q,s)\boldsymbol{\pi}_T\), it suffices to show that if \(t\in D_T\) is arbitrary then \(((p,s)\boldsymbol{\pi}_T, t)\boldsymbol{\lambda}_T = ((q,s)\boldsymbol{\pi}_T, t)\boldsymbol{\lambda}_T\).  This follows from
\begin{align*}
    ((p,s)\boldsymbol{\pi}_T, t)\boldsymbol{\lambda}_T &= ((((p,s)\boldsymbol{\pi}_T, t)\boldsymbol{\lambda}_T)\lambda_{(p,s)\boldsymbol{\lambda}_T})\lambda_{(p,s)\boldsymbol{\lambda}_T}^{-1}\\
    &= ((p,s)\boldsymbol{\lambda}_T((p,s)\boldsymbol{\pi}_T, t)\boldsymbol{\lambda}_T)\lambda_{(p,s)\boldsymbol{\lambda}_T}^{-1}\\
    &= ((p,st)\boldsymbol{\lambda}_T)\lambda_{(p,s)\boldsymbol{\lambda}_T}^{-1}\\
    &= ((q,st)\boldsymbol{\lambda}_T)\lambda_{(q,s)\boldsymbol{\lambda}_T}^{-1}\\
 &= ((q,s)\boldsymbol{\lambda}_T((q,s)\boldsymbol{\pi}_T, t)\boldsymbol{\lambda}_T)\lambda_{(q,s)\boldsymbol{\lambda}_T}^{-1}\\
  &= ((((q,s)\boldsymbol{\pi}_T, t)\boldsymbol{\lambda}_T)\lambda_{(q,s)\boldsymbol{\lambda}_T})\lambda_{(q,s)\boldsymbol{\lambda}_T}^{-1}\\
     &= ((q,s)\boldsymbol{\pi}_T, t)\boldsymbol{\lambda}_T.
\end{align*}

We next need to show that \(q_T\) is a quotient map. The maps \({q_T}_D, {q_T}_R, {q_T}_Q\) are all surjective by definition, so we need only show that \(q_T\) is a transducer homomorphism. If \(q\in Q_T\) and \(s\in D_T\), then by definition we have
 \[(q, s)\boldsymbol{\pi}_{T}{q_T}_Q = [(q, s)\boldsymbol{\pi}_{T}]_{\sim_{M_T}}=([q]_{\sim_{M_T}}, s)\boldsymbol{\pi}_{M_T}=  ((q){q_T}_Q, (s){q_T}_{D})\boldsymbol{\pi}_{M_T},\] 
 \[(q, s)\boldsymbol{\lambda}_{T}{q_A}_R =(q, s)\boldsymbol{\lambda}_{T}= ([q]_{\sim_{M_T}}, s) \boldsymbol{\lambda}_{M_T}=((q){q_A}_Q, (s){q_A}_D) \boldsymbol{\lambda}_{M_T}\]
 as required. 
 We define \(\psi:B\to M_A\) by having \(\psi_D, \psi_R\) be the identity maps and defining \(\psi_Q\) by:
\[((q)\phi_Q)\psi_Q:= (q){q_A}_Q.\]

As \({q_A}_Q\) is surjective, it suffices to show that \(\psi\) is a well-defined transducer homomorphism.
Note that all of the maps \(\phi_D, \phi_R, \psi_D, \psi_R, {q_A}_D\) and \({q_A}_R\) are identity maps so we can ignore them for the purposes of this proof. 
We first show that \(\psi\) is well-defined. 
Suppose that \(q_0, q_1\in Q_A\) satisfy \((q_0)\phi_Q= (q_1)\phi_Q\). 
We need to show that \(q_0\sim_{M_T} q_1\).
For \(s\in D_A\) be arbitrary, we have
\[(q_0, s)\boldsymbol{\lambda}_A = ((q_0)\phi_Q, s)\boldsymbol{\lambda}_B= ((q_1)\phi_Q, s)\boldsymbol{\lambda}_B= (q_1, s)\boldsymbol{\lambda}_A.\]
So indeed \(q_0\sim_{M_T} q_1\).
It remains to show that \(\psi\) is a homomorphism.
Using the fact that \(\phi_Q\) is surjective, let \((p,s)= ((q)\phi_Q,s)\in Q_B~\times~D_B\) be arbitrary.
We need to show that \(\psi\) satisfies the third condition for being a transducer homomorphism.
We have
\begin{align*}
    (p, s)\boldsymbol{\pi}_B\psi_Q&=((q)\phi_Q, s)\boldsymbol{\pi}_B\psi_Q&\text{ by the definition of }p\\
    &=(q, s)\boldsymbol{\pi}_A\phi_Q\psi_Q& \text{ because }\phi\text{ is a homomorphism} \\
    &=(((q)\phi_Q)\psi_Q, s)\boldsymbol{\pi}_{M_A}&\text{ because }q_A=\phi\psi\text{ is a homomorphism}\\
    &=((p)\psi_Q, s)\boldsymbol{\pi}_{M_A}&\text{ by the definition of }p,
\end{align*}
and similarly \((p, s)\boldsymbol{\lambda}_B=((p)\psi_Q, s)\boldsymbol{\lambda}_{M_A}\) as required.
\end{proof}

\speeddicttwo{113}{subtransducer def}{function defn}{transducers defn}
\begin{defn}[Subtransducers, \Assumed{113}]\label{subtransducer def}
If \(A, B\) are transducers, \(Q_A\subseteq Q_B, D_A\subseteq D_B\), \(R_A\subseteq R_B\), \(\boldsymbol{\pi}_A\subseteq \boldsymbol{\pi}_B\) and \(\boldsymbol{\lambda}_A\subseteq \boldsymbol{\lambda}_B\), then we call \(A\) \textit{subtransducer} of \(B\). 
We will sometimes identify a subtransducer of a transducer (with full domain and range) with its state set for simplicity.
\end{defn}

The transducers and maps from the following definitions, will be our main focus (due to Theorem~\ref{transducerable = continuous theorem}).
\speeddicttwo{114}{dn transducer def}{cones words and Cantor spaces defn}{transducers defn}
\begin{defn}[\((d, n)\)-transducers, \Assumed{114}]\label{dn transducer def}
If \(d, k\in \N\) and \(n, m\in \N\backslash \{0, 1\}\), then we define a \textit{\((d,n, k, m)\)-transducer} to be a transducer \(T\) with \((X_{n}^*)^{d}\) as its domain, \((X_{m}^*)^{k}\) as its range, and such that for all \(q\in Q_T\), we have
\[(q, \varepsilon_d)\boldsymbol{\pi}_T = q, \quad\text{ and }\quad (q, \varepsilon_d)\boldsymbol{\lambda}_T = \varepsilon_k.\]
We will be primarily concerned with the case that \(d=k\) and \(n=m\). Thus to simplify notation we define a \((d, n)\)-transducer to be a \((d, n, d, n)\)-transducer.
\end{defn}



\speeddictone{115}{reading infinite words defn}{dn transducer def}
\begin{defn}[Reading infinite words, \Assumed{115}]\label{reading infinite words defn}
Suppose that \(d, k\in \N\), \(n, m\in \N\backslash \{0, 1\}\), \(T\) is a \((d, n, k, m)\)-transducer and \(q\in Q_T\).
If \(w\in ((X_n)^\N)^d\) and \(j\in \N\), then we denote the unique element of \((X_n^j)^d\), which is a prefix of \(w\), by \(w\restriction_j\).
Note that if \(i\in \{0, 1, \ldots, k-1\}\) and \(j_0\leq j_1\), then \((q, w\restriction_{j_0})\boldsymbol{\lambda}_T\pi_i\subseteq (q, w\restriction_{j_1})\boldsymbol{\lambda}_T\pi_i\).

Thus we can define \(f_{T,q}:(X_n^\N)^d \to (X_m^\N \cup X_m^*)^k\) to be the map with
\[(w)f_{T,q}\pi_i = \union{j\in \N}(q, w\restriction_{j})\boldsymbol{\lambda}_T \pi_i\]
for all \(i\in \{0, 1, \ldots, k-1\}\).
In particular \((w)f_{T,q}\) is the smallest element of \((X_m^\N \cup X_m^*)^k\) which is greater that \((q, w')\boldsymbol{\lambda}_T\) for all \(w'\in (X_n^*)^d\) with \(w'\leq w\).
\end{defn}

\speeddicttwo{116}{induced maps preserved remark}{minimal transducer defn}{reading infinite words defn}
\begin{remark}[Induced maps are preserved, \Assumed{116}]\label{induced maps preserved remark}
Let \(d, k\in \N\) and \(n, m\in \N\backslash \{0, 1\}\). 
If \(A\) is a \((d, n, k, m)\)-transducer, \(q\in Q_A\) and \(\phi:A\to B\) is a strong transducer homomorphism, then the maps \(s\to (q, s)\boldsymbol{\lambda}_A\) and \(s\to ((q)\phi, s)\boldsymbol{\lambda}_B\) are equal. Hence \(f_{A,q}=f_{B, (q)\phi_Q}\). In particular this is true of the quotient map \(q_A\).
\end{remark}

It is routine to verify that the notion of degenerate in the following definition is equivalent to the GNS notion of degenerate when both are applicable. 
\speeddicttwo{117}{degenerate transducers defn}{cones words and Cantor spaces defn}{reading infinite words defn}
\begin{defn}[Degenerate transducers, \Assumed{117}]\label{degenerate transducers defn}
Let \(d, k\in \N\) and \(n, m\in \N\backslash \{0, 1\}\). 
We say that a \((d, n, k, m)\)-transducer \(T\) is \textit{degenerate} if there exist \(q\in Q_T\), \(x\in \mathfrak{C}_n^d\) such that \((x)f_{T,q} \notin \mathfrak{C}_m^k\).
\end{defn}

\speeddicttwo{118}{transducers are continuous lemma}{cones words and Cantor spaces defn}{degenerate transducers defn}
\begin{lemma}[Continuity of transducers, \Assumed{118}]\label{transducers are continuous lemma}
Let \(d, b\in \N\) and \(n, l\in \N\backslash \{0, 1\}\).
If \(T\) is a non-degenerate \((d, n, b, l)\)-transducer and \(q\in Q_T\), then for all \(m\in \N\) there is \(k\in \N\) such that for all \(i\in \{0, 1, \ldots, b-1\}\) and \(w\in (X_n^k)^d\) we have \(|(q, w)\boldsymbol{\lambda}_T\pi_i| \geq m\).

Moreover, the map \(f_{T, q}:\mathfrak{C}_n^d\to \mathfrak{C}_l^b\) is continuous.
\end{lemma}
\begin{proof}
Let \(T\) be a non-degenerate \((d,n, b, l)\)-transducer and \(q\in Q_T\) be arbitrary.
To prove the first claim, suppose for a contradiction that there is \(m\in \N\) such that for all \(k\in \N\) there is \(w\in (X_n^k)^d\) and \(i\in \{0,1, \ldots, b-1\}\) with \(|(q, w)\boldsymbol{\lambda}_T\pi_i| < m\).

Using this assumption we define a sequence \((w_{j})_{j\in \N}\) in \((X_n^*)^d\) by induction as follows:
\begin{enumerate}
    \item Define \(w_0 := \varepsilon_d\).
    \item Suppose that \(w_j\) is defined, choose \(w_{j+1}\in (X_n^{j+1})^d\) such that 
    \begin{enumerate}
        \item \(w_j\leq w_{j+1}\),
        \item there is \(i\in \{0, 1, \ldots, b-1\}\) such that \(|(q, w_{j+1})\boldsymbol{\lambda}_T\pi_i| \leq m\), and
        \item for all \(k\in \N\backslash \{0,1, \ldots, j\}\), there is \(w_{j + 1, k}\in (X_n^{k})^d\) such that \(w_{j +1} \leq w_{j+1, k}\) and there is \(i\in \{0, 1, \ldots, b-1\}\) such that \(|(q, w_{j+1, k})\boldsymbol{\lambda}_T\pi_i| \leq m\).
    \end{enumerate}
\end{enumerate}

As \(b\) is finite, there must be some \(i\in \{0, 1, \ldots, d-1\}\) such that for infinitely many (and hence all) \(j\in \N\) we have \(|(q, w_{j})\boldsymbol{\lambda}_T\pi_i| \leq m\).
Thus if we define \(w\in (X_n^\N)^d\) to be such that for all \(i\in \{0, 1, \ldots, b-1\}\) we have 
\[(w)\pi_i= \union{j\in \N}(w_j)\pi_i,\] then \(|(w)f_{T, q}\pi_i|\leq m\).
This is a contradiction as \(T\) was assumed to be non-degenerate.

The topology on \((X_n^\N)^d\) is the product topology.
Thus it suffices to show that if \(x\in (X_n^\N)^d, i\in \{0, 1, \ldots, n-1\}\) and \(j\in \N\) then there is a neighbourhood \(U\) of \(x\), in \( (X_n^\N)^d\), such that for all \(y\in U\) we have \((j)((i)((y)f_{T, q})) = (j)((i)((x)f_{T, q}))\).
By the first part of the lemma, we can choose \(k\in \N\) such that for all \(a\in \{0, 1, \ldots, b-1\}\) and \(w\in (X_n^k)^d\) we have \(|(q, w)\boldsymbol{\lambda}_T\pi_a| \geq j + 1\).
Thus defining
\[U := \makeset{y\in (X_n^\N)^d}{\(((y)\pi_a)\restriction_{j}=((x)\pi_a)\restriction_{j}\) for all \(a\in \{0, 1, \ldots, d-1\}\)}\]
gives the required neighbourhood of \(x\).
\end{proof}

We give two distinct ways of multiplying transducers.
The first (Definition~\ref{transducers composition defn}) is an extension of the definition given by GNS, and the second (Definition~\ref{transducer products defn}) is a categorical product object.

\speeddictone{119}{transducers composition defn}{transducers defn}
\begin{defn}[Transducer composition, \Assumed{119}]\label{transducers composition defn}
If \(A\) and \(B\) are transducers, such that the range of \(A\) is contained in the domain of \(B\), then we define their \textit{composite} by
\[AB:= (Q_{AB}, D_{AB}, R_{AB}, \boldsymbol{\pi}_{AB}, \boldsymbol{\lambda}_{AB}).\]
Where
\begin{enumerate}
    \item \(Q_{AB}:= Q_A\times Q_B\), \(D_{AB}:= D_A\), and \(R_{AB}:= R_B\).
    \item \(((a, b), s)\boldsymbol{\pi}_{AB}= ((a,s)\boldsymbol{\pi}_A, (b, (a,s)\boldsymbol{\lambda}_{A})\boldsymbol{\pi}_B)\).
    \item \(((a, b), s)\boldsymbol{\lambda}_{A,B}= (b,(a,s)\boldsymbol{\lambda}_{A})\boldsymbol{\lambda}_B\).
\end{enumerate}
\end{defn}


\speeddictthree{120}{composition works as expected lemma}{reading infinite words defn}{transducers are continuous lemma}{transducers composition defn}
\begin{lemma}[Composition works as expected, \Assumed{120}]\label{composition works as expected lemma}
If \(A\) is a non-degenerate \((a, b, c, d)\)-transducer and \(B\) is a non-degenerate \((c, d, e, f)\) transducer, and \((p,q)\in Q_A\times Q_B\), then \(f_{A, p}f_{B,q}=f_{AB,(p,q)}\) (in particular, \(AB\) is non-degenerate).
\end{lemma}
\begin{proof}
Let \(w\in (X_b^\N)^a\) be arbitrary. 
From Definition~\ref{reading infinite words defn}, it suffices to show that \((w)f_{A, p}f_{B,q}\) is the smallest element of \((X_f^\N \cup X_n^*)^e\) which is greater that \(((p, q), w')\boldsymbol{\lambda}_{AB}\) for all \(w'\in (X_b^*)^a\) with \(w'\leq w\).

Let \(w'\in (X_b^*)^a\) be arbitrary with \(w'\leq w\). We have \((p, w')\boldsymbol{\lambda}_{A} \leq (w)f_{A, p}\), thus
\[((p, q), w')\boldsymbol{\lambda}_{AB} = (q, (p, w')\boldsymbol{\lambda}_{A})\boldsymbol{\lambda}_{B} \leq ((w)f_{A, p})f_{B, q}= (w)f_{A, p}f_{B, q}.\]

We now have \((w)f_{A, p}f_{B,q} \geq (w)f_{AB,(p,q)}\).
To show equality it suffices to show that \((w)f_{AB,(p,q)} \in (X_f^\N)^e\).
Let \(m\in \N\) be arbitrary, we show that for all \(i\in \{0, 1, \ldots, e-1\}\) we have \(|(w)f_{AB, (p, q)}\pi_i| \geq m\).
By Lemma~\ref{transducers are continuous lemma}, let \(k_A, k_B\in \N\) be such that if \(w_A\in (X_b^{k_A})^a, w_B\in (X_d^{k_B})^c\) then for all valid choices of \(i\) we have
\[|(p, w_A)\boldsymbol{\lambda}_{A}\pi_i|\geq k_B, \quad \text{ and }\quad |(p, w_B)\boldsymbol{\lambda}_{B}\pi_i|\geq m.\]
It follows that if \(w'\in (X_b^{k_A})^a\), \(w'\leq w\) and \(i\in \{0, 1, \ldots, e-1\}\), then
\[|(w)f_{AB, (p, q)}\pi_i| \geq |( (p, q), w')\boldsymbol{\lambda}_{AB}\pi_i|=|(q,(p, w')\boldsymbol{\lambda}_{A})\boldsymbol{\lambda}_B\pi_i|\geq m\]
as required.
\end{proof}

\speeddictfour{121}{transducer products defn}{Products in Categories Defn}{product structures}{transducers homomorphisms defn}{transducer products are products lemma}
\begin{defn}[Transducer products, \Assumed{121}]\label{transducer products defn}
If \((A)_{i\in I}\) are transducers, then we define their \textit{product}
\(\prod_{i\in I}A_i\) to be the transducer \(P\)
where
\[Q_P:=\prod_{i\in I}Q_{A_i}, \quad D_P:= \prod_{i\in I}D_{A_i}, \quad R_P := \prod_{i\in I}R_{A_i},\]
and for all \((p_i)_{i\in I}\in Q_P\) and \((s_i)_{i\in I}\in D_P\) we have
\[((p_i)_{i\in I}, (s_i)_{i\in I})\boldsymbol{\pi}_{P}=
((p_i,s_i)\boldsymbol{\pi}_{A_i})_{i\in I},\]
\[((p_i)_{i\in I}, (s_i)_{i\in I})\boldsymbol{\lambda}_{P}=
((p_i,s_i)\boldsymbol{\lambda}_{A_i})_{i\in I}.\]
We show in Lemma~\ref{transducer products are products lemma} that \((P, ((\pi_i, \pi_i, \pi_i))_{i\in I})\) is a product in the category of transducers and transducer homomorphisms (recall Definition~\ref{Products in Categories Defn}).
\end{defn}

\speeddictfour{700}{transducer products are products lemma}{Products in Categories Defn}{product structures}{transducers homomorphisms defn}{transducer products defn}
\begin{lemma}[Transducer products are products, \Assumed{121}]\label{transducer products are products lemma}
If \((A)_{i\in I}\) are transducers, then \((\prod_{i\in I}A_i, ((\pi_i, \pi_i, \pi_i))_{i\in I})\) is a product in the category of transducers and transducer homomorphisms.
\end{lemma}
\begin{proof}
Let \(P:= \prod_{i\in I}A_i\). Note that by the definition of \(P\), for all \(i\in I\) the triple \((\pi_i, \pi_i, \pi_i)\) is a transducer homomorphism from \(P\) to \(A_i\).

Let \(B\) be a transducer and for each \(i\in I\), let \(\phi_{i}\) be a transducer homomorphism from \(B\) to \(A_i\). 
We must show that there is a unique transducer homomorphism \(\psi:B \to P\) such that for all \(i\in I\), the composition of the transducer homomorphism \(\psi\) with the transducer homomorphism \((\pi_i, \pi_i, \pi_i)\) is \(\phi_i\).

As \((Q_P, (\pi_i)_{i\in I})\), \((D_P, (\pi_i)_{i\in I})\), and \((R_P, (\pi_i)_{i\in I})\) are products in the categories of sets and functions, there exist unique functions \(\psi_Q:Q_B\to Q_{P}\), \(\psi_D:D_B\to D_{P}\), and \(\psi_R:R_B\to R_{P}\) such that 
\[\psi_Q\circ \pi_i = {\phi_i}_Q, \quad \psi_D\circ \pi_i = {\phi_i}_D, \text{ and }\psi_R\circ \pi_i = {\phi_i}_R\]
for all \(i\in I\). 
It follows that \(\psi:= (\psi_Q, \psi_D, \psi_R)\) that it is the only triple of functions whose coordinate wise composition with \((\pi_i, \pi_i, \pi_i)\) results in \(\phi_i\) for all \(i\in I\). It therefore suffices to show that \(\psi\) is a transducer homomorphism from \(B\) to \(P\).

As \((D_P, (\pi_i)_{i\in I})\) and \((R_P, (\pi_i)_{i\in I})\) are also products in the category of semigroups and homomorphisms, the maps \(\psi_D\) and \(\psi_R\) are semigroup homomorphisms. For all \(q\in Q_B\) and \(s\in D_B\) we have
\begin{align*}
    (q, s)\boldsymbol{\pi}_{B}\psi_Q&= ((q, s)\boldsymbol{\pi}_{B}{\psi_i}_Q)_{i\in I} = (((q){\psi_i}_Q, (s){\psi_i}_D)\boldsymbol{\pi}_{A_i})_{i\in I}\\
    &= (((q){\psi_i}_Q)_{i\in I}, ((s){\psi_i}_D)_{i\in I})\boldsymbol{\pi}_{P}=((q)\psi_Q, (s)\psi_D)\boldsymbol{\pi}_{P}
\end{align*}
 and similarly
 \begin{align*}
    (q, s)\boldsymbol{\lambda}_{B}\psi_R &= ((q, s)\boldsymbol{\lambda}_{B}{\psi_i}_R)_{i\in I} = (((q){\psi_i}_Q, (s){\psi_i}_D) \boldsymbol{\lambda}_{A_i})_{i\in I}\\
    &=(((q){\psi_i}_Q)_{i\in I}, ((s){\psi_i}_D)_{i\in I}) \boldsymbol{\lambda}_{P} =((q)\psi_Q, (s)\psi_D) \boldsymbol{\lambda}_{P}
\end{align*}
 It follows that \(\psi\) is a transducer homomorphism as required. 
\end{proof}

\speeddictthree{122}{transducerable = continuous theorem}{Products in Categories Defn}{cones words and Cantor spaces defn}{transducers are continuous lemma}
\begin{theorem}[Transducerable = continuous, \Assumed{122}]\label{transducerable = continuous theorem}
A function \(h: \mathfrak{C}_n^d \to \mathfrak{C}_n^d\) is continuous if and only if there is a non-degenerate \((d, n)\)-transducer \(T\) and \(q\in Q_T\) such that \(h= f_{T,q}\).
\end{theorem}
\begin{proof}
\((\Leftarrow):\) This is immediate from Lemma~\ref{transducers are continuous lemma}.

\((\Rightarrow):\) Let \(h: \mathfrak{C}_n^d \to \mathfrak{C}_n^d\) be a continuous map.
 We define a \((d, n)\)-transducer \(T\) and then proceed to show that \(T\) is well-defined, non-degenerate and satisfies \(f_{T, \varepsilon_d} = h\).
\begin{enumerate}
    \item Define \(Q_T := (X_n^*)^d\),  \(D_T := (X_n^*)^d\), \(R_T := (X_n^*)^d\).
    \item Define \(\boldsymbol{\pi}_T: Q_T\times D_T\to Q_T\) by \((q, w)\boldsymbol{\pi}_T: = q w\).
    \item For each \(i\in \{0, 1, \ldots, n-1\}\), define \(\boldsymbol{\lambda}_{T, i}: Q_T\times D_T\to X_n^*\) as follows:
    \begin{enumerate}
    \item If \(q = w = \varepsilon_d\), then \((q, w)\boldsymbol{\lambda}_{T, i} = \varepsilon\);
       \item if \(q =\varepsilon_d\) and \(w \neq  \varepsilon_d\) and \(|(w\mathfrak{C}_n^d)h\pi_i| > 1\), then \((q, w)\boldsymbol{\lambda}_{T, i}\) is defined to be the longest common prefix of all the words in \((w\mathfrak{C}_n^d)h\pi_i\);
       \item if \(q = \varepsilon_d\), \(w \neq \varepsilon_d\), \((w\mathfrak{C}_n^d)h\pi_i = \{x\}\) for some \(x\in X_n^\N\), and \((q, w')\boldsymbol{\lambda}_{T_i}\) is already defined for all \(w'< w\), then \((q, w)\boldsymbol{\lambda}_{T_i}\) is defined to be the prefix of \(x\) with length
       \[\max\left(\{|(w)\pi_0|, |(w)\pi_1|, \ldots, |(w)\pi_{d-1}|\}\cup \makeset{|(q, w')\boldsymbol{\lambda}_{T, i}|}{\(w'< w\)} \right);\]
       \item if \(q \neq \varepsilon_d\), then \((q, w)\boldsymbol{\lambda}_{T, i}\) is defined to be \(((\varepsilon_d, q w)\boldsymbol{\lambda}_{T, i})\lambda_{(\varepsilon_d, q)\boldsymbol{\lambda}_{T, i}}^{-1}\).
    \end{enumerate}
    We then define \(\boldsymbol{\lambda}_T:Q_T\times D_T \to R_T\) by \(\boldsymbol{\lambda}_T = \langle \boldsymbol{\lambda}_{T, 0}, \boldsymbol{\lambda}_{T, 1}, \ldots, \boldsymbol{\lambda}_{T, d-1} \rangle_{(X_n^*)^d}\) (Definition~\ref{Products in Categories Defn}).
\end{enumerate}
To show that \(T\) is well-defined, we first need to show that if \(i\in \{0, 1, \ldots, d-1\}\), \(q,w\in D_T\) and \(q\neq \varepsilon_d\), then \((\varepsilon_d, q)\boldsymbol{\lambda}_{T, i}\) is a prefix of \((\varepsilon_d, q w)\boldsymbol{\lambda}_{T, i}\). There are \(2\) cases to consider:
\begin{enumerate}
   \item If \(|(q w\mathfrak{C}_n^d)h\pi_i| \geq 2\), then as \((q w\mathfrak{C}_n^d)h\pi_i \subseteq (q\mathfrak{C}_n^d)h\pi_i\), it follows that \((\varepsilon_d, q)\boldsymbol{\lambda}_{T, i}\) is a common prefix of all the words in \((q\mathfrak{C}_n^d)h\pi_i\).
   Hence \((\varepsilon_d, q)\boldsymbol{\lambda}_{T, i}\) is a common prefix of all the words in \((q w\mathfrak{C}_n^d)h\pi_i\), so it is a prefix of the longest common prefix \((\varepsilon_d, q w)\boldsymbol{\lambda}_{T, i}\).
    \item If \((q w\mathfrak{C}_n^d)h\pi_i = \{x\}\) for some \(x\in X_n^\N\), then (via either \(3b\) or \(3c\)) \((\varepsilon_d, q)\boldsymbol{\lambda}_{T, i}\) is a finite prefix of \(x\).
    Moreover by definition \((\varepsilon_d, q w)\boldsymbol{\lambda}_{T, i}\) is also a prefix of \(x\) with \(|(\varepsilon_d, q w)\boldsymbol{\lambda}_{T, i}|\geq |(\varepsilon_d, q)\boldsymbol{\lambda}_{T, i}|\).
    Thus \((\varepsilon_d, q w)\boldsymbol{\lambda}_{T, i}\geq (\varepsilon_d, q)\boldsymbol{\lambda}_{T, i}\).
\end{enumerate}
We next need to show that \(T\) is a transducer. If \(q\in Q_T\) and \(s, t\in D_T\) are arbitrary, then \((q, st)\boldsymbol{\pi}_T = q s t = ((q, s)\boldsymbol{\pi}_T, t)\boldsymbol{\pi}_T\), so \(\boldsymbol{\pi}_T\) is an action of \(D_T\) on \(Q_T\). Moreover for all \(i\in \{0, 1, \ldots, d-1\}\) we have
\begin{align*}
 ((q, s)\boldsymbol{\lambda}_T(q s, t)\boldsymbol{\lambda}_T)\pi_i &=  ((\varepsilon_d, q s)\boldsymbol{\lambda}_{T}\pi_i)\lambda_{(\varepsilon_d, q)\boldsymbol{\lambda}_{T}\pi_i}^{-1}((\varepsilon_d, q s t)\boldsymbol{\lambda}_{T}\pi_i)\lambda_{(\varepsilon_d, q s)\boldsymbol{\lambda}_{T}\pi_i}^{-1} \\
&= (((\varepsilon_d, q s t)\boldsymbol{\lambda}_{T}\pi_i)\lambda_{(\varepsilon_d, q s)\boldsymbol{\lambda}_{T}\pi_i}^{-1})\lambda_{(\varepsilon_d, q s)\boldsymbol{\lambda}_{T}\pi_i}\lambda_{(\varepsilon_d, q)\boldsymbol{\lambda}_{T}\pi_i}^{-1} \\
 &=  ((\varepsilon_d, q s t)\boldsymbol{\lambda}_{T, i})\lambda_{(\varepsilon_d, q)\boldsymbol{\lambda}_{T, i}}^{-1}\\
 &=(q, st)\boldsymbol{\lambda}_T\pi_i .
\end{align*}

We next show that \(h = f_{T, \varepsilon_d}\).
Let \(i\in \{0, 1, \ldots, d-1\}\) be arbitrary.
By definition, if \(w'\in (X_n^*)^d\) and \(w'\leq w\in (X_n^\N)^d\), then the word \((\varepsilon_d, w')\boldsymbol{\lambda}_{T, i}\) is a prefix of \((w)h\pi_i\).
So we need only show that for all \(m\in \N\), there is \(k_m\in \N\) such that if \(w'\in (X_n^k)^d\), then \(|(\varepsilon_d, w')\boldsymbol{\lambda}_{T, i}|\geq m\).

For each \(v\in (X_n^m)^d\), we have by assumption that \((v\mathfrak{C}_n^d)h^{-1}\) is clopen. So there is some \(k_v\in \N\) and \(S_v\subseteq (X_n^{k_v})^d\) such that
\[(v\mathfrak{C}_n^d)h^{-1} = \union{s\in S_v} s\mathfrak{C}_n^d.\]
Let \(k_m := \max\left(\makeset{k_v}{\(v\in (X_n^m)^d\)} \cup \{m+1\}\right)\).
Thus if \(w'\in (X_n^{k_m})^d\), then the longest common prefix of all the elements of the set \(((w')\mathfrak{C}_n^d)h\pi_i\) has length at least \(m\). Thus \(|(\varepsilon_d, w')\boldsymbol{\lambda}_{T, i}|\geq m\). It follows that indeed \(h = f_{T, \varepsilon_d}\).

Finally we show that \(T\) is not degenerate. Let \(q\in Q_T\) and \(w\in (X_n^\N)^d\) be arbitrary. Suppose for a contradiction that \((w)f_{T, q}\not\in (X_n^\N)^d\).
It follows that
\[(q w)h =(w)f_{T, \varepsilon_d} = (\varepsilon_d, q)\lambda_T(w)f_{T, q} \not\in (X_n^\N)^d,\]
this is a contradiction.
\end{proof}

It will be useful to have concrete transducers representing homeomorphisms of Cantor spaces,
we constructed such transducers in the proof of the previous theorem, but these transducers have a rather messy definition.
In the case of homeomorphisms this can be made a bit simpler, as we will now see.
\speeddictfour{123}{transducing homeomorphisms defn}{automorphism groups defn}{cones words and Cantor spaces defn}{minimal transducer defn}{inj_clo}
\begin{defn}[Transducing homeomorphisms, \Assumed{123}]\label{transducing homeomorphisms defn}
If \(h\in \Aut(\mathfrak{C}_n^d)\), then we define \(T_h\) to be the \((d, n)\)-transducer with
\begin{enumerate}
    \item \(Q_{T_h}:= (X_n^*)^d\), \(D_{T_h}:= (X_n^*)^d\), and \(R_{T_h}:= (X_n^*)^d\).
    \item \((s, t)\boldsymbol{\pi}_{T_h}= st\).
    \item \(((s, t)\boldsymbol{\lambda}_{T_h})\pi_i\) is \((b)\lambda_a^{-1}\), where \(b\) is the longest common prefix of the words in the set \(((st\mathfrak{C}_n^d)h)\boldsymbol{\pi}_i\) and \(a\) is the longest common prefix of the words in the set \(((s\mathfrak{C}_n^d)h)\boldsymbol{\pi}_i\).
\end{enumerate}
Moreover, we define \(M_h:= M_{T_h}\) and \(e_h := (\varepsilon_d){q_{T_h}}_Q\). By Lemma~\ref{inj_clo} we have \(f_{M_h, e_h}= h\). Thus we call \(M_h\) the \textit{minimal transducer for }\(h\).
\end{defn}

\speeddictthree{124}{inj_clo}{induced maps preserved remark}{transducerable = continuous theorem}{transducing homeomorphisms defn}
\begin{lemma}[Minimal transducers are nice, \Assumed{124}]\label{inj_clo}
If \(h\in \Aut(\mathfrak{C}_n^d)\) then \(T_h\) is a well-defined non-degenerate \((d, n)\)-transducer and \(f_{M_h, e_h} = h\). Moreover if \(q\in Q_{M_h}\), then \(f_{M_h, q}\) is a homeomorphism from \(\mathfrak{C}_n^d\) to a clopen subset of \(\mathfrak{C}_n^d\).
\end{lemma}
\begin{proof}
Note that when \(h\in \Aut(\mathfrak{C}_n^d)\) is a homeomorphism and \(i\in \{0, 1, \ldots, d-1\}\), the longest common prefix of the words in \((\mathfrak{C}_n^d)h\pi_i\) is \(\varepsilon\).
Moreover as \(h\) maps open sets to open sets, if \(w\in (X_n^*)^d\) then \((w\mathfrak{C}_n^d)h\pi_i\) is an open subset of \(\mathfrak{C}_n\).
In particular \(|(w\mathfrak{C}_n^d)h\pi_i|\neq 1\). It follows that \(T_h\) is identical to the transducer \(T\) defined in the proof of Theorem~\ref{transducerable = continuous theorem}.
Thus \(f_{T_h, \varepsilon_d} = h\).
By Remark~\ref{induced maps preserved remark}, it follows that \(f_{T_h, \varepsilon_d}=f_{M_h, e_h}\), so in particular \(f_{M_h, e_h} = h\).

Let \(q\in Q_{M_h}\) be arbitrary and let \(w'\in Q_{T_h}\) be such that \((w'){q_{T_h}}_Q = q\). We next show that \(f_{M_h, q}\) is injective. Suppose that \(w_0, w_1\in \mathfrak{C}_n^d\) and \((w_0)f_{M_h, q}=(w_1)f_{M_h, q}\). We have 
\[(w' w_0)h = (e_h,w' )\boldsymbol{\lambda}_{M_h}(w_0)f_{M_h, q} =(e_h,w' )\boldsymbol{\lambda}_{M_h}(w_1)f_{M_h, q} =(w' w_1)h .\]
Thus \(w' w_0=w' w_1\), and hence \(w_0=w_1\) as required.

 As \(w'\mathfrak{C}_n^d\) is a clopen set, the set \((w'\mathfrak{C}_n^d)h\) is also a clopen set. Moreover if \(x\in \mathfrak{C}_n^d\) is arbitrary, then
\[(w' x)h = (e_h, w')\boldsymbol{\lambda}_{M_h}(x)f_{M_h, q}.\]
Thus as \(h\) is a homeomorphism, \(f_{M_h, q}\) is a homeomorphism onto its image. Moreover every element of \((w'\mathfrak{C}_n^d)h\) is greater than \((e_h, w')\boldsymbol{\lambda}_{M_h}\). So there is some finite \(S\subseteq (X_n^*)^d\) such that
\[(w'\mathfrak{C}_n^d)h=\union{s\in S} (e_h, w')\boldsymbol{\lambda}_{M_h}s\mathfrak{C}_n^d.\]
It follows that \((\mathfrak{C}_n^d)f_{M_h, q}=\union{s\in S} s\mathfrak{C}_n^d\) and is thus clopen as required.
\end{proof}

The following theorem establishes that the GNS notion of a reduced transducer (see Proposition 2.8 of \cite{GNS2000}) representing a homeomorphism still applies in our context.
\speeddictfive{125}{unique minimal thm}{transducers homomorphisms defn}{minimal transducer defn}{minimal transducers are valid lemma}{transducing homeomorphisms defn}{inj_clo}
\begin{theorem}[Minimal transducers are unique, \Assumed{125}]\label{unique minimal thm}
Let \(d\in \N\backslash \{0\}\) and \(n\in \N\backslash\{0, 1\}\). Suppose that \(h\in \Aut(\mathfrak{C}_n^d)\), and \((T, q)\) is a pair such that:
\begin{enumerate}
    \item \(T\) is a \((d, n)\)-transducer, \(q\in Q_T\) and \(f_{T, q} = h\);
    \item if \(p_0, p_1\in Q_T\) and \(p_0 \sim_{M_T} p_1\), then \(p_0 = p_1\);
    \item \(Q_T = (q, D_T)\boldsymbol{\pi}_T\);
    \item if \(p\in Q_T\), \(w\in R_T\) and \((\mathfrak{C}_n^d)f_{T, p}\subseteq w\mathfrak{C}_n^d\), then \(w = \varepsilon_d\) (we call this condition \textbf{complete response}).
\end{enumerate}
Then there is a unique strong transducer isomorphism \(\phi: T\to M_h\) with \((q)\phi_Q = e_h\).
Moreover if \(h\in \Aut(\mathfrak{C}_n^d)\) is arbitrary, then the pair \((M_h, e_h)\) satisfies all of the above conditions.  
\end{theorem}
\begin{proof}
Let \(T\) be a \((d, n)\)-transducer and \(q\in Q_T\) be such that the above conditions are satisfied. 
If \(\phi: T\to M_h\) is as required then we must have:
\begin{enumerate}
    \item Both \(\phi_{D}, \phi_{R}\) are the identity function on \((X_n^*)^d\);
    \item if \(w\in D_T\) then
    \(((q, w)\boldsymbol{\pi}_T)\phi_Q = ((q)\phi_Q, w)\boldsymbol{\pi}_{M_h} =  (e_h, w)\boldsymbol{\pi}_{M_h}.\)
\end{enumerate}
As \(Q_T = (q, D_T)\boldsymbol{\pi}_T\), it follows that the above two conditions define \(\phi\). Thus if the required \(\phi\) exists, then it is unique.

Define \(\psi_Q: Q_{T_h} \to Q_T\) by \((p)\psi_Q = (q, p)\pi_T\). As \(D_T = Q_{T_h}\), it follows from the second condition on \(T\) that \(\psi_Q\) is surjective. Let \(\psi_D, \psi_R\) be the identity functions on \((X_n^*)^d\). It follows that if \(\psi:=(\psi_Q, \psi_D, \psi_T)\) is a transducer homomorphism, then \(\psi\) is a strong quotient map. 
We now show that \(\psi\) is a transducer homomorphism.

Let \(p\in Q_{T_h}\), \(s, t\in (X_n^*)^d\) be arbitrary. Then
\[(p, s)\boldsymbol{\pi}_{T_h}\phi_{D} =(q,(p, s)\boldsymbol{\pi}_{T_h})\boldsymbol{\pi}_{T} =(q, p s)\boldsymbol{\pi}_{T}=((q, p)\boldsymbol{\pi}_T, s)\boldsymbol{\pi}_{T}= ((p)\phi_Q, s)\boldsymbol{\pi}_{T}.\]
We need to show that \((p, s)\boldsymbol{\lambda}_{T_h} = ((p)\phi_Q, s) \boldsymbol{\lambda}_{T}\). For all \(i\in \{0, 1, \ldots, d-1\}\) we have by definition that \((p, s)\boldsymbol{\lambda}_{T_h}\pi_i= (b)\lambda_a^{-1}\), where \(b\) is the longest common prefix of the elements of \((p s\mathfrak{C}_n^d)h\pi_i\) and \(a\) is the longest common prefix of the elements of \((p\mathfrak{C}_n^d)h\pi_i\).
Define \(a' := (q, p)\boldsymbol{\lambda}_{T}\), \(b' := (q, p s)\boldsymbol{\lambda}_{T}\), 
as \(T\) is a transducer, we have
\[a'((p)\phi_Q, s) \boldsymbol{\lambda}_{T}=(q, p)\boldsymbol{\lambda}_{T}((q, p)\boldsymbol{\pi}_T, s) \boldsymbol{\lambda}_{T}= (q, p s) \boldsymbol{\lambda}_{T}=b'.\]
Thus \(((p)\phi_Q, s) \boldsymbol{\lambda}_{T}= (b')\lambda_{a'}^{-1}\). 
To conclude that \((p, s)\boldsymbol{\lambda}_{T_h} = ((p)\phi_Q, s) \boldsymbol{\lambda}_{T}\), it suffices to show that \(a'= a\) and \(b'=b\).
We show \(a'=a\), showing that \(b'=b\) is almost identical.
As \(a' =(q, p)\boldsymbol{\lambda}_{T}\) and \(f_{T, q} = h\), it follows that \(a'\) is a prefix of all words in the set \((p\mathfrak{C}_n^d)h\pi_i\).
In particular \(a'\leq a\). 
Let \(w'\in X_n^*\) be such that \(a' w' =a\).
For all \(x\in \mathfrak{C}_n^d\), we have
\[ a'((x)f_{T, (q, p)\boldsymbol{\pi}_T}\pi_i)=(p x)f_{T, q}\pi_i=(p x)h\pi_i  \in a\mathfrak{C}_n=a' w'\mathfrak{C}_n.\]
Thus \((\mathfrak{C}_n^d)f_{T, (q, p)}\boldsymbol{\pi}_T\pi_i \subseteq w' \mathfrak{C}_n\).
As \(T\) has complete response, it follows that \(w'\leq (\varepsilon_d)\pi_i= \varepsilon\) and thus \(a=a'\).

We now have that \(\psi: T_h\to T\) is a strong quotient map. 
By Lemma~\ref{minimal transducers are valid lemma}, there is a strong quotient map \(\phi: T\to M_h\) with \(\psi\phi= q_{T_h}\).
From this equality it follows that \((q)\phi= (\varepsilon_d){q_{T_h}}_Q = e_h\).
We next show that \(\phi_Q\) is injective. Let \(p_0, p_1\in Q_T\) be arbitrary with \((p_0)\phi_Q = (p_1)\phi_Q\), we will show that \(p_0= p_1\).
Let \(p_0', p_1'\in Q_{T_h}\) be such that \((p_0')\psi_Q = p_0\) and \((p_1')\psi_Q = p_1\). We have
\[(p_0')q_{T_h}= (p_0')\psi_Q\phi_Q= (p_0)\phi_Q= (p_1)\phi_Q=(p_1')\psi_Q\phi_Q = (p_1')q_{T_h}.\]
Thus \(p_0'\sim_{M_{T_h}} p_1'\), and as \(\psi\) is a transducer homomorphism it follows that \(p_0\sim_{M_{T}} p_1\). By the second assumption on \(T\), it follows that \(p_0 = p_1\).
As \(\psi_Q\) is injective, it follows that \(\phi\) is a strong transducer isomorphism as required.

Finally we show that \(M_h\) and \(e_h\) satisfy the required conditions. The first condition is immediate from Lemma~\ref{inj_clo}. The second condition follows from the definition of \(M_h\). The third condition follows as \((p)q_{T_{h}}=(e_h, p)\boldsymbol{\pi}_{M_h}\) for all \(p\in Q_{T_h}\). The forth condition is immediate from the definition on \(T_h\) together with Remark~\ref{induced maps preserved remark}.
\end{proof}

\section{Generalizing the synchronizing homeomorphisms of Bleak, Cameron, Maissel, Navas, and Olukoya}

We now have a class of transducers for describing continuous maps.
We restrict our transducers to those which represent automorphisms of \(dV_n\).
As was the case in \cite{bleak2016}, this is done with a synchronization condition.

\speeddictthree{126}{synchronizing defn}{function defn}{binary relations defns}{dn transducer def}
\begin{defn}[Synchronizing transducers, \Assumed{126}]\label{synchronizing defn}
Suppose that \(T\) is a \((d, n)\)-transducer and \(k\in \N\). 
We say that \(T\) is \textit{synchonizing at level }\(k\) if for all \(p, q\in Q_T\) and \(w\in (X_n^*)^d\) with \(\min\left(\makeset{|(w)\pi_i|}{\(i\in \{0, 1, \ldots, d-1\}\)}\right)\geq k\), we have \[(p, w)\pi_T = (q, w)\pi_T.\]
We say that \(T\) is \textit{synchronizing} if there is a level at which it is synchronizing and we call the smallest such level the \textit{synchronizing length} of \(T\). In this case we define the function
\[\mathfrak{s}_T:= \makeset{(w, q)\in (X_n^*)^d\times Q_T}{for all \(p\in Q_T\) we have \((w,p)\boldsymbol{\pi}_T=q\)}.\]
The map \(\mathfrak{s}_T\) is essentially \(\boldsymbol{\pi}_T\) restricted to the part of its domain where the input state is not needed. The image of \(\mathfrak{s}_T\), denoted \(\C(T)\), is called the \textit{core} of \(T\).
\end{defn}

\speeddicttwo{127}{Core Transducers remark}{subtransducer def}{synchronizing defn}
\begin{remark}[Core transducers, \Assumed{127}]\label{Core Transducers remark}
Suppose that \(T\) is a \((d, n)\)-transducer with synchronizing level \(k\).
If \(q\in \C(T)\), \(s\in (\{q\})\mathfrak{s}_T^{-1}\) and \(w\in D_T\) then
\[(q, w)\boldsymbol{\pi}_T= ((s)\mathfrak{s}_T, w)\boldsymbol{\pi}_T  = (s w)\mathfrak{s}_T\in \C(T).\]
Thus \(\C(T)\) is a subtransducer of \(T\).
Moreover \(\mathfrak{s}_T\subseteq \mathfrak{s}_{\C(T)}\) and \(\C(T)= ((X_n^k)^d)\mathfrak{s}_T\), so \(\C(T)\) has at most \(n^{k d}\) states and has synchronizing length at most \(k\).
\end{remark}

The core of a synchronizing transducer, consists precisely of the state which are ``always reachable", in particular it is the smallest subtransducer which still has the same domain and range as the original transducer.

\speeddicttwo{128}{identity states defn}{subtransducer def}{unique minimal thm}
\begin{defn}[Identity states, \Assumed{128}]\label{identity states defn}
If \(I\in \Aut(\mathfrak{C}_n^d)\) is the identity map, then (by Theorem~\ref{unique minimal thm}) \(M_I\) is a transducer with one state \(e_I\) and output function given by \((e_I, w)\boldsymbol{\lambda}_{M_I} = w\) (for all \(w\in (X_n^*)^d\)).
We call a transducer strongly isomorphic to this transducer an \textit{identity transducer}.

If \(T\) is a \((d, n)\)-transducer and \(q\in Q_T\) is such that \(\{q\}\) is a subtransducer of \(T\) isomorphic to an identity transducer, then we call \(q\) an \textit{identity state}.
\end{defn}

We can now show that, as was the case with \(G_{n, r}\) \cite{bleak2016}, the elements of \(dV_n\) are precisely those elements with ``trivial" cores.

\speeddictfive{129}{nV_placement}{prefix codes and dVn defn}{unique minimal thm}{synchronizing defn}{Core Transducers remark}{identity states defn}
\begin{proposition}[\(dV_n\) transducers, \Assumed{129}]\label{nV_placement}
If \(d\in \N\backslash\{0\}\), \(n\in \N\backslash\{0, 1\}\) and \(h\in \Aut(\mathfrak{C}_n^d)\), then \(h\in dV_n\) if and only if \(M_h\) is a synchronizing transducer whose core is an identity transducer.
\end{proposition}
\begin{proof}
\((\Rightarrow):\) Let \(h\in dV_n\), and let \(F_1, F_2\) be complete prefix codes of \(\mathfrak{C}_n^d\) and \(\phi:F_1\to F_2\) be a bijection such that \(h= f_\phi\) (as in Definition~\ref{prefix codes and dVn defn}). Define
\[k := \max\left(\makeset{|(w)\pi_i|}{\(i\in \{0,1 , \ldots, d-1\}, w\in F_1\)}\right).\]

Let \(w\in (X_n^*)^{d}\) be arbitrary with \(\min\left(\makeset{|(w)\pi_i|}{\(i\in \{0, 1, \ldots, d-1\}\)}\right) \geq k\).
We show for all \(s\in (X_n^*)^d\) that \((w, s)\boldsymbol{\lambda}_{T_h} =s\).

By the definitions of \(k\) and \(F_1\), there is a unique \(w'\in F_1\) with \(w'\leq w\).
It follows that \((e_h, w)\boldsymbol{\lambda}_{M_h}=(\varepsilon_d, w)\boldsymbol{\lambda}_{T_h} = (w)\lambda_{w'}^{-1}\lambda_{(w')\phi}\).
So if \(s\in (X_n^*)^d\) is arbitrary, then we have
\[(e_h, w s)\boldsymbol{\lambda}_{M_h}= (w s)\lambda_{w'}^{-1}\lambda_{(w')\phi}= (w)\lambda_{w'}^{-1}\lambda_{(w')\phi}s\]
and thus \((w, s)\boldsymbol{\lambda}_{T_h} =s\).
It follows that if \(v, w\in (X_n^*)^{d}\) are arbitrary with \[\min\left(\makeset{|(v)\pi_i|}{\(i\in \{0, 1, \ldots, d-1\}\)}\right) \geq k,\quad  \min\left(\makeset{|(w)\pi_i|}{\(i\in \{0, 1, \ldots, d-1\}\)}\right) \geq k,\]
then \((v){q_{T_h}}_Q=(w){q_{T_h}}_Q\) and moreover this state is an identity state.
Thus \(M_h\) is a synchronizing of length at most \(k\) and the core of \(M_h\) consists of a single identity state.

\((\Leftarrow):\) Let \(h\in \Aut(\mathfrak{C}_n^d)\) be such that \(M_h\) has an identity core.
We show that \(h\in dV_n\).
Let \(k\) be the synchronizing length of \(M_h\) and let \(I\) be the core state of \(M_h\).
Define \(F_1:= (X_n^k)^d\) and note that \(F_1\) is a complete prefix code for \(\mathfrak{C}_n^d\). Define \(\phi:F_1 \to (X_n^*)^d\) by \((w)\phi = (e_h, w)\boldsymbol{\lambda}_{M_h}\). If \(x\in \mathfrak{C}_n^d\) and \(w\in (X_n^k)^d\) are arbitrary then
\[(w x)h = (w x)f_{M_h, e_h} =  (e_h, w)\boldsymbol{\lambda}_T (x)f_{M_h, I}=(w)\phi x = (w x)\lambda_w^{-1}\lambda_{(w)\phi}.\]
Let \(F_2 := \im(\phi)\). As \(h\) is a homeomorphism, it follows from the equality \((w x)h= (w x)\lambda_w^{-1}\lambda_{(w)\phi}\) that \(F_2\) is a complete prefix code and \(\phi\) is a bijection between \(F_1\) and \(F_2\).
Thus \(h = f_\phi \in dV_n\).
\end{proof}

\speeddicttwo{165}{picture definitions}{cones words and Cantor spaces defn}{dn transducer def}
\begin{defn}[Defining transducers with diagrams, \Assumed{165}]\label{picture definitions}
As we see in the Example~\ref{baker's map example}, it is useful to think about transducers as diagrams consisting of collections of dots (representing states) with arrows between them (representing the transition and output functions).
If \(T\) is a transducer and \(X\) is a generating set for \(D_T\), then we can do this by having an arrow originating from each state for each element of \(X\).
We give each arrow from a state \(q\) to a state \(p\) a label of the form \(x/s\) where \(x\in X\) and \(s\in R_T\). The meaning of this being that 
\[(q, x)\boldsymbol{\pi}_T = p \quad\text{and}\quad (q, x)\boldsymbol{\lambda}_T = s.\]
As \(X\) generates \(D_T\), the conditions in the definition of a transducer determine the rest of the maps \(\boldsymbol{\pi}_T\) and \(\boldsymbol{\lambda}_T\). 
If an element of \(D_T\) can be expressed as a product of elements of \(X\) in multiple ways then this may result in \(\boldsymbol{\pi}_T\) or \(\boldsymbol{\lambda}_T\) not being well-defined by a particular diagram, so not all diagrams of this type give well-defined transducers.

We will be particularly interested in \((d, n)\)-transducers using the generating set 
\[\makeset{x_{d,i}}{\(x\in X_n, i\in \{0, 1, \ldots, d-1\}\)}\]
(recall Definition~\ref{cones words and Cantor spaces defn}). In the case that \(n=1\) this is particularly easy as each element of \(X_n^*\) can be expressed as a product of generators in a unique fashion, thus all diagrams define valid transducers (as was the case in GNS).
\end{defn}

\speeddictfive{130}{baker's map example}{prefix codes and dVn defn}{unique minimal thm}{identity states defn}{nV_placement}{picture definitions}
\begin{example}[Baker's map, \Assumed{130}]\label{baker's map example}
Unlike for \(V_n\), the minimal transducers for elements of \(dV_n\) can sometimes be infinite. For example if we define \(F_1 := \{(0, \varepsilon), (1, \varepsilon)\}\),  \(F_2 := \{(\varepsilon, 0), (\varepsilon, 1)\}\) and \(\phi:F_1\to F_2\) by
\[(0, \varepsilon)\phi = (\varepsilon, 0), \quad (1, \varepsilon)\phi = (\varepsilon, 1).\]
Then the map \(b:= f_\phi\in 2V\) (called the \textit{baker's map}) has the property that \(M_b\) has infinitely many states.

This can be seen by observing that if \(w\in (X_2^*)^2\) and \((w)\pi_0= \varepsilon\), then \((w, (0, \varepsilon))\boldsymbol{\lambda}_T = (0, w)\). Thus if \(v, w\in (X_2^*)^2\) are distinct then \((e_b, v)\boldsymbol{\pi}_{M_b} \neq (e_b, v)\boldsymbol{\pi}_{M_b}\). 
In Figure~\ref{baker}, we see a diagram of the transducer \(M_b\).
\end{example}
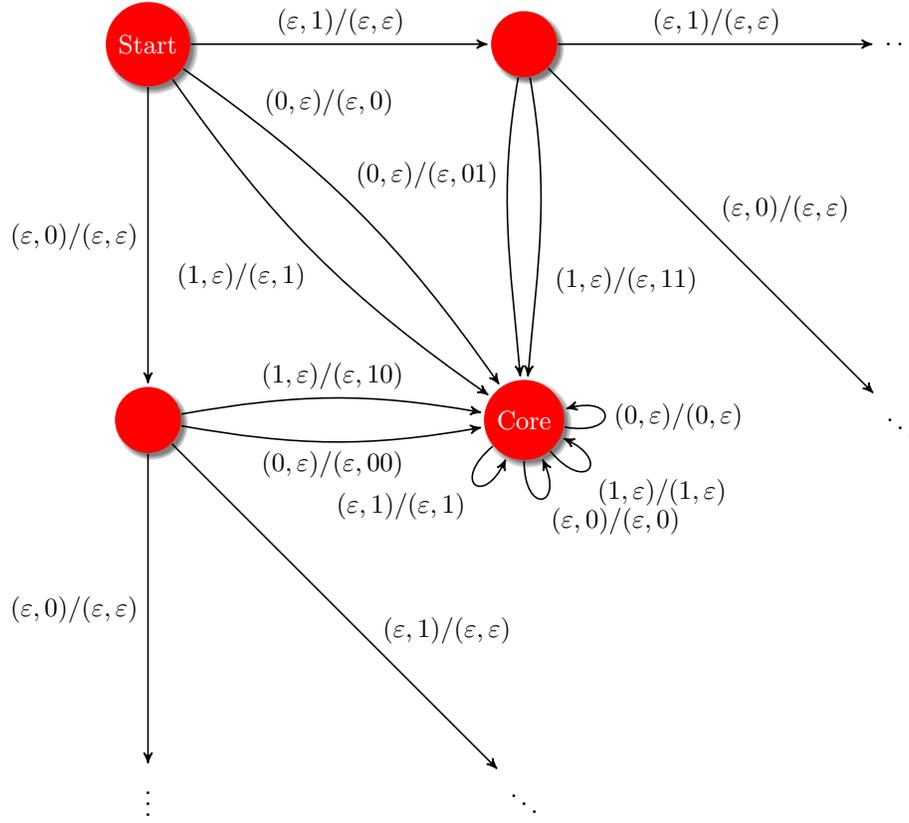
\begin{figure}
\begin{center}
\begin{tikzpicture}[->,>=stealth',shorten >=1pt,auto,node distance=5cm,on grid,semithick,
                    every state/.style={fill=red,draw=none,circular drop shadow,text=white}]  
  \node [state] (A)                {Start};
  \node [state] (B)  [right= of A] {};
  \node [state] (C)  [below= of B] {Core};
  \node [state] (D)  [below= of A] {};
  \node [] (E) [below= of D] {$\vdots$};
  \node [] (G) [right= of E] {$\ddots$};
  \node [] (F) [right= of B] {$\ldots$};
  \node [] (H) [below= of F] {$\ddots$};
 \path [->]
 (A) edge [out=305,in= 145] node [swap]
 {$(1, \varepsilon)/(\varepsilon, 1)$} (C)
 (A) edge [out=325,in=125] node [swap,xshift = 17pt,yshift = 50pt]
 {$(0, \varepsilon)/(\varepsilon, 0)$} (C)
 (A) edge [out=0,in=180] node [swap, yshift= 17pt]
 {$(\varepsilon,1)/(\varepsilon,\varepsilon)$} (B)
 (A) edge [out=270,in=90] node [swap]
 {$(\varepsilon,0)/(\varepsilon,\varepsilon)$} (D)
 
 (B) edge [out=280,in=85] node [swap, xshift=63pt, yshift = -20pt]
 {$(1, \varepsilon)/(\varepsilon, 11)$} (C)
 (B) edge [out=260, in= 95] node [swap, yshift= 20pt]
 {$(0, \varepsilon)/(\varepsilon, 01)$} (C)
 (B) edge [out=00,in=180] node [swap, yshift=17pt]
 {$(\varepsilon,1)/( \varepsilon,\varepsilon)$} (F)
 (B) edge [out=315,in=135] node []
 {$(\varepsilon,0)/( \varepsilon,\varepsilon)$} (H)
 
 (C) edge  [out=310,in=335, loop]
 node  [swap]{$(1, \varepsilon)/(1, \varepsilon)$} (C)
 (C) edge  [out=350,in=15, loop]
 node [swap]{$(0, \varepsilon)/(0, \varepsilon)$} (C)
 (C) edge [out=220,in=245, loop] node [swap]
 {$(\varepsilon,1)/( \varepsilon,1)$} (C)
 (C) edge [out=270,in=295, loop] node [swap]
 {$(\varepsilon,0)/( \varepsilon,0)$} (C)
 
 (D) edge [out=10,in=170] node [swap, yshift=17pt]
 {$(1, \varepsilon)/(\varepsilon, 10)$} (C)
 (D) edge [out=350,in=190] node [swap]
 {$(0, \varepsilon)/(\varepsilon, 00)$} (C)
 (D) edge [out=315,in=135] node [swap, xshift=70pt]
 {$(\varepsilon,1)/(\varepsilon, \varepsilon)$} (G)
 (D) edge [out=270, in= 90] node [swap]
 {$(\varepsilon,0)/(\varepsilon, \varepsilon)$} (E);
 \end{tikzpicture}
 \end{center}
\caption{Part of a minimal transducer with \((X_2^*)^2\) as domain and range. This represents the baker's map in \(2V\) (this transducer is infinite and every state has an edge with the core as its target).}\label{baker}
\end{figure}


\speeddicttwo{131}{synchronizing homeomorphisms}{unique minimal thm}{synchronizing defn}
\begin{defn}[Synchronizing homeomorphisms, \Assumed{131}]\label{synchronizing homeomorphisms}
We say an element \(h\in \Aut(\mathfrak{C}_n^d)\) is \textit{synchronizing} if \(M_h\) is synchronizing. We define \(d\mathcal{S}_{n,1}\) to be the set of synchronizing elements of \(\Aut(\mathfrak{C}_n^d)\).
\end{defn}
The subscript \(1\) in the above definition is due the the overlap between these sets of homeomorphisms and the homeomorphism groups \(\mathcal{S}_{n,r}\) from \cite{bleak2016} (\(\mathcal{S}_{n,1}=1\mathcal{S}_{n,1}\)). 
Later we will use the notations \(d\mathcal{B}_{n, 1}\) and \(d\mathcal{O}_{n, 1}\) for the same reason.

We will need to be able to compute the minimal transducer representative (see Theorem~\ref{unique minimal thm}) of a product of homeomorphisms from the minimal transducers representing the individual homeomorphisms.
For this we will need to be able to convert a transducer defining a homeomorphism into one with complete response.

A means of doing this is provided by GNS (see Proposition 2.16 of \cite{GNS2000}), we now show that a similar construction applies to our context.

\speeddictthree{132}{Removing Incomplete Response}{inj_clo}{unique minimal thm}{removing incomplete responce works}
\begin{defn}[Removing incomplete response, \Assumed{132}]\label{Removing Incomplete Response}
Suppose that \(T\) is a \((d, n)\)-transducer and for all \(q\in Q_T\) the function \(f_{T, q}\) is a homeomorphism from \(\mathfrak{C}_n^d\) to a clopen subset of \(\mathfrak{C}_n^d\).
We define 
\[CR(T):= (Q_T, D_T, R_T, \boldsymbol{\pi}_T, \boldsymbol{\lambda}_{CS(T)}).\]
The map \(\boldsymbol{\lambda}_{CR(T)}: Q_T\times D_T\to R_T\) is defined by
\[(q, s)\boldsymbol{\lambda}_{CR(T)}\pi_i = (b)\lambda_a^{-1},\]
where \(b\) is the longest common prefix of the elements of \((s\mathfrak{C}_n^d)f_{T, q}\pi_i\) and \(a\) is the longest common prefix of the elements of \((\mathfrak{C}_n^d)f_{T, q}\pi_i\). We show in Lemma~\ref{removing incomplete responce works} that \(CR(T)\) is a well-defined \((d, n)\)-transducer.
\end{defn}

\speeddictthree{133}{removing incomplete responce works}{inj_clo}{unique minimal thm}{Removing Incomplete Response}
\begin{lemma}[Removing incomplete response works, \Assumed{133}]\label{removing incomplete responce works}
Suppose that \(T\) is a \((d, n)\)-transducer and for all \(q\in Q_T\) the function \(f_{T, q}\) is a homeomorphism from \(\mathfrak{C}_n^d\) to a clopen subset of \(\mathfrak{C}_n^d\). The object \(CR(T)\) defined in Definition~\ref{Removing Incomplete Response} is a \((d, n)\)-transducer with complete response (see Theorem~\ref{unique minimal thm}).

Moreover \(CR(T)\) is non-degenerate and if \(q\in Q_T =Q_{CR(T)}\) and \(p\) is the largest common prefix of the elements of \((\mathfrak{C}_n^d)f_{T, q}\), then \(f_{CR(T), q}=f_{T, q}\lambda_{p}^{-1} \). 
\end{lemma}
\begin{proof}
We first show that \(CR(T)\) is a \((d, n)\)-transducer. We need to show that the image of \(\boldsymbol{\lambda}_{CR(T)}\) is contained in \(R_T\). For each \(p\in Q_{T}\), the map \(f_{T, q}\) maps open sets to open sets, so if \(i\in \{0,1 , \ldots, d-1\}\) then \((s\mathfrak{C}_n^d)f_{T, q}\pi_i\) is an open subset of \(\mathfrak{C}_n\). Thus \(|(s\mathfrak{C}_n^d)f_{T, q}\pi_i| \neq 1\) and so the longest common prefix of the words in \((s\mathfrak{C}_n^d)f_{T, q}\pi_i\) is finite.

Let \(p\in Q_{CR(T)}\) and \(s, t\in D_{CR(T)}\) be arbitrary.
As \(T\) is a transducer, the map \(\boldsymbol{\pi}_{T}=\boldsymbol{\pi}_{CR(T)}\) is an action.
We need to show that
\[(p, st)\boldsymbol{\lambda}_{CR(T)} = (p, s)\boldsymbol{\lambda}_{CR(T)}((p, s)\boldsymbol{\pi}_{CR(T)}, t)\boldsymbol{\lambda}_{CR(T)}.\]
We have by assumption that \((p, st)\boldsymbol{\lambda}_{T} = (p, s)\boldsymbol{\lambda}_{T}((p, s)\boldsymbol{\pi}_{T}, t)\boldsymbol{\lambda}_{T}\), and thus if \(x\in \mathfrak{C}_n^d\) then
\[(s x)f_{T, p} = (p, s)\boldsymbol{\lambda}_{CR(T)}(x)f_{T, (p, s)\boldsymbol{\pi}_{T}}\]
so if \(i\in \{0, 1, \ldots, d-1\}\), then \(a\) is a prefix of all the words in the set \((\mathfrak{C}_n^d)f_{T, p}\pi_i\) if and only if \((p, s)\boldsymbol{\lambda}_{CR(T)}\pi_i a\) is a common prefix of all the words in the set \((s\mathfrak{C}_n^d)f_{T, p}\pi_i\). The required condition on \(\boldsymbol{\lambda}_{CR(T)}\) follows.

We next show that \(CR(T)\) has complete response. Let \(p\in Q_{CR(T)}\) and suppose that \(w\in (X_n^*)^d\) is such that \((\mathfrak{C}_n^d)f_{CR(T), q}\subseteq w\mathfrak{C}_n^d\). Let \(i\in \{0,1 , \ldots, d-1\}\) be arbitrary and let \(a\) be the longest common prefix of all the words in the set \(((\mathfrak{C}_n^d)f_{T, p})\pi_i\).
By the definition of \(\boldsymbol{\lambda}_{CS(T)}\), \(a((w)\pi_i)\) is also a common prefix of all the words in the set \(((\mathfrak{C}_n^d)f_{T, p})\pi_i\).
Hence \(a\leq a((w)\pi_i) \leq a\) and \(((w)\pi_i) = \varepsilon\).
As \(i\) was arbitrary it follows that \(w = \varepsilon_d\) as required.

Now let \(q\in Q_{T}\) be arbitrary. By assumption, the set \((\mathfrak{C}_n^d)f_{T, q}\) is open, so for each \(i\in \{0, 1, \ldots, d-1\}\), the set \((\mathfrak{C}_n^d)f_{T, q}\pi_i\) is open.
Thus the element \(p\in (X_n^*)^d\) with each \((p)\pi_i\) the longest common prefix of the words in the set \((\mathfrak{C}_n^d)f_{T, q}\pi_i\) is the maximum element of \((X_n^*)^d\) smaller than all the elements of \((\mathfrak{C}_n^d)f_{T, q}\).

For all \(w\in \mathfrak{C}_n^d\) and \(w'\in (X_n^*)^d\) with \(w'\leq w\), we have that \(p(q, w')\boldsymbol{\lambda}_{CR(T)}\) is the largest element of \((X_n^*)^d\) which is smaller than all the elements of \((w'\mathfrak{C}_n^d)f_{T, q}\).
In particular \(p(q, w')\boldsymbol{\lambda}_{CR(T)} \leq (w)f_{T, q}\).
As \(h\) is continuous, it follows that if \(m\in \N\), then there is \(k\in \N\) such that for all \(w'\in (X_n^k)^d\), there is \(v\in (X_n^m)^d\) with \((w'\mathfrak{C}_n^d)f_{T, q}\subseteq v\mathfrak{C}_n^d\).
Thus \(p(w)f_{CR(T), q}\in \mathfrak{C}_n^d\).
As \(p(w)f_{CR(T), q}\) is also comparable with \((w)f_{T, q}\), it follows that \((w)f_{T, q}= p(w)f_{CR(T), q}\).
As \(w\in \mathfrak{C}_n^d\) was arbitrary, the result follows.
\end{proof}

\speeddictfour{134}{Cores From Cores lemma}{transducers composition defn}{transducers are continuous lemma}{synchronizing defn}{Core Transducers remark}
\begin{lemma}[Cores from cores, \Assumed{134}]\label{Cores From Cores lemma}
If \(A, B\) are synchronizing, non-degenerate \((d, n)\)-transducers then so is their composite \(AB\). Moreover \(\C(AB)\) is a subtransducer of the composite transducer \(\C(A)\C(B).\)
\end{lemma}
\begin{proof}
Let \(k_A, k_B\) be the synchronizing lengths of \(A\) and \(B\) respectively.
By Remark~\ref{Core Transducers remark}, we have \(\C(A)\) is finite.
Moreover by Lemma~\ref{transducers are continuous lemma}, for each \(q\in \C(A),\) there is \(k_q\in \N\) such that if \(w\in (X_n^*)^d\), \(\min\left(\makeset{(w)\pi_i}{\(i\in \{0, 1, \ldots, d-1\}\)}\right) \geq k_q\) and \(i\in \{0, 1, \ldots, d-1\}\) then \(|(q, w)\boldsymbol{\lambda}_A| \geq k_{B}\).

Let \((p_0, p_1)\in Q_{AB}\) and \(w\in D_{AB} = (X_n^*)^d\) be arbitrary such that
\[\min\left(\makeset{(w)\pi_i}{\(i\in \{0, 1, \ldots, d-1\}\)}\right) \geq k_A + \max\left(\makeset{k_q}{\(q\in \C(A)\)}\right).\]
It suffices to show that \(((p_0, p_1), w)\boldsymbol{\pi}_{AB}\) is an element of \(\C(A)\C(B)\) which does not depend of \((p_0, p_1)\). Let \(w = w_0w_1\) where \(w_0\in (X_n^{k_A})^d\). Then
\begin{align*}
    ((p_0, p_1), w)\boldsymbol{\pi}_{AB} &= ((p_0, w)\boldsymbol{\pi}_{A}, (p_1, (p_0, w)\boldsymbol{\lambda}_A)\boldsymbol{\pi}_{B})\\
&= ((w)\mathfrak{s}_A, (p_1, (p_0, w)\boldsymbol{\lambda}_A)\boldsymbol{\pi}_{B})\\
&= ((w)\mathfrak{s}_A, (p_1, (p_0, w_0)\boldsymbol{\lambda}_A((p_0, w_0)\boldsymbol{\pi}_A, w_1)\boldsymbol{\lambda}_A)\boldsymbol{\pi}_{B})\\
&= ((w)\mathfrak{s}_A, ((p_1, (p_0, w_0)\boldsymbol{\lambda}_A)\boldsymbol{\pi}_B, ((p_0, w_0)\boldsymbol{\pi}_A, w_1)\boldsymbol{\lambda}_A)\boldsymbol{\pi}_{B})\\
&= ((w)\mathfrak{s}_A, ((p_1, (p_0, w_0)\boldsymbol{\lambda}_A)\boldsymbol{\pi}_B, ((w_0)\mathfrak{s}_A, w_1)\boldsymbol{\lambda}_A)\boldsymbol{\pi}_{B})\\
&= ((w)\mathfrak{s}_A, (((w_0)\mathfrak{s}_A, w_1)\boldsymbol{\lambda}_A)\mathfrak{s}_{B}).
\end{align*}
Note that \(p_0\) and \(p_1\) are not mentioned in the final expression above. Moreover this is a state of the composite of \(\C(A)\) and \(\C(B)\).
\end{proof}

\speeddictsix{135}{sync_monoid}{semigroup defn}{inj_clo}{unique minimal thm}{synchronizing homeomorphisms}{removing incomplete responce works}{Cores From Cores lemma}
\begin{corollary}[Synchronizing monoid, \Assumed{135}]\label{sync_monoid}
If \(d\in \N\backslash \{0\}\) and \(n\in \N\backslash\{0, 1\}\), then \(d\mathcal{S}_{n, 1}\) is a monoid (under composition of functions).
\end{corollary}
\begin{proof}
First note that from Definition~\ref{identity states defn}, we know that \(\text{id}_{\mathfrak{C}_n^d}\in d\mathcal{S}_{n, 1}\).
Let \(f, g\in d\mathcal{S}_{n, 1}\) be arbitrary. It suffices to show that \(f g\in d\mathcal{S}_{n, 1}\).
By Lemmas~\ref{composition works as expected lemma} and \ref{inj_clo}, we have that \(f_{M_f M_g, (e_f, e_g)} = f g\).
Moreover by Lemma~\ref{Cores From Cores lemma}, the transducer \(M_f M_g\) is synchonizing. Let \(M'\) be the subtransducer of \(M_f M_g\) with state set \(((e_f, e_g), (X_n^*)^d)\boldsymbol{\pi}_{M_f M_g}\).

By Lemmas~\ref{composition works as expected lemma} and \ref{inj_clo}, for all \(q\in Q_{M'}\) the function \(f_{M', q}\) is a homeomorphism from \(\mathfrak{C}_n^d\) to an open subset of  \(\mathfrak{C}_n^d\).
Thus by Lemma~\ref{removing incomplete responce works}, \(CR(M')\) is a well-defined non-degenerate \((d, n)\)-transducer with complete response such that \(f_{CR(M'), (e_f, e_g)} = f_{M',(e_f, e_g)} =f g\).
Moreover as \(M_f M_g\) was synchonizing, so is \(M'\) and hence so is \(CR(M')\).

We know that the transducer \(M_{CR(M')}\) is synchonizing, and satisfies all of the conditions required by Theorem~\ref{unique minimal thm}.
Thus \(M_{f g}\) is strongly isomorphic to the synchonizing transducer \(M_{CR(M')}\), so \(M_{f g}\) is also synchonizing and \(f g\in d\mathcal{S}_{n, 1}\) as required.
\end{proof}

\speeddictfive{137}{core monoid defn}{types of binary relation defns}{semigroup defn}{transducers homomorphisms defn}{sync_monoid}{core monoids work lemma}
\begin{defn}[Core monoid, \Assumed{137}]\label{core monoid defn}
For \(d\in \N\backslash \{0\}\) and \(n\in \N\backslash \{0, 1\}\), we define \(\widetilde{d\mathcal{O}_{n,1}}:=\makeset{\C(M_f)}{\(f\in d\mathcal{S}_{n,1}\)}/\cong_S\) (recall Definitions~\ref{types of binary relation defns} and \ref{transducers homomorphisms defn}) and we define a binary operation on \(\widetilde{d\mathcal{O}_{n,1}}\) by \[[\C(M_f)]_{\cong_S}[\C(M_g)]_{\cong_S}=[\C(M_{f g})]_{\cong_S}.\]
We show in Lemma~\ref{core monoids work lemma} that the above operation is well defined and makes \(\widetilde{d\mathcal{O}_{n,1}}\) into a monoid.
\end{defn}

\speeddictthree{138}{core monoids work lemma}{Cores From Cores lemma}{sync_monoid}{core monoid defn}
\begin{lemma}[Core monoids work, \Assumed{138}]\label{core monoids work lemma}
The object \(\widetilde{d\mathcal{O}_{n,1}}\) defined in Definition~\ref{core monoid defn} is a well-defined monoid, the map \(\gamma_{d, n}: d\mathcal{S}_{n, 1} \to \widetilde{d\mathcal{O}_{n,1}}\) defined by \((h)\gamma_{d, n} = [\C(M_h)]_{\cong_S}\) is a quotient map, and for all \(f, g\in d\mathcal{S}_{n, 1}\) we have
\[\C(M_{f g})\cong_S M_{CR(\C(\C(M_f)\C(M_g)))}.\]
\end{lemma}
\begin{proof}
If \(\widetilde{d\mathcal{O}_{n,1}}\) is well defined, then the map \(\gamma_{d, n}\) is a surjective homomorphism by definition.
By Lemma~\ref{sync_monoid}, the set \(d\mathcal{S}_{n, 1}\) is a monoid.
Thus \(\widetilde{d\mathcal{O}_{n,1}}\) is a monoid if and only if the product operation is well-defined.

In the proof of Lemma~\ref{sync_monoid}, it is shown that if \(f, g\in d\mathcal{S}_{n, 1}\), then the \(M_{f g}\) is strongly isomorphic to the transducer \(M_{CR(M')}\) where \(M'\) is the subtransducer \(((e_f, e_g), (X_n^*)^d)\boldsymbol{\pi}_{M_f M_g}\) of \(M_f M_g\).
It follows that 
\[\C(M_{f g})\cong_S \C(M_{CR(M')}) \cong_S M_{\C(CR(M'))}= M_{CR(\C(M'))}.\]
As the core of a transducer \(T\) is contained in every subtransducer of \(T\) with the same domain and range, it follows that \(\C(M_{f g})\cong_S M_{CR(\C(M'))}= M_{CR(\C(M_f M_g))}\). Thus by Lemma~\ref{Cores From Cores lemma}, we have 
\[\C(M_{f g})\cong_S M_{CR(\C(\C(M_f)\C(M_g)))}.\]
Thus the strong isomorphism type of \(\C(M_{f g})\) is uniquely determined by the isomorphism types of \(\C(M_f)\) and \(\C(M_g)\), and the operation is well-defined as required.
\end{proof}

\speeddictthree{139}{dBn and dOn defn}{group defn}{core monoid defn}{core monoids work lemma}
\begin{defn}[\(d\mathcal{B}_{n, 1}\) and \(d\mathcal{O}_n\), \Assumed{139}]\label{dBn and dOn defn}
For \(d\in \N\backslash \{0\}\) and \(n\in \N\backslash \{0, 1\}\), we define \(d\mathcal{B}_{n, 1}\) to be the group of units of \(d\mathcal{S}_{n, 1}\), and \(d\mathcal{O}_{n,1}:=(d\mathcal{B}_{n, 1})\gamma_{d, n}\).
\end{defn}

The groups \(d\mathcal{B}_{n, 1}\) and \(d\mathcal{O}_{n, 1}\) are the main purpose of this section, as we will show in Theorem~\ref{autnV} these will be respectively the automorphism and outer-automorphism groups of \(dV_n\).

\speeddictfour{140}{nV_is_normal}{centralizer defn}{nV_placement}{core monoids work lemma}{dBn and dOn defn}
\begin{lemma}[\(dV_n\) is normal, \Assumed{140}]\label{nV_is_normal}
If \(d\in \N\backslash \{0\}\) and \(n\in \N\backslash\{0, 1\}\), then \(d\mathcal{V}_{n} \normalsubgroup d\mathcal{B}_{n, 1}\) and \(d\mathcal{B}_{n, 1}/d\mathcal{V}_{n} \cong d\mathcal{O}_{n,1}\).
\end{lemma}
\begin{proof}
As \(\text{id}_{\mathfrak{C}_n^d}\) is the identity of \(d\mathcal{B}_{n, 1}\), \((\text{id}_{\mathfrak{C}_n^d})\gamma_{d, n}\) is the identity of \(d\mathcal{O}_{n, 1}\).
Thus if \(f\in d\mathcal{B}_{n, 1}\), then \((f)\gamma_{d, n}\) is the identity if and only if the core of \(f\) is an identity transducer.
Thus by Proposition~\ref{nV_placement}, \((f)\gamma_{d, n}\) is the identity if and only if \(f\in dV_{n}\).
It follows that \(dV_n\) is a normal subgroup of \(d\mathcal{B}_{n,1}\).
Moreover the map \(h dV_n \to (h)\gamma_{d, n}\) is a well-defined isomorphism between \(d\mathcal{B}_{n, 1}/d\mathcal{V}_{n}\) and \(d\mathcal{O}_{n,1}\).
\end{proof}

\speeddictsix{141}{main_lemma}{centralizer defn}{transducers are continuous lemma}{composition works as expected lemma}{inj_clo}{unique minimal thm}{nV_placement}
\begin{lemma}[Local synchronization for the normalizer, \Assumed{141}]\label{main_lemma}
Let \(d\in \N\backslash \{0\}\), \(n\in \N\backslash \{0, 1\}\), \(h\in N_{\Aut(\mathfrak{C}_n^d)}(dV_n)\) and \(s,t\in (X_n^*)^d\backslash \{\varepsilon_d\}\).
There exists \(K_{h, s, t}\in \mathbb{N}\) such that for all \(a\in (X_n^*)^d\) with \(\min\left(\makeset{|(a)\pi_i|}{\(i\in \{0,1, \ldots, d-1\}\)}\right)\geq K_{h,s,t}\), we have \((e_h,s a)\boldsymbol{\pi}_{M_h} = (e_h,ta)\boldsymbol{\pi}_{M_h}\).
\end{lemma}
\begin{proof}
For all \(x\in (X_n^*)^d\), let \(q_x := (e_h, x)\boldsymbol{\pi}_{M_h}\).
As \(s, t\) are non-trivial, we can find complete prefix codes \(F_s, F_t\) for \(\mathfrak{C}_n^d\), such that \(s\in F_s, t\in F_t\) and \(|F_s|= |F_t|\). Let \(\phi:F_s \to F_t\) be a bijection such that \((s)\phi = t\), and let \(f:= f_\phi\) (Definition~\ref{prefix codes and dVn defn}).

As \(h\in N_{\Aut(\mathfrak{C}_n^d)}(dV_n)\), we have that \(h dV_n = dV_n h\).
In particular \(f h\in h dV_n\). Let \(g\in nV\) such that \(f h=h g\). Thus by Lemmas~\ref{composition works as expected lemma} and \ref{inj_clo}, we have
\[f_{M_f M_h, (e_f,e_h)} =f h=h g= f_{M_h M_g, (e_h, e_g)}.\]
By Proposition~\ref{nV_placement}, each of \(M_f, M_g\) has a single identity state as its core. Let \(I_f, I_g\) be the core states of \(M_f\) and \(M_g\) respectively. Note that
\[((e_f, e_h), s)\boldsymbol{\pi}_{M_f M_h} = (I_f, q_t),\quad ((e_h, e_g), s)\boldsymbol{\pi}_{M_h M_g} = (q_s, (e_g,(q_h, s)\boldsymbol{\lambda}_{M_h})\boldsymbol{\pi}_{M_g}).\]
By Lemma~\ref{transducers are continuous lemma} applied to the state \(q_s\) of \(M_h\), let \(K\in \mathbb{N}\) be such that for all \(w\in (X_n^*)^d\) with 
\[\min\left(\makeset{|(w)\pi_i|}{\(i\in \{0, 1, \ldots, d-1\}\)}\right)\geq K,\]
we have
\(\min\left(\makeset{|(q_s, w)\boldsymbol{\lambda}_{M_h}\pi_i|}{\(i\in \{0, 1,\ldots, d-1\}\)}\right)\)
is at least the synchronizing length of \(M_g\).

Let \(a\in (X_n^*)^d\) be arbitrary such that \(\min\left(\makeset{|(a)\pi_i|}{\(i\in \{0, 1, \ldots, d-1\}\)}\right)\geq K\).
It suffices to show that \(q_{ta}=q_{s a}\).
From the choice of \(K\), we obtain
\[((e_f, e_h), s a)\boldsymbol{\pi}_{M_f M_h} = (I_f, q_{ta}),\quad ((e_h, e_g), s a)\boldsymbol{\pi}_{M_h M_g} = (q_{s a}, I_g).\]
Thus, for all \(v\in (X_n^\N)^d\) we have

\begin{align*}
((e_f, e_h),s a)\boldsymbol{\lambda}_{M_f M_h}(v)f_{M_h, q_{ta}}
&=((e_f, e_h),s a)\boldsymbol{\lambda}_{M_f M_h}(v)f_{M_f M_h, (I_f,q_{ta})}\\
&=(s a v)f h\\
&= (s a v)h g\\
&=((e_h, e_g),s a)\boldsymbol{\lambda}_{M_h M_g}(v)f_{M_h M_g, (q_{s a},I_g)}\\
&=((e_h, e_g),s a)\boldsymbol{\lambda}_{M_h M_g}(v)f_{M_h, q_{s a}}.
\end{align*}

From the above equation we have for all \(i\in \{0, 1, \ldots, d-1\}\), that the words \(((e_f, e_h), s a)\boldsymbol{\lambda}_{M_f M_h}\pi_i\) and \(((e_h, e_g), s a)\boldsymbol{\lambda}_{M_h M_g}\pi_i\) are comparable.

If there was \(i\in \{0, 1, \ldots, d-1\}\) such that \(((e_f, e_h), s a)\lambda_{M_f M_h}\pi_i\neq((e_h, e_g), s a)\lambda_{M_h M_g}\pi_i\), it would follow that either the map \(f_{M_h, q_{ta}}\) or \(f_{M_h, q_{s a}}\) has its image contained in a proper cone.
By Theorem~\ref{unique minimal thm}, the transducer \(M_h\) has complete response, so this cannot occur and we must have \(((e_f, e_h), s a)\lambda_{M_f M_h} = ((e_h, e_g),s a)\lambda_{M_h M_g}\).

From the equality
\[((e_f, e_h), s a)\lambda_{M_f M_h}(v)f_{M_h, q_{ta}}
=((e_h, e_g), s a)\lambda_{M_h M_g}(v)f_{M_h, q_{s a}}\]
it follows that \(f_{M_h, q_{ta}}=f_{M_h, q_{s a}}\).
Thus, as \(M_h\) has complete response, we have \(q_{ta}\sim_{M_h} q_{s a}\).
Finally, as \(M_h\) is minimal, it follows that \(q_{ta}=q_{s a}\) as required.
\end{proof}

\speeddicttwo{142}{hard_aut_containment}{dBn and dOn defn}{main_lemma}
\begin{lemma}[Synchronizing normalizer, \Assumed{142}]\label{hard_aut_containment}
For \(d\in \N\backslash \{0\}\) and \(n\in \N\backslash \{0, 1\}\), the group \(N_{\A(\mathfrak{C}_n^d)}(dV_n)\) is contained in \(d\mathcal{B}_{n, 1}\).
\end{lemma}
\begin{proof}
 Let \(h\in N_{\A(\mathfrak{C}_n^d)}(dV_n)\) be arbitrary. We show that \(h\in d\mathcal{S}_n\), as \(h\) is arbitrary we then also conclude that \(h^{-1}\in d\mathcal{S}_n\) and so \(h\) must be a unit of \(d\mathcal{S}_n\).
 
 Thus we need only show that \(M_h\) is a synchronizing transducer. We define
 \[K:= \max\left(\makeset{K_{h, x_{d,i}, y_{d,j}z_{d, k}}}{\(x, y, z\in X_n, i, j, k\in \{0, 1, \ldots, d-1\}\)}\right) + 1\]
 where \(x_{d,i}, y_{d,j}, z_{d,k}\) are as in Definition~\ref{cones words and Cantor spaces defn}, and \(K_{h, x_{d, i}, y_{d,j}z_{d,k}}\) is as in Lemma~\ref{main_lemma}.
 Let \(w\in (X_n^*)^d\) be arbitrary, such that 
\[\min\left(\makeset{|(w)\pi_i|}{\(i\in \{0, 1, \ldots, d-1\}\)}\right)\geq K.\]
 It suffices to show for all \(q\in Q_{M_h}\), we have \((q, w)\boldsymbol{\pi}_{M_h} = (e_h, w)\boldsymbol{\pi}_{M_h}\).
 Let \(q\in Q_{M_h}\) be arbitrary and let \(v\in (X_n^*)^d\) be such that \((e_h, v)\boldsymbol{\pi}_{M_h} = q\). Let \(w = w_0w'\), and \(v= v_0v_1\ldots v_l\) where 
 \[v_0,v_1,\ldots, v_l,w_0\in \makeset{x_{d,i}}{\(x\in X_n, i\in \{0, 1, \ldots, d-1\}\)}.\]
 By the choice of \(w\), we obtain:
 \begin{align*}
     (q, w)\boldsymbol{\pi}_{M_h}&= (e_h, v w)\boldsymbol{\pi}_{M_h}= (e_h, v_0v_1v_2\ldots v_l w_0w')\boldsymbol{\pi}_{M_h}\\
     &= (e_h, (v_0v_1)(v_2\ldots v_l w_0w'))\boldsymbol{\pi}_{M_h}= (e_h, (v_1)(v_2\ldots v_l w_0w'))\boldsymbol{\pi}_{M_h}\\
     &= (e_h, (v_1v_2)(v_3\ldots v_l w_0w'))\boldsymbol{\pi}_{M_h}= (e_h, (v_2)(v_3\ldots v_l w_0w'))\boldsymbol{\pi}_{M_h}\\
   &\vdots\\
    &= (e_h, (v_l w_0)w')\boldsymbol{\pi}_{M_h}= (e_h, w_0w')\boldsymbol{\pi}_{M_h}= (e_h, w)\boldsymbol{\pi}_{M_h}
 \end{align*}
 as required.
\end{proof}

We can now tie everything together and use Rubin's Theorem to connect our groups to the automorphism groups of \(dV_n\).
\speeddicttwo{143}{autnV}{nV_is_normal}{hard_aut_containment}
\begin{theorem}[Automorphisms characterised, \Assumed{143}]\label{autnV}
Let \(d\in \N\backslash\{0\}\) and \(n\in \N\backslash \{0, 1\}\).
The groups \(N_{\A(\mathfrak{C}_n^d)}(dV_n)\) and \(d\mathcal{B}_{n,1}\) are equal, countably infinite and isomorphic to \(\A(dV_n)\). Moreover \(\Out(dV_n) \cong d\mathcal{O}_{n, 1}\).
\end{theorem}
\begin{proof}
We show the continents 
\[N_{\A(\mathfrak{C}_n^d)}(dV_n)\subseteq d\mathcal{B}_{n, 1}\subseteq N_{\A(\mathfrak{C}_n^d)}(dV_n)\]
in Lemmas~\ref{hard_aut_containment} and \ref{nV_is_normal} respectively.
Thus \(N_{\A(\mathfrak{C}_n^d)}(dV_n)=d\mathcal{B}_{n,1}\).

As there is a countably infinite number of complete prefix codes for \(\mathfrak{C}_n^d\), the group \(dV_n\) is countably infinite.
By Theorem~\ref{nV_is_normal}, we have \[|d\mathcal{B}_{n, 1}|= |dV_n||d\mathcal{B}_{n, 1}/dV_n|=|dV_n||d\mathcal{O}_{n, 1}|.\]
An element of \(d\mathcal{O}_{n, 1}\) is a strong isomorphism class of core transducers all of which are finite.
Moreover the monoid \((X_n^*)^d\) is finitely generated and by the definition of a transducer, the output and transition maps are determined by their actions on the generators of \((X_n^*)^d\). Hence for each \(k\in \N\) there are only finitely many \((d, n)\)-transducer with \(k\) states up to strong isomorphism. Thus the group \(d\mathcal{O}_{n, 1}\) is countable and so 
\[|d\mathcal{B}_{n, 1}|= |dV_n||d\mathcal{O}_{n, 1}|= |\N||d\mathcal{O}_{n, 1}|=|\N|.\]

By Rubin's Theorem (Corollary~\ref{Rubin Automorphisms cor}) and Lemma~\ref{nV_is_normal}, it suffices to show that the standard action \(a:\mathfrak{C}_n^d \times dV_n \to \mathfrak{C}_n^d\) given by \((p, f)a = (p)f\) is locally moving (Definition~\ref{locally moving defn}).

Let \(x\in \mathfrak{C}_n^d\) be arbitrary and let \(w\mathfrak{C}_n^d\) be a basic open neighbourhood of \(x\).
Let \(u_x, v\in (X_n^*)^d\) be arbitrary such that \(w<v\) and \(w<u_x< x\). It suffices to show that \((x)(dV_n)_{w\mathfrak{C}_n^d, a} \cap v\mathfrak{C}_n^d \neq \varnothing\).

Let \(S\subseteq (X_n^*)^d\) be such that
\(\makeset{s\mathfrak{C}_n^d}{\(s\in S\)}\)
is a partition of \(\mathfrak{C}_n^d\backslash w\mathfrak{C}_n^d\) into clopen sets. Also choose \(F_{u_x}, F_v\subseteq (X_n^*)^d\) to be such that \(|F_{u_x}|= |F_v|, u_x\in F_{u_x}, v\in F_v\) and
\(\makeset{s\mathfrak{C}_n^d}{\(s\in F_{u_x}\)},\)
\(\makeset{s\mathfrak{C}_n^d}{\(s\in F_v\)},\)
are both partitions of \(w\mathfrak{C}_n^d\) into clopen sets.
Let \(\phi': F_{u_x}\to F_{v}\) be any bijection with \((u_x)\phi'= v\), and let \(\phi := \phi' \cup \makeset{(s, s)}{\(s\in S\)}\).
It follows that the map \(f_{\phi}\) is an element of \((dV_n)_{w\mathfrak{C}_n^d, a}\) and satisfies \((x)f_{\phi}\in v\mathfrak{C}_n^d\) as required.
\end{proof}

\section{Rationality}\label{rationality sec}

In GNS \cite{GNS2000}, the notion of a rational homomorphism is introduced (GNS also proves Proposition~\ref{isomorphic rational groups prop}).

\speeddictfive{165}{rational transformations defn}{Categories defn}{automorphism groups defn}{dn transducer def}{reading infinite words defn}{composition works as expected lemma}
\begin{defn}[Rational transformations, \Assumed{165}]\label{rational transformations defn}
If say that \(f: \mathfrak{C}_n \to \mathfrak{C}_m\) is \textit{rational} if \(f= f_{T, q}\) for some \((1, n, 1, m)\)-transducer \(T\) with finitely many states and some \(q\in Q_T\).
By Lemma~\ref{composition works as expected lemma} it follows that the set of objects \[\makeset{\mathfrak{C}_n}{\(n\in \N\backslash \{0, 1\}\)}\]
together with the set of morphisms
\[\makeset{f: \mathfrak{C}_n \to \mathfrak{C}_m}{\(f\) is rational and \(n, m\in \N\backslash\{0, 1\}\)}\]
form a small category. We call this the category of rational transformations and denote it by \(\mathcal{R}\). We denote the endomorphism monoid and automorphism group of \(\mathfrak{C}_n\) in this category by \(\widetilde{\mathcal{R}_n}\) and \(\mathcal{R}_n\) respectively.
\end{defn}

\speeddicttwo{166}{isomorphic rational groups prop}{picture definitions}{rational transformations defn}
\begin{proposition}[cf \cite{GNS2000}, Rational isomorphisms, \Assumed{166}]\label{isomorphic rational groups prop}
For all \(n, m\in \N\backslash \{0, 1\}\), the objects \(\mathfrak{C}_n, \mathfrak{C}_m\) are isomorphic in the category \(\mathcal{R}\), thus 
\[\widetilde{\mathcal{R}_n} \cong \widetilde{\mathcal{R}_m},\quad \text{and} \quad \mathcal{R}_n \cong \mathcal{R}_m.\]
\end{proposition}
\begin{proof}
Let \(n\in \N\backslash \{0, 1, 2\}\) be arbitrary. We give a rational isomorphism from \(\mathfrak{C}_2 \to \mathfrak{C}_n\). We define \(F_1 := \{1, 01, 001, \ldots, 0^{n-1}1\}\), \(F_2 := X_n\) and \(\phi: F_1\to F_2\) by \(((0)^i1)\phi = i\) for all \(i\in \{0, 1, \ldots, n-1\}\). We then define \(f_{\phi}:\mathfrak{C}_n \to \mathfrak{C}_n\) and \(f_{\phi^{-1}}:\mathfrak{C}_n \to \mathfrak{C}_2\) as in Figure~\ref{figure7} (recall Definition~\ref{picture definitions}).

\begin{figure}
\begin{center}
\begin{tikzpicture}[->,>=stealth',shorten >=1pt,auto,node distance=5cm,on grid,semithick,
                    every state/.style={fill=red,draw=none,circular drop shadow,text=white}]
  \node [state, at={(1, -2.5)}] (A) [rotate = 0 ]{$q_0$};
  \node [state, at={(3, -2.5)}] (B)  {};
  \node [state, at={(5, -2.5)}] (C)   {};
  \node [state, at={(7, -2.5)}] (D)   {};
  \node [at={(10,-2.5)},state] (E) {$q_0$};
  \node [at={(10,-1.9)}]    (F)  {\(\ldots\)};
  \node [at={(6,-2.5)}]     (G)  {\(\ldots\)};
 \path [->]
 (A) edge [out=190,in=170, loop]       node {\small $1/0$} (A)
 (A) edge      node [swap]{\small $0/\varepsilon$} (B)
 (B) edge       node [swap]{\small $0/\varepsilon$} (C)
 (B) edge [out=120,in=60]      node {\small $1/1$} (A)
 (C) edge [out=90,in=90]      node {\small $1/2$} (A)
 (D) edge [out=60,in=120]      node [swap]{\small $1/(n-1)$} (A)

 (E) edge [out=20,in=00, loop]       node {\small $(n-1)/0^{n-1}1$} (E)
 (E) edge [out=60,in=40, loop]       node {\small $(n-2)/0^{n-2}1$} (E)
 (E) edge [out=140,in=120, loop]       node {\small $1/01$} (E)
 (E) edge [out=190,in=170, loop]       node {\small $0/1$} (E);
 \end{tikzpicture}
 \end{center}
\caption{We define \(f_\phi:= f_{T, q_0}\) where \(T\) is the left transducer, and \(f_{\phi^{-1}}:= f_{T, q_0}\) where \(T\) is the right transducer}\label{figure7}
\end{figure}
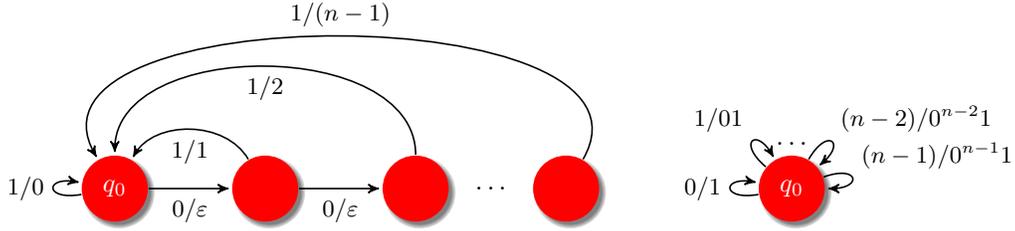
Note that
\[(w x)f_{\phi} = (w)\phi (x)f_{\phi}, \quad\text{and}\quad (v y)f_{\phi^{-1}}=(v)\phi^{-1} (y)f_{\phi^{-1}}\]
where \(w\in F_1, v\in F_2, x\in \mathfrak{C}_2, y\in \mathfrak{C}_n\) are arbitrary. In particular we have that \(f_{\phi}, f_{\phi^{-1}}\) are inverses of each other.
As these maps are rational, the result follows.
\end{proof}

If \(k, n\in \N\), then there are only finitely many words over the alphabet \(X_n\) with length at most \(k\).
 Thus from Remark~\ref{Core Transducers remark}, it follows that \(1\mathcal{B}_{n, 1}\leq \mathcal{R}_n\). 
 However if \(d\geq 2\), then there are infinitely many elements \(w\) of \((X_n^*)^d\) with \[\min\left(\makeset{|(w)\pi_i|}{\(i\in \{0, 1, \ldots, d-1\}\)}\right) <k.\]
Thus this reasoning does not follow in the case of \(d\mathcal{B}_{n, 1}\). 
In Figure~\ref{baker} we see that the bakers map cannot be defined by a finite \((2, 2)\)-transducer.

In \cite{Belk2016} Theorem 5.2, it is shown that \(2V\) can be embedded in the rational group.
They embed \(2V\) by conjugating it by the map \(f\) defined the transducer in Figure~\ref{figure8}.
As \(\mathcal{B}_{2, 1}\leq \mathcal{R}_2\) and \(f 2V f^{-1}\leq \mathcal{R}_4\), one might expect that \(f2\mathcal{B}_{2, 1}f^{-1}\leq \mathcal{R}_4\).
We show in Proposition~\ref{rat problems prop} that this is not the case.

\speeddictfour{164}{swap transducer example}{dn transducer def}{reading infinite words defn}{synchronizing defn}{picture definitions}
\begin{example}[\(0\leftrightarrow 00\) transducer, \Assumed{164}]\label{swap transducer example}
Suppose \(T\) is the left \((1, 2)\)-transducer from Figure~\ref{figure1} (recall Definition~\ref{picture definitions}).
The transducer \(T\) is synchronizing with \(\C(T)=T\).  Moreover if \(n_0\in \N\), \(n_1, \ldots, m_0,m_1, \ldots\in (\N\backslash\{0\})\cup \{\infty\}\) and \(s:\N\cup \{\infty\}\to \N\cup \{\infty\}\) is the permutation with \((1)s=2, (2)s =1, (n)s =n\) for \(n\not\in \{1, 2\}\), then
\[(0^{n_0}1^{m_0}0^{n_1}\ldots)f_{T, q_0} = 0^{(n_0)s}1^{m_0}0^{(n_1)s}\ldots.\]
In particular \(f_{T, q_0}\) is a homeomorphism.
\end{example}
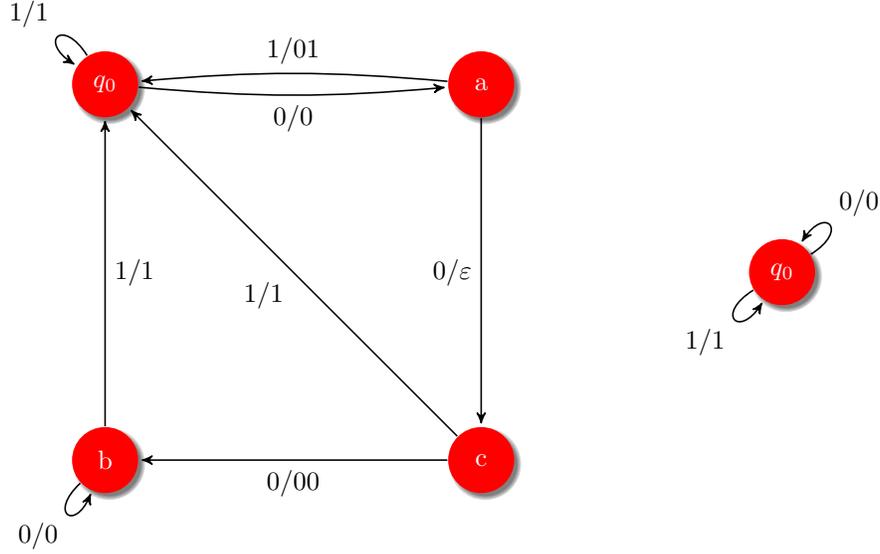
\begin{figure}
\begin{center}
\begin{tikzpicture}[->,>=stealth',shorten >=1pt,auto,node distance=5cm,on grid,semithick,
                    every state/.style={fill=red,draw=none,circular drop shadow,text=white}]  \node [state] (A)                {$q_0$};
  \node [state] (B)  [right= of A] {a};
  \node [state] (C)  [below= of B] {c};
  \node [state] (D)  [below= of A] {b};
  \node [at={(9,-2.5)},state] (E) {$q_0$};
 \path [->]
 (A) edge [out=122,in=147, loop] node [swap]{$1/1$} (A)
 
 (A) edge [out=355,in=185]        node [swap]{$0/0$} (B)
 
 (B) edge [out=175,in=5]        node [swap]{$1/01$} (A)
 (B) edge                        node [swap]{$0/\varepsilon$} (C)
 (C) edge                        node {$1/1$} (A)
 (C) edge                        node {$0/00$} (D)
 (D) edge                        node [swap]{$1/1$} (A)
 (D) edge [out=223,in=247, loop] node [swap]{$0/0$} (D)
 (E) edge [out=212,in=237, loop] node [swap]{$1/1$} (E)
 (E) edge [out=32,in=57, loop] node [swap]{$0/0$} (E);
 \end{tikzpicture}
 \end{center}
\caption{Two transducers with domain and range \(X_2^*\)}\label{figure1}
\end{figure}
\begin{proof}
By the definition of \(T\), it follows that for all \(q\in Q_T\) we have
\[(q, 000)\boldsymbol{\pi}_T = b, \quad (q, 100)\boldsymbol{\pi}_T = c, \quad (q, 10)\boldsymbol{\pi}_T = a \quad, (q, 1)\boldsymbol{\pi}_T = q_0\]
In particular if \(T\) is synchronizing at level \(3\). 
By definition we have
\[(0^{n_0}1^{m_0}0^{n_1}\ldots)f_{T, q_0} = (q_0, 0^{n_0}1^{m_0})\boldsymbol{\lambda}_T (0^{n_1}1^{m_1}0^{n_2}\ldots)f_{T, (q_0, 0^{n_0}1^{m_0})\boldsymbol{\pi}_T}.\]
Moreover, as \(m_0 \geq 1\) and \((q, 1)\boldsymbol{\pi}_T = q_0\) for all \(q\in Q_T\), we have \((q_0, 0^{n_0}1^{m_0})\boldsymbol{\pi}_T = q_0\). Thus 
\[(0^{n_0}1^{m_0}0^{n_1}\ldots)f_{T, q_0} = (q_0, 0^{n_0}1^{m_0})\boldsymbol{\lambda}_T (0^{n_1}1^{m_1}0^{n_2}\ldots)f_{T,q_0}.\]
By repeatedly applying the above equality, it suffices to show that \((q_0, 0^{n_0}1^{m_0})\boldsymbol{\lambda}_T = 0^{(n_0)s}1^{m_0}\).
If \(n\in \N\) is arbitrary, then by definition
\[(q_0, 0^11^{m_0})\boldsymbol{\lambda}_T = (0)(01)(1)^{m_0-1}=0^21^{m_0},\]
\[(q_0, 0^21^{m_0})\boldsymbol{\lambda}_T = (0)()(1)(1)^{m_0-1}=0^11^{m_0},\]
\[(q_0, 0^{3+n}1^{m_0})\boldsymbol{\lambda}_T = (0)()(00)(0)^n(1)(1)^{m_0-1}=0^{3+n}1^{m_0},\]
as required.
\end{proof}

\begin{figure}
\begin{center}
\begin{tikzpicture}[->,>=stealth',shorten >=1pt,auto,node distance=5cm,on grid,semithick,
                    every state/.style={fill=red,draw=none,circular drop shadow,text=white}]  \node [state] (A)                {$q_0$};
  \node [state] (B)  [right= of A] {};
  \node [state] (C)  [below= of B] {};
  \node [state] (D)  [below= of A] {};
 \path [->]
 (A) edge [out=122,in=147, loop] node [swap]{$(1, \varepsilon)/(1, \varepsilon)$} (A)
 (A) edge [out=355,in=185]        node [swap]{$(0, \varepsilon)/(0, \varepsilon)$} (B)
 (A) edge [out=182,in=207, loop] node [swap]{$(\varepsilon,1)/( \varepsilon,1)$} (A)
 (A) edge [out=62,in=87, loop] node [swap]{$(\varepsilon,0)/( \varepsilon,0)$} (A)
 
 (B) edge [out=175,in=5]        node [swap]{$(1, \varepsilon)/(01, \varepsilon)$} (A)
 (B) edge                        node [swap]{$(0, \varepsilon)/(\varepsilon, \varepsilon)$} (C)
 (B) edge [out=10,in=345, loop] node [swap, xshift = 55pt]{$(\varepsilon,1)/( \varepsilon,1)$} (B)
 (B) edge [out=75,in=100, loop] node [swap]{$(\varepsilon,0)/( \varepsilon,0)$} (B)
 
 (C) edge                        node {$(1, \varepsilon)/(1, \varepsilon)$} (A)
 (C) edge                        node {$(0, \varepsilon)/(00, \varepsilon)$} (D)
 (C) edge [out=260,in=285, loop] node [swap]{$(\varepsilon,1)/( \varepsilon,1)$} (C)
 (C) edge [out=340,in=5, loop] node [swap]{$(\varepsilon,0)/( \varepsilon,0)$} (C)
 
 (D) edge                        node [swap,xshift = -60pt]{$(1, \varepsilon)/(1, \varepsilon)$} (A)
 (D) edge [out=223,in=247, loop] node [swap]{$(0, \varepsilon)/(0, \varepsilon)$} (D)
 (D) edge [out=282,in=307, loop] node [swap]{$(\varepsilon,1)/( \varepsilon,1)$} (D)
 (D) edge [out=142,in=167, loop] node [swap]{$(\varepsilon,0)/( \varepsilon,0)$} (D);
 \end{tikzpicture}
 \end{center}
\caption{The transducer obtained by taking the categorical product of the transducers in Figure~\ref{figure1}}\label{figure2}
\end{figure}
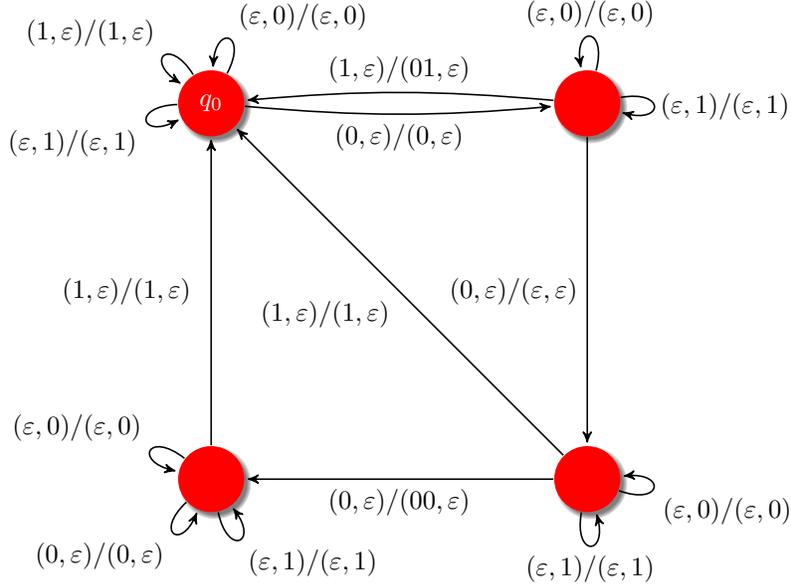

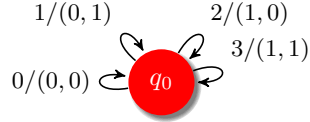
\begin{figure}
\begin{center}
\begin{tikzpicture}[->,>=stealth',shorten >=1pt,auto,node distance=5cm,on grid,semithick,
                    every state/.style={fill=red,draw=none,circular drop shadow,text=white}]
  \node [at={(10,-2.5)},state] (E) {$q_0$};
 \path [->]
 (E) edge [out=20,in=00, loop]       node {\small $3/(1, 1)$} (E)
 (E) edge [out=60,in=40, loop]       node {\small $2/(1, 0)$} (E)
 (E) edge [out=140,in=120, loop]       node {\small $1/(0, 1)$} (E)
 (E) edge [out=190,in=170, loop]       node {\small $0/(0, 0)$} (E);
 \end{tikzpicture}
 \end{center}
\caption{A single state \((1, 4, 2, 2)\)-transducer defining a homeomorphism from \(\mathfrak{C}_{4}\) to \(\mathfrak{C}_2^2\) (analogous to those in  Figure~\ref{figure7})}\label{figure8}
\end{figure}

\speeddictseven{167}{rat problems prop}{picture definitions}{nV_placement}{composition works as expected lemma}{transducing homeomorphisms defn}{words defn}{swap transducer example}{unique minimal thm}
\begin{proposition}[Rationality problems in higher dimensions, \Assumed{167}]\label{rat problems prop}
Let \(f:\mathfrak{C}_4\to \mathfrak{C}_2^2\) be the homeomorphism defined by the transducer in Figure~\ref{figure8} (recall Definition~\ref{picture definitions}). This transducer has finitely many states but \(f^{-1}\) cannot be defined by a \((2, 2,1, 2)\)-transducer with finitely many states. Moreover we have 
\[f 2V f^{-1}\subseteq \mathcal{R}_4 \quad \text{but} \quad f2\mathcal{B}_{2, 1}f^{-1}\not \subseteq \mathcal{R}_4.\]
\end{proposition}
\begin{proof}
Let \(T\) be the transducer in Figure~\ref{figure8}. Note that if \(w\in (X_2^*)^2\), then \(w\in (Q_T\times D_T)\boldsymbol{\lambda}_T\) if and only if \(|(w)\pi_0|=|(w)\pi_1|\). 
Let \(D\) be the \((2, 2, 1, 4)\)-transducer defined by
\begin{enumerate}
    \item \(Q_D := (X_2^* \times \{\varepsilon\}) \cup (\{\varepsilon\} \times X_2^*)\).
    \item \((q, s)\boldsymbol{\pi}_T = p\), where \(p\in Q_D\) is such that there is \(b\in (X_2^*)^2\) with \(|(b)\pi_0| = |(b)\pi_1|\) such that \(q s=b p\).
    \item \((q, s)\boldsymbol{\lambda}_T = w\), where \(w\in X_4^*\) is such that there is \(p\in Q_D\) is with \(q s=(q_0, w)\boldsymbol{\lambda}_T p\).
\end{enumerate}
Note that \(D\) is well-defined.
Moreover for all \(p\in Q_D\) and \(x\in \mathfrak{C}_2^2\) we have \((x)f_{D, p}= (p x)f^{-1}\).
In particular \(f_{D, \varepsilon_2} = f^{-1}\).

 We now show that \(f 2V f^{-1}\subseteq \mathcal{R}_4\). 
 Let \(g\in 2V\) be arbitrary.
 By Theorem~\ref{nV_placement}, the transducer \(M_g\) is synchonizing and has a single identity state as its core. Let \(k\) be the synchonizing length of \(M_g\). Thus by Lemma~\ref{composition works as expected lemma}, we have
\[f g f^{-1} = f_{T M_g D, (q_0, e_g, \varepsilon_2)}.\]
We define \(P:= T M_g D\) and \(e:= (q_0, e_g, \varepsilon_2)\).
As \(D\) has infinitely many states, so does \(P\). However if \(w\in X_4^*\) is arbitrary with \(|w|\geq k\), then
\[||((q_0, e_g), w)\boldsymbol{\lambda}_{TM_g}\pi_0|-|((q_0, e_g), w)\boldsymbol{\lambda}_{TM_g}\pi_1||\]
is equal to
\[||((q_0, e_g), w\restriction_k)\boldsymbol{\lambda}_{TM_g}\pi_0|-|((q_0, e_g), w\restriction_k)\boldsymbol{\lambda}_{TM_g}\pi_1||.\]
Thus the subtransducer \(S:= (e, X_4^*)\boldsymbol{\pi}_P\) of \(P\) has finitely many states and satisfies \(f_{S, e}= f g f^{-1}\), so \(f g f^{-1}\) is rational as required.

We next show that \(f2\mathcal{B}_{2, 1}f^{-1}\not\subseteq \mathcal{R}_2\).
Let \(b:= f_{B, q_0}\), where \(B\) is the transducer from Figure~\ref{figure2}.
By Example~\ref{swap transducer example}, we know that \(b\in 2\mathcal{B}_{2, 1}\).
We show that \(f b f^{-1}\notin \mathcal{R}_4\).

By definition, for all \(n\in 3\N\) we have
\begin{align*}
    (q_0, (02)^n)\boldsymbol{\lambda}_{T} &= ((01)^n, (00)^n)\\
    (q_0, ((01)^n, (00)^n))\boldsymbol{\lambda}_{B} &= ((001)^n, (00)^n)=((001)^n, (000)^{2n/3})\\
    (\varepsilon_2, ((001)^{n}, (000)^{2n/3}))\boldsymbol{\lambda}_{D} &= (\varepsilon_2, ((001)^{2n/3}, (000)^{2n/3}))\boldsymbol{\lambda}_{D}= (002)^{2n/3}.
\end{align*}
Thus \(((q_0, q_0, \varepsilon_2), (02)^n)\boldsymbol{\lambda}_{TBD} = (002)^{2n/3}\). Moreover we have
\[((02)^n\mathfrak{C}_4)f b f^{-1}= (((001)^n, (000)^{2n/3})\mathfrak{C}_2^2)f^{-1} =(002)^{2n/3}(((001)^{n/3}, \varepsilon)\mathfrak{C}_2^2)f^{-1}.\]
For all \(n\in 3\N\), let \(p_n:= ((q_0, q_0,\varepsilon_2), (02)^n)\boldsymbol{\pi}_{TBD}\). From the above equalities it follows that
\[(\mathfrak{C}_4)f_{TBD, p_n}= (((001)^{n/3}, \varepsilon)\mathfrak{C}_2^2)f^{-1}.\]
In particular, for \(n\neq m\) we have \(f_{TBD,p_n} \neq f_{TBD,p_m}\) and \(TBD\) has complete response from these states.
Consider the subtransducer \(S:=((q_0, q_0, \varepsilon_2), X_4^*)\boldsymbol{\pi}_{TBD}\) of \(TBD\).

By Theorem~\ref{unique minimal thm}, it follows that if \(S'\) is any \((1,4)\)-transducer such that there is \(q\in Q_{S'}\) with \((q, X_4^*)\boldsymbol{\pi}_{S'}=Q_{S'}\) and \(f_{S',q}=f b f^{-1}\), then \(M_{f b f^{-1}}\cong_S M_{CR(S')}\).
Thus from the above observations, it follows that \(S'\) must be infinite and so \(f b f^{-1}\) is not rational.

It remains to show that \(f^{-1}\) cannot be defined by a transducer with finitely many states. Suppose for a contradiction that \(D'\) is a finite state transducer defining \(f^{-1}\). It follows that \(TBD'\) is a transducer with finitely many states with a state defining \(f b f^{-1}\). As \(f b f^{-1}\) is not rational, this is a contradiction.
\end{proof}

\section{A closer look at the groups \(d\mathcal{O}_{n,1}\)}
In this section we investigate the core transducers used to describe the outer-automorphisms of \(dV_n\) in the previous section.
We show that an element of \(d\mathcal{O}_{n,1}\) can be ``decomposed" into \(d\) elements of \(1\mathcal{O}_{n,1}\) and use this to show that \(\Out(dV)\) is a wreath product
(Theorem~\ref{main theorem 2}).

\speeddictone{147}{semidirect product defn}{group actions defn}
\begin{defn}[Semidirect products, \Assumed{147}]\label{semidirect product defn}
If \(G, H\) are groups and \(a: G\times H \to G\) is an action of \(H\) on \(G\), then we define the semidirect product \(G \rtimes_a H\) to be the group with the underlying set \(G\times H\) and operations:
\begin{enumerate}
    \item If \((g_0, h_0), (g_1, h_1)\in G\rtimes_a H\), then
    \[(g_0, h_0)(g_1, h_1)= (g_0(g_1, h_0^{-1})a, h_0h_1),\]
    \item If \((g, h)\in G\rtimes_a H\), then
    \[(g, h)^{-1}= ((g^{-1}, h)a, h^{-1}).\]
\end{enumerate}
It is routine to verify that this defines a group.
\end{defn}

\speeddictone{153}{wreath product defn}{semidirect product defn}
\begin{defn}[Wreath products, \Assumed{153}]\label{wreath product defn}
If \(G, H\) are groups, and \(a: X\times H \to X\) is an action of \(H\) on a set \(X\), then we define the wreath product \(G \wr_a H\) to be the group \(G^X \rtimes_{a'} H\) where \(a':G^X \times H \to G^X\) is defined by
\[((g_{x})_{x\in X}, h)a' = (g_{(x, h)a})_{x\in X}.\]
\end{defn}

We start by identifying an action of the semigroup \(\widetilde{d\mathcal{O}_{n,1}}\) on the set \(\{0, 1, \ldots, d-1\}\).

\speeddictone{144}{restricted coordinates defn}{cones words and Cantor spaces defn}
\begin{defn}[Restricted coordinates, \Assumed{144}]\label{restricted coordinates defn}
For \(d\in \N\backslash \{0\}, n\in \N\backslash \{0, 1\}\) and \(S\subseteq \{0, 1, \ldots, d-1\}\), we define 
\[F_{d, n, S}:=\makeset{w\in (X_n^*)^d}{\( (w)\pi_i = \varepsilon\) for all \(i\in \{0, 1, \ldots, d-1\}\backslash S\)}.\]
So \(F_{d, n, S}\) is the submonoid of \((X_n^*)^d\) consisting of those elements only allowed to be non-trivial in the \(S\) coordinates.
\end{defn}

\speeddictsix{145}{notinjlemma}{function defn}{binary relations defns}{transducers are continuous lemma}{inj_clo}{core monoid defn}{restricted coordinates defn}
\begin{lemma}[\(\widetilde{d\mathcal{O}_{n,1}}\) acts on coordinates independently, \Assumed{145}]\label{notinjlemma}
Let \(d\in \N\backslash\{0\}\) and \(n\in \N\backslash\{0, 1\}\).
If \(T = \C(M_f)\) for some \(f\in d\mathcal{S}_{n, 1}\) then 
\[\psi_T:= \makeset{(i, j)\in \{0, 1,\ldots, d-1\}^2}{\((Q_T\times F_{d, n, \{i\}})\boldsymbol{\lambda}_T \subseteq F_{d, n , \{j\}}\)}\]
is a function with domain \(\{0, 1,\ldots, d-1\}\).
\end{lemma}
\begin{proof}
It suffices to show for all \(i\in \{0, 1, \ldots, d-1\}\), that the set \((\{i\})\psi_T\) (recall Definition~\ref{binary relations defns}) is a singleton.
Let \(i\in \{0, 1, \ldots, d-1\}\) be arbitrary, we start by showing that \((\{i\})\psi_T \neq \varnothing\).
Recall that for all \(q\in Q_T\), the map \(f_{T, q}\) is necessarily injective (Lemma~\ref{inj_clo}).

Suppose for a contradiction that there is \(i, \alpha, \beta\in \{0, 1,\ldots, d-1\}\) and \((q_0,l_0), (q_1, l_1)\in Q_T \times F_{d,n,\{i\}}\), such that \[\alpha\neq \beta, \quad\quad |(q_0, l_0)\boldsymbol{\lambda}_T\pi_{\alpha}|>0, \quad\text{ and }\quad |(q_1, l_1)\boldsymbol{\lambda}_T\pi_{\beta}|>0.\]
We may assume without loss of generality that \(\alpha = 0\) and \(\beta = 1\).

Let \(p_0\in Q_T\) be fixed.
The transducer \(T\) is non-degenerate (as it is a subtransducer of \(M_f\)) so by Lemma~\ref{transducers are continuous lemma}, there is \(s\in (X_n^*)^d\) such that \(|(p_0, s)\boldsymbol{\lambda}_T\pi_j|\geq 1\) for all \(j\in \{0, 1, \ldots, d-1\}\).
For each \(j\in \{0, 1, \ldots, d-1\}\) define \(s_j\in F_{d,n,\{j\}}\) such that \(s=s_0s_1s_2\ldots s_j\). Moreover for each \(j\in \{1, 2, 3, \ldots, d-1\}\), define \(p_j:= (p_0, s_0s_1\ldots s_{j-1})\boldsymbol{\pi}_T\).
As \(T\) is a transducer, it follows that
\[(p_0, s)\boldsymbol{\lambda}_T=(p_0, s_0s_1s_2\ldots s_j)\boldsymbol{\lambda}_T=(p_0, s_0)\boldsymbol{\lambda}_T(p_1, s_1)\boldsymbol{\lambda}_T\ldots(p_{d-1}, s_{d-1})\boldsymbol{\lambda}_T.\]

Thus for all \(j\in \{0, 1, 2, 3, \ldots, d-1\}\backslash \{0, 1\}\) we can choose some, \(q_j\in Q_T\), \(i_j\in \{0, 1, \ldots, d-1\}\) and \(l_j\in F_{d,n,\{i_j\}}\) such that \(|(q_j, l_j)\boldsymbol{\lambda}_T\pi_j|>0\). We also define both \(i_0\) and \(i_1\) to be \(i\).

For all \(j\in \{0, 1, \ldots, d-1\}\), we now have \(|(q_j, l_j)\boldsymbol{\lambda}_T\pi_{j}|>0\).
Moreover \(i_{0}=i_1\), so we can fix some \(b\in \{0, 1, \ldots, d-1\}\backslash \{i_0, i_1, \ldots, i_{d-1}\}\). We will use this observation to contradict injectivity.

For each \(j\in \{0, 1, \ldots, d-1\}\), choose some \(w_j\in (\{q_j\})\mathfrak{s}_T^{-1}\). Now if \(p\in Q_T\) is arbitrary, then 
\(|(p, w_jl_j)\boldsymbol{\lambda}_T\pi_j|\geq |(q_j, l_j)\boldsymbol{\lambda}_T\pi_j|\geq 1\).
For each \(j\in \{0, 1, \ldots, d-1\}\), define \(w_{j,b}\in F_{d, n, \{b\}}\) and \(w_j'\in F_{d, n, \{0, 1, \ldots, d-1\}\backslash\{b\}}\) such that \(w_j=w_{j,b}w_{j}'\).

For each \(m\in \N\), we define
\[ t_m:=w_{0, b}w_{1, b}\ldots w_{d-1, b}(w_{0}'l_0w_{1}'l_1\ldots w_{d-1}'l_{d-1})^m.\]
For each \(j\in \{0, 1, \ldots, n-1\}\), the element \(w_{j, b}\in F_{d, n, \{b\}}\) commutes with the product
\[w_{0}'l_0w_{1}'l_1\ldots w_{d-1}'l_{d-1}\in F_{d, n, \{0, 1, \ldots, d-1\}\backslash \{b\}}.\]
Thus for all \(m\in \N\) and \(j\in \{0, 1, \ldots, d-1\}\), we can find \(u_{j, m}, v_{j, m}\in (X_n^*)^d\) such that
\[t_m = u_{j, m}w_{j, b}(w_{0}'l_0w_{1}'l_1\ldots w_{d-1}'l_{d-1})^m v_{j, m}.\]
Moreover, \(w_{j, b}\) commutes with all elements of \( F_{d, n, \{0, 1, \ldots, d-1\}\backslash \{b\}}\) and so if \(q\in Q_T, j\in \{0, 1, \ldots, d-1\}, m\in \N\) are arbitrary, then
\begin{align*}
   |(q, t_m) \boldsymbol{\lambda}_T\pi_j|&=    |(q, u_{j, m}w_{j, b}(w_{0}'l_0w_{1}'l_1\ldots w_{d-1}'l_{d-1})^m v_{j, m}) \boldsymbol{\lambda}_T\pi_j|\\
  &\geq    |((q, u_{j, m}w_{j, b})\boldsymbol{\pi}_T, (w_{0}'l_0w_{1}'l_1\ldots w_{d-1}'l_{d-1})^m) \boldsymbol{\lambda}_T\pi_j|\\
  &\geq    |((w_j)\mathfrak{s}_T, l_j w_{j+1}'l_{j+1}\ldots w_{d-1}'l_{d-1}(w_{0}'l_0w_{1}'l_1\ldots w_{d-1}'l_{d-1})^{m-1}) \boldsymbol{\lambda}_T\pi_j|\\
 &\geq   1+ |((w_{j, b}w_{j}'l_j\ldots w_{d-1}'l_{d-1})\mathfrak{s}_T, (w_{0}'l_0w_{1}'l_1\ldots w_{d-1}'l_{d-1})^{m-1}) \boldsymbol{\lambda}_T\pi_j|\\
 &=   1+ |((w_{j}'l_j\ldots w_{d-1}'l_{d-1}w_{j, b})\mathfrak{s}_T, (w_{0}'l_0w_{1}'l_1\ldots w_{d-1}'l_{d-1})^{m-1}) \boldsymbol{\lambda}_T\pi_j|\\
 &\geq   2+ |((w_{j}'l_j\ldots w_{d-1}'l_{d-1}w_{j, b})\mathfrak{s}_T, (w_{0}'l_0w_{1}'l_1\ldots w_{d-1}'l_{d-1})^{m-2}) \boldsymbol{\lambda}_T\pi_j|\\
 &\vdots\\
 &\geq   m.
\end{align*}

Thus all elements of \((X_n^\N)^d\) which have all of the elements of \(t_m\) as prefixes have the same image under \(f_{T, q}\) (for any choice of \(q\)). This is a contradiction as there are infinitely many such elements and \(f_{T, q}\) is injective.

It remains to show that \(|(\{i\})\psi_T|\leq 1\).
Suppose for a contradiction that \(\alpha, \beta\in (\{i\})\psi_T\) and \(\alpha\neq \beta\).
It follows that 
\[(Q_T\times F_{d, n, \{i\}})\boldsymbol{\lambda}_T \subseteq F_{d, n , \{\alpha\}} \cap  F_{d, n , \{\beta\}} = \{\varepsilon_d\}.\]
For each \(j\in \{0, 1, \ldots, d-1\}\backslash \{i\}\), by the first part of the proof, we can choose some \(j'\in (\{j\})\psi_T\).
Let \(b\in \{0, 1, \ldots, d-1\}\backslash\makeset{j'}{\(j\neq i\)}\).
It follows that for all \(q\in Q_T\) and \(w\in (X_n^*)^d\), we have \(|(q, w)\boldsymbol{\lambda}_T\pi_b| = 0\).
This contradicts Lemma~\ref{transducers are continuous lemma}.
\end{proof}

\speeddicttwo{146}{On coordinate map defn}{function sets def}{notinjlemma}
\begin{defn}[\(d\mathcal{O}_n\) coordinate map, \Assumed{146}]\label{On coordinate map defn}
We define a map \(\Psi:d\mathcal{O}_n \to \{0, 1, \ldots, d-1\}^{\{0, 1, \ldots, d-1\}}\) by
\[([T]_{\cong_S})\Psi = \psi_T,\]
where \(\psi_T\) is as in Lemma~\ref{notinjlemma}. Note that the property defining \(\psi_T\) is invariant under strong transducer isomorphisms, so this is a well-defined map.
\end{defn}

\speeddictseven{148}{first_semidirect}{Products in Categories Defn}{symmetric group topology}{unique minimal thm}{core monoids work lemma}{notinjlemma}{On coordinate map defn}{semidirect product defn}
\begin{theorem}[\(d\mathcal{O}_{n, 1}\) is semidirect, \Assumed{148}]\label{first_semidirect}
The map \(\Psi:d\mathcal{O}_n \to \{0, 1, \ldots, d-1\}^{\{0, 1, \ldots, d-1\}}\) (Definition~\ref{On coordinate map defn}) is a semigroup homomorphism.
Moreover, the group \(d\mathcal{O}_{n, 1}\) is isomorphic to \(d\mathcal{K}_{n, 1}\rtimes_a \Sym(d)\), where 
\[d\mathcal{K}_{n, 1} := \makeset{S\in d\mathcal{O}_{n, 1}}{\((S)\Psi=\text{id}_{\{0, 1,\ldots, d-1\}}\)},\]
\(a:d\mathcal{K}_{n, 1}\times \Sym(d) \to d\mathcal{K}_{n, 1}\) is defined by \[([A]_{\cong_S}, g)a = [\C(M_{\mathbf{g}})]_{\cong_S}^{-1}[\C(M_{f})]_{\cong_S}[\C(M_{\mathbf{g}})]_{\cong_S}\]
and \(\mathbf{g}:=\langle (\pi_{(i)g})_{i\in \{0, 1, \ldots, d-1\}}\rangle_{\mathfrak{C}_n^d}\) (recall Definition~\ref{Products in Categories Defn}).
\end{theorem}
\begin{proof}
If \([A]_{\cong_S}, [B]_{\cong_S}\in d\mathcal{O}_{n, 1}\) are arbitrary, then by Lemma~\ref{core monoids work lemma}, we have 
\[[A]_{\cong_S} [B]_{\cong_S}=[M_{CR(\C(AB))}]_{\cong_S}.\]
If \(i\in \{0, 1, \ldots, d-1\}\) then \((Q_{A} \times F_{d, n, \{i\}})\boldsymbol{\lambda}_A\subseteq F_{d, n,\{(i)\psi_{A}\}}\) and \((Q_{B} \times F_{d, n, \{(i)\psi_{A}\}})\boldsymbol{\lambda}_B\subseteq F_{d, n,\{(i)\psi_{A}\psi_{B}\}}\).
Thus
\[(Q_{AB}\times F_{d, n, \{i\}})\boldsymbol{\lambda}_{AB}=(Q_{B} \times (Q_{A} \times F_{d, n, \{i\}})\boldsymbol{\lambda}_A)\boldsymbol{\lambda}_B\subseteq  F_{d, n,\{(i)\psi_{A}\psi_{B}\}}.\]
It follows that
\[(Q_{CR(\C(AB))}\times F_{d, n, \{i\}})\boldsymbol{\lambda}_{CR(\C(AB))}\subseteq  F_{d, n,\{(i)\psi_{A}\psi_{B}\}}\]
as well, and hence the map \(\Psi\) is a homomorphism.
If \(T\) is an identity transducer (Definition~\ref{identity states defn}), then \(\psi_T\) is the identity map on \(\{0, 1, \ldots, d-1\}\).
Therefor, as \(d\mathcal{O}_{n, 1}\) is a group, we must have that \(\Psi\) is a group homomorphism from \(d\mathcal{O}_{n,1}\) to \(\Sym(d)\). It follows that the set \(d\mathcal{K}_{n, 1} = (\{\text{id}_{\{0, 1, \ldots, d-1\}}\})\Psi^{-1}\) is a group.

For each \(g\in \Sym(d)\), let \(T_g\) be the \((d, n)\)-transducer with state set \(\{0\}\) and \(\boldsymbol{\lambda}_{T_g}\) defined by
\[(0, w)\boldsymbol{\lambda}_{T_g} = ((w)\pi_{(0)g}, (w)\pi_{(1)g}, \ldots, (w)\pi_{(d-1)g}),\]
for all \(w\in (X_n^*)^d\). For all \(g\in \Sym(d)\), we have that \(f_{T_g, 0}= \mathbf{g}\). Thus by Lemma~\ref{unique minimal thm}, we have \(\C(T_g)=T_g \cong_S M_{\mathbf{g}}=\C(M_{\mathbf{g}})\) for all \(g\in \Sym(d)\). In particular we have that \(\mathbf{g}\in d\mathcal{B}_{n, 1}\) for all \(g\in \Sym(d)\).

Define \(\Psi':\Sym(d) \to d\mathcal{O}_{n, 1}\) by
\[(g)\Psi' = [\C(M_{\mathbf{g}})]_{\cong_S}.\]
As each of the maps \(g\to \mathbf{g}\) and \(\mathbf{g} \to  [\C(M_{\mathbf{g}})]_{\cong}\) are homomorphisms, the map \(\Psi'\) is a homomorphism. Moreover as \(\C(M_{\mathbf{g}})= M_{\mathbf{g}} \cong_S T_g\), it follows that \(\Psi'\Psi\) is the identity map on \(\Sym(d)\).

We define maps \(\phi:d\mathcal{K}_{n, 1}\rtimes_a \Sym(d) \to d\mathcal{O}_n\) and \(\phi': d\mathcal{O}_n \to d\mathcal{K}_{n, 1}\rtimes_a \Sym(d)\)  by 
\[(S, g)\phi= S(g)\Psi',\quad \text{ and }\quad(S)\phi'=(S ((S)\Psi\Psi')^{-1} ,(S)\Psi).\]
Note that \(\phi'\) is well defined as for all \(S\in d\mathcal{O}_n\), we have
\[(S ((S)\Psi\Psi')^{-1})\Psi=(S( ((S)\Psi)^{-1})\Psi')\Psi=(S)\Psi((S)\Psi)^{-1}=\text{id}_{\{0, 1, \ldots, d-1\}}.\]

By construction, the maps \(\phi\phi'\) and \(\phi'\phi\) are both identity maps, thus \(\phi\) is a bijection. It now suffices to show that \(\phi\) is a homomorphism. If \((S_0, g_0), (S_1, g_1)\in d\mathcal{K}_{n, 1}\rtimes_a \Sym(d)\), then we have
\begin{align*}
    (S_0, g_0)\phi (S_1, g_1)\phi &= S_0(g_0)\Psi'S_1(g_1)\Psi'\\
     &= S_0(g_0)\Psi' S_1((g_0)\Psi')^{-1}(g_0g_1)\Psi'\\
       &=  S_0(S_1, g_0^{-1})a (g_0g_1)\Psi'\\
       &=  (S_0(S_1, g_0^{-1})a, g_0 g_1)\phi \\
         &=  ((S_0, g_0) (S_1, g_1))\phi.
\end{align*}
Thus the result follows.
\end{proof}

We have that \(d\mathcal{O}_{n, 1}\) is a semidirect product of \(\Sym(d)\) and the group \(d\mathcal{K}_{n, 1}\).
We next work towards decomposing the transducers in the group \(d\mathcal{K}_{n, 1}\) (see Lemma~\ref{product_decomposition}).
We then use our description of these transducers to embed \(d\mathcal{K}_{n, 1}\) into the group \(d\mathcal{P}_{n, n-1}\) (Theorem~\ref{alpha}).

The group \(d\mathcal{P}_{n, n-1}\) is based on the group \(\mathcal{O}_{n, n-1}\) from \cite{bleak2016}.
It is shown that \cite{bleak2016}, that \(\mathcal{O}_{n, n-1}\) is a group isomorphic to \(\Out(G_{n, n-1})\) (hence the name).
The insistence that \(Q_T\subseteq \N\) in the next definition is arbitrary, and done only to avoid proper classes.

\speeddictseven{149}{Onn-1 defn}{types of binary relation defns}{group defn}{degenerate transducers defn}{unique minimal thm}{synchronizing defn}{core monoids work lemma}{Onn-1 lemma}
\begin{defn}[\(\mathcal{O}_{n, n-1}\), \Assumed{149}]\label{Onn-1 defn}
We say that a transducer \(T\) is an \(\widetilde{\mathcal{O}_{n, n-1}}\) transducer if
\begin{enumerate}
    \item \(Q_T\subseteq \N\), and \(D_T=R_T= X_n^*\).
    \item \(T\cong_S M_T\).
    \item \(T\) has complete response.
    \item \(T\) is synchonizing and \(\C(T) = T\).
    \item For all \(q\in Q_T\), the map \(f_{T, q}\) is a homeomorphism from \(\mathfrak{C}_n^d\) to an open subset of \(\mathfrak{C}_n^d\).
\end{enumerate}
We also define \(\widetilde{\mathcal{O}_{n, n-1}}\) to be the set \(\makeset{T}{\(T\) is an \(\widetilde{\mathcal{O}_{n, n-1}}\) transducer}/\cong_S,\) with multiplication given by \([A]_{\cong_S}[B]_{\cong_S}=[C]_{\cong_S}\), where \(C\) is strongly isomorphic to \(M_{CR(\C(AB))}\) (we show in Lemma~\ref{Onn-1 lemma} that this operation is well-defined and makes \(\widetilde{\mathcal{O}_{n, n-1}}\) a monoid). We then define \(\mathcal{O}_{n, n-1}\) to be the group of units of \(\widetilde{\mathcal{O}_{n, n-1}}\), and we say that an \(\widetilde{\mathcal{O}_{n, n-1}}\) transducer \(T\) is \textit{invertible} if \([T]_{\cong_S}\in \mathcal{O}_{n, n-1}\).
\end{defn}

\speeddictfive{150}{Onn-1 lemma}{composition works as expected lemma}{induced maps preserved remark}{identity states defn}{removing incomplete responce works}{Onn-1 defn}
\begin{lemma}[\(\mathcal{O}_{n, n-1}\) works as expected, \Assumed{150}]\label{Onn-1 lemma}
The operation given in Definition~\ref{Onn-1 defn}, makes \(\widetilde{\mathcal{O}_{n, n-1}}\) into a well-defined monoid. Moreover, if \(A, B\) are synchronizing \((1, n)\)-transducers with \(A\cong_S M_{CR(A)}\), \(B\cong_S M_{CR(B)}\) and there is \(a\in Q_A, b\in Q_B\) such that \(f_{A, a}=f_{B, b}\), then \(\C(A)\cong_S \C(B)\). 
\end{lemma}
\begin{proof}
Let \(A, B\) be arbitrary \(\widetilde{\mathcal{O}_{n, n-1}}\) transducers.
For remainder of this proof (to avoid nested subscripts) if \(T\) is a \((1, n)\)-transducer, then we will denote \(M_T\) by \(M(T)\) instead.

We first need to show that \(M(CR(\C(AB)))\) is strongly isomorphic to an \(\widetilde{\mathcal{O}_{n, n-1}}\) transducer. All states of \(A\) and \(B\) define homeomorphisms from \(\mathfrak{C}_n\) to open subsets of \(\mathfrak{C}_n\), thus it follows from Lemma~\ref{composition works as expected lemma} that all the states of \(AB\) define homeomorphisms from \(\mathfrak{C}_n\) to open subsets of \(\mathfrak{C}_n\).
Therefore by Remark~\ref{induced maps preserved remark} and Lemma~\ref{removing incomplete responce works}, all states of \(M(CR(\C(AB)))\) also define homeomorphisms from \(\mathfrak{C}_n^d\) to open subsets of \(\mathfrak{C}_n^d\).

By construction \(M(CR(\C(AB)))\) is its own minimal transducer, has complete response, is its own core, and has the required domain and range.
Thus the operation of \(\widetilde{\mathcal{O}_{n, n-1}}\) is well defined.

\underline{Claim:}
Suppose that \(S, T\) are arbitrary synchronizing \((1, n)\)-transducers such that \(S\cong_S M(CR(S))\) and \(T\cong_S M(CR(T))\).
If there are \(q_s\in Q_S\) and \(q_t\in Q_T\) such that \(f_{T, q_t}=f_{S, q_s}\), then \(\C(T)\cong_S \C(S)\).\\
\underline{Proof of Claim:}
If \(q\in Q_T\), then for all \(v, w\in X_n^*\), \(x\in \mathfrak{C}_n\) we have
\begin{enumerate}
    \item \((q, w)\boldsymbol{\lambda}_T\) is determined by the sets \((\mathfrak{C}_n)f_{T, q}\) and \((w\mathfrak{C}_n)f_{T, q}\). 
    Thus \((q, w)\boldsymbol{\lambda}_T\) is determined by \(f_{T, q}\).
    \item \((w x)f_{T, q} = (q, w)\boldsymbol{\lambda}_T(x)f_{T, (q, w)\boldsymbol{\pi}_T}\) and
    \[(q, v)\boldsymbol{\pi}_T= (q, w)\boldsymbol{\pi}_T \iff f_{T, (q, w)\boldsymbol{\pi}_T}=f_{T, (q, v)\boldsymbol{\pi}_T},\]
    thus \(f_{T, q}\) determines weather or not \((q, v)\boldsymbol{\pi}_T= (q, w)\boldsymbol{\pi}_T\).
    \item \(\C(T)\subseteq (q,X_n^* )\boldsymbol{\pi}_T\).
\end{enumerate}
Combining the above observations gives the result.\(\diamondsuit\)

Identity transducers with the appropriate domain, range, and state sets  (Definition~\ref{identity states defn}) are \(\widetilde{\mathcal{O}_{n, n-1}}\) transducers, so \(\widetilde{\mathcal{O}_{n, n-1}}\) has an identity.

It remains to show that our binary operation is associative.
Let \(A, B, C\) be arbitrary \(\widetilde{\mathcal{O}_{n, n-1}}\) transducers, we need to show that
\[M(CR(\C(M(CR(\C(AB))) C))) \cong_S M(CR(\C(AM(CR(\C(BC)))))).\]
By the same argument showing that multiplication in \(\widetilde{\mathcal{O}_{n, n-1}}\) is well-defined, the transducer \(T:=M(CR(\C(ABC)))\) is an \(\widetilde{\mathcal{O}_{n, n-1}}\) transducer.

Note that the \(\C\) operator commutes with the \(M\) and \(CR\) operators, so \(T=\C(M(CR(ABC)))\).
Let \((a, b, c)\in Q_{ABC}\) be arbitrary.
By Remark~\ref{induced maps preserved remark} and Lemma~\ref{removing incomplete responce works}, there is some \(q\in Q_{M(CR(ABC))}\) and \(p\in X_n^*\) such that
\[f_{M(CR(ABC)),q} =f_{A, a}f_{B, b}f_{C, c}\lambda_p^{-1}.\]

Let \((a, b, c)\in Q_A\times Q_B \times Q_C\)  be arbitrary such that \((a, b)\in \C(AB)\). Again by Remark~\ref{induced maps preserved remark} and Lemma~\ref{removing incomplete responce works}, there is \(q\in Q_{M(CR(\C(AB)))}\) and \(p\in X_n^*\) such that
\[f_{M(CR(AB))C,(q, c)} =f_{A, a}f_{B, b}f_{C, c}\lambda_p^{-1},\]
and hence there is \(q\in Q_{M(CR(M(CR(\C(AB)))C))}\) and \(p\in X_n^*\) such that
\[f_{M(CR(M(CR(AB))C)),q} =f_{A, a}f_{B, b}f_{C, c}\lambda_p^{-1}.\]
By commuting the \(\C\) operator again we have
\[M(CR(\C(M(CR(\C(AB))) C))) = \C(M(CR(M(CR(\C(AB))) C))).\]
Thus, the transducers \(T\) and \(M(CR(\C(M(CR(\C(AB))) C)))\) are the cores of transducers which have states defining a common map (as they have complete response, the choices of \(p\) above must coincide).
By the Claim, this implies that
\[T \cong_S M(CR(\C(M(CR(\C(AB))) C))).\]
By a symmetric argument, we also have
\[T \cong_S M(CR(\C(AM(CR(\C(BC)))))).\]
so the result follows.
\end{proof}

\speeddictthree{158}{dOnn-1 monoid defn}{transducer products defn}{Onn-1 defn}{product_decomposition}
\begin{defn}[\(d\mathcal{P}_{n, n-1}\), \Assumed{158}]\label{dOnn-1 monoid defn}
We say that a transducer \(T\) is a \(\widetilde{d\mathcal{P}_{n, n-1}}\) transducer if 
\(T= \prod_{i\in \{0, 1, \ldots, d-1\}} T_i\)
where \(T_0, T_1, \ldots, T_{d-1}\) are \(\widetilde{\mathcal{O}_{n, n-1}}\) transducers. 
We also define \(\widetilde{d\mathcal{P}_{n, n-1}}\) to be the set \(\makeset{T}{\(T\) is an \(\widetilde{d\mathcal{P}_{n, n-1}}\) transducer}/\cong_S,\) with multiplication given by \([A]_{\cong_S}[B]_{\cong_S}=[C]_{\cong_S}\), where \(C\) is strongly isomorphic to \(M_{CR(\C(AB))}\) (we show in Lemma~\ref{product_decomposition} that this object is well defined).
\end{defn}

\speeddictseven{151}{product_decomposition}{types of binary relation defns}{Products in Categories Defn}{minimal transducers are valid lemma}{inj_clo}{first_semidirect}{Onn-1 defn}{dOnn-1 monoid defn}
\begin{lemma}[Product decomposition, \Assumed{151}]\label{product_decomposition}
If \([T]_{\cong_S}\in d\mathcal{K}_{n, 1}\) then \(T\) is strongly isomorphic to a \(\widetilde{d\mathcal{P}_{n, n-1}}\) transducer.
Moreover if \(A, B\) are \(\widetilde{d\mathcal{P}_{n, n-1}}\) transducers then \(T:= M_{CR(\C(AB))}\) is strongly isomorphic to a \(\widetilde{d\mathcal{P}_{n, n-1}}\) transducer and hence the object \(\widetilde{d\mathcal{P}_{n, n-1}}\) of Definition~\ref{dOnn-1 monoid defn} is well-defined.
\end{lemma}
\begin{proof}
Let \(T\) be a transducer such that either \([T]_{\cong_S}\in d\mathcal{K}_{n, 1}\) or \(T= M_{CR(\C(AB))}\) for some \(\widetilde{d\mathcal{P}_{n, n-1}}\) transducers \(A\) and \(B\).
For each \(i\in \{0, 1, \ldots, d-1\},\) let
\[\sim_{i}:= \makeset{(p, q)\in Q_T^2}{ there is \(w\in F_{d, n, \{i\}}\) with \((p, w)\boldsymbol{\pi}_T=q\)}.\]
\underline{Claim 1:} For all \(i\in \{0, 1, \ldots, d-1\}\), the binary relation \(\sim_i\) is an equivalence relation on \(Q_T\). Moreover if \(p\in Q_T\), then there is \(w_p\in F_{d, n, \{i\}}\) such that for all \(q\in [p]_{\sim_i}\) we have \((q, w_p)\boldsymbol{\pi}_T = p\).\\
\underline{Proof of Claim:} 
\begin{enumerate}
    \item Reflexive: If \(q\in Q_T\), then \((q, \varepsilon_d)\boldsymbol{\pi}_T = q\) and \(\varepsilon_d\in F_{d, n, \{i\}}\).
    Thus \(q\sim_i q\).
    \item Symmetric: Suppose that \(p, q\in Q_T\) are such that \((p, q)\in\sim_i\).
    Let \(w\in (\{p\})\mathfrak{s}_T^{-1}\) and \(v\in F_{d, n, \{i\}}\) be such that \((p, v)\boldsymbol{\pi}_T = q\). Let \(w=w_p w'\) where \(w_p\in F_{d, n, \{i\}}\) and \(w'\in F_{d, n, \{0, 1, \ldots, d-1\}\backslash \{i\}}\). We have
    \begin{align*}
        (q, w_p)\boldsymbol{\pi}_T&= ((p, v)\boldsymbol{\pi}_T, w_p)\boldsymbol{\pi}_T= ((p, w)\boldsymbol{\pi}_T, v w_p)\boldsymbol{\pi}_T= (p, w v w_p)\boldsymbol{\pi}_T\\
    &= (p, w_p w' v w_p)\boldsymbol{\pi}_T= (p, w_p v w' w_p)\boldsymbol{\pi}_T= (p, w_p v w)\boldsymbol{\pi}_T\\
          &= ((p, w_p v)\boldsymbol{\pi}_T, w)\boldsymbol{\pi}_T=(w)\mathfrak{s}_T=p.
    \end{align*}
    Thus \(q\sim_i p\).
    \item Transitive: Suppose that \(q_0\sim_i q_1\) and \(q_1 \sim_i q_2\). Thus there are \(v_0, v_1\in F_{d,n, \{i\}}\) with \((q_0, v_0)\boldsymbol{\pi}_T = q_1\) and \((q_1, v_1)\boldsymbol{\pi}_T = q_2\). 
    In particular
    \[(q_0, v_0v_1)\boldsymbol{\pi}_T=((q_0, v_0)\boldsymbol{\pi}_T, v_1)\boldsymbol{\pi}_T =(q_1, v_1)\boldsymbol{\pi}_T=q_2.\]
    So \(q_0\sim_i q_2\).
\end{enumerate}.
If we choose \(w_p\), as we did when showing \(\sim_i\) is symmetric, then the second part of the claim follows from the same sequence of equalities.
\(\diamondsuit\) \\

If \(q\in Q_T\), \(i\in \{0, 1, \ldots, d-1\}\), and we restrict the domain, range, and state set of \(T\) to \(F_{d, n, \{i\}}\), \(F_{d, n, \{i\}}\)  and \([q]_{\sim_i}\) respectively, then we obtain a subtransducer of \(T\) (in the case that \([T]_{\cong_S}\in d\mathcal{K}_{n, 1}\) this is only true because \(([T]_{\cong_S})\Psi\) is trivial).

We denote by \(S_{q, i}\), the \((1, n)\)-transducer isomorphic to this transducer obtained by replacing the monoid \(F_{d, n, \{i\}}\) with \(X_n^*\) in the natural fashion.

\noindent
\underline{Claim 2:} If \(q\in Q_T\) and \(w\in (X_n^*)^d\), then 
\[(q, w)\boldsymbol{\lambda}_T= ((q, (w)\pi_0)\boldsymbol{\lambda}_{S_{q, 0}},(q, (w)\pi_1)\boldsymbol{\lambda}_{S_{q, 1}}, \ldots,(q, (w)\pi_{d-1})\boldsymbol{\lambda}_{S_{q, d-1}}).\]
\underline{Proof of Claim:}
Let \(i\in \{0,1, \ldots, d-1\}\) be arbitrary. We show that \((q, w)\boldsymbol{\lambda}_T\pi_i = (q, (w)\pi_i)\boldsymbol{\lambda}_{S_{q, i}}\). Let \(w=w_i w'\), where \(w_i\in F_{d, n, \{i\}}\) and \(w'\in F_{d, n, \{0, 1, \ldots, d-1\}\backslash\{i\}}\). We have that
\begin{align*}
    (q, w)\boldsymbol{\lambda}_T\pi_i&=(q, w_i w')\boldsymbol{\lambda}_T\pi_i=((q, w_i)\boldsymbol{\lambda}_T((q, w_i)\boldsymbol{\pi}_T,w')\boldsymbol{\lambda}_T)\pi_i\\
    &=(q, w_i)\boldsymbol{\lambda}_T\pi_i((q, w_i)\boldsymbol{\pi}_T,w')\boldsymbol{\lambda}_T\pi_i.
\end{align*}
By the choice of \(T\), we have \(((q, w_i)\boldsymbol{\pi}_T,w')\boldsymbol{\lambda}_T\pi_i = \varepsilon\). Thus
\[(q, w)\boldsymbol{\lambda}_T\pi_i=(q, w_i)\boldsymbol{\lambda}_T\pi_i=(q, (w)\pi_i)\boldsymbol{\lambda}_{S_{q, i}}\]
as required.
\(\diamondsuit\)\\

\noindent
\underline{Claim 3:} If \(q\in Q_T\) and \(i\in \{0, 1, \ldots, d-1\}\), then \(S_{q, i}\) is strongly isomorphic to an \(\widetilde{\mathcal{O}_{n,n-1}}\) transducer.\\
\underline{Proof of Claim:}
 The transducer \(S_{q, i}\) is synchronizing and is its own core by Claim 1.
 Suppose that \(p\in [q]_{\sim_i}\) and \(w\in X_n^*\) is such that \((\mathfrak{C}_n)f_{S_{q, i}, p}\subseteq w\mathfrak{C}_n\).
 It follows from Claim 2 that
 \[(\mathfrak{C}_n^d)f_{T, p}\subseteq w'\mathfrak{C}_n^d\]
 where \((w')\pi_{i} = w\) and \((w')\pi_j= \varepsilon\) for \(j\neq i\). As \(T\) has complete response, it follows that \(w= \varepsilon\). Thus \(S_{q, i}\) has complete response.

If \([T]_{\cong_S}\in d\mathcal{K}_{n, 1}\), then (by Remark~\ref{inj_clo}) for all \(q\in Q_T\), the function \(f_{T, q}\) is injective with open image.
If \(T=M_{CR(\C(AB))}\) for some  \(\widetilde{d\mathcal{P}_{n, n-1}}\) transducers \(A, B\), and \(p\in Q_T\) is arbitrary,
then (by Lemmas~\ref{minimal transducers are valid lemma}, \ref{composition works as expected lemma} and \ref{removing incomplete responce works}) there are \(p_a\in Q_A\), \(p_b\in Q_B\) and \(w\in (X_{n}^*)^d\) such that
\(f_{T, p} = f_{A, p_a}f_{B, p_b}\lambda_w^{-1}\). So again for all \(q\in Q_T\), the function \(f_{T, q}\) is injective with open image.

By Claim 2, it follows that if \(p\in [q]_{\sim_i}\), then
\[f_{T, p}= \langle \pi_0f_{S_{q, 0}, p}, \pi_1f_{S_{q, 1}, p}, \ldots, \pi_{d-1}f_{S_{q, d-1}, p}\rangle_{\mathfrak{C}_n^d}.\]

In particular \(f_{S_{q, i}, p}\) is continuous, injective, and has image \((\mathfrak{C}_n^d)f_{T, p}\pi_i\) (which is open as it is a projection of an open set). 
By Remark~\ref{compactness facts}, \(f_{S_{q, i}, p}^{-1}\) is also continuous.

It remains to show that \(S_{q, i}\cong_S M_{S_{q, i}}\).
Suppose that \(q_0, q_1\in Q_{S_{q, i}}\) are such that \((p_0, v)\boldsymbol{\lambda}_{S_{q, i}}=(p_1, v)\boldsymbol{\lambda}_{S_{q, i}}\) for all \(v\in X_n^*\).
As \(T\cong_S M_T\), it suffices to show that \((p_0, w)\boldsymbol{\lambda}_{T}=(p_1, w)\boldsymbol{\lambda}_{T}\) for all \(w\in (X_n^*)^d\).

Let \(j\in \{0, 1, \ldots, d-1\}\), and \(w\in F_{d, n, \{j\}}\) be arbitrary. By Claim 2, it suffices to show that \((p_0, w)\boldsymbol{\lambda}_{T}\pi_j= (p_1, w)\boldsymbol{\lambda}_{T}\pi_j\).
If \(j= i\) then this follows by the assumption on \(p_0, p_1\). Otherwise let \(s\in F_{d, n, \{i\}}\) be such that \((p_0, s)\boldsymbol{\pi}_T= p_1\).
Then 
\begin{align*}
    (p_0, w)\boldsymbol{\lambda}_{T}\pi_j&=(p_0, w)\boldsymbol{\lambda}_{T}\pi_j((p_0, w)\boldsymbol{\pi}_T, s)\boldsymbol{\lambda}_{T}\pi_j=(p_0, w s)\boldsymbol{\lambda}_{T}\pi_j=(p_0, s w)\boldsymbol{\lambda}_{T}\pi_j\\
    &=((p_0, s)\boldsymbol{\lambda}_{T}(p_1, w)\boldsymbol{\lambda}_{T})\pi_j=(p_1, w)\boldsymbol{\lambda}_{T}\pi_j
\end{align*}
as required.\(\diamondsuit\)\\

\underline{Claim 4:}
For all \(w\in F_{d, n, \{0, 1, \ldots, d-1\}\backslash \{i\}}\) and \(i\in \{0, 1, \ldots, d-1\}\), the map \(p\mapsto (p, w)\pi_T\) defines a strong transducer isomorphism from \(S_{q, i}\) to \(S_{(q,w)\pi_T, i}\).\\
\underline{Proof of Claim:}
It follows from the fact that all elements of \(F_{d, n, \{0, 1, \ldots, d-1\}\backslash \{i\}}\) commute with all elements of \(F_{d, n,\{i\}}\), that this map is a homomorphism.

It remains to show that the map is invertable.
For all \(p\in [q]_{\sim_{i}}\), let \(w_p\in F_{d, n,\{i\}}\) be as in Claim 1.
Let \(w'\in (\{q\})\mathfrak{s}_T^{-1}\) be arbitrary, and let \(w''\in  F_{d, n, \{0, 1, \ldots, d-1\}\backslash \{i\}}\) be such that \(w\in w''F_{d, n,\{i\}}\).
Note that for all \(p\in Q_T\), we have \(q\sim_i (p, w'')\boldsymbol{\pi}_{T}\).
It follows that for all \(p\in [q]_{\sim_i}\), we have
\begin{align*}
    ((p, w)\boldsymbol{\pi}_T, w'')\boldsymbol{\pi}_T&=(((p, w_p)\boldsymbol{\pi}_T, w)\boldsymbol{\pi}_T, w'')\boldsymbol{\pi}_T = (p, w_pww'')\boldsymbol{\pi}_T\\
    &=(p, ww''w_p)\boldsymbol{\pi}_T=((p,w)\boldsymbol{\pi}_T, w''w_p)\boldsymbol{\pi}_T = p
\end{align*}
So the maps \(p\mapsto (p, w)\pi_T\) and \(p\mapsto (p, w'')\pi_T\) are mutually inverse homomorphisms as required.
\(\diamondsuit\)\\

By Claim 4, for each \(i\in \{0, 1,\ldots, d-1\},\) let \(T_{i}\) be an \(\widetilde{\mathcal{O}_{n, n-1}}\) transducer isomorphic to \(S_{q, i}\) for all \(q\in Q_T\).
Moreover for each \(q\in Q_T\) and \(i\in \{0, 1, \ldots, d-1\}\) let \(\phi_{q, i}:S_{q, i} \to T_i\) be a strong transducer isomorphism.
For all \(i\in \{0, 1, \ldots, d-1\}\), we define \({\phi_i}_D={\phi_i}_R\) to be the map \(\pi_i:(X_n^*)^d \to (X_n^*)\) and \({\phi_{i}}_Q\) to be the union of the maps \({\phi_{q,i}}_Q\) over all \(q\in Q_T\).

\noindent
\underline{Claim 5:} For all \(i\in \{0, 1, \ldots, d-1\}\), the triple \(\phi_i:= ({\phi_i}_Q, {\phi_i}_D, {\phi_i}_R)\) is a transducer homomorphism from \(T\) to \(T_i\).\\
\underline{Proof of Claim:}
If \(p, q\in Q_T\) and \(p\sim_i q\) and \(w\in \dom(\mathfrak{s}_{{S}_{q, i}})\), then 
\[((w)\mathfrak{s}_{{S}_{q, i}}){\phi_{q, i}}_Q=(w)\mathfrak{s}_{T_i}=((w)\mathfrak{s}_{{S}_{q, i}}){\phi_{p, i}}_Q\]
so \(\phi_{q, i} = \phi_{p,i}\). 
In particular \({\phi_{i}}_Q\) is a union of functions with disjoint domains and is hence a function.

We next show that \(\phi_i\) is a homomorphism.
Let \(q\in Q_T\) and \(w\in D_T\) be arbitrary. If \(w_q\in (\{q\})\mathfrak{s}_T^{-1}\), then
\[(q, w)\boldsymbol{\pi}_T{\phi_i}_Q=((w_q w)\pi_i)\mathfrak{s}_{T_i}= (((w_q)\pi_i)\mathfrak{s}_{T_i}, (w)\pi_i)\boldsymbol{\pi}_{T_i}=((q){\phi_i}_Q, (w)\pi_i)\boldsymbol{\pi}_{T_i}.\]
The condition on the output functions follows from Claim 2.
\(\diamondsuit\)d\\

Let \(P:= \prod_{i\in \{0, 1, \ldots, d-1\}} T_i\) (Definition~\ref{transducer products defn}), and \(\phi:= \langle \phi_0, \phi_1, \ldots, \phi_{d-1} \rangle_P\) (Definition~\ref{Products in Categories Defn}).
We need to show that \(\phi\) is a strong transducer isomorphism.
We have \(\phi_D=\phi_R=\langle \pi_0, \pi_1, \ldots, \pi_{d-1}\rangle_{(X_n^*)^d}= \text{id}_{(X_n^*)^d}\) by definition, so we need only show that \(\phi_Q\) is a bijection.

Let \((q_0, q_1, \ldots, q_{d- 1})\in Q_P\) be arbitrary. For each \(i\in \{0, 1, \ldots, d-1\},\) let \(w_i\in (\{q_i\})\mathfrak{s}_{T_i}^{-1}\). As \(\phi\) is a homomorphism, it follows that 
\[(q, (w_0, w_1, \ldots, w_{d-1}))\boldsymbol{\pi}_T\phi_Q = (q_0, q_1, \ldots, q_{d-1})\]
where \(q\in Q_T\) is arbitrary. In particular \(\phi\) is surjective.

By Lemma~\ref{minimal transducers are valid lemma}, it follows that there is a strong transducer homomorphism \(\psi\) such that \(\phi\psi =q_T\). As \(T \cong_S M_T\), the quotient map \(q_T\) is bijective, thus so is \(\phi\) as required.
\end{proof}

\speeddicttwo{159}{dOnn-1 group defn}{dOnn-1 monoid defn}{product_decomposition}
\begin{defn}[\(d\mathcal{P}_{n, n-1}\), \Assumed{159}]\label{dOnn-1 group defn}
We will show in Theorem~\ref{alpha} that \(\widetilde{d\mathcal{P}_{n, n-1}}\) is a monoid.
We define \(d\mathcal{P}_{n, n-1}\) to be the group of units of \(\widetilde{d\mathcal{P}_{n, n-1}}\).
Moreover we will say that a \(\widetilde{d\mathcal{P}_{n, n-1}}\) transducer \(T\) is \textit{invertable} if \([T]_{\cong_S}\in d\mathcal{P}_{n, n-1}\).
\end{defn}

\speeddictseven{152}{alpha}{induced maps preserved remark}{composition works as expected lemma}{removing incomplete responce works}{core monoids work lemma}{Onn-1 lemma}{product_decomposition}{dOnn-1 group defn}
\begin{theorem}[\(d\mathcal{P}_{n, n-1}\) embedding , \Assumed{152}]\label{alpha}
The function \(\phi: \widetilde{\mathcal{O}_{n,n-1}}^d\to \widetilde{d\mathcal{P}_{n, n-1}}\) defined by \(([T_0]_{\cong_S}, [T_1]_{\cong_S}, \ldots, [T_{d-1}]_{\cong_S})\phi = [\prod_{i\in \{0,1, \ldots, d-1\}} T_i]_{\cong_S}\) is a well-defined monoid isomorphism.
Moreover there is an embedding \(\Phi:d\mathcal{K}_{n, 1}\to d\mathcal{P}_{n, n-1}\) such that whenever
\(([T]_{\cong_S} )\Phi = [P]_{\cong_S}\) we have \(P\cong_S T\).
\end{theorem}
\begin{proof}
Note that if \(A_0, A_1,  \ldots, A_{d-1}, B_0, B_1,  \ldots, B_{d-1}\) are \(\widetilde{\mathcal{O}_{n,n-1}}\) transducers and
\[\prod_{i\in \{0, 1, \ldots, d-1\}} A_i \cong_S\prod_{i\in \{0, 1, \ldots, d-1\}} B_i,\]
then for all \(i\in \{0, 1, \ldots, d-1\}\) we have \(A_i\cong_S B_i\) (this is shown during the proof of Lemma~\ref{product_decomposition}).
Thus \(\phi\) is a well-defined bijection.

Let \([A]_{\cong_S}, [B]_{\cong_S}\in \widetilde{d\mathcal{P}_{n, n-1}}\) be arbitrary, and suppose that \([A]_{\cong_S}[B]_{\cong_S}=[C]_{\cong_S}\)
\[([A]_{\cong_S})\phi^{-1}=([A_0]_{\cong_S}, [A_{1}]_{\cong_S}, \ldots, [A_{d-1}]_{\cong_S}),\]
\[([B]_{\cong_S})\phi^{-1}=([B_0]_{\cong_S}, [B_{1}]_{\cong_S}, \ldots, [B_{d-1}]_{\cong_S}),\]
\[([C]_{\cong_S})\phi^{-1}=([C_0]_{\cong_S}, [C_{1}]_{\cong_S}, \ldots, [C_{d-1}]_{\cong_S}).\]

Also let \(\alpha : A\to \prod_{i\in \{0, 1, \ldots, d-1\}}A_i\), \(\beta : B\to \prod_{i\in \{0, 1, \ldots, d-1\}}B_i\), \(\gamma : C\to \prod_{i\in \{0, 1, \ldots, d-1\}}C_i\) be transducer isomorphisms.

Let \(i\in \{0, 1, \ldots, d-1\}\) be arbitrary.
To conclude that \(\phi\) is an isomorphism (and hence \(\widetilde{d\mathcal{P}_{n, n-1}}\) is a monoid) we need only show that \(M_{CR(\C(A_i B_i))} \cong_S C_i\).
Let \(q\in Q_{C}\) be arbitrary. 
By Remark~\ref{induced maps preserved remark}, Lemma~\ref{composition works as expected lemma} and Lemma~\ref{removing incomplete responce works}, there are \(q_a\in Q_A\), \(q_b\in Q_B\) and \(w\in (X_{n}^*)^d\) such that
\[f_{C, q} = f_{A, q_a}f_{B, q_b}\lambda_w^{-1}.\]
Moreover (as \(\alpha, \beta, \gamma\) are strong transducer isomorphisms) we have
\[f_{A, q}= \langle \pi_0f_{A_0,(q)\alpha_Q\pi_0}, \pi_1f_{A_1,(q)\alpha_Q\pi_1}, \ldots, \pi_{d-1}f_{A_{d-1},(q)\alpha_Q\pi_{d-1}}\rangle_{\mathfrak{C}_{n}^d},\]
\[f_{B, q}= \langle \pi_0f_{B_0,(q)\beta_Q\pi_0}, \pi_1f_{B_1,(q)\beta_Q\pi_1}, \ldots, \pi_{d-1}f_{B_{d-1},(q)\beta_Q\pi_{d-1}}\rangle_{\mathfrak{C}_{n}^d},\]
\[f_{C, q}= \langle \pi_0f_{C_0,(q)\gamma_Q\pi_0}, \pi_1f_{C_1,(q)\gamma_Q\pi_1}, \ldots, \pi_{d-1}f_{C_{d-1},(q)\gamma_Q\pi_{d-1}}\rangle_{\mathfrak{C}_{n}^d}.\]
So in particular we have
\(f_{C_i, (q)\gamma_Q\pi_i} = f_{A_i, (q_a)\gamma_Q\pi_i}f_{B_i, (q_b)\gamma_Q\pi_i}\lambda_{(w)\pi_i}^{-1}\).
Let \(P_i:= M_{CR(A_i B_i)}\).
As the \(\C\) operator commutes with the \(CR\) and \(M\) operators, we have that \(\C(P_i) = M_{CR(\C(A_i B_i))}\).

By Remark~\ref{induced maps preserved remark} and Lemma~\ref{composition works as expected lemma}, there is a state \(p\in Q_{P_i}\) and a word \(v\in X_n^*\) such that
\[f_{P_i, p} = f_{A_i, (q_a)\gamma_Q\pi_i}f_{B_i, (q_b)\gamma_Q\pi_i}\lambda_{v}^{-1}.\]
As \(P_i\) and \(C_i\) both have complete response, it follows that \(f_{P_i, p}=f_{C_i, (q)\gamma_Q\pi_i}\).
Thus by Lemma~\ref{Onn-1 lemma}, we have 
\[C_i =\C(C_i) \cong_S \C(P_i) = M_{CR(\C(A_i B_i))}\]
as required.

We now prove the final part of the theorem. 
By Lemma~\ref{product_decomposition}, there is a unique map \(\Phi:\mathcal{K}_{n, 1}\to d\widetilde{\mathcal{O}_{n,n-1}}\) such that the required isomorphism type condition holds. Moreover, because \(\Phi\) satisfies this condition, \(\Phi\) must be injective.
The fact that \(\Phi\) is a homomorphism follows from Lemma~\ref{core monoids work lemma}.

Note that \(\Phi\) is a homomorphism, \(d\mathcal{K}_{n, 1}\) is a group, and \(\Phi\) maps the identity of \(d\mathcal{K}_{n, 1}\) to the identity of \(\widetilde{d\mathcal{P}_{n, n-1}}\). Thus \((d\mathcal{K}_{n, 1})\Phi\) is contained in the group of units of \(\widetilde{d\mathcal{P}_{n, n-1}}\). So \(\Phi\) is indeed an map from \(d\mathcal{K}_{n, 1}\) to \(d\mathcal{P}_{n, n-1}\) as required.
\end{proof}

\speeddictthree{154}{wreath_embedding theorem}{first_semidirect}{alpha}{wreath product defn}
\begin{theorem}[Wreath embedding, \Assumed{154}]\label{wreath_embedding theorem}
The group \(\Out(dV_n)\) embeds in the group
\(\mathcal{O}_{n, n-1}\wr_a \Sym(d)\) where \(a\) is the usual action of \(\Sym(d)\) on \(\{0, 1, \ldots, d-1\}\).
\end{theorem}
\begin{proof}
This follows from Theorems~\ref{first_semidirect} and \ref{alpha}.
\end{proof}

Theorem~\ref{wreath_embedding theorem}, gives us embeddings of both \(\Out(dV_n)\) and \(\Out(V_n)\wr_a \Sym(d)\) into the group \(\mathcal{O}_{n, n-1}\wr_a \Sym(d)\).
Thus to compare these groups we need to establish a useful description of the images of these embeddings.
To do this we use the notion of \textit{signatures} from \cite{olukoya2019automorphisms} Definition 7.6.
Each clopen subset of \(\mathfrak{C}_n^d\) can be expressed as a finite union of cones (Definition~\ref{cones words and Cantor spaces defn}). 
Thus any homeomorphism of \(\mathfrak{C}_n^d\), must map every cone to a finite collection of cones, the signature map measures the degree to which cones are split.

\speeddictfive{155}{transducer signature defn}{Numbers defns}{centralizer defn}{transducer products defn}{Onn-1 defn}{alpha}
\begin{defn}[Clopen set signatures, \Assumed{155}]\label{transducer signature defn}
For each \(d\in \N\backslash \{0\}\) and \(n\in \N\backslash\{0, 1\}\), we define a map \(\overline{\text{ssig}_{d, n}}:\makeset{S\subseteq \mathfrak{C}_n^d}{\(S\) is clopen} \to \mathbb{Z}/(n-1)\mathbb{Z}\) by 
\((S)\overline{\text{ssig}_{d, n}} =  m + (n-1)\mathbb{\Z}\) where \(m\) is such that there are pairwise disjoint \(w_{0}\mathfrak{C}_{n}^d, w_{1}\mathfrak{C}_{n}^d, \ldots, w_{m-1}\mathfrak{C}_{n}^d\) with
\[(w\mathfrak{C}_{n})f_{T, q} = \union{i\in \{0, 1, \ldots, m-1\}}w_i\mathfrak{C}_n^d.\]
We then define a map \(\overline{\text{sig}_{d, n}}:\widetilde{d\mathcal{P}_{n, n-1}} \to \mathbb{Z}/(n-1)\mathbb{Z}\) by 
\[([P]_{\cong_S})\overline{\text{sig}_{d, n}} = ((w\mathfrak{C}_{n})f_{P, q})\overline{\text{ssig}_{d, n}}\]
where \(w\in (X_n^*)^d\), \(q \in Q_P\) are arbitrary. We show in Lemma~\ref{signatures work lemma} that these maps are well-defined.
\end{defn}

\speeddictfour{156}{signatures work lemma}{induced maps preserved remark}{composition works as expected lemma}{removing incomplete responce works}{transducer signature defn}
\begin{lemma}[Signatures work, \Assumed{156}]\label{signatures work lemma}
Let \(d\in \N\backslash \{0\}\) and \(n\in \N\backslash\{0, 1\}\). 
The map \(\overline{\text{sig}_{d, n}}:\widetilde{d\mathcal{P}_{n, n-1}} \to \mathbb{Z}/(n-1)\mathbb{Z}\) of Definition~\ref{transducer signature defn} is well-defined, moreover if we view  \(\mathbb{Z}/(n-1)\mathbb{Z}\) as a semigroup with multiplication of integers as the operation, then \(\overline{\text{sig}_{d, n}}\) is a semigroup homomorphism.
\end{lemma}
\begin{proof}
We first show that \(\overline{\text{ssig}_{d, n}}\) is well defined. Let \(S\) be a clopen subset of \(\mathfrak{C}_n^d\). Suppose that \(F_1, F_2 \subseteq (X_n^*)^d\) are such that both collections of sets \(\makeset{v\mathfrak{C}_n^d}{\(v\in F_1\)}\) and \(\makeset{v\mathfrak{C}_n^d}{\(v\in F_2\)}\) are individually pairwise disjoint, and
\[S = \union{v\in F_1} v\mathfrak{C}_n^d = \union{v\in F_2} v\mathfrak{C}_n^d.\]
As \(S\) is compact, it follows that both \(F_1\) and \(F_2\) are finite. 
We define
\[k:= \max\left(\makeset{|(v)\pi_i|}{\(v\in F_1 \cup F_2, i\in \{0, 1, \ldots, d-1\}\)}\right),\] and define \(F_3:= \makeset{v\in (X_n^k)^d}{\(v\mathfrak{C}_n^d \subseteq S\)}\).
For all \(v\in F_1\cup F_2\), we have
\(v\mathfrak{C}_n^d = \union{}\makeset{u\mathfrak{C}_n^d}{\(u\in F_3, v\leq u\)}.\)
Thus by Lemma~\ref{anti-chain sizes}, we have
\[\sum_{v\in F_1} 1 + (n-1)\Z = \sum_{v\in F_1} \left|\makeset{u\in F_3}{\(v\leq u\)}\right| + (n-1)\Z = |F_3| +(n-1)\Z.\]
In particular \(|F_1|+(n-1)\Z = |F_3| + (n-1)\Z\) and by a symmetric argument, we have \(|F_2| + (n-1)\Z=|F_3| + (n-1)\Z=|F_1|+(n-1)\Z\). Thus the map \(\overline{\text{ssig}_{d, n}}\) is well-defined.

Let \([P]_{\cong_S}\in \widetilde{d\mathcal{P}_{n, n-1}}\), \(w\in (X_n^*)^d\) and \(q\in Q_P\) be arbitrary.
Let \(k\) be the synchronizing length of \(P\).
To conclude that \(\overline{\text{sig}_{d, n}}\) is well-defined, we need only show that \(((w\mathfrak{C}_n^d)f_{P, q})\overline{\text{ssig}_{d, n}}\) is independent of \(w\) and \(q\).
This can be seen as follows:
\begin{align*}
    ((w\mathfrak{C}_n^d)f_{P, q})\overline{\text{ssig}_{d, n}} &= \left(\left(\union{v\in (X_n^k)^d}w v\mathfrak{C}_n^d\right)f_{P, q}\right)\overline{\text{ssig}_{d, n}}= \left(\union{v\in (X_n^k)^d}\left(w v\mathfrak{C}_n^d\right)f_{P, q}\right)\overline{\text{ssig}_{d, n}}\\
    &= \sum_{v\in (X_n^k)^d}\left(\left(w v\mathfrak{C}_n^d\right)f_{P, q}\right)\overline{\text{ssig}_{d, n}}= \sum_{v\in (X_n^k)^d}\left((q, w v)\boldsymbol{\lambda}_P\left(\mathfrak{C}_n^d\right)f_{P, (w v)\mathfrak{s}_P}\right)\overline{\text{ssig}_{d, n}}\\
    &= \sum_{v\in (X_n^k)^d}\left(\left(\mathfrak{C}_n^d\right)f_{P, (w v)\mathfrak{s}_P}\right)\overline{\text{ssig}_{d, n}}= \sum_{v\in (X_n^k)^d}\left(\left(\mathfrak{C}_n^d\right)f_{P, (v)\mathfrak{s}_P}\right)\overline{\text{ssig}_{d, n}}.
\end{align*}
We now need to show that \(\overline{\text{sig}_{d, n}}\) is a semigroup homomorphism.
Let \(A, B, C\in \widetilde{d\mathcal{P}_{n, n-1}}\) be such that \([A]_{\cong_S} [B]_{\cong_S}= [C]_{\cong_S}\). 
Let \(q\in Q_C\), by Remark~\ref{induced maps preserved remark}, Lemma~\ref{composition works as expected lemma} and \ref{removing incomplete responce works}, there is \((q_a, q_b)\in Q_A\times Q_B\) and \(w\in X_n^*\) such that \(f_{C, q}=f_{A, q_a}f_{B, q_b}\lambda_{w}^{-1}\).
Moreover, let \(F\subseteq (X_n^*)^d\) be arbitrary such that \((\mathfrak{C}_n)f_{A, q_a} = \union{v\in F}v\mathfrak{C}_n\) and such that the sets of the collection \(\makeset{v\mathfrak{C}_n}{\(v\in F\)}\) are pairwise disjoint. We now have
\begin{align*}
    ([C]_{\cong_S})\overline{\text{sig}_{d, n}}&=((\mathfrak{C}_n)f_{C, q})\overline{\text{ssig}_{d, n}}=((\mathfrak{C}_n)f_{A, q_a}f_{B, q_b}\lambda_{w}^{-1})\overline{\text{ssig}_{d, n}}=((\mathfrak{C}_n)f_{A, q_a}f_{B, q_b})\overline{\text{ssig}_{d, n}}\\
    &=\left(\left(\union{v\in F}v\mathfrak{C}_n\right)f_{B, q_b}\right)\overline{\text{ssig}_{d, n}}=\left(\union{v\in F}\left(v\mathfrak{C}_n\right)f_{B, q_b}\right)\overline{\text{ssig}_{d, n}}\\
     &=\sum_{v\in F}\left(\left(v\mathfrak{C}_n\right)f_{B, q_b}\right)\overline{\text{ssig}_{d, n}}     =\sum_{v\in F}\left([B]_{\cong_S}\right)\overline{\text{sig}_{d, n}}=\left([B]_{\cong_S}\right)\overline{\text{sig}_{d, n}}|F|\\
&=\left([A]_{\cong_S}\right)\overline{\text{sig}_{d, n}}\left([B]_{\cong_S}\right)\overline{\text{sig}_{d, n}}.
\end{align*}
\end{proof}

\speeddictfour{160}{higher dimensional signatures}{cones words and Cantor spaces defn}{alpha}{transducer signature defn}{signatures work lemma}
\begin{lemma}[Higher dimensional signatures from decomposition, \Assumed{160}]\label{higher dimensional signatures}
Let \(d\in \N\backslash\{0\}\) and \(n\in \N\backslash \{0, 1\}\).
Let \(\phi\) be as in Theorem~\ref{alpha}.
If \([T]_{\cong_S}\in \widetilde{d\mathcal{P}_{n, n - 1}} \), then
\[([T]_{\cong_S})\overline{\text{sig}_{d, n}}=\prod_{i\in \{0, 1, \ldots, d-1\}} (([T]_{\cong_S})\phi^{-1}\pi_i)\overline{\text{sig}_{1, n}}\]
\end{lemma}
\begin{proof}
Let \(T_0, T_1,  \ldots, T_{d-1}\) be arbitrary \(\widetilde{\mathcal{O}_{n, n - 1}}\) transducers, and let \(P:= \prod_{i\in \{0, 1, \ldots, d-1\}} T_i\).
For each \(i\in \{0, 1, \ldots, d-1\}\), let \(q_i\in Q_{T_i}\), and \(F_i\subseteq X_n^*\) be such that the elements of \(F_i\) are pairwise incomparable and
\[(\mathfrak{C}_n)f_{T_i, q_i}= \union{w\in F_i} w\mathfrak{C}_n.\]

We now have
\begin{align*}
    ((\mathfrak{C}_n^d)f_{P, (q_0, q_1, \ldots, q_{d-1})})\overline{\text{ssig}_{d, n}}&= \left(\prod_{i\in \{0, 1, \ldots, d-1\}}(\mathfrak{C}_n)f_{T_i, q_i}\right)\overline{\text{ssig}_{d, n}}= \left(\prod_{i\in \{0, 1, \ldots, d-1\}}\union{w\in F_i} w\mathfrak{C}_n\right)\overline{\text{ssig}_{d, n}}\\
    &= \left(\union{w\in F_0\times F_1 \times\ldots \times F_{d-1}} w\mathfrak{C}_n^d\right)\overline{\text{ssig}_{d, n}}= \prod_{i\in \{0, 1, \ldots, d-1\}} |F_i| + (n-1)\Z.
\end{align*}
By definition, \(((\mathfrak{C}_n^d)f_{P, (q_0, q_1, \ldots, q_{d-1})})\overline{\text{ssig}_{d, n}}= ([P]_{\cong_S})\overline{\text{sig}_{d, n}}\) and for each \(i\in \{0, 1, \ldots, d-1\}\) we have \(|F_i|+(n-1)\Z=([T_i]_{\cong_S})\overline{\text{sig}_{1, n}}\). Thus the result follows.
\end{proof}

In \cite{olukoya2019automorphisms} Proposition 7.7, the notion of signature is used to describe \(\mathcal{O}_{n, r}\) as subgroup of \(\mathcal{O}_{n, n-1}\).
We can now to the same with \(d\mathcal{K}_{n, 1}\) and \(d\mathcal{P}_{n, n-1}\).

\speeddictseven{161}{lemma2}{anti-chain sizes}{transducers are continuous lemma}{unique minimal thm}{induced maps preserved remark}{removing incomplete responce works}{alpha}{signatures work lemma}
\begin{lemma}[Identifying \(d\mathcal{K}_{n, 1}\) with signatures, \Assumed{161}]\label{lemma2}
Let \(d\in \N\backslash \{0\}\) and \(n\in \N\backslash \{0, 1\}\).
 If \(\Phi:d\mathcal{K}_{n, 1}\to d\mathcal{P}_{n, n-1}\) is as in Theorem~\ref{alpha}, then 
 \[(d\mathcal{K}_{n, 1})\Phi = (\{1 + (n-1)\Z\})\overline{\text{sig}_{d, n}}^{-1}\cap d\mathcal{P}_{n, n-1}.\]
\end{lemma}

\begin{proof}
\((\subseteq):\) Let \([T]_{\cong_S}\in d\mathcal{K}_{n, 1}\) be arbitrary, and let \(f~\in~d\mathcal{B}_{n, 1}\) be such that \(T= \C(M_f)\). Let \(k\) be the synchronizing length of \(M_f\). It follows that
\begin{align*}
    1+(n-1)\mathbb{Z}  &= (\mathfrak{C}_n^d)\overline{\text{ssig}_{d, n}}= ((\mathfrak{C}_n^d)f)\overline{\text{ssig}_{d, n}}= \left(\union{w\in (X_n^k)^d}(w\mathfrak{C}_n^d)f\right)\overline{\text{ssig}_{d, n}}\\
     &= \left(\union{w\in (X_n^k)^d}(e_f, w)\boldsymbol{\lambda}_{M_f}(\mathfrak{C}_n^d)f_{T, (w)\mathfrak{s}_{T}}\right)\overline{\text{ssig}_{d, n}}\\
     &= \sum_{w\in (X_n^k)^d}\left((e_f, w)\boldsymbol{\lambda}_{M_f}(\mathfrak{C}_n^d)f_{T, (w)\mathfrak{s}_{T}}\right)\overline{\text{ssig}_{d, n}}=\sum_{w\in (X_n^k)^d} ((\mathfrak{C}_n^d)f_{T,(w) \mathfrak{s}_{T}})\overline{\text{ssig}_{d, n}}\\
    &=\sum_{w\in (X_n^k)^d} (([T]_{\cong_S})\Phi)\overline{\text{sig}_{d, n}}=n^{k d} (([T]_{\cong_S})\Phi)\overline{\text{sig}_{d, n}}=(([T]_{\cong_S})\Phi)\overline{\text{sig}_{d, n}}
\end{align*}
as required.

\((\supseteq):\) Let \([P]_{\cong_S}\in (\{1 + (n-1)\Z\})\overline{\text{sig}_{d, n}}^{-1}\cap d\mathcal{P}_{n, n-1} \), and \(q\in Q_P\) be arbitrary.
Then let \((a\mathfrak{C}_n^d)_{{a\in A}}\) be disjoint cones with union \((\mathfrak{C}_n^d)f_{P, q}\).
By the choice of \(P\), we have \(|A| \in 1 + (n-1)\mathbb{Z}\).
By Lemma~\ref{transducers are continuous lemma}, we can choose \(k\in \mathbb{N}\) greater that \(|A|\), such that for all \(w\in (X_n^k)^d\) and \(i\in \{0, 1, \ldots, d-1\}\), we have 
\[|(q, w)\boldsymbol{\lambda}_P\pi_i| \geq \max\makeset{|(a)\pi_j|}{\(a\in A, j\in \{0, 1, \ldots, d-1\}\)}.\]
For all \(w\in (X_n^k)^d\) let \(a_w\in A\) be the unique element such that \(a_w\leq (q, w)\boldsymbol{\lambda}_P\).
By the definition of \(A\), for all \(a\in A\) we now have
\[\union{w\in (X_n^k)^d \\ a_w = a} (w\mathfrak{C}_n^d)f_{P, q} = a\mathfrak{C}_n^d.\]
In particular
\[\makeset{((q, w)\boldsymbol{\lambda}_P\lambda_{a_w}^{-1})(\mathfrak{C}_n^d)f_{P, (q, w)\pi_T}}{\(w\in (X_n^k)^d\) satisfies \(a_w = a\)}\]
is a partition of \(\mathfrak{C}_n^d\).

As \(|A|\in 1 + (n-1)\mathbb{N}\), we have by Lemma~\ref{anti-chain sizes} that there is a complete prefix code \(B\) for \(\mathfrak{C}_n^d\) with \(|B|=|A|\).
Let \(\phi: A \to B\) be a bijection.
It follows that
\[\makeset{(a_w)\phi((q, w)\boldsymbol{\lambda}_P\lambda_{a_w}^{-1})(\mathfrak{C}_n^d)f_{P, (q, w)\pi_T}}{\(w\in (X_n^k)^d\)}\]
 is a partition of \(\mathfrak{C}_n^d\).
We define a \((d, n)\)-transducer \(D\) as follows:
\begin{enumerate}
    \item \(N :=\makeset{w\in (X_n^*)^d}{there is \(i\in \{0, 1,\ldots, d-1\}\) with \((w)\pi_i<k\)}\) (we assume without loss of generality that \(N\cap Q_P = \varnothing\)).
    \item \(Q_D := Q_P\cup N\).
    \item If \(p\in N\), \(w\in (X_n^*)^d\) and \(p w\in N\), then \((p, w)\boldsymbol{\pi}_D = p w\) and \((p, w)\boldsymbol{\lambda}_D = \varepsilon_d\).
    \item If \(p\in N\), \(w\in (X_n^*)^d\) and \(p w=b t\) where \(b\in (X_n^k)^d\) and \(t\in (X_n^*)^d\), then 
    \[(p, w)\boldsymbol{\pi}_D := ((q, w)\boldsymbol{\pi}_P, t)\boldsymbol{\pi}_P, \quad (p, w)\boldsymbol{\lambda}_D := (a_w)\phi((q, w)\boldsymbol{\lambda}_P\lambda_{a_w}^{-1})((q, w)\boldsymbol{\pi}_P, t)\boldsymbol{\lambda}_P.\]
    \item If \(p\in Q_P\) and \(w\in (X_n^*)^d\) then 
    \[(p, w)\boldsymbol{\pi}_D :=(p, w)\boldsymbol{\pi}_P, \quad (p, w)\boldsymbol{\lambda}_D :=(p, w)\boldsymbol{\lambda}_P.\]
\end{enumerate}
It is routine to verify that \(D\) is a \((d, n)\)-transducer. 
By definition, the transducer \(D\) is synchronizing with synchronizing length at most \(k\) plus the synchonizing length of \(P\). 
Moreover \(\C(D) = P\). We now define \(f:=f_{D, \varepsilon_d}\)

If \(w\in (X_n^k)^d\), \(x\in \mathfrak{C}_n^d\), then
\[(w x)f = (a_w)\phi((q, w)\boldsymbol{\lambda}_P\lambda_{a_w}^{-1}) (x)f_{P, (q, w)\boldsymbol{\pi}_P},\]
as each of the maps \(f_{P, (q, w)\boldsymbol{\pi}_P}\) is a homeomorphism to an open subset of \(\mathfrak{C}_n^d\), it follows that \(f\in \Aut(\mathfrak{C}_n^d)\).

By Remark~\ref{induced maps preserved remark} and Lemma~\ref{removing incomplete responce works}, together with Theorem~\ref{unique minimal thm}, we see that \(M_{CR(D)}\cong_S M_{f}\), and \(\C(M_{CR(D)}) \cong_S P\).
Thus if \(f\in d\mathcal{B}_{n, 1}\), then \([M_{f}]_{\cong_S}\in d\mathcal{O}_{n, 1}\) and \(([M_{f}]_{\cong_S})\Phi = [P]_{\cong_S}\) as required.

We have already shown that \(f\in d\mathcal{S}_{n, 1}\), thus we need only show that \(f^{-1}\in d\mathcal{S}_{n,1}\).
By assumption, the transducer \(P\) is invertable, so let \(P^{-1}\in [P]_{\cong_S}^{-1}\) be arbitrary.
By the same construction given for \(P\), let \(g\in d\mathcal{S}_{n, 1}\) be such that \(\C(M_g) \cong_S P^{-1}\).

We thus have \([\C(M_{f g})]_{\cong_S} = [\C(M_f)]_{\cong_S}[\C(M_g)]_{\cong_S}= 1_{d\mathcal{O}_{n, 1}}\). So by Lemma~\ref{nV_placement}, we have \(f g\in dV_n\) and in particular \((f g)^{-1}\in d\mathcal{S}_{n,1}\). Therefore \(f^{-1} = g(f g)^{-1}\in d\mathcal{S}_{n, 1}\) as well.
\end{proof}

\speeddictfive{162}{main theorem 1}{autnV}{first_semidirect}{alpha}{higher dimensional signatures}{lemma2}
\begin{theorem}[Connecting \(\Out(V_n)\) with \(\Out(dV_n)\), \Assumed{162}]\label{main theorem 1}
For all \(d\in \N\backslash\{0\}\) and \(n\in \N\backslash\{0, 1\}\), we have
\[\Out(dV_n) \cong \makeset{(\boldsymbol{T}, f)\in\mathcal{O}_{n, n-1}\wr_a  \Sym(d)}{\((\boldsymbol{T})\phi\overline{\text{sig}_{d, n}} = 1+(n-1)\Z\)},\]
Moreover, the group \(\Out(V_n)\wr_a \Sym(d)\) embeds in the group \(\Out(dV_n)\). 
Here \(a\) is the standard action of \(\Sym(d)\) on \(\{0, 1, \ldots, d-1\}\), and \(\phi\) is as in Theorem~\ref{alpha}.
 \(\Out(dV_n)\).
\end{theorem}
\begin{proof}
By Theorem~\ref{autnV} we have \(\Out(dV_n) \cong d\mathcal{O}_{n,1}\). 
Thus from Theorem~\ref{first_semidirect} and Lemma~\ref{lemma2}, we have
\[\Out(dV_n) \cong \makeset{(\boldsymbol{T}, f)\in\mathcal{O}_{n, n-1}\wr_a  \Sym(d)}{\((\boldsymbol{T})\phi\overline{\text{sig}_{d, n}} = 1+(n-1)\Z\)}.\]
In particular \(\Out(V_n) \cong \makeset{(\boldsymbol{T}, f)\in\mathcal{O}_{n, n-1}}{\((\boldsymbol{T})\phi\overline{\text{sig}_{1, n}} = 1+(n-1)\Z\)}\).
The result the follows from Lemma~\ref{higher dimensional signatures}.
\end{proof}

\speeddictone{163}{main theorem 2}{main theorem 1}
\begin{theorem}[Connecting \(\Out(V)\) with \(\Out(dV)\), \Assumed{163}]\label{main theorem 2}
For all \(d\in \N\backslash \{0\}\), we have \(\Out(dV) \cong \Out(V)\wr_a \Sym(d)\) (where \(a\) is the standard action of \(\Sym(d)\) on \(\{0, 1, \ldots, d-1\}\)).
\end{theorem}
\begin{proof}
Note that \(|\Z/(2-1)\Z| = 1\), thus by Theorem~\ref{main theorem 1}, we have
\(\Out(V)\cong \mathcal{O}_{n, n-1}\) and \(\Out(dV)\cong \mathcal{O}_{n, n-1}\wr_a \Sym(d)\).  
\end{proof}

\cite{https://doi.org/10.48550/arxiv.2203.11577}


\bibliography{bib.bib}

@misc{bleak2016,
    title={The further chameleon groups of {R}ichard {T}hompson and {G}raham {H}igman: Automorphisms via dynamics for the {H}igman groups ${G}_{n,r}$},
    author={Collin Bleak and Peter Cameron and Yonah Maissel and Andrés Navas and Feyishayo Olukoya},
    year={2016},
    eprint={1605.09302},
    archivePrefix={arXiv},
    primaryClass={math.GR}
}

@Article{Brin2004,
author="Brin, Matthew G.",
title="Higher Dimensional {T}hompson Groups",
journal="Geometriae Dedicata",
year="2004",
month="Oct",
day="01",
volume="108",
number="1",
pages="163--192",
issn="1572-9168",
doi="10.1007/s10711-004-8122-9",
url="https://doi.org/10.1007/s10711-004-8122-9"
}

@misc{https://doi.org/10.48550/arxiv.2203.11577,
  doi = {10.48550/ARXIV.2203.11577},
  
  url = {https://arxiv.org/abs/2203.11577},
  
  author = {Elliott, L. and Jonušas, J. and Mitchell, J. D. and Péresse, Y. and Pinsker, M.},
  
  title = {Polish topologies on endomorphism monoids of relational structures},
  
  publisher = {arXiv},
  
  year = {2022},
  
  copyright = {arXiv.org perpetual, non-exclusive license}
}

@article{Brin2005,
   title={Presentations of higher dimensional {T}hompson groups},
   volume={284},
   ISSN={0021-8693},
   url={http://dx.doi.org/10.1016/j.jalgebra.2004.10.028},
   DOI={10.1016/j.jalgebra.2004.10.028},
   number={2},
   journal={Journal of Algebra},
   publisher={Elsevier BV},
   author={Brin, Matthew G.},
   year={2005},
   month={Feb},
   pages={520–558}}

@misc{hennig2011,
    title={Presentations for the higher dimensional {T}hompson's groups n{V}},
    author={Johanna Hennig and Francesco Matucci},
    year={2011},
    eprint={1105.3714},
    archivePrefix={arXiv},
    primaryClass={math.GR}
}

@misc{quick2019,
    title={Permutation-based presentations for {B}rin's higher-dimensional {T}hompson groups $n{V}$},
    author={Martyn Quick},
    year={2019},
    eprint={1901.04409},
    archivePrefix={arXiv},
    primaryClass={math.GR}
}

@Article{Bleak2010,
author="Bleak, Collin
and Lanoue, Daniel",
title="A family of non-isomorphism results",
journal="Geometriae Dedicata",
year="2010",
month="Jun",
day="01",
volume="146",
number="1",
pages="21--26",
issn="1572-9168",
doi="10.1007/s10711-009-9423-9",
url="https://doi.org/10.1007/s10711-009-9423-9"
}

@article{Belk2016,
   title={Some undecidability results for asynchronous transducers and the {B}rin-{T}hompson group $2{V}$},
   volume={369},
   ISSN={1088-6850},
   url={http://dx.doi.org/10.1090/tran/6963},
   DOI={10.1090/tran/6963},
   number={5},
   journal={Transactions of the American Mathematical Society},
   publisher={American Mathematical Society (AMS)},
   author={Belk, James and Bleak, Collin},
   year={2016},
   month={Dec},
   pages={3157–3172}
}

@misc{olukoya2019automorphisms,
    title={Automorphisms of the generalised {T}hompson's group ${T}_{n,r}$},
    author={Feyishayo Olukoya},
    year={2019},
    eprint={1908.03816},
    archivePrefix={arXiv},
    primaryClass={math.GR}
}

@article{GNS2000,
    AUTHOR = {Grigorchuk, R. I. and Nekrashevich, V. V. and
              Sushchanski{\u\i}, V. I.},
     TITLE = {Automata, dynamical systems, and groups},
   JOURNAL = {Proc. Steklov Inst. Math},
  FJOURNAL = {Trudy Matematicheskogo Instituta Imeni V. A. Steklova.
              Rossi\u\i skaya Akademiya Nauk},
    VOLUME = {231},
      YEAR = {2000},
     PAGES = {128--203},
 }

@preamble{
   "\def\cprime{$'$} "
}

@article {Brin_autF,
    AUTHOR = {Brin, Matthew G.},
     TITLE = {The chameleon groups of {R}ichard {J}. {T}hompson:
              automorphisms and dynamics},
   JOURNAL = {Inst. Hautes \'Etudes Sci. Publ. Math.},
  FJOURNAL = {Institut des Hautes \'Etudes Scientifiques. Publications
              Math\'ematiques},
    NUMBER = {84},
      YEAR = {1996},
     PAGES = {5--33 (1997)},
}

@misc{OlukoyaAutTnr,
title={Automorphisms of the generalised {T}hompson's group {$T_{n,r}$}},
author={Feyishayo Olukoya},
pages={1--35},
note={In Preparation, 2018},}

@misc{lawson2019higher,
    title={Higher dimensional generalizations of the {T}hompson groups},
    author={Mark V Lawson and Alina Vdovina},
    year={2019},
    eprint={1909.13254},
    archivePrefix={arXiv},
    primaryClass={math.RA}
}

@article{Brin_1998,
   title={Automorphisms of Generalized {T}hompson Groups},
   volume={203},
   ISSN={0021-8693},
   url={http://dx.doi.org/10.1006/jabr.1997.7315},
   DOI={10.1006/jabr.1997.7315},
   number={1},
   journal={Journal of Algebra},
   publisher={Elsevier BV},
   author={Brin, Matthew G. and Guzmán, Fernando},
   year={1998},
   month={May},
   pages={285–348}
}

@article{good_generic_cantor,
   title={The group of homeomorphisms of the {C}antor set has ample generics},
   volume={44},
   ISSN={0024-6093},
   url={http://dx.doi.org/10.1112/blms/bds039},
   DOI={10.1112/blms/bds039},
   number={6},
   journal={Bulletin of the London Mathematical Society},
   publisher={Wiley},
   author={Kwiatkowska, Aleksandra},
   year={2012},
   month={Apr},
   pages={1132–1146}
}

@article{kechris2007turbulence,
  title={Turbulence, amalgamation, and generic automorphisms of homogeneous structures},
  author={Kechris, Alexander S and Rosendal, Christian},
  journal={Proceedings of the London Mathematical Society},
  volume={94},
  number={2},
  pages={302--350},
  year={2007},
  publisher={Wiley Online Library}
}

@article{projective_fr,
author = {Irwin, Trevor and Solecki, Slawomir},
year = {2006},
month = {07},
pages = {},
title = {Projective {Fra\"{\i}ss\'e} limits and the pseudo-arc},
volume = {358},
journal = {Transactions of the American Mathematical Society},
doi = {10.2307/3845566}
}

@article{rubin1989reconstruction,
  title={On the reconstruction of topological spaces from their groups of homeomorphisms},
  author={Rubin, Matatyahu},
  journal={Transactions of the American Mathematical Society},
  volume={312},
  number={2},
  pages={487--538},
  year={1989}
}

@misc{AcPaper1,
      title={Automatic continuity, unique {P}olish topologies, and {Z}ariski topologies on monoids and clones}, 
      author={L. Elliott and J. Jonušas and Z. Mesyan and J. D. Mitchell and M. Morayne and Y. Péresse},
      year={2021},
      eprint={1912.07029},
      archivePrefix={arXiv},
      primaryClass={math.RA}
}

@article{Paolini2020,
  doi = {10.4064/fm723-4-2019},
  url = {https://doi.org/10.4064/fm723-4-2019},
  year = {2020},
  publisher = {Institute of Mathematics,  {P}olish Academy of Sciences},
  volume = {248},
  number = {3},
  pages = {301--307},
  author = {Gianluca Paolini and Saharon Shelah},
  title = {Automorphism groups of countable stable structures},
  journal = {Fundamenta Mathematicae}
}

@article{Dixon1986aa,
	Author = {Dixon, John D. and Neumann, Peter M. and Thomas, Simon},
	Coden = {LMSBBT},
	Fjournal = {The Bulletin of the London Mathematical Society},
	Issn = {0024-6093},
	Journal = {Bull. London Math. Soc.},
	Mrclass = {20B35},
	Mrnumber = {MR859950 (88i:20004)},
	Mrreviewer = {H. Dugald Macpherson},
	Number = {6},
	Pages = {580--586},
	Title = {Subgroups of small index in infinite symmetric groups},
	Volume = {18},
	Year = {1986}}

@article{Tsankov2013aa,
	Author = {Tsankov, Todor},
	Doi = {10.1090/S0002-9939-2013-11666-7},
	Fjournal = {Proceedings of the American Mathematical Society},
	Issn = {0002-9939},
	Journal = {Proc. Amer. Math. Soc.},
	Mrclass = {46H40 (54H11)},
	Mrnumber = {3080189},
	Mrreviewer = {Jos{\'e} Extremera},
	Number = {10},
	Pages = {3673--3680},
	Title = {Automatic continuity for the unitary group},
	Url = {https://doi.org/10.1090/S0002-9939-2013-11666-7},
	Volume = {141},
	Year = {2013}}

@article{Yaacov2010aa,
	Author = {Yaacov, Ita{\"\i} and Berenstein, Alexander and Melleray, Julien},
	Doi = {10.1090/S0002-9947-2013-05773-X},
	Journal = {Trans. Amer. Math. Soc.},
	Month = {07},
	Title = {{P}olish topometric groups},
	Volume = {365},
	Year = {2010}}

@article{Mann2016aa,
	Author = {Kathryn Mann},
	Doi = {10.2140/gt.2016.20.3033},
	Journal = {Geometry {\&} Topology},
	Month = oct,
	Number = {5},
	Pages = {3033--3056},
	Publisher = {Mathematical Sciences Publishers},
	Title = {Automatic continuity for homeomorphism groups and applications},
	Url = {https://doi.org/10.2140/gt.2016.20.3033},
	Volume = {20},
	Year = {2016}}

@article{Herwig1998aa,
	Author = {Bernhard Herwig},
	Journal = {Israel Journal of Mathematics},
	Month = dec,
	Number = {1},
	Pages = {93--123},
	Publisher = {Springer Science and Business Media {LLC}},
	Title = {Extending partial isomorphisms for the small index property of many $\omega$-categorical structures},
	Volume = {107},
	Year = {1998}}

@article{Hodges1993ab,
  doi = {10.1112/jlms/s2-48.2.204},
  url = {https://doi.org/10.1112/jlms/s2-48.2.204},
  year = {1993},
  publisher = {Wiley},
  volume = {s2-48},
  number = {2},
  pages = {204--218},
  author = {Wilfrid Hodges and Ian Hodkinson and Daniel Lascar and Saharon Shelah},
  title = {The {S}mall {I}ndex {P}roperty for $\omega$-{S}table ($\omega$-{C}ategorical {S}tructures and for the {R}andom {G}raph)},
  journal = {Journal of the London Mathematical Society}
}

@article{Truss1989aa,
	Author = {J. K. Truss},
	Journal = {Journal of Algebra},
	Month = feb,
	Number = {2},
	Pages = {494--515},
	Publisher = {Elsevier {BV}},
	Title = {Infinite permutation groups {II}. {S}ubgroups of small index},
	Volume = {120},
	Year = {1989}}

@article{Evans1986aa,
	Author = {David M. Evans},
	Journal = {Bulletin of the London Mathematical Society},
	Number = {6},
	Pages = {587--590},
	Publisher = {Wiley},
	Title = {Subgroups of small Index in infinite General Linear Groups},
	Volume = {18},
	Year = {1986}}

@article{Behrisch2017aa,
	Author = {Behrisch, Mike and Truss, John K. and Vargas-Garc\'{\i}a, Edith},
	Fjournal = {Studia Logica. An International Journal for Symbolic Logic},
	Issn = {0039-3215},
	Journal = {Studia Logica},
	Mrclass = {08A35 (06A05 22F50)},
	Mrnumber = {3607630},
	Mrreviewer = {Keith A. Kearnes},
	Number = {1},
	Pages = {65--91},
	Title = {Reconstructing the topology on monoids and polymorphism clones of the rationals},
	Volume = {105},
	Year = {2017}}

@article{Bodirsky2017aa,
	Author = {Bodirsky, Manuel and Pinsker, Michael and Pongr\'{a}cz, Andr\'{a}s},
	Fjournal = {Trans. Amer. Math. Soc.},
	Issn = {0002-9947},
	Journal = {Trans. Amer. Math. Soc.},
	Mrclass = {03C05 (03C40 08A35 08A40 08A70 20B27 22A30)},
	Mrnumber = {3605985},
	Mrreviewer = {Keith A. Kearnes},
	Number = {5},
	Pages = {3707--3740},
	Title = {Reconstructing the topology of clones},
	Volume = {369},
	Year = {2017}}

@book{Mill2001aa,
	Author = {van Mill, Jan},
	Isbn = {0-444-50557-1},
	Mrclass = {57N20 (46T10 54C35 57N17 58D15)},
	Mrnumber = {1851014},
	Mrreviewer = {Taras Banakh},
	Pages = {xii+630},
	Publisher = {North-Holland Publishing Co., Amsterdam},
	Series = {North-Holland Mathematical Library},
	Title = {The infinite-dimensional topology of function spaces},
	Volume = {64},
	Year = {2001}}

@article{Pech:2018aa,
	Author = {Pech, Christian and Pech, Maja},
	Doi = {10.1007/s00012-018-0504-1},
	Isbn = {1420-8911},
	Journal = {Algebra universalis},
	Number = {2},
	Pages = {35},
	Title = {Polymorphism clones of homogeneous structures: gate coverings and automatic homeomorphicity},
	Ty = {JOUR},
	Url = {https://doi.org/10.1007/s00012-018-0504-1},
	Volume = {79},
	Year = {2018}
}

@book{givant2008introduction,
    AUTHOR = {Givant, Steven and Halmos, Paul},
     TITLE = {Introduction to {B}oolean algebras},
    SERIES = {Undergraduate Texts in Mathematics},
 PUBLISHER = {Springer, New York},
      YEAR = {2009},
     PAGES = {xiv+574},
      ISBN = {978-0-387-40293-2},
   MRCLASS = {06-01 (06E05)},
  MRNUMBER = {2466574},
       DOI = {10.1007/978-0-387-68436-9},
       URL = {https://doi-org.ezproxy.st-andrews.ac.uk/10.1007/978-0-387-68436-9},
}

@misc{elliott2020description,
      title={A Description of Aut(dVn) and Out(dVn) Using Transducers}, 
      author={Luke Elliott},
      year={2020},
      eprint={2009.05450},
      archivePrefix={arXiv},
      primaryClass={math.GR}
}

@book{kechris2012classical,
  title={Classical descriptive set theory},
  author={Kechris, Alexander},
  volume={156},
  year={2012},
  publisher={Springer Science \& Business Media}
}

@article{effrosvan2004note,
  title={A note on the {E}ffros Theorem},
  author={van Mill, Jan},
  journal={The American Mathematical Monthly},
  volume={111},
  number={9},
  pages={801--806},
  year={2004},
  publisher={Taylor \& Francis}
}

@incollection{munkres2016topology,
  title={Topology},
  author={Munkres, James R},
  year={2016},
  publisher={Pearson India Education Services Pvt. Ltd.}
}
\bibliographystyle{plain}

\end{document}